%% file: paper.tex
\documentclass{article}

\input{ryantibs}
\DeclareSymbolFont{bbold}{U}{bbold}{m}{n}
\DeclareSymbolFontAlphabet{\mathbm}{bbold}

\def\Id{\mathrm{Id}}
\def\S{\mathcal{S}}
\def\DS{\mathcal{DS}}
\def\NS{\mathcal{NS}}
\def\DNS{\mathcal{DNS}}
\def\TV{\mathrm{TV}}
\def\A{\mathbb{A}}
\def\B{\mathbb{B}}
\def\D{\mathbb{D}}
\def\H{\mathbb{H}}
\def\I{\mathbb{I}}
\def\F{\mathbb{F}}
\def\G{\mathbb{G}}
\def\K{\mathbb{K}}
\def\L{\mathbb{L}}
\def\M{\mathbb{M}}
\def\N{\mathbb{N}}
\def\Q{\mathbb{Q}}
\def\U{\mathbb{U}}
\def\V{\mathbb{V}}
\def\W{\mathbb{W}}
\def\Z{\mathbb{Z}}
\def\T{\mathsf{T}}
\def\cumsum{\mathrm{cumsum}}
\def\diff{\mathrm{diff}}
\def\rev{\mathrm{rev}}

\usepackage{caption}
\usepackage{nicefrac}
\usepackage{ragged2e}
\def\nfrac{\nicefrac}

\title{Divided Differences, Falling Factorials, and Discrete Splines \\
{\Large Another Look at Trend Filtering and Related Problems}}
\author{Ryan J.\ Tibshirani}
\date{}

\begin{document}
\maketitle
\tableofcontents
\newpage

\begin{abstract}
This paper reviews a class of univariate piecewise polynomial functions known as
{\it discrete splines}, which share properties analogous to the better-known
class of spline functions, but where continuity in derivatives is replaced by (a  
suitable notion of) continuity in {\it divided differences}. As it happens,
discrete splines bear connections to a wide array of developments in applied
mathematics and statistics, from divided differences and Newton interpolation
(dating back to over 300 years ago) to trend filtering (from the last 15 years).
We survey these connections, and contribute some new perspectives and new
results along the way.    
\end{abstract}

\section{Introduction}
\label{sec:intro}

Nonparametric regression is a fundamental problem in statistics, in which we
seek to flexibly estimate a smooth trend from data without relying on specific 
assumptions about its form or shape. The standard setup is to assume that data
comes from a model (often called the ``signal-plus-noise'' model):
$$
y_i = f_0(x_i) + \epsilon_i, \quad i=1,\ldots,n.
$$
Here, $f_0 : \cX \to \R$ is an unknown function to be estimated, referred to as 
the {\it regression function}; $x_i \in \cX$, $i=1,\ldots,n$ are {\it design
  points}, often (though not always) treated as nonrandom; $\epsilon_i \in \R$,
$i=1,\ldots,n$ are random errors, usually assumed to be i.i.d.\ (independent and
identically distributed) with zero mean; and $y_i \in \R$, $i=1,\ldots,n$ are
referred to as {\it response points}. Unlike in a {\it parametric} problem,
where we would assume $f_0$ takes a particular form (for example, a polynomial 
function) that would confine it to some finite-dimensional function space, in a 
{\it nonparametric} problem we make no such restriction, and instead assume $f_0$ 
satisfies some broader smoothness properties (for example, it has two bounded
derivatives) that give rise to an infinite-dimensional function space.

The modern nonparametric toolkit contains an impressive collection of diverse 
methods, based on ideas like kernels, splines, and wavelets, to name just a
few. Many estimators of interest in nonparametric regression can be formulated
as the solutions to optimization problems based on the observed data. At a high 
level, such optimization-based methods can be divided into two camps. The
first can be called the {\it continuous-time approach}, where we optimize over a 
function $f : \cX \to \R$ that balances some notion of goodness-of-fit (to the 
data) with another notion of smoothness. The second can be called the {\it 
  discrete-time approach}, where we optimize over function evaluations
$f(x_1),\ldots,f(x_n)$ at the design points, again to balance goodness-of-fit
with smoothness.\footnote{The use of the word ``time'' here is completely
  informal. In some applications, the input $x \in \cX$ might actually index
  time, and thus the names ``continuous-time'' and ``discrete-time'' would take
  on a direct meaning; but in general, they are only to be understood loosely, in 
  reference to the distinction between modeling an entire function, and modeling
  function evaluations, as in \eqref{eq:tv_denoising_cont} and
  \eqref{eq:tv_denoising}, respectively.}    

The main difference between these approaches lies in the optimization variable:
in the first it is a function $f$, and in the second it is a vector $\theta =
(f(x_1),\ldots,f(x_n)) \in \R^n$. Each perspective comes with its advantages.
The discrete-time approach is often much simpler, conceptually speaking, as it
often requires only a fairly basic level of mathematics in order to explain and
understand the formulation at hand. Consider, for example, a setting with $\cX
= [a,b]$ (the case of univariate design points), where we assume without a loss
of generality that $x_1 < x_2 < \cdots < x_n$, and we define an estimator by the
solution of the optimization problem: 
\begin{equation} 
\label{eq:tv_denoising}
\minimize_\theta \; \frac{1}{2} \sum_{i=1}^n (y_i - \theta_i)^2 + \lambda
\sum_{i=1}^{n-1} |\theta_i - \theta_{i+1}|. 
\end{equation}
In the above criterion, each $\theta_i$ plays the role of a function evaluation
$f(x_i)$; the first term measures the goodness-of-fit (via squared error loss)
of the evaluations to the responses; the second term measures the jumpiness of
the evaluations across neighboring design points, $\theta_i = f(x_i)$ and
$\theta_{i+1} = f(x_{i+1})$; and $\lambda \geq 0$ is a tuning parameter
determining the relative importance of the two terms for the overall
minimization, with a larger $\lambda$ translating into a higher importance on 
encouraging smoothness (mitigating jumpiness). 

Reasoning about the discrete-time problem \eqref{eq:tv_denoising} can be done
without appealing to sophisticated mathematics, both conceptually and formally. 
Arguably, this could be appropriate for an introductory course on nonparametric
statistical estimation. On the other hand, consider the estimator defined by the
solution of the optimization problem:\footnote{Here and throughout, we say ``the
  solution'' only for simplicity. Problem \eqref{eq:tv_denoising_cont}, and more
  generally problem \eqref{eq:trend_filter_cont}, need not admit unique
  solutions. The discrete-time problems \eqref{eq:tv_denoising} and 
  \eqref{eq:trend_filter_old} do, however, always admit unique solutions,
  because their criteria are strictly convex.}
\begin{equation} 
\label{eq:tv_denoising_cont}
\minimize_f \; \frac{1}{2} \sum_{i=1}^n \big(y_i - f(x_i)\big)^2 +  
\lambda \, \TV(f).
\end{equation}
The minimization is taken over functions (for which the criterion is
well-defined and finite); the first term measures the goodness-of-fit of the
evaluations to the response points, as before; the second term measures the
jumpiness of $f$, now using the total variation operator $\TV(\cdot)$ acting on
univariate functions; and $\lambda \geq 0$ is again a tuning parameter.
Relative to \eqref{eq:tv_denoising}, the continuous-time problem
\eqref{eq:tv_denoising_cont} requires an appreciably higher level of
mathematical sophistication, in order to develop any conceptual or formal
understanding. However, problem \eqref{eq:tv_denoising_cont} does have the
distinct advantage of delivering a {\it function} as its solution, call it
\smash{$\hf$}: this allows us to predict the value of the response at any point
$x \in [a,b]$, via \smash{$\hf(x)$}.

From the solution in \eqref{eq:tv_denoising}, call it \smash{$\htheta$}, it is
not immediately clear how to predict the response value at an arbitrary point $x
\in [a,b]$. This is about choosing the ``right'' method for interpolating (or
extrapolating, on $[a,x_1) \cup (x_n,b]$) a set of $n$ function evaluations. To
be fair, in the particular case of problem \eqref{eq:tv_denoising}, its solution
is generically piecewise-constant over its components
\smash{$\htheta_i$}, $i=1,\ldots,n$, which suggests a natural interpolant. In
general, however, the task of interpolating the estimated function evaluations
from a discrete-time optimization problem into an entire estimated function is
far from clear-cut. Likely for this reason, the statistics literature---which
places a strong emphasis, both applied and theoretical, on prediction at a new
points $x \in [a,b]$---has focused primarily on the continuous-time approach  
to optimization-based nonparametric regression. While the discrete-time
approach is popular in signal processing and econometrics, the lines of
work on discrete- and continuous-time smoothing seem to have evolved
mostly in parallel, with limited interplay.

The optimization problems in \eqref{eq:tv_denoising},
\eqref{eq:tv_denoising_cont} are not arbitrary examples of the discrete- and
continuous-time perspectives, respectively; they are in fact deeply related to
the main points of study in this paper. Interestingly, problems
\eqref{eq:tv_denoising}, \eqref{eq:tv_denoising_cont} are equivalent in the
sense that their solutions, denoted \smash{$\htheta,\hf$} respectively, satisfy 
\smash{$\htheta_i = \hf(x_i)$}, $i=1,\ldots,n$. In other words, the solution in
\eqref{eq:tv_denoising} reproduces the evaluations of the solution in 
\eqref{eq:tv_denoising_cont} at the design points. The common estimator here is 
well-known, called {\it total variation denoising} \citep{rudin1992nonlinear}
in some parts of applied mathematics, and the {\it fused lasso}
\citep{tibshirani2005sparsity} in statistics.   

The equivalence between \eqref{eq:tv_denoising}, \eqref{eq:tv_denoising_cont} is
a special case of a more general equivalence between classes of discrete- and
continuous-time optimization problems, in which the differences $\theta_i -
\theta_{i+1}$ in \eqref{eq:tv_denoising} are replaced by higher-order discrete
derivatives (based on divided differences), and $\TV(f)$ in
\eqref{eq:tv_denoising_cont} is replaced by the total variation of a suitable
derivative of $f$. The key mathematical object powering this connection is a
linear space of univariate piecewise polynomials called {\it discrete splines},
which is the central focus of this paper. We dive into the details, and explain
the importance of such equivalences, in the next subsection.

\subsection{Motivation} 

The jumping-off point for the developments that follow is a generalization of
the discrete-time total variation denoising problem \eqref{eq:tv_denoising},
proposed independently by \citet{steidl2006splines,kim2009trend} (though similar
ideas were around earlier, see Section \ref{sec:trend_filter_review}), defined
for an integer $k \geq 0$ by:
\begin{equation}
\label{eq:trend_filter_old}
\minimize_\theta \; \frac{1}{2} \|y-\theta\|_2^2 + \lambda \|\C^{k+1}_n 
\theta\|_1. 
\end{equation}
Here, $\lambda \geq 0$ is a tuning parameter, $y=(y_1,\ldots,y_n) \in \R^n$ is 
the vector of response points, \smash{$\C^{k+1}_n \in \R^{(n-k-1) \times
n}$} is an explicit banded matrix that corresponds to a weighted $(k+1)$st 
order discrete derivative operator (this can be defined in terms of the
$(k+1)$st order divided difference coefficients across the design points; see
the construction in \eqref{eq:diff_mat}--\eqref{eq:discrete_deriv_mat_old}), and 
$\|\cdot\|_2$ and $\|\cdot\|_1$ are the standard $\ell_2$ and $\ell_1$ norms
acting on vectors.    
 
The estimator defined by solving problem \eqref{eq:trend_filter_old} is known as
{\it $k$th order trend filtering}. A important aspect to highlight right away
is computational: since \smash{$\C^{k+1}_n$} is a banded matrix (with bandwidth 
$k+2$), the trend filtering problem \eqref{eq:trend_filter_old} can be solved
efficiently using various convex optimization techniques that take advantage of
this structure (see, for example,
\citet{kim2009trend,arnold2016efficient,ramdas2016fast}). The original papers
on trend filtering \citet{steidl2006splines,kim2009trend} considered the special
case of evenly-spaced design points, $x_{i+1}-x_i=v>0$, $i=1,\ldots,n-1$, where
the penalty term in \eqref{eq:trend_filter_old} takes a perhaps more familiar
form:
\begin{equation}
\label{eq:trend_filter_penalties}
\|\C^{k+1}_n \theta\|_1 = 
\begin{cases}
\displaystyle
\frac{1}{v} \sum_{i=1}^{n-1} |\theta_i - \theta_{i+1}| & \text{if $k=0$} \\
\displaystyle
\frac{1}{v^2} \sum_{i=1}^{n-2} |\theta_i - 2\theta_{i+1} + \theta_{i+2}| & 
\text{if $k=1$} \\ 
\displaystyle
\frac{1}{v^3} \sum_{i=1}^{n-3} |\theta_i - 3\theta_{i+1} + 3\theta_{i+2} - 
\theta_{i+3}| & \text{if $k=2$},
\end{cases}
\end{equation}
and so forth, where for a general $k \geq 0$, the penalty is a $1/v^{k+1}$ times a  
sum of absolute $(k+1)$st forward differences. (The factor of $1/v^{k+1}$ can
always be abosrbed into the tuning parameter $\lambda$; and so we can see 
that \eqref{eq:trend_filter_old} reduces to \eqref{eq:tv_denoising} for $k=0$,
modulo a rescaling of $\lambda$). The extension of trend filtering to arbitrary
(unevenly-spaced) design points is due to \citet{tibshirani2014adaptive}. The
continuous-time (functional) perspective on trend filtering is also due to
\citet{tibshirani2014adaptive}, which we describe next. 

\paragraph{Connections to continuous-time.}

To motivate the continuous-time view, consider \smash{$\C^{k+1}_n\theta$}, the
vector of (weighted) $(k+1)$st discrete derivatives of $\theta$ across the
design points: since discrete differentiation is based on iterated differencing,
we can equivalently interpret \smash{$\C^{k+1}_n\theta$} as a vector of {\it
  differences} of $k$th discrete derivatives of $\theta$ at adjacent design 
points. By the sparsity-inducing property of the $\ell_1$ norm, the penalty in
problem \eqref{eq:trend_filter_old} thus drives the $k$th discrete derivatives
of $\theta$ to be equal at adjacent design points, and the trend filtering
solution \smash{$\htheta$} generically takes on the structure of a $k$th degree
piecewise polynomial (as its $k$th discrete derivative will be piecewise
constant), with adaptively-chosen knots (points at which the $k$th discrete
derivative changes). This intuition is readily confirmed by empirical examples;
see Figure \ref{fig:tf}.

\begin{figure}[tb]
\centering
\includegraphics[width=0.325\textwidth]{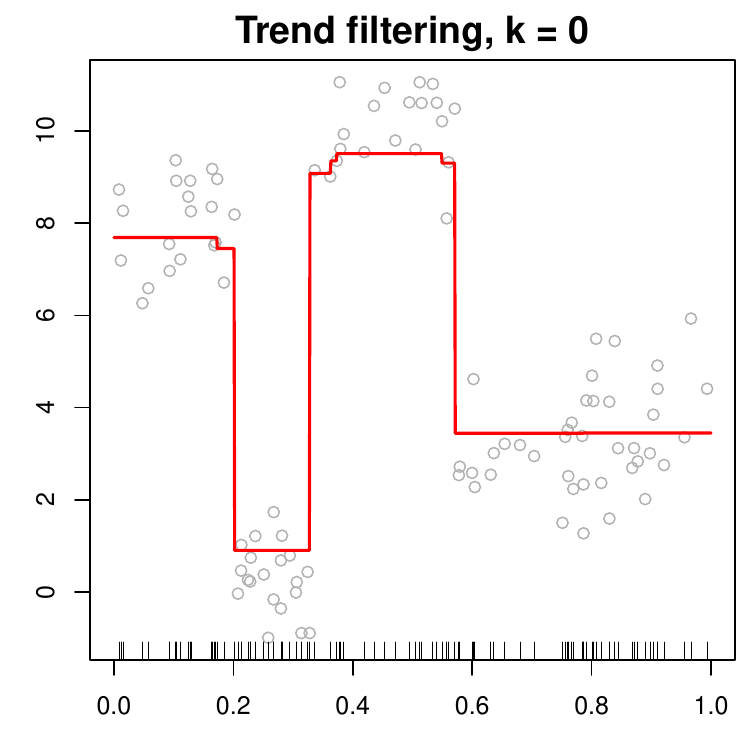}
\includegraphics[width=0.325\textwidth]{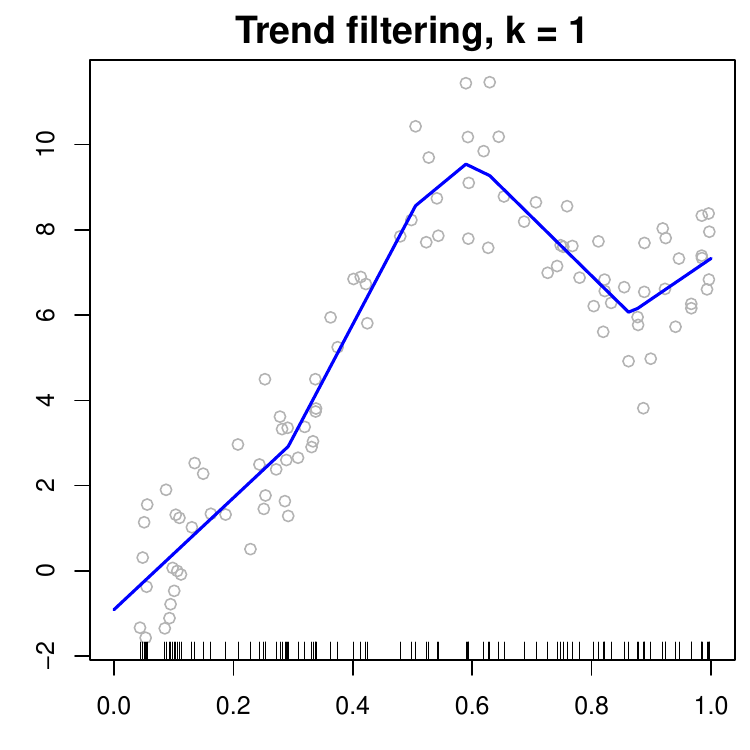}
\includegraphics[width=0.325\textwidth]{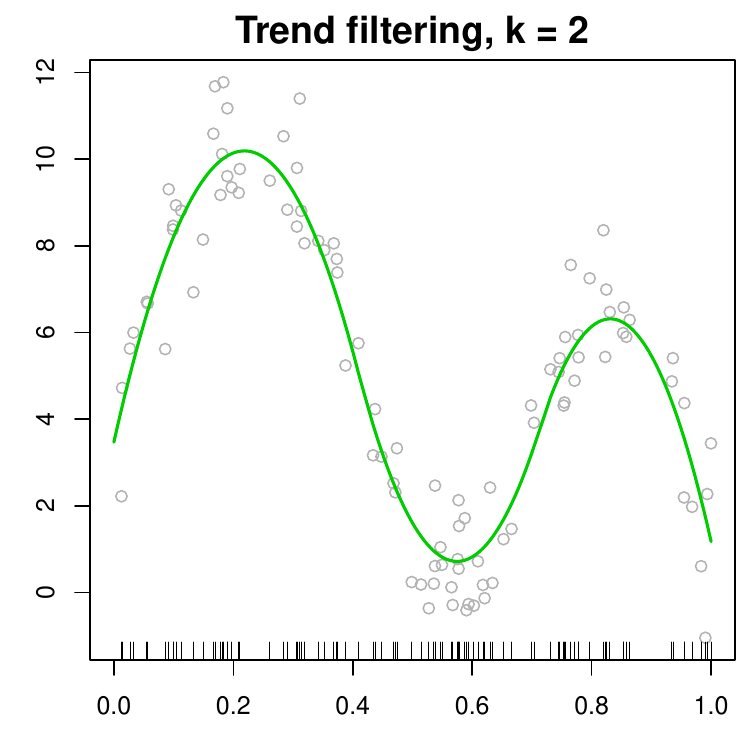}
\caption{\small (Adapted from \citet{tibshirani2014adaptive}.)  Example trend
  filtering estimates for $k=0$, $k=1$, and $k=2$, exhibiting piecewise
  constant, piecewise linear, and piecewise quadratic behavior, respectively. In
  each panel, the $n=100$ design points are marked by ticks on the horizontal
  axis (note that they are not evenly-spaced).}   
\label{fig:tf}
\end{figure}

These ideas were formalized in \citet{tibshirani2014adaptive}, and then
developed further in \citet{wang2014falling}. These papers introduced what were
called {\it $k$th degree falling factorial basis}, a set of functions defined
as 
\begin{equation}
\label{eq:ffb}
\begin{aligned}
h^k_j(x) &= \frac{1}{(j-1)!} \prod_{\ell=1}^{j-1}(x-x_\ell), 
\quad j=1,\ldots,k+1, \\
h^k_j(x) &= \frac{1}{k!} \prod_{\ell=j-k}^{j-1} (x-x_\ell) \cdot  
1\{x > x_{j-1}\}, \quad j=k+2,\ldots,n. 
\end{aligned}
\end{equation}
(Note that this basis depends on the design points $x_1,\ldots,x_n$, though this
is notationally suppressed.)  The functions in \eqref{eq:ffb} are $k$th degree
piecewise polynomials, with knots at $x_{k+1},\ldots,x_{n-1}$. Here and
throughout, we interpret the empty product to be equal to 1, for convenience
(that is, \smash{$\prod_{i=1}^0 a_i = 1$}). Note the similarity of the above
basis and the standard truncated power basis for splines, with knots at
$x_{k+1},\ldots,x_{n-1}$ (see \eqref{eq:tpb}); in fact, when $k=0$ or $k=1$, the
two bases are equal, and the above falling factorial functions are exactly
splines; but when $k \geq 2$, this is no longer true---the above falling
factorial functions are piecewise polynomials with {\it discontinuities} in
their derivatives of orders $1,\ldots,k-1$ (see \eqref{eq:ffb_deriv},
\eqref{eq:ffb_deriv_lim}), and thus span a different space than that of $k$th
degree splines.

The key result connecting \eqref{eq:ffb} and \eqref{eq:trend_filter_old} was
given in Lemma 5 of \citet{tibshirani2014adaptive} (see also Lemma 2 of 
\citet{wang2014falling}), and can be explained as follows. For each $\theta
\in \R^n$, there is a function in the span of the falling factorial basis, 
\smash{$f \in \spa\{h^k_1,\ldots,h^k_n\}$}, with two properties: first,
$f$ interpolates each $\theta_i$ at $x_i$, which we write as
$\theta=f(x_{1:n})$, where $f(x_{1:n})=(f(x_1),\ldots,f(x_n)) \in \R^n$ 
denotes the vector of evaluations of $f$ at the design points; and second
\begin{equation}
\label{eq:ffb_tv_mat_old}
\TV(D^k f) =  \big\|\C^{k+1}_n f(x_{1:n}) \big\|_1.
\end{equation}
On the right-hand side is the trend filtering penalty, which, recall, we can
interpret as a sum of absolute differences of $k$th discrete derivatives of $f$
over the design points, and therefore as a type of total variation penalty on
the $k$th discrete derivative. On the left-hand side above, we denote by $D^k
f$ the $k$th derivative of $f$ (which we take to mean the $k$th left
derivative when this does not exist), and by $\TV(\cdot)$ the usual total
variation operator on functions. Hence, taking total variation of the $k$th
derivative as our smoothness measure, the property in \eqref{eq:ffb_tv_mat_old}
says that the interpolant $f$ of $\theta$ is {\it exactly as smooth} in
continuous-time as $\theta$ is in discrete-time.

Reflecting on this result, the first property---that $f$ interpolates $\theta_i$
at $x_i$, for $i=1,\ldots,n$---is of course not special in it of itself. Any
rich enough function class, of dimension at least $n$, will admit such a
function. However, paired with the second property \eqref{eq:ffb_tv_mat_old},
the result becomes interesting, and even somewhat surprising. Said differently,
any function $f$ lying in the span of the $k$th degree falling factorial basis
has the property that its discretization to the design points is {\it lossless}
with respect to the total variation smoothness functional $\TV(D^k f)$: this
information is exactly preserved by $\theta=f(x_{1:n})$. Denoting by
\smash{$\cH^k_n=\spa\{h^k_1,\ldots,h^k_n\}$} the span of falling factorial
functions, we thus see that the trend filtering problem
\eqref{eq:trend_filter_old} is equivalent to the variational problem:
\begin{equation}
\label{eq:trend_filter_cont}
\minimize_{f \in \cH^k_n} \; \frac{1}{2} \sum_{i=1}^n \big(y_i - f(x_i)\big)^2 +  
\lambda \, \TV(D^k f),
\end{equation}
in the sense that at the solutions \smash{$\htheta,\hf$} in problems
\eqref{eq:trend_filter_old}, \eqref{eq:trend_filter_cont}, respectively, we have
\smash{$\htheta=\hf(x_{1:n})$}. Moreover, it turns out that forming
\smash{$\hf$} from \smash{$\htheta$} is straightforward: starting with the
falling factorial basis expansion \smash{$\hf=\sum_{j=1}^n \halpha_j h^k_j$},
and then writing the coefficient vector in block form \smash{$\halpha =
(\hat{a}, \hat{b}) \in \R^{k+1} \times \R^{n-k-1}$}, the piecewise polynomial
basis coefficients are given by \smash{$\hat{b} = \C^{k+1}_n \htheta$}, and the
polynomial basis coefficients \smash{$\hat{a}$} can also be expressed simply in
terms of lower-order discrete derivatives. This shows that \smash{$\hf$} is a
$k$th degree piecewise polynomial, with knots occurring at the nonzeros of
\smash{$\C^{k+1}_n \htheta$}, that is, at changes in the $k$th discrete
derivative of \smash{$\htheta$}, formally justifying the intuition about the
structure of \smash{$\htheta$} given above.

\paragraph{Reflections on the equivalence.}

One might say that the developments outlined above bring trend filtering closer
to the ``statistical mainstream'': we move from being able to estimate the
values of the regression function $f_0$ at the design points $x_1,\ldots,x_n$ to
being able to estimate $f_0$ itself. This has several uses:
practical---we can use the interpolant \smash{$\hf$} to estimate $f_0(x)$ at
unseen values of $x$; conceptual---we can better understand what kinds of 
``shapes'' trend filtering is inclined to produce, via the representation in
terms of falling factorial functions; and theoretical---we can tie
\eqref{eq:trend_filter_cont} to an unconstrained variational problem, where we
minimize the same criterion over {\it all} functions $f$ (for which the
criterion is well-defined and finite): 
\begin{equation}
\label{eq:local_spline}
\minimize_f \; \frac{1}{2} \sum_{i=1}^n \big(y_i - f(x_i)\big)^2 + \lambda 
\, \TV(D^k f). 
\end{equation}
This minimization is in general computationally difficult, but its solution,
called the {\it locally adaptive regression spline} estimator
\citep{mammen1997locally} has favorable theoretical properties, in terms of its
rate of estimation of $f_0$ (see Section \ref{sec:local_spline_review} for a
review). By showing that the falling factorial functions are ``close'' to
certain splines, \citet{tibshirani2014adaptive,wang2014falling} showed that the
solution in \eqref{eq:trend_filter_cont} is ``close'' to that in
\eqref{eq:local_spline}, and thus trend filtering inherits the favorable
estimation guarantees of the locally adaptive regression spline (which is
important because trend filtering is computationally easier; for more, see
Sections \ref{sec:local_spline_review} and \ref{sec:trend_filter_review}).

The critical device in all of this were the falling factorial basis functions
\eqref{eq:ffb}, which provide the bridge between the discrete and continuous
worlds. This now brings us to the motivation for the current paper. One has to
wonder: did we somehow get ``lucky'' with trend filtering and this basis?  Do
the falling factorial functions have other properties aside from
\eqref{eq:ffb_tv_mat_old}, that is, aside from equating
\eqref{eq:trend_filter_old} and \eqref{eq:trend_filter_cont}?  At the time of 
writing \citet{tibshirani2014adaptive,wang2014falling} (and even in subsequent
work on trend filtering), we were not fully aware of the relationship of the
falling factorial functions and what appears to be fairly classical work in
numerical analysis. First and foremost: 
\begin{quote}
The span \smash{$\cH^k_n = \spa\{h^k_1,\ldots,h^k_n\}$} of the $k$th degree
falling factorial basis functions is a special space of piecewise polynomials
known as {\it $k$th degree discrete splines}.
\end{quote}
Discrete splines have been studied since the early 1970s by applied
mathematicians, beginning with
\citet{mangasarian1971discrete,mangasarian1973best}. 
The current paper recasts some of our previous work on trend filtering to better
connect it to the discrete spline literature, reviews some relevant existing
results on discrete splines and discusses the implications for trend filtering
and related problems, and lastly, contributes some new results and perspectives
on discrete splines.

\subsection{Summary}
\label{sec:summary}

An outline and summary of this paper is as follows. 

\begin{itemize}
\item In Section \ref{sec:background}, we provide relevant background and
  historical remarks. 

\item In Section \ref{sec:fall_fact}, we give a new perspective on how to
  construct the falling factorial basis ``from scratch''. We start by defining
  a natural discrete derivative operator and its inverse, a discrete
  integrator. We then show that the falling factorial basis functions are given
  by $k$th order discrete integration of appropriate step functions (Theorem
  \ref{thm:ffb_discrete_integ}). 

\item In Section \ref{sec:smoothness}, we verify that the span of the falling
  factorial basis is indeed a space of discrete splines (Lemma
  \ref{lem:ffb_span}), and establish that functions in this span satisfy a key 
  matching derivatives property: their $k$th discrete derivative matches their
  $k$th derivative everywhere, and moreover, they are the {\it only} $k$th
  degree piecewise polynomials with this property (Corollary
  \ref{cor:deriv_match}). 

\item In Section \ref{sec:dual_basis}, we give a dual basis to the falling
  factorial basis, based on evaluations of discrete derivatives. As a primary
  use case, we show how to use such a dual basis to perform efficient
  interpolation in the falling factorial basis, which generalizes Newton's
  divided difference interpolation formula (Theorem \ref{thm:ffb_interp}).  
  We also show that this interpolation formula can be recast in an implicit
  manner, which reveals that interpolation using discrete splines can be done in 
  {\it constant-time} (Corollary \ref{cor:ffb_interp_implicit}), and further,
  discrete splines are uniquely determined by this implicit result: they are the
  {\it only} functions that satisfy such an implicit interpolation formula
  (Corollary \ref{cor:ffb_interp_implicit_conv}).     

\item In Section \ref{sec:matrix_comp}, we present a matrix-centric view of the
  results given in previous sections, drawing connections to the way some
  related results have been presented in past papers. We review specialized
  methods for fast matrix operations with discrete splines from
  \citet{wang2014falling}. 

\item In Section \ref{sec:discrete_bs}, we present a new discrete B-spline basis 
  for discrete splines (it is new for arbitrary designs, and our construction
  here is a departure from the standard one): we first define these basis
  functions as discrete objects, by fixing their values at the design points,
  and we then define them as continuum functions, by interpolating these values 
  within the space of discrete splines, using the implicit interpolation view
  (Lemma \ref{lem:discrete_nbs_supp}). We show how this discrete B-spline basis
  can be easily modified to provide a basis for discrete natural splines (Lemma 
  \ref{lem:discrete_natural_nbs}). 

\item In Section \ref{sec:sparse_knots}, we demonstrate how the previous results 
  and developments can be ported over to the case where the knot set that
  defines the space of discrete splines is an arbitrary (potentially sparse)
  subset of the design points. An important find here is that the discrete
  B-spline basis provides a much more stable (better-conditioned) basis for
  solving least squares problems involving discrete splines.   

\item In Section \ref{sec:representation}, we present two representation results 
  for discrete splines. First, we review a result from
  \citet{tibshirani2014adaptive,wang2014falling} on representing the total
  variation functional $\TV(D^k f)$ for a $k$th degree discrete spline $f$ in 
  terms of a sum of absolute differences of its $k$th discrete derivatives
  (Theorem \ref{thm:ffb_tv}). (Recall that we translated this in
  \eqref{eq:ffb_tv_mat_old}.) Second, we establish a new result on
  representing the $L_2$-Sobolev functional \smash{$\int_a^b (D^m f)(x)^2 \,
    dx$} for a $(2m-1)$st degree discrete spline $f$ in terms of a quadratic
  form of its $m$th discrete derivatives (Theorem \ref{thm:ffb_sobolev}).  

\item In Section \ref{sec:approximation}, we derive some simple (crude) 
  approximation bounds for discrete splines, over bounded variation
  spaces.   

\item In Section \ref{sec:trend_filter}, we revisit trend filtering. We discuss 
  some potential computational improvements, stemming from the development of 
  discrete B-splines and their stability properties. We also show that the  
  optimization domain in trend filtering can be further restricted to the space
  of discrete natural splines by adding simple linear constraints to the
  original problem, and that this modification can lead to better boundary
  behavior.   

\item In Section \ref{sec:bw_filter}, we revisit Bohlmann-Whittaker (BW)
  filtering. In the case of arbitrary design points, we propose a simple
  modification of the BW filter using a weighted penalty, which for $m=1$
  reduces to the linear smoothing spline. For $m=2$, we derive a
  deterministic bound on the $\ell_2$ distance between the weighted cubic BW 
  filter and the cubic smoothing spline (Theorem
  \ref{thm:ss_bw_bound}). We use this, in combination with classical
  nonparametric regression theory for smoothing splines, to prove that the
  weighted BW filter attains minimax optimal estimation rates over the
  appropriate $L_2$-Sobolev classes (Corollary \ref{cor:ss_bw_bound}).
\end{itemize}

Most proofs are deferred to Appendix \ref{app:proofs}. Other relevant technical
details (background and otherwise) are deferred to Appendices \ref{app:bs_dbs}
and \ref{app:fast_mult}. 

\subsection{Notation}

Here is an overview of some general notation used in this paper. For integers $a 
\leq b$, we use $z_{a:b} = \{z_a, z_{a+1}, \ldots, z_b\}$. For a set $C$, we use
$1_C$ for the indicator function of $C$, that is, $1_C(x)=1\{x \in C\}$. We
write $f|_C$ for the restriction of a function $f$ to $C$. We use $D$ for the
differentiation operator, and $I$ for the integration operator: acting on
functions $f$ on $[a,b]$, we take $If$ to itself be a function on $[a,b]$,
defined by 
$$
(If)(x) = \int_a^x f(t) \, dt.
$$  
For a nonnegative integer $k$, we use $D^k$ and $I^k$ to denote $k$ repeated
applications (that is, $k$ times composition) of the differentiation and
integration operators, respectively. In general, when the derivative of a
function $f$ does not exist, we interpret $Df$ to mean the {\it left}
derivative, assuming the latter exists, and the same with $D^k f$.

An important note: we refer to a $k$th degree piecewise polynomial that has
$k-1$ continuous derivatives as a spline of {\it degree} $k$, whereas much of
the classical literature refers to this as a spline of {\it order} $k+1$; we
specifically avoid the use of the word ``order'' when it comes to such functions
or functions spaces, to avoid confusion.

Finally, throughout, we use ``blackboard'' fonts for matrices
(such as $\F,\G$, etc.), in order to easily distinguish them from operators that
act on functions (for which we use $F,G$, etc.). The only exceptions are that
we reserve $\R$ to denote the set of real numbers and $\E$ to denote the 
expectation opterator.

For a more detailed summary of notation, and discrete-continuum analogies or
equivalences, see Appendix \ref{app:notation}.

\section{Background}
\label{sec:background}

We provide background on various topics that will play important roles in the
remainder of the paper. Of course, we do not intend to give a comprehensive
review of any of the subjects covered, just the basic elements needed for what
follows. We also use this space to make historical remarks and discuss related
work.

\subsection{Divided differences}

Divided differences have a very old, rich history in mathematics, and are
usually attributed to Newton (due to 
\citet{newton1687philosophiae,newton1711methodus}). They also serve a one 
of the primary building blocks in classical numerical analysis (for example, see 
\citet{whittaker1924calculus}). For a beautiful review of divided differences,
their properties, and connections, see \citet{deboor2005divided}. Given a
univariate function $f$, the divided difference of $f$ at distinct points   
$z_1,z_2$ is defined by  
$$
f[z_1,z_2] =  \frac{f(z_2)-f(z_1)}{z_2-z_1},
$$
and more generally, for an integer $k \geq 1$, the $k$th order divided
difference at distinct $z_1,\ldots,z_{k+1}$ is defined by 
$$
f[z_1,\ldots,z_{k+1}] = \frac{f[z_2,\ldots,z_{k+1}] -
f[z_1,\ldots,z_k]}{z_{k+1}-z_1}.
$$
(For this to reduce to the definition in the previous display, when $k=1$, we
take by convention $f[z]=f(z)$.) We refer to the points $z_1,\ldots,z_{k+1}$
used to define the divided difference above as {\it centers}. Note that these
centers do not need to be in sorted order for this definition to make sense,
and the definition of a divided difference is invariant to the ordering of
centers: \smash{$f[z_1,\ldots,z_{k+1}]=f[z_{\sigma(1)},\ldots,z_{\sigma(k+1)}]$}
for any permutation $\sigma$ acting on $\{1,\ldots,k+1\}$. (We also note that
requiring the centers to be distinct is not actually necessary, but we will
maintain this assumption for simplicity; for a more general definition that
allows for repeated centers, see, for example, Definition 2.49 in
\citet{schumaker2007spline}.)      

A notable special case is when the centers are evenly-spaced, say, $z+iv$,
$i=0,\ldots,k$, for some spacing $v>0$, in which case the divided difference
becomes a (scaled) {\it forward difference}, or equivalently a (scaled) {\it
  backward difference}, 
$$
k! \cdot f[z,\ldots,z+kv] = \frac{1}{v^k} (F^k_v f)(z) = \frac{1}{v^k}  
(B^k_v f)(z+kv),
$$
where we use \smash{$F^k_v,B^k_v$} to denote the $k$th order forward and 
backward difference operators, respectively; to be explicit, we recall that
\smash{$(F^k_v f)(z) = \sum_{i=0}^k (-1)^{k-i} {k \choose i} f(z+iv)$}. 

\paragraph{Linear combination formulation.}  
 
It is not hard to see that divided differences are linear combinations of
function evaluations. A simple calculation reveals the exact form of the
coefficients in this linear combination, for example, 
\begin{align*}
f[z_1,z_2,z_3] &= \frac{f[z_1,z_2]}{z_1-z_3}
+\frac{f[z_2,z_3]}{z_3-z_1} \\
&= \frac{f(z_1)}{(z_1-z_2)(z_1-z_3)} + 
\frac{f(z_2)}{(z_2-z_1)(z_1-z_3)} + 
\frac{f(z_2)}{(z_2-z_3)(z_3-z_1)} +
\frac{f(z_3)}{(z_3-z_2)(z_3-z_1)} \\ 
&= \frac{f(z_1)}{(z_1-z_2)(z_1-z_3)} +
\frac{f(z_2)}{(z_2-z_1)(z_2-z_3)} +
\frac{f(z_3)}{(z_3-z_2)(z_3-z_1)}.
\end{align*}
By an inductive argument (whose inductive step is similar to the calculation
above), we may also write for a general order $k \geq 1$,
\begin{equation}
\label{eq:divided_diff_linear}
f[z_1,\ldots,z_{k+1}] = \sum_{i=1}^{k+1} \frac{f(z_i)}
{\prod_{j \in \{1,\ldots,k+1\} \setminus \{i\}} (z_i-z_j)}. 
\end{equation}
This expression is worth noting because it is completely explicit, but it is
not often used, and the recursive formulation given previously is the more
common view of divided differences.

\paragraph{Newton interpolation.}  

For distinct points $t_{1:r}=\{t_1,\ldots,t_r\}$, we denote the {\it Newton
  polynomial} based on $t_{1:r}$ by   
\begin{equation}
\label{eq:newton_poly}
\eta(x; t_{1:r}) = \prod_{j=1}^r (x-t_j).
\end{equation}
Here, when $r=0$, we set $t_{1:0}=\emptyset$ and $\eta(x;t_{1:0})=1$ for 
notational convenience. It is important to note that the pure polynomial
functions in the falling factorial basis, given in the first line of
\eqref{eq:ffb}, are simply Newton polynomials, and the piecewise polynomial
functions, given in the second line of \eqref{eq:ffb}, are {\it truncated}
Newton polynomials:
\begin{gather*}
h^k_j(x) = \frac{1}{(j-1)!} \eta(x; x_{1:j}), \quad j=1,\ldots,k+1, \\ 
h^k_j(x) = \frac{1}{k!} \eta(x; x_{(j-k):(j-1)}) \cdot 1\{x > x_{j-1}\}, 
\quad j=k+2,\ldots,n. 
\end{gather*}
In this light, it would also be appropriate to call the basis in \eqref{eq:ffb}
the {\it truncated Newton polynomial basis}, but we stick to the name falling
factorial basis for consistency with our earlier work (and Chapter 8.5 of 
\citet{schumaker2007spline}). 

Interestingly, Newton polynomials and divided differences are closely connected,
via {\it Newton's divided difference interpolation formula} (see, for example,
Proposition 7 in \citet{deboor2005divided}), which says that for a polynomial
$p$ of degree $k$, and any centers $t_1,\ldots,t_{k+1}$,
\begin{equation}
\label{eq:newton_interp}
p(x) = \sum_{j=1}^{k+1} p[t_1,\ldots,t_j] \cdot \eta(x; t_{1:(j-1)}).
\end{equation}
One of our main developments later, in Theorem \ref{thm:ffb_interp}, may be seen
as extending \eqref{eq:newton_interp} to interpolation with truncated Newton
polynomials (that is, with the falling factorial basis). In particular,
compare \eqref{eq:newton_interp} and \eqref{eq:ffb_interp_explicit}. 

An important fact about the representation in \eqref{eq:newton_interp} is that
it is unique (meaning, any $k$th degree polynomial can only be written as a
linear combination of Newton polynomials in one particular way, which is given
by \eqref{eq:newton_interp}). This property has the following implication for 
divided differences of Newton polynomials (that we will use extensively in
later parts of this paper): for any integer $r \geq 0$, and any centers
$t_1,\ldots,t_j$,    
\begin{equation}
\label{eq:newton_poly_divided_diff}
\eta(\cdot; t_{1:r}) [t_1,\ldots,t_j] = 
\begin{cases}
1 & \text{if $j = r+1$} \\
0 & \text{otherwise}.
\end{cases}
\end{equation}
The result is clear when $j = r+1$ and $j > r+1$ (in these cases, it is a
statement about a $j$th order divided difference of a polynomial of degree at
most $j$, for example, see \eqref{eq:deriv_match_poly}). However, it is perhaps
less obvious for $j < r+1$ (in this case it is a statement about a $j$th order
divided difference of a polynomial of degree greater than $j$).

\subsection{Splines}

Splines play a central role in numerical analysis, approximation theory, and
nonparametric statistics. The ``father'' of spline theory is widely considered
to be Schoenberg (due to 
\citet{schoenberg1946contributions1,schoenberg1946contributions2}, where
Schoenberg also introduces the terminology ``spline function''). It should be
noted that in the early 1900s, there were many papers written about splines
(without using this name), and piecewise polynomial interpolation, more
generally; for a survey of this work, see \citet{greville1944general}. For two
wonderful books on splines, see
\citet{deboor1978practical,schumaker2007spline}. 
We will draw on the latter book extensively throughout this paper.

In simple terms, a spline is a piecewise polynomial having continuous
derivatives of all orders lower than the degree of the polynomial. We can make
this definition more precise as follows.

\begin{definition}
\label{def:spline}
For an integer $k \geq 0$, and knots $a=t_0 < t_1 < \cdots < t_r < t_{r+1}=b$,
we define the space of {\it $k$th degree splines} on $[a,b]$ with knots
$t_{1:r}$, denoted $\S^k(t_{1:r}, [a,b])$, to contain all functions $f$ on
$[a,b]$ such that    
\begin{equation}
\label{eq:spline}
\begin{gathered}
\text{for each $i=0,\ldots,r$, there is a $k$th degree polynomial $p_i$
  such that $f|_{I_i}=p_i|_{I_i}$, and} \\
\text{for each $i=1,\ldots,r$, it holds that $(D^\ell p_{i-1})(t_i) = (D^\ell
  p_i)(t_i)$, $\ell=0,\ldots,k-1$}, 
\end{gathered}
\end{equation}
where $I_0=[t_0,t_1]$ and $I_i=(t_i,t_{i+1}]$, $i=1,\ldots,r$. 
\end{definition}

We write Definition \ref{def:spline} in this particular way because it makes it
easy to compare the definition of discrete splines in Definition
\ref{def:discrete_spline_even} (and in Definition \ref{def:discrete_spline} for
the case of arbitrary design points). The simplest basis for the space
$\S^k(t_{1:r}, [a,b])$ is the {\it $k$th degree truncated power basis}, defined
by  
\begin{equation}
\begin{gathered}
\label{eq:tpb}
g^k_j(x) = \frac{1}{(j-1)!} x^j, \quad j=1,\ldots,k+1, \\
g^k_{j+k+1}(x) = \frac{1}{k!} (x-t_j)^k_+, \quad j=1,\ldots,r,
\end{gathered}
\end{equation}
where $x_+=\max\{x,0\}$. When $k=0$, we interpret \smash{$(x-t)^0_+ = 
  1\{x >  t\}$}; this choice (strict versus nonstrict inequality) is arbitrary,
but convenient, and consistent with our choice for the falling factorial basis
in \eqref{eq:ffb}. 

An alternative basis for splines, which has local support and is therefore
highly computationally appealing, is given by the {\it B-spline} basis. In
fact, most authors view B-splines as {\it the} basis for splines---not only for
computational reasons, but also because building splines out of linear
combinations of B-splines makes so many of their important properties
transparent. To keep this background section (relatively) short, we defer
discussion of B-splines until Appendix \ref{app:bs}.

\subsection{Discrete splines}

Discrete splines were introduced by
\citet{mangasarian1971discrete,mangasarian1973best}, then further developed by
\citet{schumaker1973constructive,lyche1975discrete,deboor1976splines}, among
others. As far as we know, the most comprehensive summary of discrete splines 
and their properties appears to be Chapter 8.5 of \citet{schumaker2007spline}. 

In words, a discrete spline is similar to a spline, except in the required
smoothness conditions, forward differences are used instead of
derivatives. This can be made precise as follows.

\begin{definition}
\label{def:discrete_spline_even}
For an integer $k \geq 0$, design points $[a,b]_v=\{a,a+v,\ldots,b\}$ with
$v>0$ and $b=a+Nv$, and knots $a=t_0 < t_1 < \cdots < t_r < t_{r+1}=b$ 
with $t_{1:r} \subseteq [a,b]_v$ and $t_r \leq b-kv$, we define the space of 
{\it $k$th degree discrete splines} on $[a,b]$ with knots $t_{1:r}$, denoted 
\smash{$\DS^k_v(t_{1:r}, [a,b]_v)$}, to contain all functions $f$ on $[a,b]_v$  
such that 
\begin{equation}
\label{eq:discrete_spline_even}
\begin{gathered}
\text{for each $i=0,\ldots,r$, there is a $k$th degree polynomial $p_i$ 
  such that $f|_{I_{i,v}}=p_i|_{I_{i,v}}$, and} \\
\text{for each $i=1,\ldots,r$, it holds that $(F^\ell_v p_{i-1})(t_i) =
  (F^\ell_v p_i)(t_i)$, $\ell=0,\ldots,k-1$},  
\end{gathered}
\end{equation}
where $I_{0,v}=[t_0,t_1] \cap [a,b]_v$ and $I_{i,v}=(t_i,t_{i+1}] \cap [a,b]_v$, 
$i=1,\ldots,r$. 
\end{definition}


\begin{remark}
Comparing the conditions in \eqref{eq:discrete_spline_even} and
\eqref{eq:spline}, we see that when $k=0$ or $k=1$, the space
\smash{$\DS^k_v(t_{1:r}, [a,b]_v)$} of $k$th degree discrete splines with knots 
$t_{1:r}$ is the essentially equivalent to the space $\S^k(t_{1:r}, [a,b])$ of
$k$th degree splines with knots $t_{1:r}$ (precisely, for $k=0$ and $k=1$,
functions in \smash{$\DS^k_v(t_{1:r}, [a,b]_v)$} are the restriction of
functions in \smash{$\S^k(t_{1:r}, [a,b])$} to $[a,b]_v$). This is not true for
$k \geq 2$, in which case the two spaces are genuinely different.
\end{remark}

As covered in Chapter 8.5 of \citet{schumaker2007spline}, various properties of  
discrete splines can be developed in a parallel fashion to splines. For example,
instead of the truncated power basis \eqref{eq:tpb}, the following is a basis
for \smash{$\DS^k_v(t_{1:r}, [a,b]_v)$} (Theorem 8.51 of
\citet{schumaker2007spline}):    
\begin{equation}
\label{eq:ffb_even}
\begin{gathered}
f^k_j(x) = \frac{1}{(j-1)!} (x-a)_{j-1,v}
\quad j=1,\ldots,k+1, \\
f^k_j(x) = \frac{1}{k!} (x - t_j)_{k,v} \cdot 1\{x > t_j \}, 
\quad j=1,\ldots,r, 
\end{gathered}
\end{equation}
where we write $(x)_{\ell,v} = x (x-v) \cdots (x-(\ell-1)v)$ for the falling
factorial polynomial of degree $\ell$ with gap $v$, which we take to be equal to 
1 when $\ell=0$. Note that the above basis is an evenly-spaced analog
of the falling factorial basis in \eqref{eq:ffb}; in fact, Schumaker
refers to \smash{$f^k_j$}, $j=1,\ldots,r+k+1$ as ``one-sided factorial 
  functions'', which is (coincidentally) a very similar name to that we gave to 
\eqref{eq:ffb}, in our previous papers. In addition, a local basis for
\smash{$\DS^k_v(t_{1:r}, [a,b]_v)$}, akin to B-splines and hence called {\it
  discrete B-splines}, can be formed in an analogous fashion to that for
splines; we defer discussion of this until Appendix \ref{app:discrete_bs_even}.  

It should be noted that most of the classical literature, as well as Chapter 8.5
of \citet{schumaker2007spline}, studies discrete splines in the special case of
evenly-spaced design points $[a,b]_v$. Furthermore, the classical literature
treats discrete splines as discrete objects, that is, as {\it vectors}: see
Definition \ref{def:discrete_spline_even}, which is concerned only with the
evaluations of $f$ over the discrete set $[a,b]_v$. The assumption of
evenly-spaced design points is not necessary, and in the current paper we 
consider discrete splines with arbitrary design points. We also treat discrete
splines as continuum objects, namely, as {\it functions} defined over the
continuum interval $[a,b]$. To be clear, we do not intend to portray such 
extensions alone as particularly original or important contributions. Rather,
it is the {\it perspective} that we offer on discrete splines that (we believe)
is important---this starts with constructing a basis via discrete 
integration of indicator functions in Section \ref{sec:fall_fact}, which then
leads to the development of new properties, such as the matching derivatives
property in Section \ref{sec:deriv_match}, and the implicit interpolation
formula in Section \ref{sec:ffb_interp_implicit}.   

\subsection{Smoothing splines}
\label{sec:smooth_spline_review}

Let $x_{1:n}=\{x_1,\ldots,x_n\} \in [a,b]$ be design points, assumed to be
ordered, as in $x_1 < \cdots < x_n$, and let $y_1,\ldots,y_n$ be associated 
response points. For an odd integer $k=2m-1 \geq 1$, the {\it $k$th degree
  smoothing spline} estimator is defined as the solution of the variational 
optimization problem:
\begin{equation}
\label{eq:smooth_spline}
\minimize_f \; \sum_{i=1}^n \big(y_i - f(x_i)\big)^2 + \lambda \int_a^b 
(D^m f)(x)^2 \, dx, 
\end{equation}
where $\lambda \geq 0$ is a regularization parameter, and the domain of the
minimization in \eqref{eq:smooth_spline} is all functions $f$ on $[a,b]$ that
are $m$ times weakly differentiable, with \smash{$\int_a^b (D^m f)(x)^2 \, dx <
  \infty$}; this is known as as the $L_2$-Sobolev space of order $m$, and
denoted $\cW^{m,2}([a,b])$. The smoothing spline estimator was first proposed by  
\citet{schoenberg1964spline}, where he asserts (appealing to logic from previous
work on spline interpolation) that the solution in \eqref{eq:smooth_spline} is
unique, and is a $k$th degree spline belonging to $\S^k(x_{1:n},[a,b])$. In
fact, the solution in \eqref{eq:smooth_spline} is a special type of spline that
reduces to a polynomial of degree $m-1$ on the boundary intervals $[a,x_1]$ and
$[x_n,b]$, which is called a {\it natural spline} of degree $k=2m-1$. To fix
notation, we will denote the space of $k$th degree natural splines on $[a,b]$
with knots $x_{1:n}$ by \smash{$\NS^k(x_{1:n}, [a,b])$}.

Following Schoenberg's seminal contributions, smoothing splines have become the
topic of a vast body of work in both applied mathematics and statistics, with
work in the latter community having been pioneered by Grace Wahba and coauthors;
see, for example, \citet{craven1978smoothing} for a notable early paper. Two
important books on the statistical perspective underlying smoothing splines are
\citet{wahba1990spline,green1993nonparametric}. Today, smoothing splines are
undoubtedly one of the most widely used tools for univariate nonparametric
regression.

\paragraph{Connections to discrete-time.}

An interesting historical note, which is perhaps not well-known (or at least it
seems to have been largely forgotten in discussions on motivation for the
smoothing spline from a modern point of view), is that in creating the smoothing
spline, Schoenberg was motivated by the much earlier discrete-time smoothing
(graduation) approach of \citet{whittaker1923new}, stating this explicitly in
\citet{schoenberg1964spline}. Whittaker's approach, see
\eqref{eq:whittaker_filter}, estimates smoothed values by minimizing the sum of
a squared loss term and a penalty term of squared $m$th divided differences
(Whittaker takes $m=3$); meanwhile, Schoenberg's approach
\eqref{eq:smooth_spline}, ``in an attempt to combine [spline interpolation ...]
with Whittaker's idea'', replaces $m$th divided differences with $m$th
derivatives. Thus, while Schoenberg was motivated to move from a discrete-time
to a continuous-time perspective on smoothing, we are, as one of the main themes
in this paper, interested in returning to the discrete-time perspective, and
ultimately, connecting the two.

Given this, it is not really a surprise that Schoenberg himself derived the
first concrete connection between the two perspectives, continuous and discrete.
Next we transcribe his result from \citet{schoenberg1964spline}, and we
include a related result from \citet{reinsch1967smoothing}.

\begin{theorem}[\citet{schoenberg1964spline,reinsch1967smoothing}] 
\label{thm:nsp_sobolev}
For any odd integer $k = 2m-1 \geq 1$, and any $k$th degree natural spline
\smash{$f \in \NS^k(x_{1:n}, [a,b])$} with knots in $x_{1:n}$, it holds that 
\begin{equation}
\label{eq:nsp_sobolev}
\int_a^b (D^m f)(x)^2 \, dx =
\big\|(\K^m_n)^{\hspace{-1pt}\frac{1}{2}} \D^m_n f(x_{1:n}) \big\|_2^2,  
\end{equation}
where $f(x_{1:n})=(f(x_1),\ldots,f(x_n)) \in \R^n$ is the vector of evaluations
of $f$ at the design points, and \smash{$\D^m_n \in \R^{(n-m) \times n}$} is the
$m$th order discrete derivative matrix, as in \eqref{eq:discrete_deriv_mat}.
Furthermore, \smash{$\K^m_n \in \R^{(n-m) \times (n-m)}$} is a symmetric matrix
(that depends only on $x_{1:n}$), with a banded inverse of bandwidth $2m-1$.  If
we abbreviate, for $i=1,\ldots,n-m$, the function \smash{$P^{m-1}_i =
P^{m-1}(\cdot; x_{i:(i+m)})$}, which is the degree $m-1$ B-spline with knots
$x_{i:(i+m)}$, defined in \eqref{eq:bs} in Appendix \ref{app:bs}, then we can
write the entries of \smash{$(\K^m_n)^{-1}$} as
\begin{equation}
\label{eq:nsp_sobolev_kmat}
(\K^m_n)^{-1}_{ij} = m^2 \int_a^b P^{m-1}_i(x) P^{m-1}_j(x) \, dx. 
\end{equation}
For $m=1$, this matrix is diagonal, with entries
\begin{equation}
\label{eq:nsp_sobolev_kmat_m1}
(\K_n)^{-1}_{ii} = \frac{1}{x_{i+1}-x_i}.
\end{equation}
For $m=2$, this matrix is tridiagonal, with entries
\begin{equation}
\label{eq:nsp_sobolev_kmat_m2}
(\K^2_n)^{-1}_{ij} = 
\begin{cases}
\displaystyle
\frac{4}{3(x_{i+2}-x_i)} & \text{if $i=j$} \\
\displaystyle
\frac{2(x_{i+1}-x_i)}{3(x_{i+2}-x_i)(x_{i+1}-x_{i-1})} & \text{if $i=j+1$}. 
\end{cases}
\end{equation}
\end{theorem}

The matrix \smash{$\D^m_n \in \R^{(n-m) \times n}$} appearing in Theorem
\ref{thm:nsp_sobolev} is to be defined (and studied in detail) later, in
\eqref{eq:discrete_deriv_mat}. Acting on a vector
$f(x_{1:n})=(f(x_1),\ldots,f(x_n)) \in \R^n$, it gives $m!$ times the 
appropriate divided differences of $f$, namely,   
\begin{equation}
\label{eq:discrete_deriv_mat_act}
\big(\D^m_n f(x_{1:n})\big)_i = m! \cdot f[x_i,\ldots,x_{i+m}], 
\quad i=1,\ldots,n-m. 
\end{equation}
\citet{schoenberg1964spline} states the result in \eqref{eq:nsp_sobolev_kmat}
without proof. \citet{reinsch1967smoothing} derives the explicit form in
\eqref{eq:nsp_sobolev_kmat_m2}, for $m=2$, using a somewhat technical proof 
that stems from the Euler-Lagrange conditions for the variational problem
\eqref{eq:smooth_spline}. We give a short proof all results
\eqref{eq:nsp_sobolev_kmat}, \eqref{eq:nsp_sobolev_kmat_m1},
\eqref{eq:nsp_sobolev_kmat_m2} in Theorem \ref{thm:nsp_sobolev} in Appendix
\ref{app:nsp_sobolev}, based on the Peano representation for the B-spline (to be
clear, we make no claims of originality, this is simply done for completeness). 

\begin{remark}
Theorem \ref{thm:nsp_sobolev} reveals that the variational smoothing spline 
problem \eqref{eq:smooth_spline} can be recast as a finite-dimensional convex
quadratic program (relying on the fact that the solution in this problem lies in
\smash{$\NS^k(x_{1:n},[a,b])$} for $k=2m-1$):  
\begin{equation}
\label{eq:smooth_spline_discrete}
\minimize_\theta \; \|y-\theta\|_2^2 + \lambda
\big\|(\K^m_n)^{\hspace{-1pt}\frac{1}{2}} \D^m_n \theta \big\|_2^2, 
\end{equation}
for $y=(y_1,\ldots,y_n) \in \R^n$, and \smash{$\K^m_n$} as defined in
\eqref{eq:nsp_sobolev_kmat}. The solutions \smash{$\htheta,\hf$} in problems
\eqref{eq:smooth_spline_discrete}, \eqref{eq:smooth_spline}, respectively,
satisfy \smash{$\htheta=\hf(x_{1:n})$}. Furthermore, from
\eqref{eq:smooth_spline_discrete}, the solution is easily seen to be
\begin{equation}
\label{eq:smooth_spline_sol}
\htheta = \big(\I_n + \lambda (\D^m_n)^\T \K^m_n \, \D^m_n\big)^{-1} y,  
\end{equation}
where $\I_n$ denotes the $n \times n$ identity matrix. Despite the fact that
\smash{$\K^m_n$} is itself dense for $m \geq 2$ (recall that its inverse is
banded with bandwidth $2m-1$), the smoothing spline solution \smash{$\htheta$}
in \eqref{eq:smooth_spline_sol} can be computed in linear-time using a number of
highly-efficient, specialized approaches (see, for example, Chapter XIV of
\citet{deboor1978practical}). 
\end{remark}

\begin{remark}
It is interesting to compare \eqref{eq:smooth_spline_discrete} and what we
call the Bohlmann-Whittaker (BW) filter \eqref{eq:bw_filter_old} (note that the
traditional case studied by Bohlmann and Whittaker was unit-spaced design
points, as in \eqref{eq:bw_filter_unit}, and problem \eqref{eq:bw_filter_old}
was Whittaker's proposed extension to arbitrary design points). We can see that
the smoothing spline problem reduces to a modified version of the discrete-time
BW problem, where \smash{$\|\D^m_n \theta\|_2^2$} is replaced by the quadratic
form \smash{$\|(\K^m_n)^{\hspace{-1pt}\frac{1}{2}} \D^m_n \theta \|_2^2$}, for a
matrix \smash{$\K^m_n$} having a banded inverse. To preview one of our later
results, in Theorem \ref{thm:ffb_sobolev}: by restricting the domain in problem
\eqref{eq:smooth_spline} to discrete splines, it turns out we can obtain another
variant of the BW filter where the corresponding matrix \smash{$\K^m_n$} is now 
{\it itself} banded. 
\end{remark}

\subsection{Locally adaptive splines}
\label{sec:local_spline_review}

Smoothing splines have many strengths, but adaptivity to changes in the
local level of smoothness is not one of them. That is, if the underlying
regression function $f_0$ is smooth in some parts of its domain and wiggly in
other parts, then the smoothing spline will have trouble estimating $f_0$
adequately throughout. It is not alone: any {\it linear smoother}---meaning, an
estimator \smash{$\hf$} of $f_0$ whose fitted values
\smash{$\htheta=\hf(x_{1:n})$} are a linear function of the responses $y$---will
suffer from the same problem, as made precise by the influential work of
\citet{donoho1998minimax}. (From \eqref{eq:smooth_spline_sol}, it is easy to
check that the smoothing spline estimator is indeed a linear smoother.)   We
will explain this point in more detail shortly. 

Aimed at addressing this very issue, \citet{mammen1997locally} proposed an 
estimator based on solving the variational problem \eqref{eq:local_spline},
which recall, for a given integer $k \geq 0$, is known as the $k$th degree
locally adaptive regression spline estimator. (It is worth noting that the same
idea was proposed earlier by \citet{koenker1994quantile}, who studied total
variation smoothing of the first derivative, $k=1$, in nonparametric quantile 
regression.)  We can see that \eqref{eq:local_spline} is like the smoothing
spline problem \eqref{eq:smooth_spline}, but with the $L_2$-Sobolev penalty is
replaced by a (higher-order) total variation penalty on $f$. Note that when $f$
is $k+1$ times weakly differentiable on an interval $[a,b]$, we have  
\begin{equation}
\label{eq:tv_sobolev}
\TV(D^k f) = \int_a^b |(D^{k+1} f)(x)| \, dx.
\end{equation}
In this sense, we can interpret problem \eqref{eq:local_spline} as something
like the $L_1$ analog of problem \eqref{eq:smooth_spline}. Importantly, note
that the fitted values \smash{$\htheta = \hf(x_{1:n})$} from the locally
adaptive regression spline estimator are {\it not} a linear function of $y$,
that is, the locally adaptive regression spline estimator is not a linear
smoother.   

\paragraph{Local adaptivity.}

True to its name, the locally adaptive regression spline estimator is more
attuned to the local level of smoothness in $f_0$ compared to the smoothing
spline. This is evident both empirically and theoretically. See Figure
\ref{fig:adapt} for an empirical example. In terms of theory, there are clear  
distinctions in the optimality properties belonging to linear and nonlinear methods. In 
classical nonparametric regression, linear smoothers such as smoothing splines 
are typically analyzed for their rates of estimation of an underlying function
$f_0$ when the latter is assumed to lie in a function class like a Sobolev or 
Holder class. In a minimax sense, smoothing splines (as well as several other 
linear methods, such as kernel smoothers) are rate optimal for Sobolev or Holder 
classes (for example, see Chapter 10 of \citet{vandegeer2000empirical}).
But for ``larger'' function classes like certain total variation, Besov, or
Triebel classes, they are notably suboptimal. 

As an example, the following is an implication of the results in
\citet{donoho1998minimax} (see Section 5.1 of \citet{tibshirani2014adaptive}
for an explanation). Let $\cV^k([a,b])$ denote the space of functions $f$ on
$[a,b]$ that are $k$ times weakly differentiable, with $\TV(D^k f) < \infty$;
and denote the associated seminorm ball of radius $C>0$ by
$$
\cV^k(C; [a,b]) = \Big\{ f : [a,b] \to \R : \TV(D^k f) \leq C \Big\}.  
$$
Abbreviating $\cV^k = \cV^k(C; [a,b])$ for fixed $C,a,b$ (not depending on $n$),
and placing standard assumptions on the data generation model (that is,
assumptions on the design points $x_i$, $i=1,\ldots,n$, and errors $\epsilon =
y_i - f_0(x_i)$, $i=1,\ldots,n$), the minimax rate in mean squared $L_2$ error
over the design points is 
\begin{equation}
\label{eq:minimax_rate}
\inf_{\hf} \sup_{f_0 \in \cV^k} \; 
\E \bigg[\frac{1}{n} \big\| \hf(x_{1:n}) - f_0(x_{1:n}) \big\|_2^2 \bigg]
\lesssim n^{-\frac{2k+2}{2k+3}}, 
\end{equation}
where the infimum above is taken over all estimators \smash{$\hf$}. However, 
the minimax {\it linear} rate is
\begin{equation}
\label{eq:minimax_linear_rate}
\inf_{\hf \, \text{linear}} \sup_{f_0 \in \cV^k} \;  
\E \bigg[\frac{1}{n} \big\| \hf(x_{1:n}) - f_0(x_{1:n}) \big\|_2^2 \bigg]
\gtrsim n^{-\frac{2k+1}{2k+2}},
\end{equation}
where the infimum above is taken over all linear smoothers \smash{$\hf$}. 
(Here, we use \smash{$a_n \lesssim b_n$} to mean $a_n \leq c b_n$ for a constant
$c>0$ and large enough $n$, and $a_n \gtrsim b_n$ to mean \smash{$1/a_n \lesssim 
  1/b_n$}.)  \citet{mammen1997locally} proved that locally adaptive regression
splines achieve the optimal rate in \eqref{eq:minimax_rate} (note that wavelet
smoothing also achieves the optimal rate, as shown by
\citet{donoho1998minimax}). Importantly, from \eqref{eq:minimax_linear_rate}, we 
can see that smoothing splines---and further, any linear smoother
whatsoever---are suboptimal. 

For a concrete case, we can take $k=0$,
and then the rates \eqref{eq:minimax_rate} and \eqref{eq:minimax_linear_rate}
are \smash{$n^{-\frac{2}{3}}$} and \smash{$n^{-\frac{1}{2}}$}, respectively,
which we can interpret as follows: for estimating a function of bounded
variation, the smoothing spline requires (on the order of)
\smash{$n^{\frac{4}{3}}$} data points to achieve the same error guarantee that 
the locally adaptive regression spline has on $n$ data points. See Figure
\ref{fig:minimax} for an illustration of the rates for general $k$. 

\begin{figure}[p]
\centering
\includegraphics[width=0.495\textwidth]{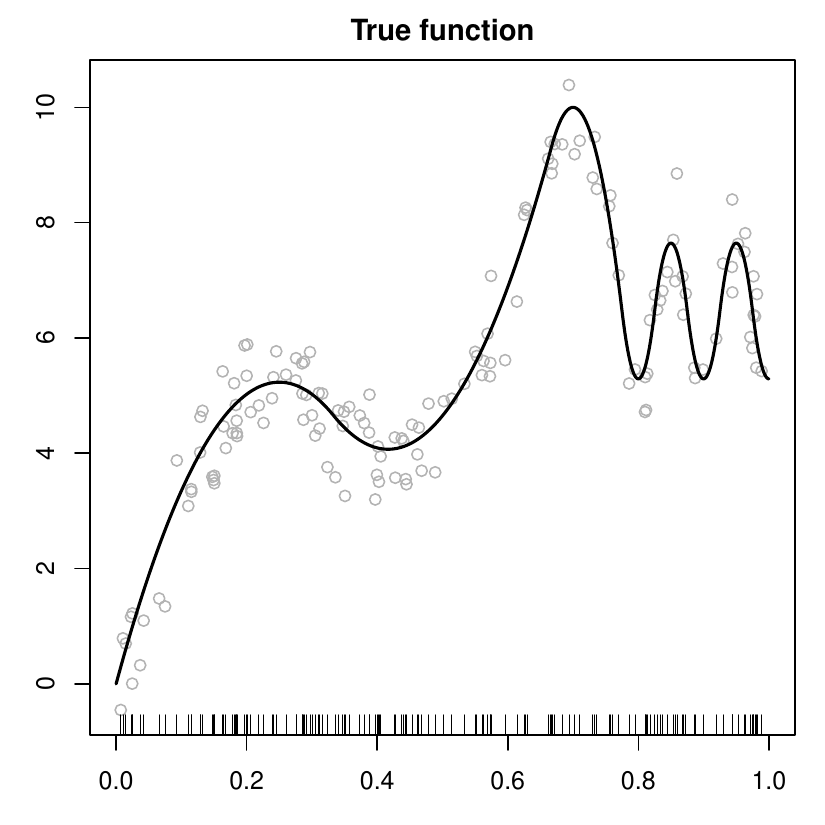}
\includegraphics[width=0.495\textwidth]{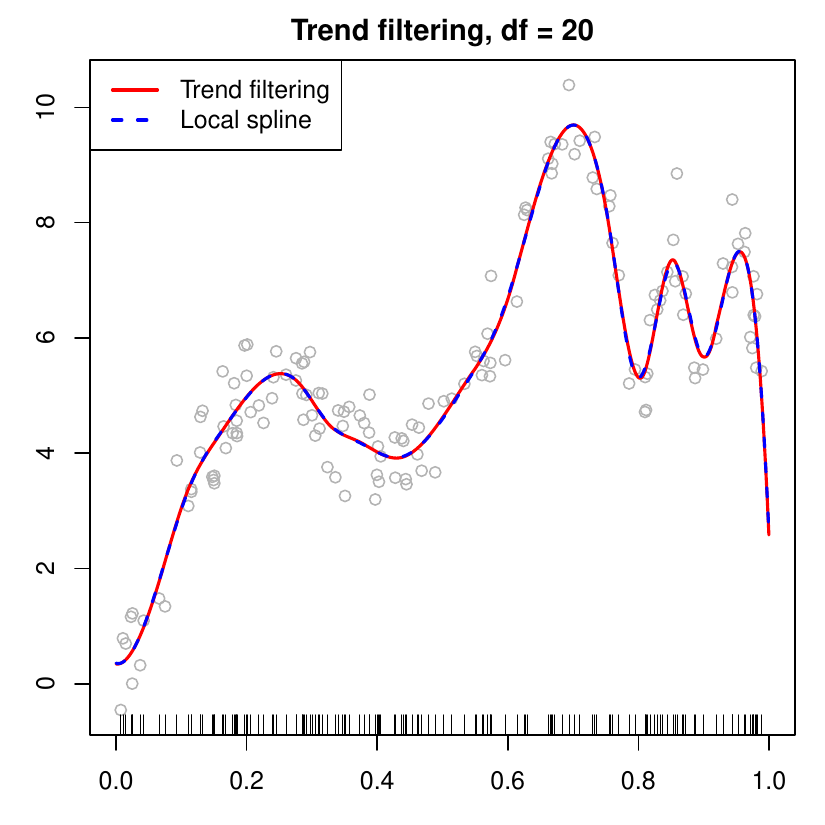}
\includegraphics[width=0.495\textwidth]{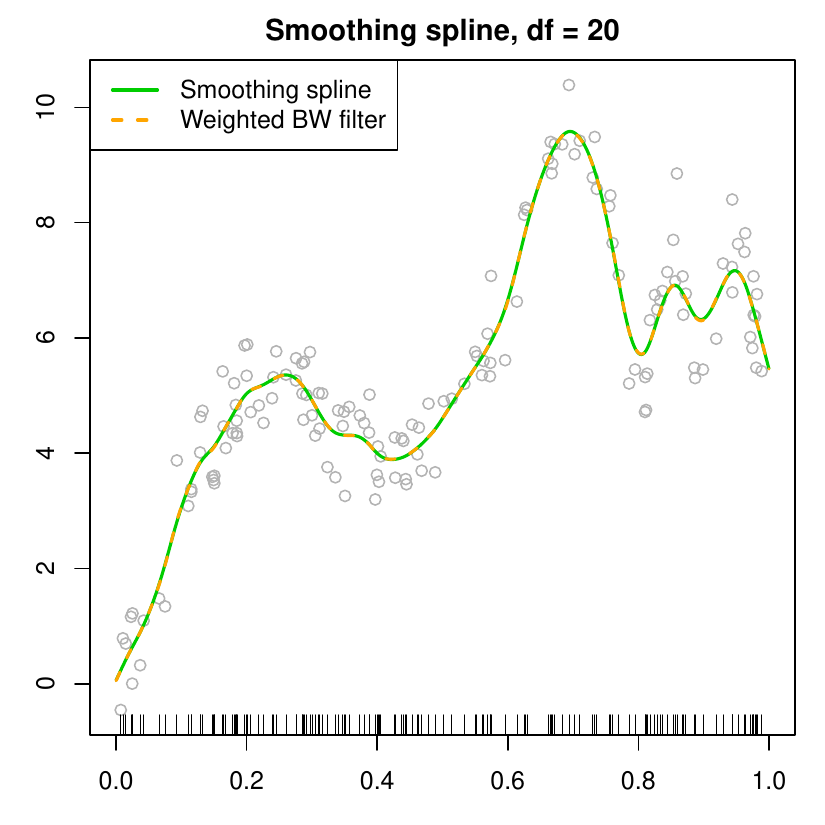}
\includegraphics[width=0.495\textwidth]{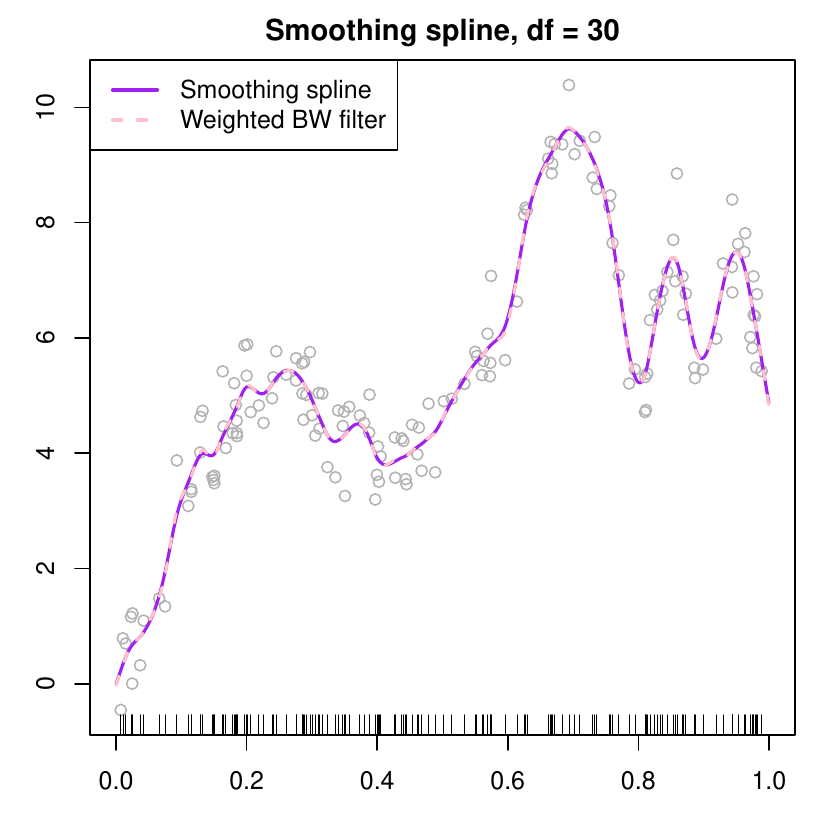}
\caption{\small (Adapted from \citet{tibshirani2014adaptive}.) Comparison of
  trend filtering and smoothing splines on an example with heterogeneous
  smoothness. The top left panel shows the true underlying regression function
  and $n=150$ sampled response points, with the design points are marked by
  ticks on the horizontal axis (they are not evenly-spaced). The top right
  panel shows the cubic trend filtering solution ($k=3$), in solid blue, with a 
  ``hand-picked'' value of the tuning parameter $\lambda$. This solution results
  in an estimated degrees of freedom (df) of 20 (see
  \citet{tibshirani2011solution}). Note that it adapts well to the smooth part 
  of the true function on the left side of the domain, as well as the wiggly
  part on the right side. Also plotted is the restricted locally
  adaptive regression spline solution ($k=3$), in dashed red, at the same value
  of $\lambda$, which looks visually identical. The bottom left panel is the
  cubic smoothing spline solution ($m=2$), in solid green, whose df is matched
  to that of the trend filtering solution; notice that it oversmooths on the
  right side of the domain. The bottom right panel is the smoothing spline
  solution when its df has been increased to 30, the first point at which it
  begins appropriately pick up the two peaks on the right side; but note that 
  it now undersmooths on the left side. Finally, in the bottom two panels, the
  cubic BW filter ($m=2$) is also plotted, in dotted orange and dotted pink---to
  be clear, this is actually our proposed weighted extension of the BW filter to 
  arbitrary designs, as given in Section \ref{sec:bw_filter}. In each case it uses
  same value of $\lambda$ as the smoothing spline solution, and looks identical 
  to the latter.} 
\label{fig:adapt}
\end{figure}

\begin{figure}[tb]
\centering
\includegraphics[width=0.85\textwidth]{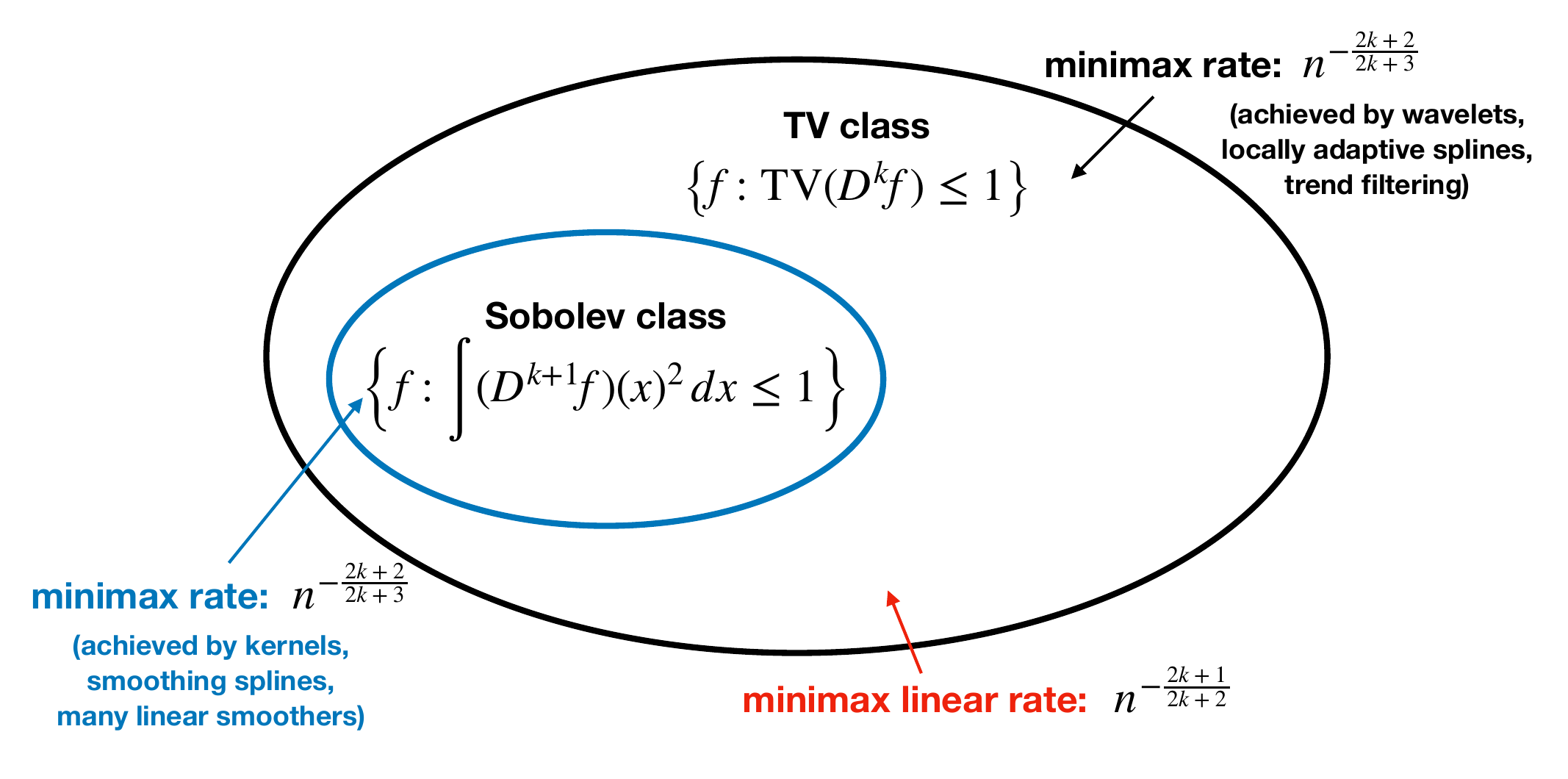}
 \captionsetup{singlelinecheck=off}
 \caption[.]{\small Minimax rates for the unit ball in the TV space of order
   $k$, \smash{$\cV^k(1; [0,1]) = \{ f : [0,1] \to \R : \TV(D^k f) \leq 1 \}$},
   and for the unit ball in the Sobolev space of order $k+1$,
   \smash{$\cW^{k+1,2}(1; [0,1]) = \{ f : [0,1] \to \R : \int_0^1 (D^{k+1}
     f)(x)^2 \, dx \leq 1 \}$}. Observe that 
   $$
   \cV^k(1; [0,1]) \supseteq \cW^{k+1,1}(1; [0,1]) 
   \supseteq \cW^{k+1,2}(1; [0,1]);
   $$ 
   the first containment is due to the equality in \eqref{eq:tv_sobolev} for
   $k+1$ times weakly differentiable functions, and the second containment is
   due to the relation between $L_1$ and $L_2$ norms on $[0,1]$. The minimax
   rates between \smash{$\cV^k(1; [0,1])$} and the \smash{$\cW^{k+1,2}(1;
     [0,1])$} are the same; but critically, linear smoothers, which can be
   optimal on the smaller set \smash{$\cW^{k+1,2}(1; [0,1])$}, {\it cannot} be
   optimal on the larger set \smash{$\cV^k(1; [0,1])$}, which contains functions
   that display more heterogeneous smoothness.}  
\label{fig:minimax}
\end{figure}

\paragraph{Computational difficulties.}

\citet{mammen1997locally} proved that the solution \smash{$\hf$} in
\eqref{eq:local_spline} is a $k$th degree spline. For $k=0$ or $k=1$, they show 
that the knots in \smash{$\hf$} must lie in particular subset of the design
points, denoted \smash{$T_{n,k} \subseteq x_{1:n}$}, with cardinality 
\smash{$|T_{n,k}|=n-k-1$}; that is, for $k=0$ or $k=1$, we know that  
\smash{$\hf \in \S^k(T_{n,k},[a,b])$}, which reduces \eqref{eq:local_spline} to
a finite-dimensional problem. But for $k \geq 2$, this is no longer true; the
knots in \smash{$\hf$} may well lie outside of $x_{1:n}$, and
\eqref{eq:local_spline} remains an infinite-dimensional problem (since we have
to optimize over all possible knot sets).  

As a proposed fix, for a general degree $k \geq 0$, \citet{mammen1997locally}
defined (what we refer to as) the {\it $k$th degree restricted locally adaptive
  regression spline} estimator, which solves
\begin{equation}
\label{eq:local_spline_rest}
\minimize_{f \in \cG^k_n} \; \frac{1}{2} \sum_{i=1}^n \big(y_i - f(x_i)\big)^2 +   
\lambda \, \TV(D^k f),
\end{equation}
for a certain space \smash{$\cG^k_n=\S^k(T_{n,k},[a,b])$} of $k$th degree
splines with knots \smash{$T_{n,k} \subseteq x_{1:n}$}, where
\smash{$|T_{n,k}|=n-k-1$} (they define \smash{$T_{n,k}$} by excluding $k+1$ 
points at the extremes of the design). To be clear, for $k=0$ and $k=1$,
problems \eqref{eq:local_spline_rest} and \eqref{eq:local_spline} are
equivalent; but for $k \geq 2$, they are not, and the former is an approximation
of the latter.  

The proposal in \eqref{eq:local_spline_rest} is useful because it is equivalent
to a finite-dimensional convex optimization problem: letting \smash{$\G^k_n \in 
  \R^{\times n}$} be the truncated power basis matrix, with entries 
\smash{$(\G^k_n)_{ij}=g^k_j(x_i)$}, where \smash{$g^k_j$}, $j=1,\ldots,n$ are
the truncated power basis \eqref{eq:tpb} for \smash{$\cG^k_n$}, we can rewrite
\eqref{eq:local_spline_rest} as
\begin{equation}
\label{eq:local_spline_basis}
\minimize_\alpha \; \frac{1}{2} \big\|y - \G^k_n \alpha\big\|_2^2 + 
\lambda \sum_{j=k+2}^n |\alpha_j|,
\end{equation}
where the solutions \smash{$\halpha,\hf$} in \eqref{eq:local_spline_basis},
\eqref{eq:local_spline_rest}, respectively, satisfy \smash{$\hf = \sum_{j=1}^n
\halpha_j g^k_j$}. \citet{mammen1997locally} proved that the restricted locally
adaptive regression spline estimator (under weak conditions on the design
points) still achieves the optimal rate in \eqref{eq:minimax_rate}. Readers
familiar with the high-dimensional regression literature will recognize
\eqref{eq:local_spline_basis} as a type of {\it lasso} problem
\citep{tibshirani1996regression,chen1998atomic}, for which many efficient
algorithms exist (for just one example, see \citet{friedman2007pathwise}). But
for large sample sizes $n$, it can still be computationally difficult to solve,
owing to the fact that the design matrix \smash{$\G^k_n$} is dense (it is
actually lower-triangular, but generally poorly-conditioned, which causes
trouble for first-order optimization algorithms).

\subsection{Trend filtering}
\label{sec:trend_filter_review}

Building on the background and motivation for trend filtering given in the
introduction, and the motivation for locally adaptive regression splines just
given, we arrive at the following perspective. Trend filtering is an
approximation to the locally adaptive regression spline problem
\eqref{eq:local_spline}, which is similar to the proposal for restricted locally
adaptive regression splines in \eqref{eq:local_spline_rest}, but with a
different restriction for the optimization domain: it uses the $k$th degree
discrete spline space \smash{$\cH^k_n$}, as we saw in
\eqref{eq:trend_filter_cont}, rather than the $k$th degree spline space
\smash{$\cG^k_n$}. To retrieve an equivalent lasso form, similar to
\eqref{eq:local_spline_basis}, we can let \smash{$\H^k_n \in \R^{n\times n}$}
denote the falling factorial basis matrix, with entries
\smash{$(\H^k_n)_{ij}=h^k_j(x_i)$}, where \smash{$h^k_j$}, $j=1,\ldots,n$ are as
in \eqref{eq:ffb}, and then \eqref{eq:trend_filter_basis} becomes
\begin{equation}
\label{eq:trend_filter_basis}
\minimize_\alpha \; \frac{1}{2} \big\|y - \H^k_n \alpha\big\|_2^2 + 
\lambda \sum_{j=k+2}^n |\alpha_j|.
\end{equation}
where the solutions \smash{$\halpha,\hf$} in problems
\eqref{eq:trend_filter_basis}, \eqref{eq:trend_filter_cont}, respectively, are
related by \smash{$\hf = \sum_{j=1}^n \halpha_j h^k_j$}. Fortunately, trend
filtering retains (under mild conditions on the design) the minimax optimal rate
in \eqref{eq:minimax_rate}. This was shown in
\citet{tibshirani2014adaptive,wang2014falling} by bounding the distance between
solutions in \eqref{eq:trend_filter_basis}, \eqref{eq:local_spline_basis}.

Finally---and critically for practical use---the problem
\eqref{eq:trend_filter_basis} has an equivalent form given in
\eqref{eq:trend_filter_old}. The latter, original form of trend filtering is
more amenable to efficient computation, thanks to the structured, banded nature
of its penalty term; computation here scales considerably better than that in
either \eqref{eq:trend_filter_basis} or the restricted locally adaptive
regression spline problem \eqref{eq:local_spline_rest}. (We have found that in
most empirical examples, trend filtering and restricted locally adaptive spline
solutions are more or less visually identical anyway; see Figure
\ref{fig:adapt}.)

On the topic of \eqref{eq:trend_filter_old}, we remark that this problem can 
be equivalently written as
\begin{equation}
\label{eq:trend_filter}
\minimize_\theta \; \frac{1}{2} \|y - \theta\|_2^2 + 
\lambda \big\|\W^{k+1}_n \D^{k+1}_n \theta \big\|_1,
\end{equation}
with \smash{$\D^{k+1}_n \in \R^{(n-k-1) \times n}$} the $(k+1)$st order 
discrete derivative matrix to be defined in \eqref{eq:discrete_deriv_mat}   
(recall, this matrix acts by producing divided differences over the design
points), and \smash{$\W^{k+1}_n \in \R^{(n-k-1) \times (n-k-1)}$} is the
$(k+1)$st order diagonal weight matrix to be defined in \eqref{eq:weight_mat}.
The penalty in the above problem is hence
\begin{equation}
\label{eq:trend_filter_wpen}
\big\|\W^{k+1}_n \D^{k+1}_n \theta \big\|_1 
= \sum_{i=1}^{n-k-1} \big| (\D^{k+1}_n \theta)_i \big| \cdot
\frac{x_{i+k+1} - x_i}{k+1}. 
\end{equation}
Thus \eqref{eq:trend_filter_old} versus \eqref{eq:trend_filter} is a matter of
whether the natural operator is viewed as \smash{$\C^{k+1}_n$} or
\smash{$\D^{k+1}_n$}. We should note that when the design points are
evenly-spaced, we have \smash{$\W^{k+1}_n=\I_{n-k-1}$}, the identity matrix, so
this choice makes no difference; in general though, it does, and we now view
\eqref{eq:trend_filter} as a more natural way of presenting trend filtering,
which differs from the choice \eqref{eq:trend_filter_old} that we made in
\citet{tibshirani2014adaptive,wang2014falling} and our subsequent work. In
Remarks \ref{rem:discrete_deriv_mat_old} and \ref{rem:ffb_tv_mat_old}, and
Section \ref{sec:trend_filter}, we return to this point.

\paragraph{Historical remarks.}

As already mentioned, trend filtering for evenly-spaced designs was
independently proposed by \citet{steidl2006splines,kim2009trend}. However,
similar ideas were around much earlier. \citet{kim2009trend} were clear
about being motivated by \citet{hodrick1981business}, who considered an
$\ell_2$ analog of trend filtering, that is, with an $\ell_2$ penalty on forward
differences, rather than an $\ell_1$ penalty. (Actually, such $\ell_2$ analogs  
were proposed over 100 years ago, long before Hodrick and Prescott, first by
Bohlmann and then by Whittaker, as we discuss in the next subsection.)
Moreover, \citet{schuette1978linear} and \citet{koenker1994quantile} studied   
estimators defined using piecewise linear ($k=1$) trend filtering penalties,
but where the squared $\ell_2$ loss is replaced with an $\ell_1$ loss or
quantile loss, respectively. Lastly, we remark again that for the piecewise 
constant case ($k=0$), trend filtering reduces to what is known as total
variation denoising \citep{rudin1992nonlinear} in signal processing, and the
fused lasso in statistics \citep{tibshirani2005sparsity}. 

In writing \citet{tibshirani2014adaptive}, we were motivated by
\citet{kim2009trend}; these authors called their method ``$\ell_1$ trend
filtering'', which we shortened to ``trend filtering'' in our work. At this
time, we had not heard of discrete splines, but we were aware that the trend
filtering solution displayed a kind of continuity in its lower-order discrete
derivatives: this was demonstrated empirically in Figure 3 of
\citet{tibshirani2014adaptive}. By the time of our follow-up paper
\citet{wang2014falling}, we learned that \citet{steidl2006splines} had proposed 
the same idea as \citet{kim2009trend}. It was in the former paper that we
first learned of discrete splines and the foundational work by
\citet{mangasarian1971discrete,mangasarian1973best} on the topic, but it was 
not until much later---until we read the book by \citet{schumaker2007spline},
where the development of discrete splines is laid out systematically in a
parallel fashion to the development of splines---that we truly appreciated the  
connection between discrete splines and trend filtering, and the value that such
a connection can bring to both lines of work. The current paper grew from an 
attempt to pay homage to discrete splines and to make all such connections
explicit. 

\subsection{Bohlmann-Whittaker filtering}
\label{sec:bw_filter_review}

Over 120 years ago, \citet{bohlmann1899ausgleichungsproblem} studied the
solution of the problem:
\begin{equation}
\label{eq:bohlmann_filter}
\minimize_\theta \; \|y-\theta\|_2^2 + 
\lambda \sum_{i=1}^{n-1} (\theta_i - \theta_{i+1})^2,  
\end{equation}
as a smoother of responses $y_i$, $i=1,\ldots,n$ observed at evenly-spaced
(unit-spaced) design points $x_i=i$, $i=1,\ldots,n$. This is one of the
earliest references that we know of for discrete-time smoothing (or smoothing of
any kind) based on optimization. Over 20 years after this,
\citet{whittaker1923new} proposed a variant of \eqref{eq:bohlmann_filter} where 
first differences are replaced by third differences: 
\begin{equation}
\label{eq:whittaker_filter}
\minimize_\theta \; \|y-\theta\|_2^2 + 
\lambda \sum_{i=1}^{n-3} (\theta_i - 3\theta_{i+1} + 3\theta_{i+2} -
\theta_{i+3})^2.
\end{equation}
Whittaker seems to have been unaware of the work by Bohlmann, and unfortunately,
Bohlmann's work has remained relatively unknown (it is still not cited in most
references on discrete-time smoothing and its history). Meanwhile, the work of
\citet{whittaker1923new} was quite influtential and led a long line of
literature, centered in the actuarial community, where
\eqref{eq:whittaker_filter} is often called the {\it Whittaker-Henderson method
of graduation}, honoring the contributions of \citet{henderson1924new}.
Moreover, as explained previously, recall it was Whittaker's work that
inspired \citet{schoenberg1964spline} to develop the smoothing spline.

Almost 60 years after this, \citet{hodrick1981business} proposed a variation on   
\eqref{eq:whittaker_filter} in which third differences are replaced by second
differences:  
\begin{equation}
\label{eq:hp_filter}
\minimize_\theta \; \|y-\theta\|_2^2 + 
\lambda \sum_{i=1}^{n-2} (\theta_i - 2\theta_{i+1} + \theta_{i+2})^2.
\end{equation}
Hodrick and Prescott were aware of the work of Whittaker, but not of Bohlmann.
The paper by \citet{hodrick1981business}, which was later published as
\citet{hodrick1997business}, has become extremely influential in econometrics,
where \eqref{eq:hp_filter} is known as the {\it Hodrick-Prescott filter}.
Recall, as explained previously, that it was Hodrick and Prescott's work that
inspired \citet{kim2009trend} to develop trend filtering. 

Generalizing \eqref{eq:bohlmann_filter}, \eqref{eq:whittaker_filter}, 
\eqref{eq:hp_filter}, consider for an integer $m \geq 0$, the problem:  
\begin{equation}
\label{eq:bw_filter_unit}
\minimize_\theta \; \|y - \theta\|_2^2 + 
\lambda \sum_{i=1}^{n-m} (F^m \theta)(i)^2
\end{equation}
where \smash{$(F^m \theta)(i) = \sum_{\ell=0}^k (-1)^{k-\ell} {k \choose \ell} 
  \theta_{i+\ell}$} is the standard (integer-based) $m$th order forward
differences of $\theta$ starting at an integer $i$. To honor their early 
contributions, we call the solution in \eqref{eq:bw_filter_unit} the {\it 
  Bohlmann-Whittaker (BW) filter}.  

\paragraph{Arbitrary designs.}

For a set of arbitrary design points $x_{1:n}$, it would seem natural to use
divided differences in place of forward differences in
\eqref{eq:bw_filter_unit}, resulting in    
\begin{equation}
\label{eq:bw_filter_old}
\minimize_\theta \; \|y - \theta\|_2^2 + \lambda 
\big\| \D^m_n \theta\big\|_2^2, 
\end{equation}
where \smash{$\D^m_n \in \R^{(n-m) \times n}$} is the $m$th order discrete
derivative matrix defined in \eqref{eq:discrete_deriv_mat}. In fact, such an
extension \eqref{eq:bw_filter_old} for arbitrary designs was suggested by
\citet{whittaker1923new}, in a footnote of his paper. This idea caught on with
many authors, including \citet{schoenberg1964spline}, who in describing
Whittaker's method as the source of inspiration for his creation of the
smoothing spline, used the form \eqref{eq:bw_filter_old}.

In Section \ref{sec:bw_filter}, we argue that for arbitrary designs it is
actually in some ways more natural to replace the penalty in
\eqref{eq:bw_filter_old} by a weighted squared $\ell_2$ penalty,    
\begin{equation}
\label{eq:bw_filter_wpen}
\big\| (\W^m_n)^{\hspace{-1pt}\frac{1}{2}} \D^m_n \theta\big\|_2^2 = 
\sum_{i=1}^{n-m} (\D^m_n \theta)_i^2 \cdot \frac{x_{i+m} - x_i}{m}. 
\end{equation}
Here \smash{$\W^m_n \in \R^{(n-m) \times (n-m)}$} is the $m$th order diagonal 
weight matrix, defined later in \eqref{eq:weight_mat}. Notice the close
similarity between the weighting in \eqref{eq:bw_filter_wpen} and in the trend
filtering penalty \eqref{eq:trend_filter_wpen}. The reason we advocate for the
penalty \eqref{eq:bw_filter_wpen} is that the resulting estimator admits a close
tie to the smoothing spline: when $m=1$, these two exactly coincide (recall
\eqref{eq:nsp_sobolev} and \eqref{eq:nsp_sobolev_kmat_m1} from Theorem
\ref{thm:nsp_sobolev}), and when $m=2$, they are provably ``close'' in $\ell_2$
distance (for appropriate values of their tuning parameters), as we show later
in Theorem \ref{thm:ss_bw_bound}. Moreover, empirical examples support the idea
that the estimator associated with the weighted penalty
\eqref{eq:bw_filter_wpen} can be closer than the solution in
\eqref{eq:bw_filter_old} to the smoothing spline.  

Finally, unlike trend filtering, whose connection to discrete splines is
transparent and clean (at least in hindsight), the story with the BW filter is
more subtle. This is covered in Section \ref{sec:bw_discrete_spline}. 

\section{Falling factorials}
\label{sec:fall_fact}

In this section, we define a discrete derivative operator based on divided
differences, and its inverse operator, a discrete integrator, based on
cumulative sums. We use these discrete operators to construct the falling
factorial basis for discrete splines, in a manner analogous to the construction
of the truncated power basis for splines.

\subsection{Discrete differentiation}

Let $f$ be a function defined on an interval $[a,b]$\footnote{There is no real
  need to consider an interval $[a,b]$ containing the points $x_1,\ldots,x_n$.
  We introduce this interval simply because we think it may be conceptually
  helpful when defining the discrete derivative and integral operators, but the
  same definitions make sense, with minor modifcations, when we consider $f$ as
  a function on all of $\R$.}, 
and let $a \leq x_1 < \cdots < x_n \leq b$. To motivate the discrete
derivative operator that we study in this subsection, consider the following
question: given a point $x \in [a,b]$, how might we use
$f(x_1),\ldots,f(x_n)$, along with one more evaluation $f(x)$, to approximate
the $k$th derivative $(D^k f)(x)$, of $f$ at $x$?       

A natural answer to this question is given by divided differences. For an
integer $k \geq 1$, we write $\Delta^k(\cdot; x_{1:n})$ for an operator that
maps a function $f$ to a function $\Delta^k (f; x_{1:n})$, which we call the
{\it $k$th discrete derivative} (or the {\it discrete $k$th derivative}) of $f$,
to be defined below. A remark on notation: $\Delta^k (f; x_{1:n})$ emphasizes
the dependence on the underlying design points $x_{1:n}=\{x_1,\ldots,x_n\}$; 
henceforth, we abbreviate \smash{$\Delta^k_n f = \Delta^k (f; x_{1:n})$} (and
the underlying points $x_{1:n}$ should be clear from the context). Now, we
define the function \smash{$\Delta^k_n f$} at a point $x \in [a,b]$ as       
\begin{equation}
\label{eq:discrete_deriv}
(\Delta^k_n f) (x) = 
\begin{cases}
k! \cdot f[x_{i-k+1},\ldots,x_i,x] & \text{if $x \in (x_i,x_{i+1}]$, $i \geq k$} \\
i! \cdot f[x_1,\ldots,x_i,x] & \text{if $x \in (x_i,x_{i+1}]$, $i < k$} \\
f(x) & \text{if $x \leq x_1$}.
\end{cases}
\end{equation}
Here and throughout, we use $x_{n+1}=b$ for notational convenience. Note that,
on ``most'' of the domain $[a,b]$, that is, for $x \in (x_k,b]$, we define
\smash{$(\Delta^k_n f)(x)$} in terms of a (scaled) $k$th divided difference of
$f$, where the centers are the $k$ points immediately to the left of $x$, and
$x$ itself. Meanwhile, on a ``small'' part of the domain, that is, for $x \in
[a,x_k]$, we define \smash{$(\Delta^k_n f)(x)$} to be a (scaled) divided
difference of $f$ of the highest possible order, where the centers are the
points to the left of $x$, and $x$ itself.

\paragraph{Linear combination formulation.} 

As divided differences are linear combinations of function evaluations, it is
not hard to see from its definition in \eqref{eq:discrete_deriv} that
\smash{$(\Delta^k_nf)(x)$} is a linear combination of (a subset of size at most
$k+1$ of) the evaluations $f(x_1),\ldots,f(x_n)$ and $f(x)$. In fact, from the 
alternative representation for divided differences in
\eqref{eq:divided_diff_linear}, we can rewrite \eqref{eq:discrete_deriv} as 
\begin{equation}
\label{eq:discrete_deriv_linear}
(\Delta^k_n f)(x) =
\begin{cases}
\displaystyle
\sum_{j=i-k+1}^i \frac{k! \cdot f(x_j)} 
{\big(\prod_{\ell \in \{i-k+1,\ldots,i\} \setminus \{j\}}
(x_j-x_\ell)\big)(x_j-x)} + 
\frac{k! \cdot f(x)}{\prod_{\ell=i-k+1}^i (x-x_\ell)} 
& \text{if $x \in (x_i,x_{i+1}]$, $i \geq k$} \\
\displaystyle
\sum_{j=1}^i \frac{i! \cdot f(x_j)} 
{\big(\prod_{\ell \in \{1,\ldots,i\} \setminus \{j\}}
(x_j-x_\ell)\big)(x_j-x)} + 
\frac{i! \cdot f(x)}{\prod_{\ell=1}^i (x-x_\ell)} 
& \text{if $x \in (x_i,x_{i+1}]$, $i < k$} \\
f(x) & \text{if $x \leq x_1$}.
\end{cases}
\end{equation}
It is worth presenting this formula as it is completely explicit. However, it
is not directly used in the remainder of the paper. On the other hand, a more
useful formulation can be expressed via recursion, as we develop next.    

\paragraph{Recursive formulation.} 

The following is an equivalent recursive formulation for the discrete derivative
operators in \eqref{eq:discrete_deriv}. We start by explicitly defining the
first order operator $\Delta_n$ (omitting the superscript here, for $k=1$, which
we will do commonly henceforth) by
\begin{equation}
\label{eq:discrete_deriv_rec1}
(\Delta_n f)(x) =
\begin{cases}
\displaystyle
\frac{f(x) - f(x_i)}{x - x_i} & \text{if $x \in (x_i,x_{i+1}]$} \\ 
f(x) & \text{if $x \leq x_1$}.
\end{cases}
\end{equation}
For $k \geq 2$, due to the recursion obeyed by divided differences, we can   
equivalently define the $k$th discrete derivative operator by 
\begin{equation}
\label{eq:discrete_deriv_rec2}
(\Delta^k_n f)(x) =
\begin{cases}
\displaystyle
\frac{(\Delta^{k-1}_n f)(x) - (\Delta^{k-1}_n f)(x_i)}{(x - x_{i-k+1})/k} 
& \text{if $x \in (x_i,x_{i+1}]$} \\  
(\Delta^{k-1}_n f)(x) & \text{if $x \leq x_k$}.
\end{cases}
\end{equation}
To express this recusion in a more compact form, we define the simple
difference operator \smash{$\widebar\Delta_n=\widebar\Delta(\cdot; x_{1:n})$} by
\begin{equation}
\label{eq:simple_diff}
(\widebar\Delta_n f)(x) =
\begin{cases}
f(x) - f(x_i) & \text{if $x \in (x_i,x_{i+1}]$} \\ 
f(x) & \text{if $x \leq x_1$}.
\end{cases}
\end{equation}
and for $k \geq 1$, we define the weight map \smash{$W^k_n=W^k(\cdot; x_{1:n})$}
by         
\begin{equation}
\label{eq:weight_map} 
(W^k_n f)(x) =
\begin{cases}
f(x) \cdot (x-x_{i-k+1})/k  & \text{if $x \in (x_i,x_{i+1}]$, $i \geq k$} \\ 
f(x) & \text{if $x \leq x_k$}.
\end{cases}
\end{equation}
Then the recursion in \eqref{eq:discrete_deriv_rec1},
\eqref{eq:discrete_deriv_rec2} can be rewritten as 
\begin{equation}
\label{eq:discrete_deriv_rec}
\begin{aligned}
\Delta_n &= (W_n)^{-1} \circ \widebar\Delta_n, \\
\Delta^k_n &= (W^k_n)^{-1} \circ \widebar\Delta_{n-k+1} \circ \Delta^{k-1}_n,  
\quad \text{for $k \geq 2$}.
\end{aligned}
\end{equation}
An important note: here, we denote by
\smash{$\widebar\Delta_{n-k+1}=\widebar\Delta(\cdot; x_{k:n})$}, the simple
difference operator in \eqref{eq:simple_diff} when we use the $n-k+1$ underlying
points $x_{k:n}=\{x_k,\ldots,x_n\}$ (rather than the original $n$ points 
$x_{1:n}=\{x_1,\ldots,x_n\}$). 

The compact recursive formulation in \eqref{eq:discrete_deriv_rec} is quite
useful, since it allows us to define a certain discrete integrator, which acts
as the inverse to discrete differentiation, to be described in the next
subsection. 

\paragraph{Evenly-spaced design points.} 

When the design points are evenly-spaced, $x_{i+1}-x_i=v>0$, for
$i=1,\ldots,n-1$, the discrete derivative operator \eqref{eq:discrete_deriv}
can be expressed at design points as a (scaled) forward difference, or
equivalently a (scaled) backward difference,   
$$
(\Delta^k_n f) (x_i) = 
\begin{cases}
F^k_v(x_i-kv) = B^k_v(x_i) & \text{if $i \geq k+1$} \\
F^i_v(x_1) = B^i_v(x_i) & \text{if $i < k+1$}.
\end{cases}
$$
where recall we use \smash{$F^k_v,B^k_v$} for the $k$th order forward and  
backward difference operators, respectively. In the case of evenly-spaced
design points, there are some special properties of discrete derivatives
(forward/backward differences), such as
$$
(\Delta^k_n f)(x_i) = (\Delta^d_{n-k+d} \, \Delta^{k-d}_n f)(x_i),
$$
for all $i$ and all $0 \leq d \leq k$. This unfortunately does not hold more
generally (for arbitrary designs); from \eqref{eq:discrete_deriv_rec1},
\eqref{eq:discrete_deriv_rec2}, we see that for arbitrary $x_{1:n}$, the above
property holds at $x_i$ with $d=1$ if and only if $(x_i-x_{i-k})/k =
x_i-x_{i-1}$. (Further, it should be noted that the above property never
holds---whether in the evenly-spaced case, or not---at points  $x \notin
x_{1:n}$.)  

\subsection{Discrete integration}

Consider the same setup as the last subsection, but now with the following
question as motivation: given $x \in [a,b]$, how might we use
$f(x_1),\ldots,f(x_n)$, along with $f(x)$, to approximate the $k$th integral
$(I^k f)(x)$, of $f$ at $x$?

We write $S^k(\cdot; x_{1:n})$ to denote an operator that maps a function $f$ to
a function $S^k(f; x_{1:n})$, which we call the {\it $k$th discrete integral}
(or the {\it discrete $k$th integral}) of $f$, to be defined below. As before,
we abbreviate $S^k_n=S^k(\cdot; x_{1:n})$. To define the function $S^k_n f$, we
take a recursive approach, mirroring our approach in \eqref{eq:simple_diff},
\eqref{eq:weight_map}, \eqref{eq:discrete_deriv_rec}. We start by defining the
simple cumulative sum operator \smash{$\widebar{S}_n = \widebar{S}(\cdot;
x_{1:n})$} by
\begin{equation}
\label{eq:cum_sum}
(\widebar{S}_n f)(x) =
\begin{cases}
\displaystyle
\sum_{j=1}^i f(x_j) + f(x) & \text{if $x \in (x_i,x_{i+1}]$} \\ 
f(x) & \text{if $x \leq x_1$}.
\end{cases}
\end{equation}
We then define the discrete integral operators by 
\begin{equation}
\label{eq:discrete_integ_rec}
\begin{aligned}
S_n &= \widebar{S}_n \circ W_n, \\
S^k_n &= S^{k-1}_n \circ \widebar{S}_{n-k+1} \circ W^k_n,
\quad \text{for $k \geq 2$}.
\end{aligned}
\end{equation}
An important note: as before, we abbreviate
\smash{$\widebar{S}_{n-k+1}=\widebar{S}(\cdot; x_{k:n})$} the discrete integral 
operator in \eqref{eq:discrete_integ_rec} over the $n-k+1$ underlying points
$x_{k:n}=\{x_k,\ldots,x_n\}$ (instead of over the original $n$ points
$x_1,\ldots,x_n$).

\paragraph{Linear combination formulation.}  

As with discrete derivatives (recall \eqref{eq:divided_diff_linear}), the
discrete integral of a function $f$ can be written in terms of linear
combinations of evaluations of $f$. This can be seen by working through the
definitions \eqref{eq:cum_sum} and \eqref{eq:discrete_integ_rec}, which would
lead to a formula for \smash{$(S^k_n f)(x)$} as a linear combination of
$f(x_1),\ldots,f(x_n)$ and $f(x)$, with the coefficients being $k$th order
cumulative sums of certain gaps between the design points $x_1,\ldots,x_n$ and
$x$. 

A subtle fact is that this linear combination can be written in a more explicit
form, that does not involve cumulative sums at all. Letting $h^{k-1}_j$,
$j=1,\ldots,n$ denote the falling factorial basis functions as in
\eqref{eq:ffb}, but of degree $k-1$, it holds that
\begin{equation}
\label{eq:discrete_integ_linear}
(S^k_n f)(x) = 
\begin{cases}
\displaystyle
\sum_{j=1}^k h^{k-1}_j(x) \cdot f(x_j) \;+
\sum_{j=k+1}^i h^{k-1}_j(x) \cdot \frac{x_j-x_{j-k}}{k} \cdot f(x_j) 
 \,+\, h^{k-1}_{i+1}(x) \cdot \frac{x-x_{i-k+1}}{k} \cdot f(x) & \\
& \hspace{-75pt}\text{if $x \in (x_i,x_{i+1}]$, $i \geq k$} \\
\displaystyle
\sum_{j=1}^i h^{k-1}_j(x) \cdot f(x_j) \,+\, h^{k-1}_{i+1}(x) \cdot f(x)  
& \hspace{-75pt}\text{if $x \in (x_i,x_{i+1}]$, $i <  k$} \\ 
f(x) & \hspace{-75pt}\text{if $x \leq x_1$},
\end{cases}
\end{equation}
The above is a consequence of results that we will develop in subsequent parts
of this paper: the inverse relationship between discrete differentation and
discrete integration (Lemma \ref{lem:discrete_deriv_integ_inv}, next), and the
dual relationship between discrete differentiation and the falling factorial
basis (Lemmas \ref{lem:ffb_discrete_deriv_extra_pp} and
\ref{lem:ffb_discrete_deriv_extra_poly}, later). We defer its proof to
Appendix \ref{app:discrete_integ_linear}. As with the discrete derivative
result \eqref{eq:discrete_deriv_linear}, it is worth presenting
\eqref{eq:discrete_integ_linear} because its form is completely explicit.
However, again, we note that this linear combination formulation is not itself
directly used in the remainder of this paper. 

\paragraph{Inverse relationship.}  

The next result shows an important relationship between discrete differentiation
and discrete integration: they are precisely inverses of each other. The proof
follows by induction and is given in Appendix
\ref{app:discrete_deriv_integ_inv}.

\begin{lemma}
\label{lem:discrete_deriv_integ_inv}
For any $k \geq 1$, it holds that $(\Delta^k_n)^{-1} = S^k_n$, that is, 
$\Delta^k_n S^k_n f = f$ and $S^k_n \Delta^k_n f = f$ for all functions $f$. 
\end{lemma}

\begin{remark}
It may be surprising, at first glance, that the $k$th order discrete derivative
operator $\Delta^k_n$ has an inverse at all. In continuous-time, by comparison,
the $k$th order derivative operator $D^k$ annihilates all polynomials of degree
$k$, thus we clearly cannot have $I^k D^k f = f$ for all $f$. Viewed as an
operator over all functions with sufficient regularity, $D^k$ only has a right
inverse, that is, $D^k I^k f = f$ for all $f$ (by the fundamental theorem of
calculus). The fact that $\Delta^k_n$ has a proper (both left and right) inverse
$S^k_n$ is due to the special way in which $\Delta^k_n f$ is defined towards the
left side of the underlying domain: recall that $(\Delta^k_n f)(x)$ does not
involve a divided difference of order $k$ for $x \in [a,x_k]$, but rather, a
divided difference of order $k-1$ for $x \in (x_{k-1},x_k]$, of order $k-2$ for
$x \in (x_{k-2},x_{k-1}]$, etc. This ``fall off'' in the order of the divided
difference being taken, as $x$ approaches the left boundary point $a$, is what
renders $\Delta^k_n$ invertible. For example, when $p$ is $kt$h degree
polynomial, we have $(\Delta^k_n p)(x)=0$ for $x \in (x_k,b]$, and yet the lower
order divided differences, $(\Delta^k_n p)(x)$ for $x \in [a,x_k]$, encode
enough information that we can recover $p$ via discrete integration. 
\end{remark}

\subsection{Constructing the basis}

We recall a simple way to construct the truncated power basis for splines.
Let us abbreviate $1_t=1_{(t,b]}$, that is, the step function with step at $t
\in [a,b]$,  
$$
1_t(x) = 1\{x > t\}.
$$
(The choice of left-continuous step function is arbitrary, but convenient for
our development). It can be easily checked by induction that for all $k \geq
0$,  
$$
(I^k 1_t)(x) = \frac{1}{k!} (x - t)^k_+,
$$
where recall $x_+=\max\{x,0\}$, and we denote by $I^0=\Id$, the identity map,
for notational convenience. We can thus see that the truncated power basis in
\eqref{eq:tpb}, for the space $\S^k(t_{1:r}, [a,b])$ of $k$th degree splines
with knot set $t_{1:r}$, can be constructed by starting with the polynomials
$x^{j-1}$, $j=1,\ldots,k+1$ and including the $k$th order antiderivatives of the
appropriate step functions,
\begin{equation}
\label{eq:tpb_integ}
\frac{1}{k!} (x-t_j)^k_+ = (I^k 1_{t_j})(x), \quad j=1,\ldots,r. 
\end{equation}

We now show that an analogous construction gives rise to the falling factorial
basis functions in \eqref{eq:ffb}. 

\begin{theorem}
\label{thm:ffb_discrete_integ}
For any $k \geq 0$, the piecewise polynomials in the $k$th degree falling
factorial basis, given in the second line of \eqref{eq:ffb}, satisfy     
\begin{equation}
\label{eq:ffb_discrete_integ}
\frac{1}{k!} \prod_{\ell=j-k}^{j-1} (x-x_\ell) \cdot 1\{x > x_{j-1}\} = 
(S^k_n 1_{x_{j-1}})(x),  \quad j=k+2,\ldots,n.
\end{equation}
Here, we use $S^0_n=\Id$, the identity map, for notational convenience.
\end{theorem}

Theorem \ref{thm:ffb_discrete_integ} shows that the falling factorial basis
functions arise from $k$ times discretely integrating step functions with jumps 
at $x_{k+1},\ldots,x_{n-1}$. These are nothing more than truncated Newton
polynomials, with the left-hand side in \eqref{eq:ffb_discrete_integ} being
$\eta(x; x_{(j-k):(j-1)}) 1\{x > x_{j-1}\} / k!$, using the compact notation for 
Newton poynomials, as defined in \eqref{eq:newton_poly}. 

Recalling that the discrete integrators are defined recursively, in
\eqref{eq:discrete_integ_rec}, one might guess that the result in
\eqref{eq:ffb_discrete_integ} can be established by induction on $k$. While
this is indeed true, the inductive proof for Theorem
\ref{thm:ffb_discrete_integ} does not follow a standard approach that one might
expect: it is not at all clear from the recursion in
\eqref{eq:discrete_integ_rec} how to express each \smash{$h^k_j$} in terms of a
discrete integral of \smash{$h^{k-1}_j$}. Instead, it turns out that we can
derive what we call a {\it lateral} recursion, where we express \smash{$h^k_j$}
as a weighted sum of \smash{$h^{k-1}_\ell$} for $\ell \geq j$, and similarly for
their discrete derivatives. This is the key driver behind the proof of Theorem
\ref{thm:ffb_discrete_integ}, and is stated next.

\begin{lemma}
\label{lem:ffb_lateral_rec}
For any $k \geq 1$, the piecewise polynomials in the $k$th degree falling
factorial basis, given in the second line of \eqref{eq:ffb}, satisfy the
following recursion. For each $d \geq 0$, $j \geq k+2$, and $x \in
(x_i,x_{i+1}]$, where $i \geq j-1$,        
\begin{equation}
\label{eq:ffb_lateral_rec}
(\Delta^d_n h^k_j)(x) = \sum_{\ell=j}^i (\Delta^d_n h^{k-1}_\ell)(x) \cdot
\frac{x_\ell-x_{\ell-k}}{k} + (\Delta^d_n h^{k-1}_{i+1})(x) \cdot
\frac{x-x_{i-k+1}}{k}.    
\end{equation}
Here, we use $\Delta^0_n=\Id$, the identity map, for notational convenience. 
\end{lemma}

The proof of Lemma \ref{lem:ffb_lateral_rec} is elementary and is deferred   
until Appendix \ref{app:ffb_lateral_rec}. We now show how it can be used to
prove Theorem \ref{thm:ffb_discrete_integ}.  

\allowdisplaybreaks
\begin{proof}[Proof of Theorem \ref{thm:ffb_discrete_integ}]  
Note that, by the invertibility of $S^k_n$, from Lemma
\ref{lem:discrete_deriv_integ_inv}, it suffices to show that for all $k \geq 0$, 
\begin{equation}
\label{eq:ffb_discrete_deriv}
(\Delta^k_n h^k_j)(x) = 1\{x > x_{j-1}\}, \quad j=k+2,\ldots,n.
\end{equation}
We proceed by induction on $k$. When $k=0$, the result is immediate from the 
definition of the falling factorial basis functions in \eqref{eq:ffb}. Assume
the result holds for the degree $k-1$ falling factorial basis. Fix $j \geq
k+2$. If $x \leq x_{j-1}$, then it is easy to check that \smash{$(\Delta^k_n
  h^k_j)(x) = 0$}. Thus let $x \in (x_i,x_{i+1}]$ where $i \geq j-1$. By the
recursive representation \eqref{eq:ffb_lateral_rec},
\begin{align*}
&(\Delta^{k-1}_n h^k_j)(x) - (\Delta^{k-1}_n h^k_j)(x_i) \\
&= \sum_{\ell=j}^i (\Delta^{k-1}_n h^{k-1}_\ell)(x)
  \cdot \frac{x_\ell-x_{\ell-k}}{k} + (\Delta^{k-1}_n h^{k-1}_{i+1})(x) 
  \cdot \frac{x-x_{i-k+1}}{k} - \sum_{\ell=j}^i (\Delta^{k-1}_n
  h^{k-1}_\ell)(x_i) \cdot \frac{x_\ell-x_{\ell-k}}{k} \\
&= \sum_{\ell=j}^i \big((\Delta^{k-1}_n h^{k-1}_\ell)(x) -
 (\Delta^{k-1}_n h^{k-1}_\ell)(x_i)\big) \cdot
  \frac{x_\ell-x_{\ell-k}}{k} + (\Delta^{k-1}_n h^{k-1}_{i+1})(x) \cdot
  \frac{x-x_{i-k+1}}{k} \\ 
&= \sum_{\ell=j}^i \big(1\{x>x_{\ell-1}\} - 1\{x_i>x_{\ell-1}\}\big)
  \cdot \frac{x_\ell-x_{\ell-k}}{k} + 1\{x>x_i\} \cdot \frac{x-x_{i-k+1}}{k}, 
\end{align*}
where in the last line we used the inductive hypothesis. As all indicators in
above line are equal to 1, the sum is equal to 0, and hence by the definition 
in \eqref{eq:discrete_deriv_rec2}, 
\begin{align*}
(\Delta^k_n h^k_j)(x) &= \frac{(\Delta^{k-1}_n h^k_j)(x) - 
  (\Delta^{k-1}_n h^k_j)(x_i)}{(x-x_{i-k+1})/k}  \\
&= \frac{(x-x_{i-k+1})/k}{(x-x_{i-k+1})/k} = 1. 
\end{align*}
This completes the proof. 
\end{proof}
\allowdisplaybreaks[0]

Now that we have constructed the piecewise polynomials in the $k$th degree
falling factorial basis functions, using $k$th order discrete integration of
step functions in \eqref{eq:ffb_discrete_integ}, we can add any set of $k+1$
linearly independent $k$th degree polynomials to these piecewise polynomials to
form an equivalent basis: the falling factorial basis. For example, the
monomials $x^{j-1}$, $j=1,\ldots,k+1$ would be a simple choice. However, as
originally defined in \eqref{eq:ffb}, we used a different set of $k$th degree
polynomials: Newton polynomials of degrees $0,\ldots,k$. This is a natural
pairing, because the falling factorial basis can be seen as a set of truncated
Newton polynomials; furthermore, as we show later in Section
\ref{sec:dual_basis}, this choice leads to a convenient dual basis to the
falling factorials \smash{$h^k_j$}, $j=1,\ldots,n$. 

\paragraph{Evenly-spaced design points.}

When the design points are evenly-spaced, $x_{i+1}-x_i=v>0$, for
$i=1,\ldots,n-1$, the falling factorial basis functions in \eqref{eq:ffb} reduce
to  
\begin{gather*}
h^k_j(x) = \frac{1}{(j-1)!} (x-x_1)_{j-1,v}, 
\quad j=1,\ldots,k+1, \\
h^k_j(x) = \frac{1}{k!} (x-x_{j-k})_{k,v} \cdot 1\{x > x_{j-1}\}, 
\quad j=k+2,\ldots,n, 
\end{gather*}
where recall we write $(x)_{\ell,v} = x (x-v) \cdots (x-(\ell-1)v)$ for the
falling factorial polynomial of degree $\ell$ with gap $v$, which we interpret 
to be equal to 1 when $\ell=0$. This connection inspired the name of these
basis functions as given in \citet{tibshirani2014adaptive,wang2014falling}.
Further, it follows by a simple inductive argument (for example, see Lemma 2 in 
\citet{tibshirani2014adaptive}) that, evaluated at a design point $x_i$, the
basis functions become  
\begin{gather*}
h^k_j(x_i) = v^{j-1} \sigma^{j-1}_{i-j+1} \cdot 1\{i > j-1\}, 
\quad j=1,\ldots,k+1, \\ 
h^k_j(x_i) = v^k \sigma^k_{i-j+1} \cdot 1\{i > j-1\}, 
\quad j=k+2,\ldots,n,
\end{gather*}
where we define $\sigma^0_i=1$ for all $i$ and \smash{$\sigma^\ell_i =
  \sum_{j=1}^i \sigma^{\ell-1}_j$}, the $\ell$th order cumulative sum of
$1,\ldots,1$ (repeated $i$ times).


\section{Smoothness properties}
\label{sec:smoothness}

We study some properties relating to the structure and smoothness of functions
in the span of the falling factorial basis. To begin, we point out an important
{\it lack of} smoothness in the usual sense: the piecewise polynomial falling
factorial basis functions \smash{$h^k_j$}, $j=k+2,\ldots,n$, given in the second
line of \eqref{eq:ffb}, do not have continuous derivatives. To see this, write,
for each $j \geq k+2$, 
$$
h^k_j(x) = \frac{1}{k!} \eta(x; x_{(j-k):(j-1)}) \cdot 1\{x > x_{j-1}\}, 
$$
where recall \smash{$\eta(x; x_{(j-k):(j-1)}) = \prod_{i=j-k}^{j-1} (x-x_i)!$}
is the $k$th degree Newton polynomial, as introduced in \eqref{eq:newton_poly}. 
Note that for any $0 \leq d \leq k$, and $x < x_{j-1}$, we have 
\smash{$(D^d h^k_j)(x) = 0$}, whereas for $x>x_{j-1}$, 
\begin{equation}
\label{eq:ffb_deriv}
(D^d h^k_j)(x) = \frac{d!}{k!} 
 \sum_{\substack{I \subseteq (j-k):(j-1) \\ |I| = k-d}} \eta(x; x_I).
\end{equation}
where for a set $I$, we let $x_I=\{x_i : i\in i\}$. We can hence see that, for
$d \geq 1$, 
\begin{equation}
\label{eq:ffb_deriv_lim}
\lim_{x \to x_{j-1}^+}  (D^d h^k_j)(x) = \frac{d!}{k!} 
 \sum_{\substack{I \subseteq (j-k):(j-2) \\ |I| = k-d}} \eta(x_{j-1}; x_I) > 0, 
\end{equation}
which is strictly positive because the design points are assumed to be distinct, 
and hence the left and right derivatives do not match at $x_{j-1}$.  

In other words, we have just shown that the falling factorial basis functions 
$h^k_j$, $j=1,\ldots,n$, when $k \geq 2$, are not $k$th degree splines, as
their derivatives lack continuity at the knot points. On the other hand, as we
show next, the falling factorial functions are not void of smoothness, it is
simply expressed in a different way: their {\it discrete} derivatives end up
being continuous at the knot points. 

\subsection{Discrete splines}
\label{sec:discrete_splines}

We begin by extending the definition of discrete splines in Definition
\ref{def:discrete_spline_even} to the setting of arbitrary design points, where 
naturally, divided differences appear in place of forward differences.

\begin{definition}
\label{def:discrete_spline}
For an integer $k \geq 0$, design points $a \leq x_1 < \cdots < x_n \leq b$
(that define the operators \smash{$\Delta^\ell_n=\Delta^\ell(\cdot; x_{1:n})$},
$\ell=1,\ldots,k-1$), and knots $a=t_0 < t_1 < \cdots < t_r < t_{r+1}=b$ such
that $t_{1:r} \subseteq x_{1:n}$ and $t_1 \geq x_{k+1}$, we define the space of 
{\it $k$th degree discrete splines} on $[a,b]$ with knots $t_{1:r}$, denoted 
\smash{$\DS^k_n(t_{1:r}, [a,b])$}, to contain all functions $f$ on $[a,b]$ such
that 
\begin{equation}
\label{eq:discrete_spline}
\begin{gathered}
\text{for each $i=0,\ldots,r$, there is a $k$th degree polynomial $p_i$ 
  such that $f|_{I_{i,v}}=p_i$, and} \\
\text{for each $i=1,\ldots,r$, it holds that $(\Delta^\ell_n p_{i-1})(t_i) = 
  (\Delta^\ell_n p_i)(t_i)$, $\ell=0,\ldots,k-1$}, 
\end{gathered}
\end{equation}
where $I_0=[t_0,t_1]$ and $I_i=(t_i,t_{i+1}]$, $i=1,\ldots,r$. 
\end{definition}


\begin{remark}
It is worth emphasizing again that we treat (in Definition
\ref{def:discrete_spline}, and throughout) a discrete spline as a {\it
  function}, defined on the continuum interval $[a,b]$, whereas the classical
literature (recall Definition \ref{def:discrete_spline_even}) treats a discrete
spline as a {\it vector}: a sequence of function evaluations made on a discrete
(and evenly-spaced) subset $[a,b]_v \subseteq [a,b]$.
\end{remark}

\begin{remark}
When $k=0$ or $k=1$, the space \smash{$\DS^k_n(t_{1:r}, [a,b])$} of $k$th 
degree discrete splines with knots $t_{1:r}$ is equal to the space
$\S^k(t_{1:r}, [a,b])$ of $k$th degree splines with knots $t_{1:r}$, as the
conditions in \eqref{eq:spline} and \eqref{eq:discrete_spline} match
(for $k=0$, there is no smoothness condition at the knots, and for $k=1$,
there is only continuity at the knots). When $k \geq 2$, this is no longer true,
and the two spaces are different; however, they contain ``similar'' piecewise 
polynomial functions for large $n,r$, which will be made precise in Section
\ref{sec:approximation}.  
\end{remark}

Now denote the linear span of the $k$th degree falling factorial basis functions 
defined in \eqref{eq:ffb} by  
\begin{equation}
\label{eq:ffb_span}
\cH^k_n = \spa\{h^k_1,\ldots,h^k_n\} = \Bigg\{ \sum_{j=1}^n \alpha_j h^k_j :
\alpha_j \in \R, \; j=1,\ldots,n \Bigg\}. 
\end{equation}
Next we show that the span of falling factorial basis functions is a space of
discrete splines. The arguments are similar to those for the case of
evenly-spaced design points, see, for example, Theorem 8.51 of
\citet{schumaker2007spline}.   

\begin{lemma}
\label{lem:ffb_span}
For any $k \geq 0$, the span \smash{$\cH^k_n$} of the $k$th degree falling
factorial basis functions, in \eqref{eq:ffb_span}, can be equivalently
represented as 
$$
\cH^k_n = \DS^k_n (x_{(k+1):(n-1)}, [a,b]),
$$
the space of $k$th degree discrete splines on $[a,b]$ with knots in
$x_{(k+1):(n-1)}=\{x_{k+1},\ldots,x_{n-1}\}$.
\end{lemma}

\begin{proof}
We first show that each basis function \smash{$h^k_j$}, $j=1,\ldots,n$ is an
element of \smash{$\DS^k_n(x_{(k+1):(n-1)}, [a,b])$}. Note that
\smash{$h^k_j$}, $j=1,\ldots,k+1$ are clearly $k$th degree discrete splines
because they are $k$th degree polynomials. Fix $j \geq k+2$. The function 
\smash{$h^k_j$} has just one knot to consider, at $x_{j-1}$. Observe
$$
h^k_j|_{[a,x_{j-1}]} = 0, \quad \text{and} \quad
h^k_j|_{[x_{j-1},b]} = \frac{1}{k!} \eta(\cdot; x_{(j-k):(j-1)}).
$$
Recall the property \eqref{eq:newton_poly_divided_diff} of divided differences
of Newton polynomials; this gives \smash{$(\Delta^\ell_n \eta(\cdot;
x_{(j-k):(j-1)}))(x_{j-1}) = 0$} for $\ell=0,\ldots,k-1$, certifying the
required property \eqref{eq:discrete_spline} for a $k$th degree discrete spline.

It is straightforward to show from the structure of their supports that
\smash{$h^k_j$}, $j=1,\ldots,n$ are linearly independent (we can evaluate them
at the design points $x_{1:n}$, yielding a lower triangular matrix, which
clearly has linearly independent columns). Furthermore, a standard
dimensionality argument shows that the linear space
\smash{$\DS^k_n(x_{(k+1):(n-1)}, [a,b])$} has dimension $(n-k-1)+(k+1)=n$ (we
can expand any function in this space as a linear combination of piecewise
polynomials the segments $I_0,\ldots,I_{n-k-1}$ then subtract the number of
constraints at the knot points). Thus the span of \smash{$h^k_j$},
$j=1,\ldots,n$ is all of \smash{$\DS^k_n(x_{(k+1):(n-1)}, [a,b])$}, completing
the proof.
\end{proof}

As we saw in \eqref{eq:ffb_deriv_lim}, functions in the span of the falling
factorial basis do not have continuous derivatives, and thus lack the
smoothness of splines, in this particular sense. However, as Lemma
\ref{lem:ffb_span} reveals, functions in this span are in fact discrete splines;
therefore they have an equal number of constraints (as splines) on their degrees
of freedom, and this is just expressed in a different way (using discrete 
derivatives in place of derivatives).    


\subsection{Matching derivatives}
\label{sec:deriv_match}

In this subsection, we investigate which kinds of functions $f$ have discrete
$k$th derivatives that everywhere match their $k$th derivatives,   
\begin{equation}
\label{eq:deriv_match}
(\Delta^k_n f)(x) = (D^k f)(x), \quad \text{for $x \in (x_k,b]$}. 
\end{equation}
Notice that, although we call the property \eqref{eq:deriv_match} an
``everywhere'' match of derivatives, we restrict our consideration to
$x \in (x_k,b]$. This is because for the $k$th discrete derivative 
operator \eqref{eq:discrete_deriv}, recall, it is only for $x \in 
(x_k,b]$ that $\Delta^k_n(f)$ is defined in terms of a $k$th divided
difference (for $x \in [a,x_k]$, it is defined in terms of a lower order
divided difference for the purposes of invertibility). 

It is a well-known fact that a $k$th degree polynomial,
\smash{$p(x)=\sum_{j=0}^k c_j x^j$}, has a $k$th divided difference equal to its
leading coefficient, with respect to any choice of $k+1$ distinct centers
$z_1,\ldots,z_{k+1}$, 
\begin{equation}
\label{eq:deriv_match_poly}
p[z_1,\ldots,z_{k+1}] = c_k = \frac{1}{k!} (D^k p)(x), \quad \text{for all
  $x$}. 
\end{equation}
(See, for example, Theorem 2.51 in \citet{schumaker2007spline}.)  Hence degree 
$k$ polynomials satisfy the matching derivatives property \eqref{eq:deriv_match} 
(note that this covers degree $\ell$ polynomials, with $\ell \leq k$, for which 
both sides in \eqref{eq:deriv_match} are zero). 

What about piecewise polynomials?  By the same logic, a $k$th degree piecewise
polynomial function $f$ will have a discrete $k$th derivative matching its
$k$th derivative at a point $x$, {\it provided that $f$ evaluates to a single 
  polynomial over the centers $x_{i-k+1},\ldots,x_i,x$} used to define
$(\Delta^k_n f)(x)$. But, if $x_{i-k+1},\ldots,x_i,x$ straddle (at least)
two neighboring segments on which $f$ is a different polynomial, then this will
not generally be true. Take as an example the truncated power function $g(x)
= (x-t)^k_+ / k!$, for $k \geq 2$. Let $x \in (x_i,x_{i+1}]$. Consider three
cases. In the first, $x \leq t$. Then\footnote{Here we are taking $(D^k g)(x)
  = 1\{x > t\}$, the choice of left-continuous step function being arbitrary but
  convenient, and consistent with our treatment of the falling factorial
  functions.}  
$$
(\Delta^k_n g)(x) = (D^k g)(x) = 0.
$$
In the second case, $x > t$ and $x_{i-k+1} \geq t$. Then 
$$
(\Delta^k_n g)(x) = (D^k g)(x) = 1.
$$
In the third case, $x > t$ and $x_{i+k-1} < t$.\footnote{Note that if $t$ is one
  of the design points $x_{1:n}$, then this case can only occur when $k \geq
  2$ (when $k=1$, we have $x_{i+k-1}=x_i$, which is defined to be the largest 
  design point strictly less than $x$, thus we cannot have $x_i<t<x$).}
Then $(D^k f)(x) = 1$, but $(\Delta^k_n f)(x)$ will vary between 0 and 1. See
Figure \ref{fig:deriv} for a simple empirical example. To summarize: if $x$ is
far enough from the underlying knot $t$ in the truncated power function---either
to the left of $t$, or to the right of $t$ and separated by $k$ underlying
design points---then the $k$th discrete derivative and $k$th derivative at $x$
will match; otherwise, they will not. (This restriction is quite problematic
once we think about trying to match derivatives \eqref{eq:deriv_match} for a
$k$th degree spline with with knots at the design points.) 

\begin{figure}[tb]
\centering
\includegraphics[width=0.495\textwidth]{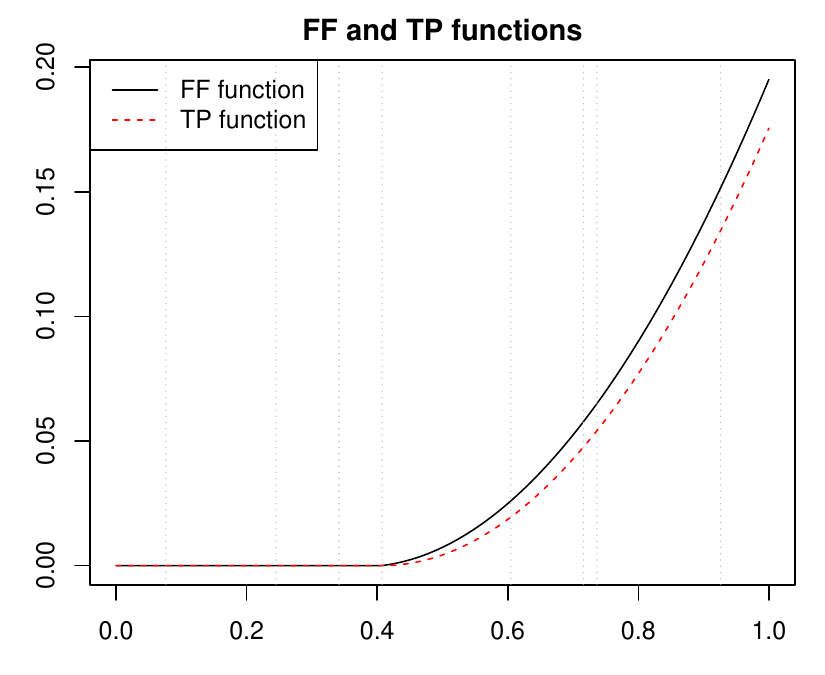}
\includegraphics[width=0.495\textwidth]{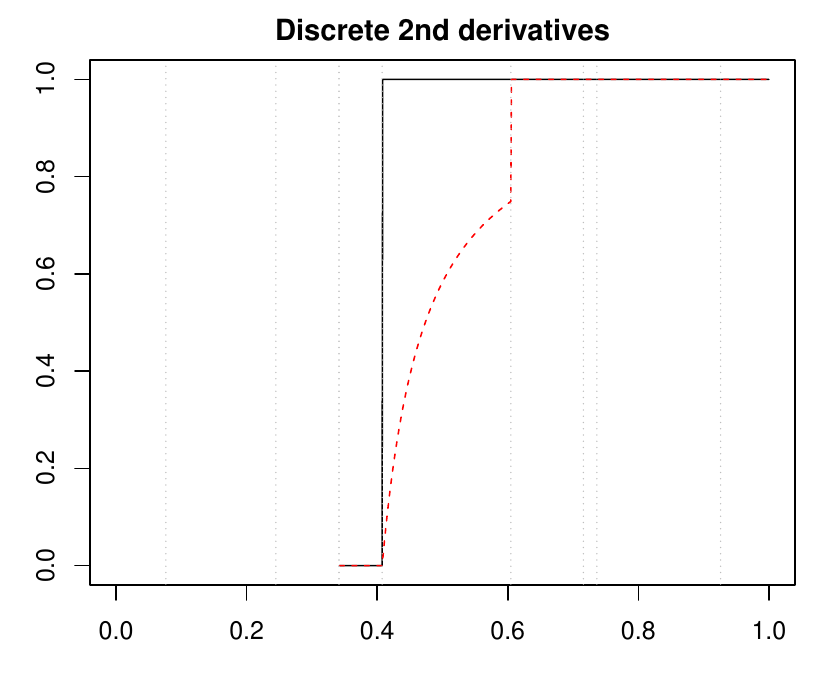}
\caption{\small Left panel: falling factorial (FF) function
  \smash{$\frac{1}{2}(x-x_3)(x-x_4)_+$} in black and truncated power (TP)
  function \smash{$\frac{1}{2}(x-x_4)^2_+$} in dashed red. The $n=8$ design 
  points are marked by dotted gray vertical lines. Both functions have a knot at
  $x_4$, and both have 2nd derivative equal to $1\{x>x_4\}$. Right panel:
  discrete 2nd derivatives of the same functions (FF in black, TP in dashed
  red), for $x \in [x_3,1]$. Note the discrete 2nd derivative of the FF
  function matches its 2nd derivative everywhere (this being $1\{x>x_4\}$), but 
  this is not true for the TP function, specifically, it fails for $x \in
  (x_4,x_5]$.}     
\label{fig:deriv}
\end{figure}

A remarkable fact about the $k$th degree falling factorial basis functions
\eqref{eq:ffb} is that their $k$th discrete derivatives and $k$th derivatives
match {\it at all $x$}, regardless of how close $x$ lies to their underlying
knot points. This result was actually already established in
\eqref{eq:ffb_discrete_deriv}, in the proof of Theorem
\ref{thm:ffb_discrete_integ} (this is for the piecewise polynomial basis
functions, and for the polynomial basis functions, it follows from the property
\eqref{eq:deriv_match_poly} on discrete derivatives of polynomials). For
emphasis, we state the full result next as a corollary. We also prove a
converse result.  

\begin{corollary}
\label{cor:deriv_match}
For any $k \geq 0$, each of the $k$th degree falling factorial basis functions
in \eqref{eq:ffb} have matching $k$th discrete derivatives and $k$th
derivatives, at all $x > x_k$, as in \eqref{eq:deriv_match}. Hence, by
linearity, any function in the span \smash{$\cH^k_n$} of the $k$th degree
falling factorial basis \eqref{eq:ffb_span}, that is, any $k$th degree discrete  
spline with knots in $x_{(k+1):(n-1)}$, also satisfies \eqref{eq:deriv_match}. 

Conversely, if $f$ is a $k$th degree piecewise polynomial with knots in
$x_{(k+1):(n-1)}$ and $f$ satisfies property \eqref{eq:deriv_match}, then $f$
must be in the span \smash{$\cH^k_n$} of the $k$th degree falling factorial
basis \eqref{eq:ffb_span}, that is, $f$ must be a $k$th degree discrete spline. 
\end{corollary}

\begin{proof}
As already discussed, the first statement was already shown, for the piecewise
polynomial basis functions, in \eqref{eq:ffb_discrete_deriv} in the proof of
Theorem \ref{thm:ffb_discrete_integ}, and for the polynomial basis functions, it
is a reflection of the basic fact \eqref{eq:deriv_match_poly}. To prove the
converse statement, observe that if $f$ is a $k$th degree piecewise polynomial
and has knots in \smash{$x_{(k+1):(n-1)}$}, then its $k$th derivative is
piecewise constant with knots in \smash{$x_{(k+1):(n-1)}$}, and thus can be 
written as 
$$
(D^k f)(x) = \alpha_0 + \sum_{j=k+2}^n \alpha_j 1\{x > x_{j-1}\},
\quad \text{for $x \in [a,b]$}.
$$
for coefficients $\alpha_0,\alpha_{k+2},\ldots,\alpha_n$. But because $f$ 
satisfies property \eqref{eq:deriv_match}, we have
$$
(\Delta^k_n f)(x) = \alpha_0 + \sum_{j=k+2}^n \alpha_j 1\{x > x_{j-1}\},
\quad \text{for $x \in (x_k,b]$}.
$$
Inverting using Lemma \ref{lem:discrete_deriv_integ_inv}, then using linearity
of $S^k_n$, and Theorem \ref{thm:ffb_discrete_integ}, we have 
$$
f(x) = \alpha_0 (S^k_n 1)(x) + \sum_{j=k+2}^n \alpha_j h^k_j(x), 
\quad \text{for $x \in (x_k,b]$}.
$$
A staightforward inductive argument shows that $S^k_n 1$ is a $k$th degree
polynomial. Therefore, the above display, along with the fact that $f$ must
be a $k$th degree polynomial on $[a,x_k]$ (it is a $k$th degree piecewise
polynomial and its first knot is at $x_{k+1}$), shows that $f$ lies in the span
of the $k$th degree falling factorial basis \eqref{eq:ffb}.  
\end{proof}


\begin{remark}
In light of the discussion preceeding Corollary \ref{cor:deriv_match}, it is
somewhat remarkable that a $k$th degree piecewise polynomial with knots at each
$x_{k+1},\ldots,x_{n-1}$ can have a matching $k$th discrete derivative and $k$th 
derivative, at all $x \in (x_k,b]$. Recall that for a $k$th degree polynomial,
its $k$th discrete derivative and $k$th derivative match at all points, stemming
from the property \eqref{eq:deriv_match_poly} of divided differences of
polynomials. For a $k$th degree piecewise polynomial $f$ with knots
$x_{k+1},\ldots,x_{n-1}$, we have {\it just one evaluation of $f$ on each
segment in which $f$ is a polynomial}, yet Corollary \ref{cor:deriv_match} says
that the $k$th divided difference $f[x_{i-k+1},\ldots,x_i,x]$ still perfectly
reflects the local structure of $f$ around $x$, in such a way that $(\Delta^k_n
f)(x)=(D^k f)(x)$. This is a very different situation than that in
\eqref{eq:deriv_match_poly}, and only happens when $f$ has a particular
piecewise polynomial structure---given by the span of falling factorial
functions.
\end{remark}

\begin{remark}
It is interesting to emphasize the second part of Corollary
\ref{cor:deriv_match}. As highlighted in \eqref{eq:ffb_deriv_lim}, the $k$th
degree falling factorial functions \eqref{eq:ffb} have discontinuous lower order
derivatives at their knots, and hence so do functions in their span
\eqref{eq:ffb_span}, that is, so do $k$th degree discrete splines with knots in
\smash{$x_{(k+1):(n-1)}$}. This may seem like an undesirable property of a
piecewise polynomial (although discrete splines do enjoy continuity in discrete
derivatives across their knot points). However, if we want our piecewise
polynomial to satisfy the matching derivatives property \eqref{eq:deriv_match},
then Corollary \ref{cor:deriv_match} tells us that such discontinuities are
inevitable, as discrete splines are the {\it only} ones that satisfy this
property. 
\end{remark}

\begin{remark}
The result in \eqref{eq:ffb_discrete_deriv} can be shown to hold at the design
points $x=x_i$, $i=k+1,\ldots,n$ by directly invoking the fact in
\eqref{eq:newton_poly_divided_diff}, on divided differences of Newton
polynomials. In other words, that the matching derivatives property
\eqref{eq:deriv_match} holds for the $k$th degree falling factorial functions at
$x=x_i$, $i=k+1,\ldots,n$ has a simple proof based on the fact they are
truncated Newton polynomials, and \eqref{eq:newton_poly_divided_diff}. However,
the fact that it is true {\it for all} $x \in (x_k,b]$ is much less
straightforward, and is due to the lateral recursion obeyed by these basis
functions, from Lemma \ref{lem:ffb_lateral_rec}. 
\end{remark}

\section{Dual basis}
\label{sec:dual_basis}

In this section, we construct a natural dual basis to the falling factorial
basis in \eqref{eq:ffb}, based on discrete derivatives. We begin by building
on the matching $k$th order derivatives property
\eqref{eq:ffb_discrete_deriv} of the piecewise polynomial functions in the $k$th
degree falling factorial basis, to investigate discrete $(k+1)$st order discrete 
derivatives of such functions.    

\subsection{Discrete differentiation of one ``extra'' order}

The next lemma reveals a special form for the $(k+1)$st order discrete
derivatives of the piecewise polynomials in the $k$th degree falling factorial
basis.

\begin{lemma}
\label{lem:ffb_discrete_deriv_extra_pp}
For any $k \geq 0$, the piecewise polynomials in the $k$th degree falling
factorial basis, given in the second line of \eqref{eq:ffb}, satisfy for each $j
\geq k+2$ and $x \in [a,b]$,  
\begin{equation}
\label{eq:ffb_discrete_deriv_extra_pp}
(\Delta^{k+1}_n h^k_j)(x) = 
\begin{cases}
\frac{k+1}{x-x_{j-k-1}} & \text{if $x \in (x_{j-1}, x_j]$} \\  
0 & \text{otherwise}.
\end{cases}
\end{equation}
\end{lemma}

\begin{proof}
For $x \leq x_{j-1}$, it is easy to see \smash{$\Delta^{k+1}_n h^k_j(x)=0$}.
Thus consider $x \in (x_i,x_{i+1}]$ with $i \geq j-1$. By definition,
\begin{align*}
(\Delta^{k+1}_n h^k_j)(x) 
&= \frac{(\Delta^k_n h^k_j)(x) - (\Delta^k_n h^k_j)(x_i)} {(x-x_{i-k})/(k+1)} \\
&= \frac{1\{x > x_{j-1}\} - 1\{x_i > x_{j-1}\}}{(x-x_{i-k})/(k+1)}.
\end{align*}
where in the second line we used property \eqref{eq:ffb_discrete_deriv}, from
the proof of Theorem \ref{thm:ffb_discrete_integ}. When $x > x_j$, we have $i
\geq j$, and both indicators above are equal to 1, so \smash{$(\Delta^{k+1}_n  
  h^k_j)(x)$} = 0. Otherwise, when $x \in (x_{j-1},x_j]$, we have $i = j-1$,
and only the first indicator above is equal to 1, therefore we get
\smash{$(\Delta^{k+1}_n h^k_j)(x) = (k+1)/(x-x_{j-k-1})$}, as claimed.  
\end{proof}

Meanwhile, for the pure polynomials in the $k$th degree falling factorial basis,
their $(k+1)$st order discrete derivatives take an even simpler form. 

\begin{lemma}
\label{lem:ffb_discrete_deriv_extra_poly}
For any $k \geq 0$, the polynomial functions in the $k$th degree falling
factorial basis, given in the first line of \eqref{eq:ffb}, satisfy for each $j
\leq k+1$,  
\begin{equation}
\label{eq:ffb_discrete_deriv_extra_poly}
\Delta^{k+1}_n h^k_j =
\begin{cases}
1_{(x_{j-1},x_j]} &\text{if $j \geq 2$} \\
1_{[a,x_1]} & \text{if $j = 1$}.
\end{cases}
\end{equation}
\end{lemma}

\begin{proof}
Fix any $j \leq k+1$. If $x > x_j$, then \smash{$(\Delta^{k+1}_n h^k_j)(x)$} is 
given by a $(j+1)$st order divided difference of the $(j-1)$st degree polynomial 
\smash{$h^k_j$}, and is hence equal to 0. If $x \in (x_i, x_{i+1}]$ for $i < j$
(or $x \in [a,x_1]$ when $i=0$), then  
$$
(\Delta^{k+1}_n h^k_j)(x) = \eta(\cdot; x_{1:(j-1)}) [x_1,\ldots,x_i,x] = 
\begin{cases}
1 & \text{if $i=j-1$} \\
0 & \text{otherwise},
\end{cases}
$$
where we have used the important property of divided differences of Newton 
polynomials in \eqref{eq:newton_poly_divided_diff}. Observe that $i=j-1$ implies 
$x \in (x_{j-1},x_j]$ (or $x \in [a,x_1]$ when $j=1$), which completes the proof. 
\end{proof}

\subsection{Constructing the dual basis}

Simply identifying natural points of evaluation for the discrete derivative
results in Lemmas \ref{lem:ffb_discrete_deriv_extra_pp} and
\ref{lem:ffb_discrete_deriv_extra_poly} gives us a dual basis for the $k$th
degree falling factorial basis. The proof of the next lemma is immediate and
hence omitted. 

\begin{lemma}
\label{lem:dual_basis}
For any $k \geq 0$, define the linear functionals \smash{$\lambda^k_i$}, 
$i=1,\ldots,n$ according to
\begin{equation}
\label{eq:dual_basis}
\begin{gathered}
\lambda^k_i f = (\Delta^{k+1}_n f)(x_i), \quad i=1,\ldots, k+1, \\
\lambda^k_i f = (\Delta^{k+1}_n f)(x_i) \cdot \frac{x_i-x_{i-k-1}}{k+1}, \quad 
i=k+2,\ldots,n. 
\end{gathered}
\end{equation}
Then \smash{$\lambda^k_i$}, $i=1,\ldots,n$ is a dual basis to the $k$th degree
falling factorial basis in \eqref{eq:ffb}, in the sense that for all $i,j$, 
\begin{equation}
\label{eq:dual_prop}
\lambda^k_i h^k_j = 
\begin{cases}
1 & \text{if $i=j$} \\
0 & \text{otherwise}.
\end{cases}
\end{equation}
\end{lemma}

One general property of a dual basis is that it allows us to explicitly
compute coefficients in a corresponding basis expansion: if
\smash{$f=\sum_{i=1}^n \alpha_i h^k_i$}, then for each $i=1,\ldots,n$, applying
the linear functional \smash{$\lambda^k_i$} to both sides gives \smash{$\alpha_i
  = \lambda^k_i f$}, by \eqref{eq:dual_prop}. Next we develop the implications
of this for interpolation with the falling factorial basis.  

\subsection{Falling factorial interpolation}
\label{sec:ffb_interp}

An immediate consequence of the dual basis developed in Lemma
\ref{lem:dual_basis} is the following interpolation result. 

\begin{theorem}
\label{thm:ffb_interp}
Let $y_i$, $i=1,\ldots,n$ be arbitrary. For any $k \geq 0$, we can construct a
$k$th degree discrete spline interpolant \smash{$f \in \cH^k_n$} with  
knots in \smash{$x_{(k+1):(n-1)}$}, satisfying $f(x_i)=y_i$, $i=1,\ldots,n$, via  
\begin{equation}
\label{eq:ffb_interp}
f(x) = \sum_{i=1}^{k+1} (\Delta^{k+1}_n f)(x_i) \cdot h^k_i(x) \;+
\sum_{i=k+2}^n (\Delta^{k+1}_n f)(x_i) \cdot \frac{x_i-x_{i-k-1}}{k+1} 
\cdot h^k_i(x). 
\end{equation}
(Note that the discrete derivatives \smash{$(\Delta^{k+1}_n f)(x_i)$},
$i=1,\ldots,n$ above, though notationally dependent on $f$, actually only 
depend on the points $y_i$, $i=1,\ldots,n$.)  Moreover, the representation in  
\eqref{eq:ffb_interp} is unique.
\end{theorem}

\begin{proof}
As \smash{$\cH^k_n$} is $n$-dimensional, we can find a unique interpolant $f$ 
passing through any $n$ points $y_i$, $i=1,\ldots,n$. Let
\smash{$f=\sum_{i=1}^n \alpha_i h^k_i$}. As explained after Lemma
\ref{lem:dual_basis}, for each $i=1,\ldots,n$, we have \smash{$\alpha_i =
  \lambda^k_i f$} for the dual basis defined in \eqref{eq:dual_basis}, that is,
\smash{$\alpha_i = (\Delta^{k+1}_n f)(x_i)$} for $i \leq k+1$ and
\smash{$\alpha_i = (\Delta^{k+1}_n f)(x_i) \cdot (x_i-x_{i-k-1})/(k+1)$} for $i
\geq k+2$.
\end{proof}

\begin{remark}
The result in \eqref{eq:ffb_interp} can be written in a more explicit form,
namely,  
\begin{equation}
\label{eq:ffb_interp_explicit}
f(x) = \sum_{i=1}^{k+1} f [x_1,\ldots,x_i] \cdot \eta(x; x_{1:(i-1)}) \;+
\sum_{i=k+2}^n f [x_{i-k-1}, \ldots, x_i] \cdot (x_i-x_{i-k-1}) \cdot 
\eta_+(x; x_{(i-k):(i-1)}),
\end{equation}
where we introduce the notation \smash{$\eta_+(x; t_{1:r})=\eta(x; t_{1:r})
  \cdot 1\{x > \max(t_{1:r}) \}$} for a truncated Newton polynomial. In this 
form, we can see it as a natural extension of Newton interpolation in
\eqref{eq:newton_interp}. The latter \eqref{eq:newton_interp} constructs a
polynomial of degree $n-1$ passing through any $n$ points, whereas the former
\eqref{eq:ffb_interp_explicit} separates the degree of the polynomial from the
number of points, and allows us to construct a piecewise polynomial
(specifically, a discrete spline) of degree $k$, with $n-k-1$ knots, passing
through any $n$ points. A nice feature of this generalization is that it
retains the property of the classical Newton formula that the coefficients in
the interpolatory expansion are simple, explicit, and easy to compute (they are
just based on sliding divided differences).
\end{remark}

\subsection{Implicit form interpolation}
\label{sec:ffb_interp_implicit}

To proceed in an opposite direction from our last remark, we now show that the 
interpolation result in Theorem \ref{thm:ffb_interp} can be written in a more
implicit form. 

\begin{corollary}
\label{cor:ffb_interp_implicit}
Let $y_i$, $i=1,\ldots,n$ be arbitrary. For any $k \geq 0$, we can construct a
$k$th degree discrete spline interpolant \smash{$f \in \cH^k_n$} with  
knots in \smash{$x_{(k+1):(n-1)}$}, satisfying $f(x_i)=y_i$, $i=1,\ldots,n$, in
the following manner. For $x \in [a,b] \setminus x_{1:n}$, if $x > x_{k+1}$ and
$i$ is the smallest index such that $x_i>x$ (with $i=n$ when $x>x_n$), then
$f(x)$ is the unique solution of the linear system      
\begin{equation}
\label{eq:ffb_interp_implicit1}
f[x_{i-k}, \ldots, x_i, x] = 0.
\end{equation}
If instead $x < x_{k+1}$, then $f(x)$ is the unique solution of the linear
system 
\begin{equation}
\label{eq:ffb_interp_implicit2}
f[x_1, \ldots, x_{k+1}, x] = 0.
\end{equation}
We note that \eqref{eq:ffb_interp_implicit1}, \eqref{eq:ffb_interp_implicit2}
are each linear systems in just one unknown, $f(x)$.
\end{corollary}

\begin{proof}
First consider the case $x>x_{k+1}$, and $x \in (x_{i-1},x_i)$ with $i \leq n$. 
Define sequences of augmented design points and target points by  
\begin{equation}
\label{eq:augmented1}
\begin{alignedat}{7}
&\tilde{x}_1 = x_1, \;\; &&\ldots, \;\; 
&&\tilde{x}_{i-1} = x_{i-1}, \;\;  
&&\tilde{x}_i = x, \;\; 
&&\tilde{x}_{i+1} = x_i, \;\; &&\ldots, \;\; 
&&\tilde{x}_{n+1} = x_n, \;\; \\
&\tilde{y}_1 = y_1, \;\; &&\ldots, \;\; 
&&\tilde{y}_{i-1} = y_{i-1}, \;\; 
&&\tilde{y}_i = f(x), \;\; 
&&\tilde{y}_{i+1} = y_i, \;\; &&\ldots, \;\;
&&\tilde{y}_{n+1} = y_n. 
\end{alignedat}
\end{equation}
In what follows, we use a subscript $n+1$ (in place of a subscript $n$) to
denote the ``usual'' quantities of interest defined with respect to design
points \smash{$\tilde{x}_{1:(n+1)}$} (instead of $x_{1:n}$). In particular, we
use \smash{$\Delta^{k+1}_{n+1}$} to denote the $(k+1)$st order discrete
derivative operator defined using \smash{$\tilde{x}_{1:(n+1)}$}, and
\smash{$\cH^k_{n+1}$} to denote the space of $k$th degree discrete splines with
knots in \smash{$\tilde{x}_{(k+1):n}$}. By Theorem \ref{thm:ffb_interp}, we
can construct an interpolant \smash{$f \in \cH^k_{n+1}$} passing through
\smash{$\tilde{y}_{1:(n+1)}$} at \smash{$\tilde{x}_{1:(n+1)}$}. Note that, by
construction, $f$ is also the unique interpolant in \smash{$\cH^k_n$} passing
through $y_{1:n}$ at $x_{1:n}$. Denote the falling factorial basis
for \smash{$\cH^k_{n+1}$} by
$$
\tilde{h}^k_j, \; j=1,\ldots,n+1.
$$
As \smash{$f \in \cH^k_n$}, the coefficient of \smash{$\tilde{h}^k_{i+1}$} in
the basis expansion of $f$ with respect to \smash{$\tilde{h}^k_j$},
$j=1,\ldots,n+1$ must be zero (this is because \smash{$\tilde{h}^k_{i+1}$} has a
knot at \smash{$\tilde{x}_i=x$}, so if its coefficient is nonzero, then $f$ will
also have a knot at $x$ and cannot be in \smash{$\cH^k_n$}). By
\eqref{eq:ffb_interp} (applied to \smash{$\tilde{x}_{1:(n+1)}$}), this means
\smash{$(\Delta^{k+1}_{n+1} f)(\tilde{x}_{i+1}) = 0$}, or equivalently by
\eqref{eq:ffb_interp_explicit} (applied to \smash{$\tilde{x}_{1:(n+1)}$}), this
means \smash{$f[x_{i-k}, \ldots, x_{i-1}, x, x_i] = 0$}. The desired result
\eqref{eq:ffb_interp_implicit1} follows by recalling that divided differences
are invariant to the ordering of the centers.

For the case $x > x_n$, a similar argument applies, but instead of augmenting
the design and target points as in \eqref{eq:augmented1} we simply append $x$
to the end of $x_{1:n}$ and $f(x)$ to the end of $y_{1:n}$. 

For the case $x < x_{k+1}$, note that as $f$ is simply a $k$th degree
polynomial on $[a,x_{k+1}]$, it hence satisfies
\eqref{eq:ffb_interp_implicit2} (any $(k+1)$st order divided difference
with centers in $[a,x_{k+1}]$ is zero, recall \eqref{eq:deriv_match_poly}). This
completes the proof.    
\end{proof}

\begin{remark}
A key feature of the implicit representation for the discrete spline interpolant
as described in Corollary \ref{cor:ffb_interp_implicit} is that it reveals
$f(x)$ can be computed in {\it constant-time}\footnote{This is not including the
  time it takes to rank $x$ among the design points: finding the index $i$ before 
  solving \eqref{eq:ffb_interp_implicit1} will have a computational cost that,
  in general, depends on $n$; say, $O(\log{n})$ if the design points are sorted
  and we use binary search. However, note that this would be constant-time if
  the design points are evenly-spaced, and we use integer divison.}, or more 
precisely, in $O(k)$ operations (independent of the number of knots in the
interpolant, and hence of $n$). This is because we can always express a $k$th
order divided difference as a linear combination of function evaluations (recall
\eqref{eq:divided_diff_linear}): writing 
\smash{$f[x_{i-k}, \ldots, x_i, x] = \sum_{j=1}^{k+1} \omega_j f(x_{i-k-1+j}) +
\omega_{k+2} f(x)$}, we see that \eqref{eq:ffb_interp_implicit1} reduces to
\smash{$f(x) = -(\sum_{j=1}^{k+1} \omega_j f(x_{i-k-1+j}))/\omega_{k+2}$}, and
similarly for \eqref{eq:ffb_interp_implicit2}.
\end{remark}

Interestingly, as we prove next, discrete splines are the {\it only}
interpolatory functions satisfying \eqref{eq:ffb_interp_implicit1},
\eqref{eq:ffb_interp_implicit2} for all $x$. In other words, equations
\eqref{eq:ffb_interp_implicit1}, \eqref{eq:ffb_interp_implicit2} uniquely define
$f$, which is reminiscent of the implicit function theorem (and serves as
further motivation for us to call the approach in Corollary
\ref{cor:ffb_interp_implicit} an ``implicit'' form of interpolation).

\begin{corollary}
\label{cor:ffb_interp_implicit_conv}
Given any evaluations $f(x_i)$, $i=1,\ldots,n$ and $k \geq 0$, if for all $x \in
[a,b] \setminus x_{1:n}$, the function $f$ satisfies
\eqref{eq:ffb_interp_implicit1} for $x > x_{k+1}$ (where $i$ is the smallest
index such that $x_i>x$, with $i=n$ when $x>x_n$), and
\eqref{eq:ffb_interp_implicit2} for $x < x_{k+1}$, then \smash{$f \in \cH^k_n$}, 
that is, $f$ must be the $k$th degree discrete spline with knots in
$x_{(k+1):(n-1)}$ that interpolates $y_{1:n}$. 
\end{corollary}

\begin{proof}
This proof is similar to the proof of the converse statement in Corollary
\ref{cor:deriv_match}. First, note that \eqref{eq:ffb_interp_implicit1}
implies that the $k$th discrete derivative of $f$ is piecewise constant on
$[x_{k+1},b]$ with knots in $x_{(k+1):(n-1)}$. Moreover, a simple inductive
argument (deferred until Lemma \ref{lem:poly_implicit} in Appendix 
\ref{app:poly_implicit}) shows that \eqref{eq:ffb_interp_implicit2} implies $f$
is a $k$th degree polynomial on $[a,b]$. Therefore we may write      
$$
(\Delta^k_n f)(x) = \alpha_0 + \sum_{j=k+2}^n \alpha_j 1\{x > x_{j-1}\},
\quad \text{for $x \in [a,b]$},
$$
and proceeding as in the proof of Corollary \ref{cor:deriv_match} (inverting
using Lemma \ref{lem:discrete_deriv_integ_inv}, using linearity of $S^k_n$,
then Theorem \ref{thm:ffb_discrete_integ}) shows that $f$ is in the span of the
falling factorial basis, completing the proof.
\end{proof}

\section{Matrix computations}
\label{sec:matrix_comp}

We translate several of our definitions and results derived thus far to a
slightly different perspective. While there will be no new results established
in this section, phrasing our results in terms of matrices (which act on
function values at the design points) will help draw clearer connections to
results in previous papers \citep{tibshirani2014adaptive,wang2014falling}, and
will be notationally convenient for some subsequent parts of the paper. We
remind the reader that we use ``blackboard'' fonts for matrices (as in $\A,\B$,
etc.), in order to easily distinguish them from operators that act on functions.

\subsection{Discrete differentiation}

First define the simple difference matrix \smash{$\widebar\D_n \in \R^{(n-1)
    \times n}$} by  
\begin{equation}
\label{eq:diff_mat}
\widebar\D_n = 
\left[\begin{array}{rrrrrr} 
-1 & 1 & 0 & \ldots & 0 & 0 \\
0 & -1 & 1 & \ldots & 0 & 0 \\
\vdots & & & & & \\
0 & 0 & 0 & \ldots & -1 & 1 
\end{array}\right],
\end{equation}
and for $k \geq 1$, define the weight matrix \smash{$\W^k_n \in \R^{(n-k)
    \times (n-k)}$} by 
\begin{equation}
\label{eq:weight_mat}
\W^k_n = \diag\bigg( \frac{x_{k+1}-x_1}{k}, \ldots, \frac{x_n-x_{n-k}}{k}
\bigg). 
\end{equation}
Then we define the $k$th order discrete derivative matrix \smash{$\D^k_n \in 
  \R^{(n-k) \times n}$} by the recursion  
\begin{equation}
\label{eq:discrete_deriv_mat}
\begin{aligned}
\D_n &= (\W_n)^{-1} \widebar\D_n, \\
\D^k_n &= (\W^k_n)^{-1} \widebar\D_{n-k+1} \, \D^{k-1}_n,  
\quad \text{for $k \geq 2$}.
\end{aligned}
\end{equation}
We emphasize that \smash{$\D_{n-k+1}$} above denotes the $(n-k) \times  
(n-k+1)$ version of the simple difference matrix in \eqref{eq:diff_mat}. For a  
function $f$, denote by $f(x_{1:n})=(f(x_1),\ldots,f(x_n)) \in \R^n$ the vector 
of its evaluations at the design points $x_{1:n}$. It is not hard to see that
the $k$th discrete derivative matrix \smash{$\D^k_n$}, applied to $f(x_{1:n})$,
yields the vector of the $k$th discrete derivatives of $f$ at the points
$x_{(k+1):n}$, that is,   
\begin{equation}
\label{eq:discrete_deriv_conn1}
\D^k_n f(x_{1:n}) = (\Delta^k_n f)(x_{(k+1):n}). 
\end{equation}
Lastly, we note that \smash{$\D^k_n$} is a banded matrix, with bandwidth
$k+1$. 

\begin{remark}
\label{rem:discrete_deriv_mat_old}
Our definition of the discrete derivative matrices in
\eqref{eq:discrete_deriv_mat} differs from that in
\citet{tibshirani2014adaptive,wang2014falling} and subsequent papers on trend  
filtering. In these papers, the discrete derivative matrices are defined as
\begin{equation}
\label{eq:discrete_deriv_mat_old}
\begin{aligned}
\C_n &= \widebar\D_n, \\
\C^k_n &= \widebar\D_{n-k+1} (\W^{k-1}_n)^{-1} \C^{k-1}_n,
\quad \text{for $k \geq 2$}.
\end{aligned}
\end{equation}
We can hence see that \smash{$\D^k_n = (\W^k_n)^{-1} \C^k_n$} for each $k \geq
1$, that is, the discrete derivative matrices in
\eqref{eq:discrete_deriv_mat_old} are just like those in
\eqref{eq:discrete_deriv_mat}, but without the leading (inverse) weight
matrices. The main purpose of \eqref{eq:discrete_deriv_mat_old} in
\citet{tibshirani2014adaptive,wang2014falling} was to derive a convenient
formula for the total variation of derivatives of discrete splines (represented
in terms of discrete derivatives), and as we will see in Theorem
\ref{thm:ffb_tv}, and we will arrive at the same formula using
\eqref{eq:discrete_deriv_mat} (see also Remark \ref{rem:ffb_tv_mat_old}). In
this sense, the discrepancy between \eqref{eq:discrete_deriv_mat} and
\eqref{eq:discrete_deriv_mat_old} is not problematic (and if the design points
are evenly-spaced, then the two definitions coincide). However, in general, we
should note that the current definition \eqref{eq:discrete_deriv_mat} offers a
more natural perspective on discrete derivatives: recalling
\eqref{eq:discrete_deriv_conn1}, we see that it connects to \smash{$\Delta^k_n$}
and therefore to divided differences, a celebrated and widely-studied discrete
analogue of differentiation.
\end{remark}

\subsection{Extended discrete differentiation}

We can extend the construction in \eqref{eq:diff_mat}, \eqref{eq:weight_mat},
\eqref{eq:discrete_deriv_mat} to yield discrete derivatives at all points
$x_{1:n}$, as follows. 

For $k \geq 1$, define an extended difference matrix \smash{$\widebar\B_{n,k}
  \in \R^{n \times n}$} by    
\begin{equation}
\label{eq:diff_mat_ext}
\widebar\B_{n,k} = 
\left[\begin{array}{rrrrrrrrr}
1 & 0 & \ldots & 0 & \multicolumn{5}{c}{\multirow{4}{*}{0}} \\
0 & 1 & \ldots & 0 & \multicolumn{5}{c}{} \\
\vdots & & & & \multicolumn{5}{c}{} \\
0 & 0 & \ldots & 1 & \multicolumn{5}{c}{} \\
\multicolumn{3}{c}{\multirow{4}{*}{0}} & -1 & 1 & 0 & \ldots & 0 & 0 \\ 
\multicolumn{3}{c}{} & 0 & -1 & 1 & \ldots & 0 & 0 \\
\multicolumn{3}{c}{} & \vdots & & & & & \\
\multicolumn{3}{c}{} & 0 & 0 & 0 & \ldots & -1 & 1 
\end{array}\right]
\renewcommand\arraystretch{1.1}
\begin{array}{ll}
\left.\vphantom{\begin{array}{c} 1 \\ 0 \\ \cdots \\ 0 \end{array}}
\right\} & \hspace{-5pt} \text{$k$ rows} \\
\left.\vphantom{\begin{array}{c} 1 \\ 0 \\ \cdots \\ 0 \end{array}}
\right\} & \hspace{-5pt} \text{$n-k$ rows}
\end{array}
\renewcommand\arraystretch{1}
\end{equation}
(note that the top-left $k \times k$ submatrix is the identity matrix $\I_k$,
and the bottom-right $(n-k) \times (n-k+1)$ submatrix is $\D_{n-k}$), and also 
define an extended weight matrix \smash{$\Z^k_n \in \R^{n \times n}$} by  
\begin{equation}
\label{eq:weight_mat_ext}
\Z^k_n = \diag\bigg(\underbrace{1,\ldots,1
  \vphantom{\frac{x_{k+1}-x_1}{k}}}_{\text{$k$ times}},   
\frac{x_{k+1}-x_1}{k}, \ldots, \frac{x_n-x_{n-k}}{k} \bigg).
\end{equation}
Then we define the extended $k$th order discrete derivative matrix
\smash{$\B^k_n \in \R^{n\times n}$} by the recursion
\begin{equation}
\label{eq:discrete_deriv_mat_ext}
\begin{aligned}
\B_n &= (\Z_n)^{-1} \widebar\B_{n,1}, \\
\B^k_n &= (\Z^k_n)^{-1} \widebar\B_{n,k} \, \B^{k-1}_n, 
\quad \text{for $k \geq 2$}.
\end{aligned}
\end{equation}
The construction \eqref{eq:diff_mat_ext}, \eqref{eq:weight_mat_ext},
\eqref{eq:discrete_deriv_mat_ext} is precisely analogous to what was done in
\eqref{eq:simple_diff}, \eqref{eq:weight_map}, \eqref{eq:discrete_deriv_rec},
but it is just specialized to the design points, and yields
\begin{equation} 
\label{eq:discrete_deriv_conn2}
\B^k_n f(x_{1:n}) = (\Delta^k_n f)(x_{1:n}),
\end{equation}
which is the extension of property \eqref{eq:discrete_deriv_conn1} to the full
set of the design points $x_{1:n}$. Lastly, we note that \smash{$\B^k_n$} is
again banded, with bandwidth $k+1$, and that the discrete derivative matrix
\smash{$\D^k_n$} is simply given by the last $n-k$ rows of the extended matrix
\smash{$\B^k_n$}.

\subsection{Falling factorial basis}

Now define for $k \geq 0$ the falling factorial basis matrix \smash{$\H^k_n \in
  \R^{n \times n}$} to have entries
\begin{equation}
\label{eq:ffb_mat}
(\H^k_n)_{ij} = h^k_j(x_i),
\end{equation}
where \smash{$h^k_j$}, $j=1,\ldots,n$ are the falling factorial basis functions
in \eqref{eq:ffb}. The lateral recursion in Lemma \ref{lem:ffb_lateral_rec}
implies 
\begin{equation}
\label{eq:ffb_lateral_rec_mat}
\H^k_n = \H^{k-1}_n \Z^k_n
 \left[\begin{array}{cc} 
\I_k & 0 \\
0 & \L_{n-k}
\end{array}\right],
\end{equation}
where $\I_k$ denotes the $k \times k$ identity matrix, and $\L_{n-k}$ denotes 
the $(n-k) \times (n-k)$ lower triangular matrix of all 1s. Furthermore, the
dual result between discrete differentiation and the falling factorial basis in
Lemma \ref{lem:dual_basis} can be written as    
\begin{equation}
\label{eq:ffb_discrete_deriv_inv}
\Z^{k+1}_n \, \B^{k+1}_n \, \H^k_n = \I_n.
\end{equation}
We note that the results in \eqref{eq:ffb_lateral_rec_mat} and
\eqref{eq:ffb_discrete_deriv_inv} were already established in Lemmas 1 and 2 
of \citet{wang2014falling} (and for the case of evenly-spaced design points, in
Lemmas 2 and 4 of \citet{tibshirani2014adaptive}). To be clear, the analogous
results in the current paper (Lemmas \ref{lem:ffb_lateral_rec},
\ref{lem:ffb_discrete_deriv_extra_pp}, \ref{lem:ffb_discrete_deriv_extra_poly}, 
and \ref{lem:dual_basis}) are slightly more general, as they hold for
arbitrary $x$, and not just at the design points. (Their proofs are also
simpler; in particular Lemma \ref{lem:ffb_lateral_rec}, whose proof is quite 
different and considerably simpler than the proof of Lemma 1 in
\citet{wang2014falling}.)  

\subsection{Fast matrix multiplication}

A nice consequence of \eqref{eq:ffb_lateral_rec_mat} and
\eqref{eq:ffb_discrete_deriv_inv}, as developed by \citet{wang2014falling}, 
is that matrix-vector multiplication using any of
\smash{$\H^k_n, (\H^k_n)^{-1},(\H^k_n)^\T,(\H^k_n)^{-\T}$} can be done in
$O(nk)$ operations using simple, in-place algorithms, based on iterated scaled
cumulative sums, and iterated scaled differences---to be precise, each of these
algorithms requires at most $4nk$ flops ($kn$ additions, subtractions,
multiplications, and divisions). For convenience, we recap the details in  
Appendix \ref{app:fast_mult}. 

\section{Discrete B-splines}
\label{sec:discrete_bs}

We develop a local basis for \smash{$\cH^k_n$}, the space of $k$th degree
discrete splines with knots in $x_{(k+1):(n-1)}$. This basis bears similarities
to the B-spline basis for splines, and is hence called the {\it discrete
B-spline basis}. In this section (as we do throughout this paper), we consider
discrete splines with arbitrary design points $x_{1:n}$, defining the underlying  
discrete derivative operators \smash{$\Delta^\ell_n=\Delta^\ell(\cdot;
x_{1:n})$}, $\ell=1,\ldots,k-1$. For the construction of discrete B-splines, in
particular, this presents an interesting conceptual challenge (that is absent 
in the case of evenly-spaced design points).

To explain this, we note that a key to the construction of B-splines, 
reviewed in Appendix \ref{app:bs}, is a certain kind of symmetry possessed by
the truncated power functions. At its core, the $k$th degree B-spline with
knots $z_1 < \cdots < z_{k+2}$ is defined by a pointwise divided difference of a
truncated power function; this is given in \eqref{eq:bs_orig}, but for
convenience, we copy it here:
\begin{equation}
\label{eq:bs_copy}
P^k(x; z_{1:(k+2)}) = (\cdot - x)^k_+[z_1,\ldots,z_{k+2}].
\end{equation}
To be clear, here the notation \smash{$(\cdot - x)^k_+[z_1,\ldots,z_{k+2}]$} 
means that we are taking the divided difference of the function \smash{$z
  \mapsto (z - x)^k_+$} with respect to the centers $z_1,\ldots,z_{k+2}$. 
The following are two critical observations. First, for fixed $x$, the map
\smash{$z \mapsto (z - x)^k_+$} is a $k$th degree polynomial for $z>x$, and
thus if $z_1>x$, then the divided difference at centers $z_1,\ldots,z_{k+2}$ 
will be zero (this is a $(k+1)$st order divided difference of a $k$th degree
polynomial, recall \eqref{eq:deriv_match_poly}). Trivially, we also have that
the divided difference will be zero if $z_{k+2}<x$, because then we will be
taking a divided difference of all zeros. This shows that $P^k(\cdot;
z_{1:(k+2)})$ is supported on $[z_1,z_{k+2}]$ (see also
\eqref{eq:nbs_supp}). Second (and this is where the symmetry property is 
invoked), for fixed $z$, the map \smash{$x \mapsto (z - x)^k_+$} is a $k$th  
degree spline that has a single knot at $z$, and hence
\smash{$P^k(\cdot;z_{1:(k+2)})$}, a linear combination of such functions, is a
$k$th degree spline with knots $z_{1:(k+2)}$. 

For evenly-spaced design points, an analogous construction goes through for
discrete splines, replacing truncated power functions with truncated rising
factorial polynomials, as reviewed in Appendix \ref{app:discrete_bs_even}. The
key is again symmetry: now $(z-x)(z-x+v) \cdots (z-x+(k-1)v) \cdot 1\{z>x\}$,
for fixed $x$, acts as a polynomial in $z$ over $z>x$, giving the desired
support property (when we take divided differences); and for fixed $z$, it acts
as a truncated falling factorial function in $x$, giving the desired discrete
spline property (again after divided differences). 

But for arbitrary design points, there is no apparent way to view the argument
and the knots in a truncated Newton polynomial in a symmetric fashion.
Therefore it is unclear how to proceed in the usual manner as outlined above
(and covered in detail in Appendices \ref{app:bs} and
\ref{app:discrete_bs_even}). Our solution is to first define a discrete B-spline
at the design points only (which we can do in analogous way to the usual
construction), and then prove that the discrete spline interpolant of such
values has the desired support structure. For the latter step, the 
interpolation results in Theorem \ref{thm:ffb_interp} and Corollary
\ref{cor:ffb_interp_implicit} (especially the implicit result in Corollary
\ref{cor:ffb_interp_implicit}) end up being very useful. 

\subsection{Construction at the design points}
\label{sec:discrete_bs_evals}

Here we define discrete B-splines directly at the design points $x_{1:n}$. We
begin by defining boundary design points 
$$
x_{-(k-1)} < \cdots < x_{-1} < x_0 = a, \quad \text{and} \quad x_{n+1}=b. 
$$
(Any such values for $x_{-(k-1)},\ldots,x_{-1}$ will suffice for our ultimate
purpose of defining a basis.) For a degree $k \geq 0$, and for each
$j=1,\ldots,n$, now define evaluations of a function \smash{$Q^k_j$} at the
design points by  
\begin{equation}
\label{eq:discrete_bs_evals}
Q^k_j(x_i) = \eta_+(\cdot; x_{(i-k+1):i}) [x_{j-k},\ldots,x_{j+1}], 
\quad i = 1,\ldots,n.
\end{equation}
where recall \smash{$\eta_+(x; t_{1:r})=\eta(x; t_{1:r}) \cdot 1\{x >
\max(t_{1:r}) \}$} denotes a truncated Newton polynomial, and the notation
$\eta_+(\cdot; x_{(i-k+1):i}) [x_{j-k},\ldots,x_{j+1}]$ means that we are taking
the divided difference of the map $z \mapsto \eta_+(z; x_{(i-k+1):i})$ with
respect to the centers $x_{j-k},\ldots,x_{j+1}$. Comparing \eqref{eq:bs_copy}
and \eqref{eq:discrete_bs_evals}, we see that \smash{$Q^k_j$}, $j=1,\ldots,n$
are defined (over the design points) in a similar manner to $P^k(\cdot;
z_{1:(k+2)})$, using sliding sets of centers for the divided differences,
and with truncated Newton polynomials instead of truncated power
functions.\footnote{Moreover, our definition in \eqref{eq:discrete_bs_evals}  
  is in the same spirit (at the design points) as the standard definition of a
  discrete B-spline in the evenly-spaced case, as given in
  Appendix \ref{app:discrete_bs_even}. It is not exactly equivalent, as the
  standard definition \eqref{eq:discrete_bs_even_orig} uses a truncated rising
  factorial polynomial, whereas our preference is to use truncated Newton
  polynomial that more closely resembles a truncated falling factorial in the 
  evenly-spaced case. In the end, this just means that our discrete B-splines 
  look like those from Appendix \ref{app:discrete_bs_even} after reflection
  about the vertical axis; compare Figures \ref{fig:dbs} and \ref{fig:bs}.}     


It is often useful to deal with a normalized version of the function evaluations
in \eqref{eq:discrete_bs_evals}. Define, for $j=1,\ldots,n$, the function
\smash{$N^k_j$} at the design points by
\begin{equation}
\label{eq:discrete_nbs_evals}
N^k_j(x_i) = (x_{j+1}-x_{j-k}) \cdot \eta_+(\cdot; x_{(i-k+1):i}) 
[x_{j-k},\ldots,x_{j+1}], \quad i = 1,\ldots,n.
\end{equation}
Next we show a critical property of these normalized function evaluations. 

\begin{lemma}
\label{lem:discrete_nbs_evals_supp}
For any $k \geq 0$, the function evaluations in \eqref{eq:discrete_nbs_evals}
satisfy: 
\begin{equation}
\label{eq:discrete_nbs_evals_supp}
N^k_j(x_i) = \delta_{ij}, \quad i,j, = 1,\ldots,n,
\end{equation}
where $\delta_{ij}=1$ if $i=j$, and $\delta_{ij}=0$ otherwise. 
\end{lemma}

\begin{proof}
Fix any $j=1,\ldots,n$. For $i \geq j+1$, we have $x_{j-k} < \cdots < x_{j+1}
\leq x_i$, hence \smash{$N^k_j(x_i)$} is defined by a divided difference of all 
zeros, and is therefore zero. For $j \geq i+1$, we claim that
$$
\eta_+(x_\ell; x_{(i-k+1):i}) = \eta(x_\ell; x_{(i-k+1):i}), \quad
\ell=j-k,\ldots,j+1.
$$
This is true because for $\ell=j-k,\ldots,i$, the left-hand side is zero (by
truncation), but the right-hand side is also zero, as $(j-k) \hspace{-2pt} :
\hspace{-2pt} i \subseteq (i-k+1) \hspace{-2pt} : \hspace{-2pt} i$. The above 
display implies  
$$
N^k_j(x_i) = (x_{j+1}-x_{j-k}) \cdot \eta(\cdot;
x_{(i-k+1):i})[x_{j-k},\ldots,x_{j+1}] = 0,   
$$
with the last equality due to the fact that a $(k+1)$st order divided
difference of a $k$th order polynomial is zero (recall, for example, 
\eqref{eq:deriv_match_poly}).  It remains to consider $j=i$. In this case, 
writing \smash{$f[x_{i-k}, \ldots, x_{i+1}] = \sum_{\ell=1}^{k+2} \omega_\ell  
  f(x_{i-k-1+\ell})$} by linearity of divided differences (recall
\eqref{eq:divided_diff_linear}), we have   
$$
N^k_i(x_i) =  \omega_{k+2} (x_{i+1}-x_{i-k}) \cdot \eta(x_{i+1};
x_{(i-k+1):i}) = 1,  
$$
where we have used the explicit form of $\omega_{k+2}$ from
\eqref{eq:divided_diff_linear}. 
\end{proof}

\subsection{Interpolation to $[a,b]$} 

We now interpolate the values defined in \eqref{eq:discrete_nbs_evals} to a
discrete spline defined on all $[a,b]$. In particular, for $j=1,\ldots,n$, let 
\begin{equation}
\label{eq:discrete_nbs}
\text{$N^k_j$ be the interpolant in $\cH^k_n$ passing through $\delta_{ij}$ at
  $x_i$, for $i=1,\ldots,n$}. 
\end{equation}
We refer to the resulting functions \smash{$N^k_j$}, $j=1,\ldots,n$ as $k$th 
degree normalized {\it discrete B-splines} or DB-splines. Since
\smash{$\cH^k_n$}, the space of $k$th degree discrete splines with knots 
$x_{(k+1):(n-1)}$, is an $n$-dimensional linear space, and each
\smash{$N^k_j$} is determined by interpolating $n$ values, it is well-defined.
We also note that \smash{$N^k_j$}, $j=1,\ldots,n$ are linearly independent
(this is clear from Lemma \ref{lem:discrete_nbs_evals_supp}), and thus they
form a basis for \smash{$\cH^k_n$}. 

Next we establish that key property the functions \smash{$N^k_j$},
$j=1,\ldots,n$ have local supports. 

\begin{lemma}
\label{lem:discrete_nbs_supp}
For any $k \geq 0$, the $k$th degree normalized DB-spline basis functions, as
defined in \eqref{eq:discrete_nbs}, have the following support structure:
\begin{equation}
\label{eq:discrete_nbs_supp}
\text{$N^k_j$ is supported on} \;
\begin{cases}
[a,x_{j+k}] & \text{if $j \leq k+1$} \\
[x_{j-1},x_{j+k}] & \text{if $k+2 \leq j \leq n-k-1$} \\
[x_{j-1},b] &  \text{if $j \geq n-k$}.
\end{cases}
\end{equation}
Furthermore, for each $j=1,\ldots,n$, we have the explicit expansion in the
falling factorial basis: 
\begin{equation}
\label{eq:discrete_nbs_expl}
N^k_j = \sum_{i=j}^{(j+k+1) \wedge n} (\Z^{k+1}_n \, \B^{k+1}_n)_{ij} 
\cdot h^k_i,   
\end{equation}
where \smash{$\B^{k+1}_n \in \R^{n \times n}$} is the $(k+1)$st order extended
discrete derivative matrix, as in \eqref{eq:discrete_deriv_mat_ext}, and
\smash{$\Z^{k+1}_n \in \R^{n \times n}$} is the $(k+1)$st order extended
diagonal weight matrix, as in \eqref{eq:weight_mat_ext}; also, we use the 
abbreviation $x \wedge y = \min\{x,y\}$.
\end{lemma}

\begin{proof}
We will apply the implicit interpolation result from Corollary 
\ref{cor:ffb_interp_implicit}. First consider the middle case, $k+2 \leq j \leq
n-k-1$. If $x > x_{k+1}$ and $i$ is the smallest index such that $x_i>x$, then
by \eqref{eq:ffb_interp_implicit1} we know that \smash{$N^k_j(x)$} is
determined by solving the linear system 
$$
N^k_j[x_{i-k}, \ldots, x_i, x] = 0.
$$
But Lemma \ref{lem:discrete_nbs_evals_supp} tells us that \smash{$N^k_j$},
restricted to the design points, is only nonzero at $x_j$. Therefore the above
linear system will have all \smash{$N^k_j(x_{i-k})=\cdots=N^k_j(x_i)=0$},
and thus trivially \smash{$N^k_j(x)=0$} as the solution, unless $i-k \leq j
\leq i$, that is, unless $x \in [x_{j-1},x_{j+k}]$. If $x < x_{k+1}$, then by 
\eqref{eq:ffb_interp_implicit2} we know that \smash{$N^k_j(x)$} is determined 
by solving the linear system  
$$
N^k_j[x_1, \ldots, x_{k+1}, x] = 0.
$$
But \smash{$N^k_j(x_{i-k})=\cdots=N^k_j(x_i)=0$}, and again
\smash{$N^k_j(x)=0$} is the solution, since $j \geq k+2$. This proves the
middle case in \eqref{eq:discrete_nbs_supp}. 

Now consider the first case, $j \leq k+1$. If $x > x_{k+1}$ and $i$ is the
smallest index such that $x_i>x$, then by \eqref{eq:ffb_interp_implicit1} we
know that \smash{$N^k_j(x)$} is determined by solving the linear system in the
second to last display, but this gives \smash{$N^k_j(x)=0$} unless $i-k \leq j 
\leq i$. As $i \geq k+2$ (since we are assuming $x>x_{k+1}$) and $j \leq
k+1$, the condition $j \leq i$ is always satisfied. The condition $i-k \leq j$
translates into $x \leq x_{j+k}$, as before, which proves the first case in
\eqref{eq:discrete_nbs_supp}.  The last case, $j \geq n-k$, is similar.  

Finally, the result in \eqref{eq:discrete_nbs_expl} is a direct consequence of
the explicit interpolation result in \eqref{eq:ffb_interp} from Theorem
\ref{thm:ffb_interp}. 
\end{proof}

\begin{remark}
Lemma \ref{lem:discrete_nbs_supp} shows the $k$th degree DB-spline basis
functions are supported on intervals that each contain at most $k+2$ knots: for
$j=k+2,\ldots,n-k-1$, \smash{$N^k_j$} is supported on $[x_{j-1}, x_{j+k}]$,
which contain knots $x_{(j-1):(j+k)}$; and for $j=1,\ldots,k+1$ or
$j=n-k,\ldots,n$, \smash{$N^k_j$} is supported on $[a,x_{j+k}]$ or
$[x_{j-1},b]$, respectively, which contain knots $x_{(k+1):(j+k)}$ or
$x_{(j-1):(n-1)}$, respectively. This matches the ``support width'' of the
usual B-splines: recall that the $k$th degree B-spline basis functions are also
supported on intervals containing at most $k+2$ knots, see \eqref{eq:nbsb}.
In fact, when $k=0$ or $k=1$, the normalized DB-spline basis \smash{$N^k_j$},
$j=1,\ldots,n$ is ``almost'' the same as the normalized B-spline basis
\smash{$M^k_j$}, $j=1,\ldots,n$ defined in \eqref{eq:nbsb}; it only differs in
the left side of the supports of the first $k+1$ basis functions, and the right
side of the supports of the last $k+1$ basis functions. This should not be a
surprise, as discrete splines of degrees $k=0$ and $k=1$ are simply splines.  
\end{remark}

\begin{remark}
A curious fact about DB-splines, as defined in \eqref{eq:discrete_nbs}, is that
they are not always positive on their support. This is in contrast to the
usual B-splines, which are always positive when nonzero, see \eqref{eq:nbs_supp}.
(However, it is consistent with the behavior of standard DB-splines for
evenly-spaced design points, see Appendix \ref{app:discrete_bs_even}.)  For $k
\geq 2$, DB-splines have a negative ``ripple'' close to their rightmost knot
point. See Figure \ref{fig:dbs} for examples of DB-splines of degree 2.   
\end{remark}

\begin{remark}
As DB-splines are discrete splines, in \smash{$\cH^k_n$} (by construction in
\eqref{eq:discrete_nbs}, via interpolation within this function space), they
have the property that their $k$th derivatives and $k$th discrete derivatives
match everywhere, by Corollary \ref{cor:deriv_match}. This means that for each
$j=1,\ldots,n$, the piecewise constant function \smash{$\Delta^k_n N^k_j$}
shares the local support of \smash{$N^k_j$}, as given by Lemma
\ref{lem:discrete_nbs_supp}. Figure \ref{fig:dbs} confirms this numerically.
For $k \geq 2$, B-splines---being splines and not discrete splines---do not
share their property, as also confirmed in the figure.  
\end{remark}

\begin{figure}[p]
\centering
\includegraphics[width=0.99\textwidth]{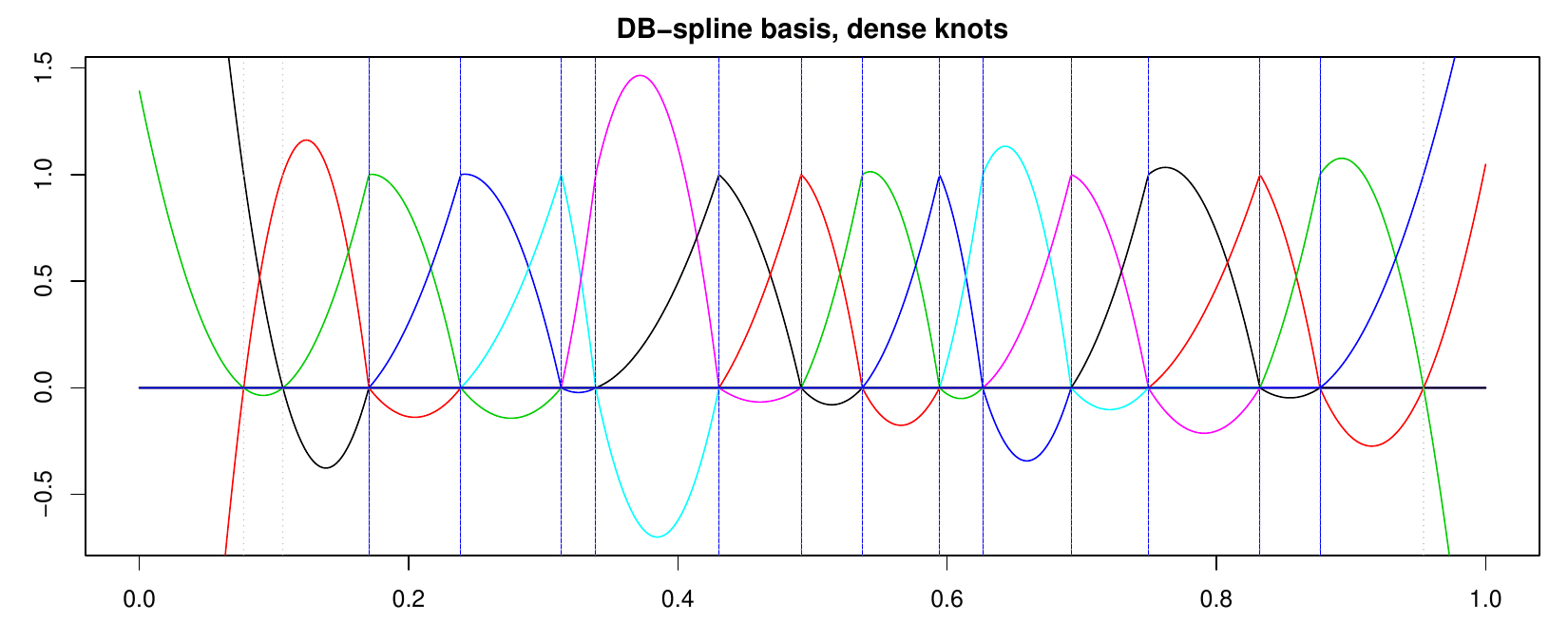}
\includegraphics[width=0.495\textwidth]{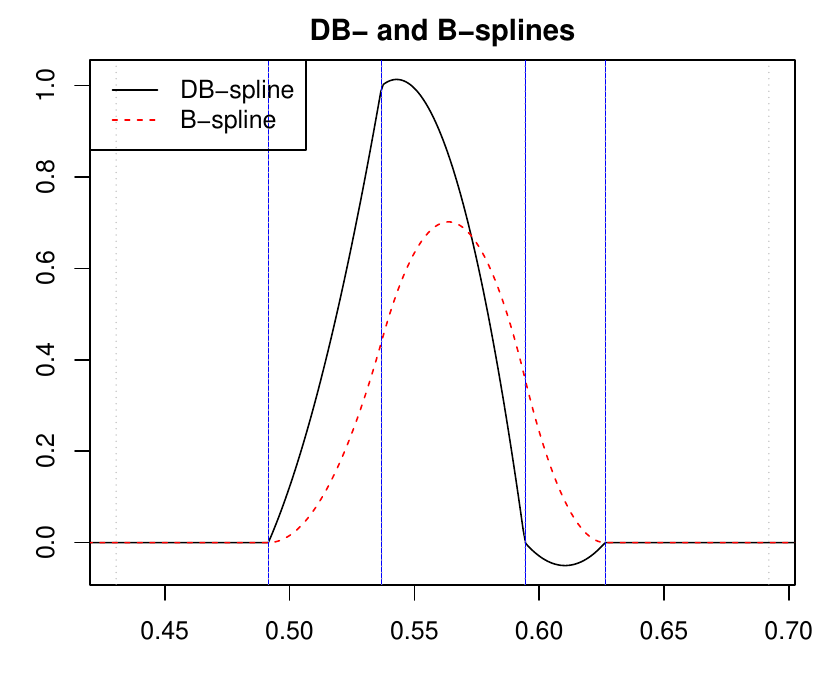}
\includegraphics[width=0.495\textwidth]{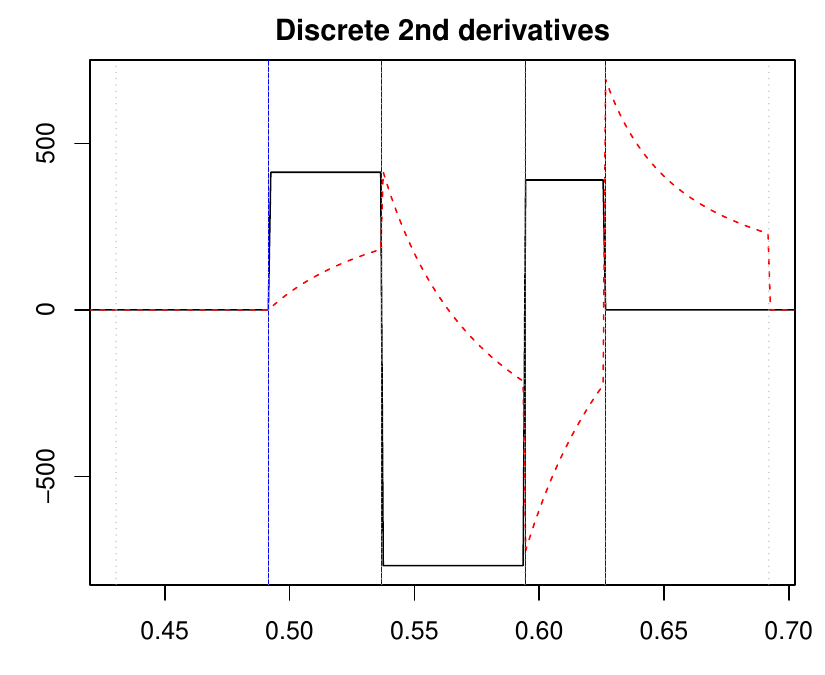}
\includegraphics[width=0.99\textwidth]{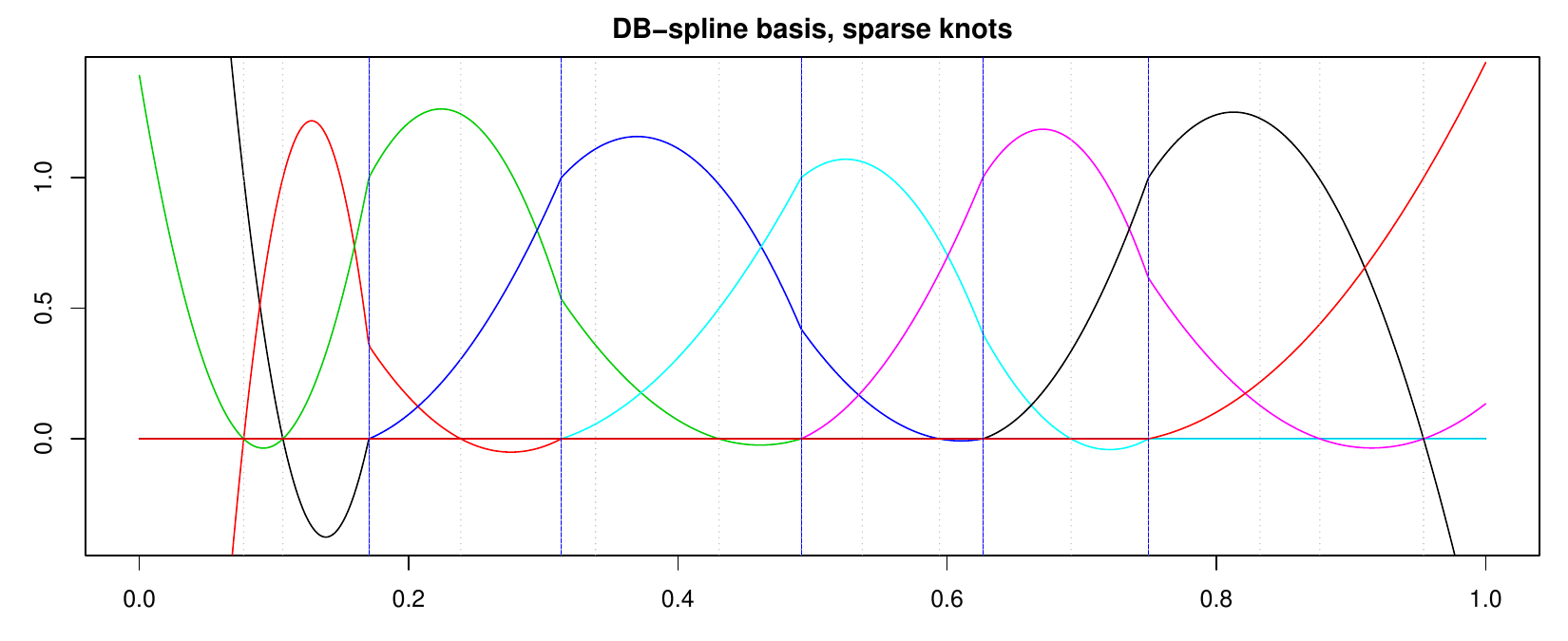}
\caption{\small Top row: normalized DB-spline basis for
  \smash{$\cH^k_n=\DS^k_n(x_{(k+1):(n-1)},[a,b])$}, where $k=2$, and the $n=16$  
  design points marked by dotted gray vertical lines. The knots are marked by
  blue vertical lines. We can see that the DB-splines have a negative
  ``ripple'' near their rightmost knot point. Middle row, left panel: comparison
  of a single DB-spline basis function in black, and its B-spline counterpart in
  dashed red. Right panel: their discrete 2nd derivatives; notice the discrete 
  2nd derivative matches the 2nd derivative for the DB-spline, but not for the
  B-spline. Bottom row: normalized DB-spline basis for 
  \smash{$\cH^k_n=\DS^k_n(t_{1:r},[a,b])$}, for a sparse subset $t_{1:r}$ of the
  design points of size $r=5$.}  
\label{fig:dbs}
\end{figure}

The discrete B-spline basis developed in this section finds two primary
applications in the remainder of this paper. First, it can be easily modified
to provide a basis for the space of discrete natural splines, which we describe
in the next subsection. Second, it provides a significantly more stable
(that is, better-conditioned) basis for solving least squares problems in
discrete splines, described later in Section \ref{sec:least_squares}. 

\subsection{Discrete natural splines}
\label{sec:discrete_natural_spline}

Similar to the usual definition of natural splines, we can modify the definition
of discrete splines to require lower-degree polynomial behavior on the
boundaries, as follows. 

\begin{definition}
\label{def:discrete_natural_spline}
As in Definition \ref{def:discrete_spline}, but with $k=2m-1 \geq 1$
constrained to be odd, we define the space of {\it $k$th degree discrete natural
  splines} on $[a,b]$ with knots $t_{1:r}$, denoted \smash{$\DNS^k_n(t_{1:r},
  [a,b])$}, to contain all functions $f$ on $[a,b]$ such that
\eqref{eq:discrete_spline} holds, and additionally,   
\begin{equation}
\label{eq:discrete_natural_spline}
\text{$(\Delta^\ell_n p_0)(t_1) = 0$ and $(\Delta^\ell_n p_r)(t_r) = 0$,
  $\ell=m,\ldots,k$}. 
\end{equation}
We note that this is equivalent to restricting $p_0$ and $p_r$ to be polynomials
of degree $m-1$.  
\end{definition}

As has been our focus thus far, we consider in this subsection the knot set 
$x_{(k+1):(n-1)}$, and study the $k$th degree discrete natural spline space 
\smash{$\cN^k_n= \DNS^k_n(x_{(k+1):(n-1)}, [a,b])$}. (In the next section, we
will discuss the case in which $t_{1:r}$ is an arbitrary subset of the design
points, in generality.)  On the one hand, since \smash{$\cN^k_n \subseteq
  \cH^k_n$} by construction, many properties of \smash{$\cH^k_n$} carry over
automatically to \smash{$\cN^k_n$}: for example, the matching derivatives
property in Corollary \ref{cor:deriv_match} and the interpolation results in
Theorem \ref{thm:ffb_interp} and Corollary \ref{cor:ffb_interp_implicit} all hold
for discrete natural splines. On the other hand, other aspects require some
work: for example, constructing a basis for \smash{$\cN^k_n$} is
nontrivial. Certainly, it seems to be highly nontrivial to modify the falling
factorial basis \smash{$h^k_j$}, $j=1,\ldots,n$ in \eqref{eq:ffb} for
\smash{$\cH^k_n$} in order to obtain a basis for \smash{$\cN^k_n$}. 
Fortunately, as we show in the next lemma, it is relatively easy to modify the
DB-spline basis \smash{$N^k_j$}, $j=1,\ldots,n$ in \eqref{eq:discrete_nbs}
(written explicitly in \eqref{eq:discrete_nbs_expl}) to form a basis for
\smash{$\cN^k_n$}. 

\begin{lemma}
\label{lem:discrete_natural_nbs}
For any odd $k = 2m-1 \geq 1$, the space
\smash{$\cN^k_n=\DNS^k_n(x_{(k+1):(n-1)}, [a,b])$} of $k$th degree discrete
natural splines on $[a,b]$ with knots $x_{(k+1):(n-1)}$ is spanned by the
following $n-k-1$ functions:  
\begin{equation}
\label{eq:discrete_natural_nbs} 
\begin{aligned}
&L^k_j = \sum_{i=1}^{k+1} x_i^{j-1} \cdot N^k_i, 
\quad j=1,\ldots,m, \\
&N^k_j, \quad j=k+2,\ldots,n-k-1, \\
&R^k_j = \sum_{i=n-k}^n (x_i- x_{n-k-1})^{j-1} \cdot N^k_i,   
\quad j=1,\ldots,m,
\end{aligned}
\end{equation}
where recall \smash{$N^k_j$}, $j=1,\ldots,n$ are the DB-spline basis functions
in \eqref{eq:discrete_nbs}. 
\end{lemma}

\begin{proof}
A dimensionality argument shows that the linear space \smash{$\cN^k_n$} has
dimension $n-k-1$. Clearly, the functions \smash{$N^k_j$}, $j=k+2,\ldots,n-k-1$
are $k$th degree discrete natural splines: each \smash{such $N^k_j$} is zero on
$[a, x_{j-1}] \supseteq [a, x_{k+1}]$ and is thus a polynomial of degree $m-1$
on this interval; further, it evaluates to zero over the points $x_{(j+1):n}
\supseteq x_{(n-k):(n-1)}$ and hence its restriction to $[x_{n-1},b]$ can also
be taken to be a polynomial of degree $m-1$.

It remains to show that the functions \smash{$L^k_j$}, $j=1,\ldots,m$ and
\smash{$R^k_j$}, $j=1,\ldots,m$ defined in the first and third lines of
\eqref{eq:discrete_natural_nbs} are discrete natural splines, since, given 
the linear independence of the $n-k-1$ functions in
\eqref{eq:discrete_natural_nbs} (an implication of the structure of their
supports), this would complete the proof.  Consider the ``left'' side functions 
\smash{$L^k_j$}, $j=1,\ldots,m$ (which will have local supports on the left 
side of the domain). Suppose we seek a linear combination
\smash{$\sum_{j=1}^{k+1} \alpha_j N^k_j$} of the first $k+1$ DB-splines that
meet the conditions in \eqref{eq:discrete_natural_spline}; since these
DB-splines will evaluate to zero on $x_{(n-k):(n-1)}$, we only need to check the  
first condition in \eqref{eq:discrete_natural_spline}, that is, 
$$
\bigg(\Delta^\ell_n \sum_{j=1}^{k+1} \alpha_j N^k_j \bigg)(x_{k+1}) = 0, \quad 
\ell=m,\ldots,k, 
$$
Using linearity of the discrete derivative operator, and recalling that 
\smash{$N^k_j(x_i)=\delta_{ij}$} by definition in \eqref{eq:discrete_nbs}, we
conclude the above condition is equivalent to $\F \alpha = 0$, where $\F \in 
\R^{m \times (k+1)}$ has entries  
\smash{$\F_{\ell-m+1, j} = (\B^\ell_n)_{k+1, j}$} for $\ell=m,\ldots,k$ and 
$j=1,\ldots,k+1$,
and where \smash{$\B^\ell_n \in \R^{n \times n}$} is the $\ell$th order extended
discrete derivative matrix, as in \eqref{eq:discrete_deriv_mat_ext}. 
The null space of $\F$ is simply given by evaluating all degree $m-1$
polynomials over the design points $x_{1:(k+1)}$ (each such vector is certainly in
the null space, because its $m$th through $k$th discrete derivatives are zero,
and there are $m$ such linearly independent vectors, which is the nullity
of $\F$). Thus with $\P \in \R^{(k+1) \times m}$ defined to have entries
\smash{$\P_{ij} = x_i^{j-1}$}, we may write any $\alpha \in \R^{k+1}$ such that 
$\F \alpha = 0$ as $\alpha = \P \beta$ for some $\beta \in \R^m$, and any linear 
combination satisfying the above condition (in the last display) must therefore
be of the form 
$$
\sum_{j=1}^{k+1} \alpha_j N^k_j = 
\sum_{j=1}^m \beta_j \sum_{i=1}^{k+1} \P_{ij} N^k_i,
$$
which shows that \smash{$L^k_j$}, $j=1,\ldots,m$ are indeed $k$th degree
natural splines. The argument for the ``right'' side functions \smash{$R^k_j$},
$j=1,\ldots,m$ follows similarly. 
\end{proof}

Later in Section \ref{sec:trend_filter}, we discuss restricting the domain in
trend filtering problem \eqref{eq:trend_filter} (equivalently,
\eqref{eq:trend_filter_cont}) to the space of discrete natural splines
\smash{$\cN^k_n$}, and give an empirical example where this improves its
boundary behavior. See  Figure \ref{fig:nat}, where we also plot the discrete
natural B-spline basis in \eqref{eq:discrete_natural_nbs} of degree 3.

\section{Sparse knot sets}
\label{sec:sparse_knots}

While our focus in this paper is the space
\smash{$\cH^k_n=\DS^k_n(x_{(k+1):(n-1)}, [a,b])$}, of $k$th degree discrete  
splines with knots in $x_{(k+1):(n-1)}$, all of our developments thus far can be
appropriately generalized to the space \smash{$\DS^k_n(t_{1:r}, [a,b])$},  
for arbitrary knots $t_{1:r} \subseteq x_{1:n}$ (this knot set could be a sparse 
subset of the design points, that is, with $r$ much smaller than $n$). We
assume (without a loss of generality) that $t_1 < \cdots < t_r$, where $t_1 \geq  
x_{k+1}$ (as in Definition \ref{def:discrete_spline}), and $t_r \leq x_{n-1}$
(for simplicity). Defining $i_j$, $j=1,\ldots,r$ such that 
$$
t_j = x_{i_j}, \quad j=1,\ldots,r,
$$
it is not hard to see that a falling factorial basis \smash{$h^k_j$},
$j=1,\ldots,r+k+1$ for \smash{$\DS^k_n(t_{1:r}, [a,b])$} is given by  
\begin{equation}
\label{eq:ffb_sk}
\begin{aligned}
h^k_j(x) &= \frac{1}{(j-1)!} \prod_{\ell=1}^{j-1}(x-x_\ell), 
\quad j=1,\ldots,k+1, \\
h^k_{j+k+1}(x) &= \frac{1}{k!} \prod_{\ell=i_j-k+1}^{i_j} (x-x_\ell) \cdot   
1\{x > x_{i_j} \}, \quad j=1,\ldots,r.
\end{aligned}
\end{equation}
(In the ``dense'' knot set case, we have $t_{1:r}=x_{(k+1):(n-1)}$, thus
$r=n-k-1$ and $i_j=j+k$, $j=1,\ldots,n-k-1$, in which case \eqref{eq:ffb_sk}
matches \eqref{eq:ffb}.)  Further, as \smash{$\DS^k_n(t_{1:r}, [a,b]) \subseteq 
  \cH^k_n$}, many results on \smash{$\cH^k_n$} carry over immediately to the
``sparse'' knot set case: we can still view the basis in \eqref{eq:ffb_sk}
from the same constructive lens (via discrete integration of step functions) as
in Theorem \ref{thm:ffb_discrete_integ}; functions in
\smash{$\DS^k_n(t_{1:r},[a,b])$} still exhibit the same matching derivatives
property as in Corollary \ref{cor:deriv_match}; a dual basis to
\eqref{eq:ffb_sk} is given by a subset of the functions in \eqref{eq:dual_basis}
from Lemma \ref{lem:dual_basis} (namely, the functions corresponding to the
indices $1,\ldots,k+1$ and $i_j$, $j=1,\ldots,r$); and interpolation within the
space \smash{$\DS^k_n(t_{1:r}, [a,b])$} can be done efficiently, precisely as in
Theorem \ref{thm:ffb_interp} or Corollary \ref{cor:ffb_interp_implicit}
(assuming we knew evaluations of $f$ at the design points, $f(x_i)$,
$i=1,\ldots,n$). 

Meanwhile, other developments---such as key matrix computations involving the 
falling factorial basis matrix, and the construction of discrete B-splines---do 
not carry over trivially, and require further explanation; we give the details
in the following subsections.   

\subsection{Matrix computations}

The fact that the dual basis result from Lemma \ref{lem:dual_basis} implies  
\smash{$\Z^{k+1}_n \, \B^{k+1}_n$} is the inverse of \smash{$\H^k_n$}, as in 
\eqref{eq:ffb_discrete_deriv_inv}, hinges critically on the fact that these 
matrices are square, which would not be the case for a general knot set
$t_{1:r}$, where the corresponding basis matrix would have dimension $n \times
(r+k+1)$. However, as we show next, this inverse result can be suitably and
naturally extended to a rectangular basis matrix. 

\begin{lemma}
\label{lem:ffb_discrete_deriv_pinv}
For any $k \geq 0$, and knots $t_1 < \cdots < t_r$ with $t_{1:r} \subseteq
x_{(k+1):(n-1)}$, let us abbreviate $T=t_{1:r}$ and let \smash{$\H^k_T \in \R^{n
    \times (r+k+1)}$} denote the $k$th degree falling factorial basis matrix
with entries     
$$
(\H^k_T)_{ij} = h^k_j(x_i), 
$$
where \smash{$h^k_j$}, $j=1,\ldots,r+k+1$ are the falling factorial basis
functions in \eqref{eq:ffb_sk} for \smash{$\DS^k_n(t_{1:r}, [a,b])$}.  

Let \smash{$\H^k_n \in \R^{n \times n}$} denote the ``usual'' $k$th degree
falling factorial basis matrix \eqref{eq:ffb_mat}, defined over
$x_{(k+1):(n-1)}$, and let $J = \{1,\ldots,k+1\} \cup \{i_j+1 : j=1,\ldots,r\}$, 
where \smash{$t_j = x_{i_j}$} for $j=1,\ldots,r$. Observe that  
\begin{equation}
\label{eq:ffb_submat}
\H^k_T = (\H^k_n)_J,
\end{equation}
where we write \smash{$(\H^k_n)_S$} to represent the submatrix defined by
retaining the columns of \smash{$\H^k_n$} in a set $S$. Furthermore, let  
\smash{$\A^{k+1}_n=\Z^{k+1}_n \, \B^{k+1}_n$}, where \smash{$\Z^{k+1}_n,
  \B^{k+1}_n \in \R^{n\times n}$} are the ``usual'' $(k+1)$st order extended  
weight and extended discrete derivative matrix, defined over the knot set 
$x_{(k+1):(n-1)}$, as in \eqref{eq:weight_mat_ext} and
\eqref{eq:discrete_deriv_mat_ext}, respectively. Then the (Moore-Penrose)
generalized inverse of \smash{$\H^k_T$} can be expressed as 
\begin{equation}
\label{eq:ffb_discrete_deriv_pinv} 
(\H^k_T)^\dagger = (\A^{k+1}_n)_J 
\big(\I_n - (\A^{k+1}_n)_{J^c}^\dagger (\A^{k+1}_n)_{J^c}\big),     
\end{equation}
where we use \smash{$(\A^{k+1}_n)_S$} to denote the submatrix formed by
retaining the rows of \smash{$\A^{k+1}_n$} in a set $S$, and recall we use
$\I_n$ for the $n \times n$ identity matrix. A direct consequence of the above
is 
\begin{equation}
\label{eq:ffb_discrete_deriv_spaces}
\col\big((\H^k_n)_J\big) = \nul\big((\A^{k+1}_n)_{J^c}\big),
\end{equation}
where we use $\col(\M)$ and $\nul(\M)$ to denote the column space and null space
of a matrix $\M$, respectively.
\end{lemma}

\begin{proof}
We abbreviate \smash{$\H=\H^k_n$}, \smash{$\A=\A^{k+1}_n$}, and further, 
\smash{$\H_1=(\H^k_n)_J$}, \smash{$\H_2=(\H^k_n)_{J^c}$},
\smash{$\A_1=(\A^{k+1}_n)_J$}, \smash{$\A_2=(\H^{k+2}_n)_{J^c}$} for notational 
simplicity. Let $y \in \R^n$ be arbitrary, and consider solving the linear
system 
$$
\H_1^\T \H_1 \alpha = \H_1^\T y. 
$$
We can embed this into a larger linear system 
$$
\begin{bmatrix}
\H_1^\T \H_1 & \H_1^\T \H_2 \\
\H_2^\T \H_1 & \H_2^\T \H_2
\end{bmatrix}
\begin{bmatrix} 
\alpha \\ \beta 
\end{bmatrix} 
= 
\begin{bmatrix} 
\H_1^\T y \\ z
\end{bmatrix},
$$
which will yield the same solution $\alpha$ as our original system provided we
choose $z$ so that we have $\beta=0$ at the solution in the larger system. Now
inverting (using $\A \H = \I_n$), the above system is equivalent to 
$$
\begin{bmatrix} 
\alpha \\ \beta 
\end{bmatrix} 
= 
\begin{bmatrix}
\A_1 \A_1^\T & \A_1 \A_2^\T \\ 
\A_2 \A_1^\T & \A_2 \A_2^\T
\end{bmatrix}
\begin{bmatrix} 
\H_1^\T y \\ z
\end{bmatrix},
$$
that is, 
\begin{align*}
\alpha &= \A_1 y + \A_1 \A_2^\T z, \\
\beta &= \A_2 y + \A_2 \A_2^\T z.
\end{align*}
Setting the second line equal to zero gives \smash{$z = - \A_2^\dagger \A_2
  y$}, and plugging this back into the first gives \smash{$\alpha = \A_1 (y -
  \A_2^\dagger \A_2 y)$}. As $y$ was arbitrary, this proves the desired
result. 
\end{proof}

\begin{remark}
An important implication of \eqref{eq:ffb_discrete_deriv_pinv} is that we can
reduce least squares problems in the falling factorial basis \smash{$\H^k_T$} to
linear systems involving discrete derivatives. This is important for two
reasons: first, these discrete derivative systems can be solved in linear-time,
due to the bandedness of the discrete derivative matrices; second, these
discrete derivative systems are typically much better-conditioned than falling
factorial systems. However, it should be noted that these discrete derivative
systems can still suffer from poor conditioning for large problem sizes, and
discrete B-splines, as developed in Section \ref{sec:discrete_bs_sk}, offer a
much more stable computational route. This is demonstrated in Section 
\ref{sec:least_squares}. 
\end{remark}

\begin{remark}
Given the relationship in \eqref{eq:ffb_submat}, it is clear that multiplication
by \smash{$\H^k_T$} and \smash{$(\H^k_T)^\T$} can be done in linear-time, using 
the specialized, in-place algorithms described in Appendix \ref{app:fast_mult}.
To see this, note that for any $\alpha \in \R^{r+k+1}$ we can write
\smash{$\H^k_T \alpha = \H^k_n \beta$}, where we set the entries of $\beta \in
\R^n$ on $J$ according to \smash{$\beta_J = \alpha$}, and we set
\smash{$\beta_{J^c}=0$}. Also, for any $y \in \R^n$ we can write
\smash{$(\H^k_T)^\T y = ((\H^k_n)^\T y)_J$}.

Owing to \eqref{eq:ffb_discrete_deriv_pinv}, multiplication by
\smash{$(\H^k_T)^\dagger$} and \smash{$((\H^k_T)^\dagger)^\T$} can also be done
in linear-time; but it is unclear if these can be done entirely with
specialized, in-place algorithms. For multiplication by
\smash{$(\H^k_T)^\dagger$}, we can see that this reduces to multiplication by 
\smash{$((\A^{k+1}_n)_J$}, \smash{$(\A^{k+1}_n)_{J^c}$}, and
\smash{$(\A^{k+1}_n)_{J^c}^\dagger$}; while the first two are handled by the
algorithms in Appendix \ref{app:fast_mult}, the third requires solving a linear
system in the banded matrix \smash{$(\A^{k+1}_n)_{J^c} (\A^{k+1}_n)_{J^c}^\T$},
which as far as we can tell, cannot be done in-place in
generality. Multiplication by \smash{$((\H^k_T)^\dagger)^\T$} is similar.
\end{remark}

\subsection{Discrete B-splines}
\label{sec:discrete_bs_sk}

To construct a discrete B-spline or DB-spline basis for \smash{$\DS^k_n(t_{1:r},
  [a,b])$}, we assume that $r \geq k+2$ (otherwise it would not be possible to
construct $k$th degree DB-splines that have local support). First, we define
boundary design points    
$$
b = x_{n+1} < x_{n+2} < \cdots < x_{n+k+2},
$$
a boundary endpoint \smash{$\tilde{b} > x_{n+k+2}$}, and boundary knots 
$$
t_{r+1} = x_{n+1}, t_{r+2} = x_{n+2}, \ldots, t_{r+k+1} = x_{n+k+1}. 
$$
(Any such choice of \smash{$x_{n+2},\ldots,x_{n+k+2},\tilde{b}$} will suffice;
though our construction may appear to have different boundary considerations
compared to the ``dense'' case in Section \ref{sec:discrete_bs_evals}, these
differences are only notational, and our construction in what follows will
reduce exactly to the previous DB-splines when $t_{1:r}=x_{(k+1):(n-1)}$.)  Now,
for $j=1,\ldots,k+1$, we define the normalized DB-spline \smash{$N^k_j$} as   
follows:  
\begin{equation}
\begin{gathered}
\label{eq:discrete_nbs_sk1}
\text{$N^k_j$ is the unique function $f \in \DS^k_n(t_{1:j},[a,b])$
  satisfying} \\ 
f(x_1) = \cdots = f(x_{j-1}) = 0, \quad f(x_j) = 1, \quad \text{and} \\ 
f(x_{i_j-k+1}) = \cdots = f(x_{i_j}) = f(x_{i_j+1}) = 0.
\end{gathered}
\end{equation}
(Note that the space \smash{$\DS^k_n(t_{1:j},[a,b])$} is $(j+k+1)$-dimensional,
and above there are $j+k+1$ linearly independent constraints, hence this
system has a unique solution.)  Moreover, for $j=1,\ldots,r$, we define the 
normalized DB-spline \smash{$N^k_{j+k+1}$} as follows: 
\begin{equation}
\begin{gathered}
\label{eq:discrete_nbs_sk2}
\text{$N^k_{j+k+1}=f|_{[a,b]}$ for the unique function $f \in
  \DS^k_{n+k+2}(t_{j:(j+k+1)},[a, \tilde{b}])$ satisfying} \\ 
f(x_{i_j-k}) = \cdots = f(x_{i_j-1}) = f(x_{i_j}) = 0, 
\quad f(x_{i_{j+1}}) = 1,  \quad \text{and} \\ 
f(x_{i_{j+k+1}-k+1}) = \cdots = f(x_{i_{j+k+1}-1}) = f(x_{i_{j+k+1}}) =
f(x_{i_{j+k+1}+1}) = 0,
\end{gathered}
\end{equation}
where \smash{$\DS^k_{n+k+2}(t_{j:(j+k+1)},[a, \tilde{b}])$} is defined over the
extended design points $x_{1:(n+k+2)}$. (Note again this is
$(2k+3)$-dimensional, and the above system has $2k+3$ linearly independent 
constraints, so it has a unique solution.)  

The normalized DB-splines \smash{$N^k_j$}, $j=1,\ldots,r+k+1$, defined above in
\eqref{eq:discrete_nbs_sk1}, \eqref{eq:discrete_nbs_sk2}, form a basis for
\smash{$\DS^k_n(t_{1:r}, [a,b])$} (they lie in this space by design; and their
defining evaluations imply linear independence). The next result establishes the
key local support property; we omit its proof, as it follows from arguments
similar to Lemma \ref{lem:discrete_nbs_supp}.  

\begin{lemma}
\label{lem:discrete_nbs_sk_supp}
For any $k \geq 0$, the $k$th degree normalized DB-spline basis functions, as 
defined in \eqref{eq:discrete_nbs_sk1}, \eqref{eq:discrete_nbs_sk2}, have
support structure: 
\begin{equation}
\label{eq:discrete_nbs_sk_supp}
\text{$N^k_j$ is supported on} \;
\begin{cases}
[a,t_j] & \text{if $j \leq k+1$} \\
[t_{j-k-1}, t_j \wedge b] & \text{if $j \geq k+2$},
\end{cases}
\end{equation}
where recall we abbreviate $x \wedge y = \min\{x,y\}$.
\end{lemma}

Just as before, in the ``dense'' knot set case, we see that each $k$th degree
DB-spline is supported on at most $k+2$ knot points. Furthermore, all of the
other remarks following Lemma \ref{lem:discrete_nbs_supp} carry over
appropriately to the current setting. We refer back to Figure
\ref{fig:dbs} for examples of DB-splines of degree 2, with a ``sparse''
knot set. 

\subsection{Evaluation at the design points}

This subsection covers a critical computational development: starting from the
definitions \eqref{eq:discrete_nbs_sk1}, \eqref{eq:discrete_nbs_sk2}, we can
fill in the evaluations of each basis function \smash{$N^k_j$} at the design
points $x_{1:n}$ using an entirely ``local'' scheme involving discrete
derivative systems. This ``local'' scheme is both numerically stable (much more
stable than solving \eqref{eq:discrete_nbs_sk1}, \eqref{eq:discrete_nbs_sk2}
using say the falling factorial basis) and linear-time. 

Fix $j \geq k+2$, and consider the following ``local'' strategy for computing  
\smash{$N^k_j(x_{1:n})$}. First recall that \smash{$N^k_j(x_i)=0$} for $x_i
\leq t_{j-k-1}$ and $x_i \geq t_j \wedge b$, by \eqref{eq:discrete_nbs_sk_supp}
in Lemma \ref{lem:discrete_nbs_sk_supp}, so we only need to calculate
\smash{$N^k_j(x_i)$} for $t_{j-k-1} < x_i < t_j \wedge b$. For notational
simplicity, and without a loss of generality, set $j=k+2$, and abbreviate
\smash{$f = N^k_{k+2}$}. Between the first knot and second knot, $t_1$ and 
$t_2$, note that we can compute the missing evaluations by solving the linear
system:       
\begin{equation}
\label{eq:discrete_nbs_local1}
\begin{aligned}
f[x_{i_1-k+1}, \ldots, x_{i_1+1}, x_{i_1+2}] &= 0, \\
f[x_{i_1-k+2}, \ldots, x_{i_1+2}, x_{i_1+3}] &= 0, \\
& \;\; \vdots \\
f[x_{i_2-k-1}, \ldots, x_{i_2-1}, x_{i_2}] &= 0.
\end{aligned}
\end{equation}
This has $i_2-i_1-1$ equations and the same number of unknowns,
\smash{$f(x_{(i_1+1):(i_2-1)})$}. Between the second and last knot, $t_2$ and
$t_{k+2}$, we can set up a similar linear system in order to perform
interpolation. From \eqref{eq:discrete_nbs_sk2}, recall that \smash{$f(x_i)=0$}
for \smash{$x_i \geq x_{i_{k+2}-k+1}$}, and thus we only need to interpolate
from \smash{$x_{i_2+1}$} to \smash{$x_{i_{k+2}-k}$}. Our linear system of
discrete derivatives is comprised of the equations:   
\begin{equation}
\label{eq:discrete_nbs_local2}
f[x_{\ell-k},\ldots,x_{\ell+1}] = 0, \quad \text{for $i_2+1 \leq \ell \leq
  i_{k+2}-1$, $\ell \notin \{i_m : m=3,\ldots,k+2\}$}.
\end{equation}
These are discrete derivatives at each $x_{\ell+1}$ such that $x_\ell$ is {\it 
  not} a knot point. There are exactly $i_{k+2}-i_2-1-(k-1)=i_{k+2}-i_2-k$ such
equations and the same number of unknowns,
\smash{$f(x_{(i_2+1):(i_{k+2}-k)})$}. Hence, putting this all together, we have
shown how to compute all of the unknown evaluations of \smash{$f = N^k_{k+2}$}.  

For $j \leq k+1$, the ``local'' strategy for computing \smash{$N^k_j(x_{1:n})$}
is similar but even simpler. Abbreviating \smash{$f = N^k_j$}, we solve the
linear system:   
\begin{equation}
\label{eq:discrete_nbs_local3}
f[x_{\ell-k},\ldots,x_{\ell+1}] = 0, \quad \text{for $k+1 \leq \ell \leq i_j-1$,
  $\ell \notin \{i_m : m=1,\ldots,j-1\}$}.
\end{equation}
This has $i_j-k-1-(j-1)=i_j-k-j$ equations and the same number of unknowns, 
\smash{$f(x_{(j+1):(i_j-k)})$}.

A critical feature of the linear systems \eqref{eq:discrete_nbs_local1},
\eqref{eq:discrete_nbs_local2}, \eqref{eq:discrete_nbs_local3} that we must
solve in order to calculate the evaluations of the DB-spline basis functions is
that they are ``local'', meaning that they are defined by discrete derivatives
over a local neighborhood of $O(k)$ design points. Therefore these systems will
be numerically stable to solve, as the conditioning of the discrete derivative
matrices of such a small size (just $O(k)$ rows) will not be an
issue. Furthermore, since each design point $x_i$ appears in the support of at
most $k+2$ basis functions, computing all evaluations of all basis functions,
$N^k_j(x_i)$ for $j=1,\ldots,r+k+1$ and $i=1,\ldots,n$, takes linear-time.

\subsection{Least squares problems}
\label{sec:least_squares}

Finally, we investigate solving least squares problems in the falling factorial
basis, of the form 
$$
\minimize_\alpha \; \|y - \H^k_T \alpha\|_2^2
$$
In particular, suppose we are interested in the least squares projection 
\begin{equation}
\label{eq:ffb_system}
\hy = \H^k_T (\H^k_T)^\dagger y.
\end{equation}
By \eqref{eq:ffb_discrete_deriv_spaces} in Lemma
\ref{lem:ffb_discrete_deriv_pinv}, we know that we can alternatively compute 
this by projecting onto \smash{$\nul((\A^{k+1}_n)_{J^c})$},   
\begin{equation}
\label{eq:discrete_deriv_system}
\hy = \big(\I_n - (\A^{k+1}_n)_{J^c}^\dagger (\A^{k+1}_n)_{J^c}\big) y.    
\end{equation}
Another alternative is to use the DB-spline basis constructed in the last
subsection. Denoting by \smash{$\N^k_T \in \R^{n \times (r+k+1)}$} the matrix
with entries \smash{$(\N^k_T)_{ij} = N^k_j(x_i)$}, where \smash{$N^k_j$},
$j=1,\ldots,r+k+1$ are defined in \eqref{eq:discrete_nbs_sk1}, 
\eqref{eq:discrete_nbs_sk2} (recall from the last subsection that their
evaluations can be computed in linear-time), we have   
\begin{equation}
\label{eq:discrete_nbs_system}
\hy = \N^k_T (\N^k_T)^\dagger y.
\end{equation}

Naively, solving the falling factorial linear system \eqref{eq:ffb_system}
requires $O(n(r+k)^2)$ operations. A larger issue is that this system will be 
typically very poorly-conditioned. The discrete derivative linear system
\eqref{eq:discrete_deriv_system} gives an improvement in both computation
time and conditioning: it requires $O(nk^2)$ operations (because it requires us
to solve a linear system in the banded matrix \smash{$(\A^{k+1}_n)_{J^c}
((\A^{k+1}_n)_{J^c})^\T$}), and will typically be better-conditioned than the
falling factorial system. Finally, the DB-spline linear system
\eqref{eq:discrete_nbs_system} is computationally the same but improves
conditioning even further: it again takes $O(nk^2)$ operations (as it requires
us to solve a linear system in the banded matrix \smash{$(\N^k_T)^\T
  (\N^k_T)$}), and will typically be much better-conditioned than the discrete
derivative system. 

To substantiate these claims about conditioning, we ran an empirical experiment
with the following setup. For each problem size $n=100,200,500,1000,2000,5000$,
we considered both a fixed evenly-spaced design $x_{1:n}$ on $[0,1]$, and a
random design given by sorting i.i.d.\ draws from the uniform distribution
$[0,1]$. In each case (fixed or random design), we then selected $r=n/10$
points to serve as knots, drawing these uniformly at random from the allowable
set of design points $x_{(k+1):(n-1)}$, where $k=3$. Next we formed the key
matrices \smash{$\H^k_T, (\A^{k+1}_n)_{J^c}, \N^k_n$} appearing in the 
linear systems \eqref{eq:ffb_system}, \eqref{eq:discrete_deriv_system},
\eqref{eq:discrete_nbs_system}, and computed their condition numbers, where
we define the condition number of a matrix $\M$ by
$$
\kappa(\M) = \frac{\lambda_{\max}(\M^\T \M)}{\lambda_{\min}(\M^\T \M)}
$$
(with $\lambda_{\max}(\cdot)$ and $\lambda_{\min}(\cdot)$ returning the maximum
and minimum eigenvalues of their arguments). We set $\kappa(\M)=\infty$ when
$\lambda_{\min}(\M^\T \M)<0$ due to numerical inaccuracy. Figure
\ref{fig:sim_dbs} plots the condition numbers for these systems versus the
problem size $n$, where the results are aggregated over multiple repetitions:
for each $n$, we took the median condition number over 30 repetitions of forming
the design points and choosing a subset of knots. We see that for evenly-spaced
design points (fixed design case), the falling factorial systems degrade quickly
in terms of conditioning, with an infinite median condition number after
$n=500$; the discrete derivative and DB-spline systems are much more stable, and
the latter marks a huge improvement over the former (for example, its median
condition is more than 2000 times smaller for $n=5000$). For unevenly-spaced
design points (random design case), the differences are even more dramatic: now
both the falling factorial and discrete derivative systems admit an infinite
median condition number at some point (after $n=200$ and $n=1000$,
respectively), yet the DB-spline systems remain stable throughout.

\begin{figure}[tb]
\centering
\includegraphics[width=0.495\textwidth]{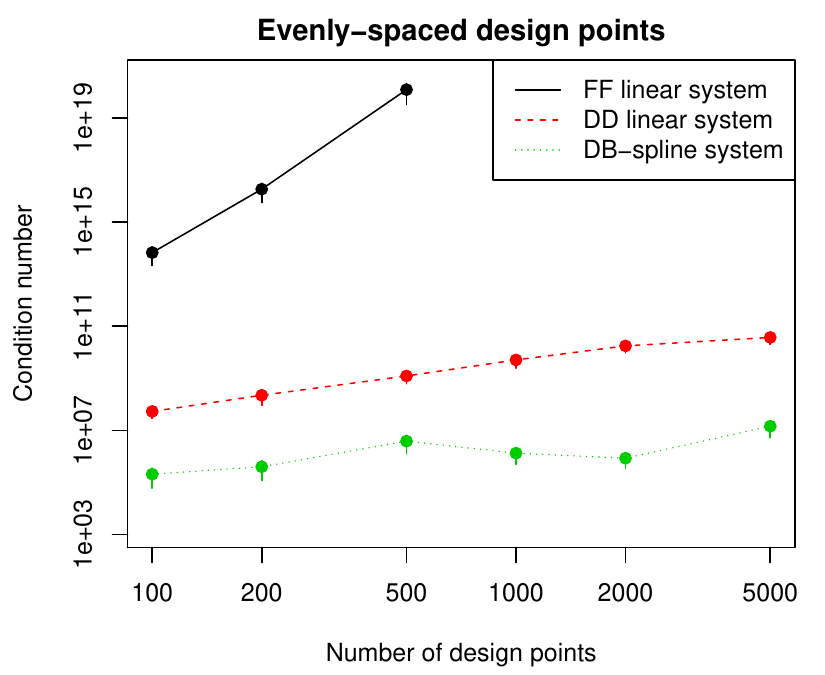}
\includegraphics[width=0.495\textwidth]{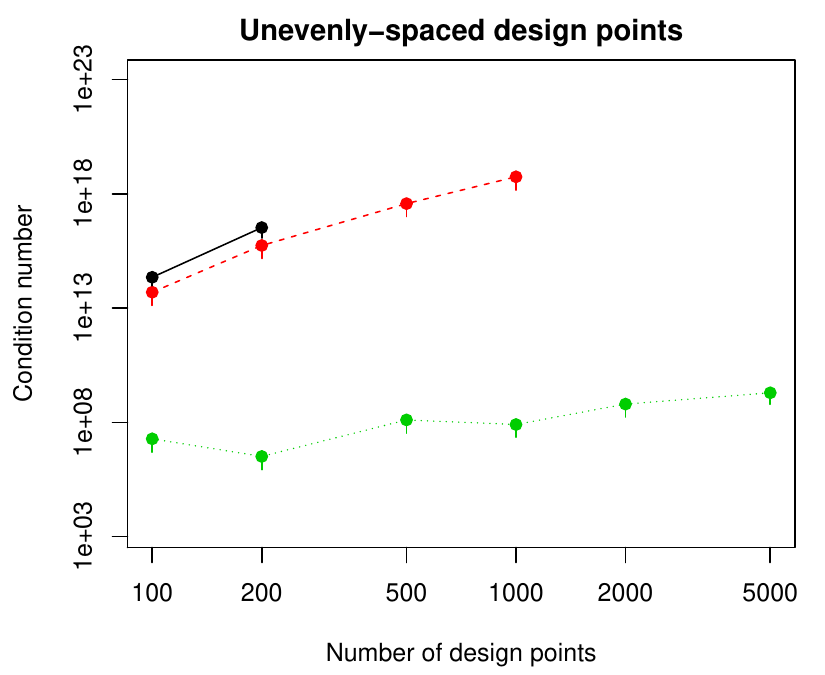}
\caption{\small Comparison of condition numbers for three equivalent linear
  systems, based on falling factorials (FF), discrete derivatives (DD), and
  discrete B-splines (DB-splines), for projecting onto a space of cubic discrete
  splines with $n$ design points and $r=n/10$ knots, as $n$ varies from 100 to
  5000. The left panel shows a case where the design points are evenly-spaced
  on $[0,1]$, and the right shows a case where they are sorted i.i.d.\ draws
  from the uniform distribution on $[0,1]$. In each case, the knots are assigned
  to be a random subset of the design points, and the median condition number is
  computed over 30 repetitions. Median absolute deviations are also shown as
  vertical lines. We can see that the FF systems, in solid black, are by far
  the worse in either case; the DD systems, in dashed red, provide a   marked
  improvement over FF systems in the evenly-spaced case, but just a moderate
  improvement in the unevenly-spaced case; and the DB-spline systems, in dotted
  green, provide a considerable improvement over both, in either case.} 
\label{fig:sim_dbs}
\end{figure}

\section{Representation}
\label{sec:representation}

In this section, we study the representational properties of discrete splines
with respect to two smoothness functionals: total variation and Sobolev
seminorms, which serve as the penalty functionals in the variational
optimization problems for locally adaptive regression splines and smoothing
splines, respectively. In particular, we show that such smoothness functionals,
for a discrete spline $f$, have exact representations in terms of discrete
derivatives of $f$ at the design points. We recall that for a $k$th degree
discrete spline $f$, and $\ell=1,\ldots,k$, we use $(D^\ell f)(x)$ to denotes
the $\ell$th derivative at $x$ when this exists, and the $\ell$th left
derivative when it does not (when $x$ is one of the knot points supporting
$f$). 

\subsection{Total variation functionals}
\label{sec:ffb_tv}

Below we show that for a $k$th degree discrete spline, the total variation of
its $k$th derivative can be written in terms of a weighted $\ell_1$ norm of its
$(k+1)$st discrete derivatives at the design points. Recall that the total
variation of a function $f$ on an interval $[a,b]$ is defined by
$$
\TV(f) = \sup_{a=z_0<z_1<\cdots<z_N=b} \; \sum_{i=1}^N |f(z_i) - f(z_{i-1})|. 
$$
The next result is an implication of Corollary
\ref{cor:deriv_match}. It serves as one of the main motivating points behind
trend filtering (as an approximation to locally adaptive regression splines);
essentially the same result can be found in Lemma 5 of 
\citet{tibshirani2014adaptive} (for evenly-spaced design points), and Lemma 2 of       
\citet{wang2014falling} (for arbitrary design points). 

\begin{theorem}
\label{thm:ffb_tv}
For any $k \geq 0$, and any $k$th degree discrete spline $f \in \cH^k_n$ (with   
knots in \smash{$x_{(k+1):(n-1)}$}), as defined in \eqref{eq:ffb_span}, it holds 
that \begin{equation} 
\label{eq:ffb_tv}
\TV(D^k f) =  \sum_{i=k+2}^n \big| (\Delta^{k+1}_n f)(x_i) \big| \cdot 
\frac{x_i - x_{i-k-1}}{k+1}.
\end{equation}
Equivalently, with $f(x_{1:n})=(f(x_1),\ldots,f(x_n)) \in \R^n$ denoting the
vector of evaluations of $f$ at the design points,
\begin{equation}
\label{eq:ffb_tv_mat}
\TV(D^k f) =  \big\|\W^{k+1}_n \D^{k+1}_n f(x_{1:n}) \big\|_1,  
\end{equation}
where \smash{$\D^{k+1}_n \in \R^{(n-k-1) \times n}$} is the $(k+1)$st
order discrete derivative matrix, as in \eqref{eq:discrete_deriv_mat}, and
\smash{$\W^{k+1}_n \in \R^{(n-k-1) \times (n-k-1)}$} is the $(k+1)$st order
diagonal weight matrix, as in \eqref{eq:weight_mat}.
\end{theorem}

\begin{proof}
As $D^k f$ is piecewise constant with knots in \smash{$x_{(k+1):(n-1)}$} (and
our convention is to treat it as left-continuous),
\begin{align}
\label{eq:ffb_tv_sum1}
\TV(D^k f) &= \sum_{i=k+2}^n \big|(D^k f)(x_i) - (D^k f)(x_{i-1})\big| \\
\label{eq:ffb_tv_sum2}
&= \sum_{i=k+2}^n \big| (\Delta^k_n f)(x_i) - (\Delta^k_n f)(x_{i-1})\big|,
\end{align}
where in the second line we used the matching derivatives result from
Corollary \ref{cor:deriv_match}. Recalling the recursive formulation for
\smash{$\Delta^{k+1}_n$} from \eqref{eq:discrete_deriv_rec2} establishes the   
result.  
\end{proof}

\begin{remark}
\label{rem:ffb_tv_mat_old}
As discussed previously, recall that
\citet{tibshirani2014adaptive,wang2014falling} defined the discrete derivative
operators differently, specifically, they defined the operators according the
recursion \eqref{eq:discrete_deriv_mat_old} (compare this to the recursion
\eqref{eq:discrete_deriv_mat} in the current paper). These papers also expressed
the total variation result in \eqref{eq:ffb_tv_mat} differently, recall
\eqref{eq:ffb_tv_mat_old}, where the modified operator \smash{$\C^{k+1}_n =
  \W^{k+1}_n \D^{k+1}_n$} results from the construction in
\eqref{eq:discrete_deriv_mat_old}. While the results \eqref{eq:ffb_tv_mat} and 
\eqref{eq:ffb_tv_mat_old} are equivalent, the latter is arguably a more natural
presentation of the same result, as it invokes the more natural notion of
discrete differentiation from this paper (recall Remark
\ref{rem:discrete_deriv_mat_old}). Using this notion, it then represents the
total variation functional via differences of discrete derivatives (which equal
differences of derivatives, recall \eqref{eq:ffb_tv_sum2} in the proof of
Theorem \ref{thm:ffb_tv}).
\end{remark}

\begin{remark}
\label{rem:ffb_tv_def}
Once we assume $f$ lies in an $n$-dimensional linear space of $k$th degree 
piecewise polynomials with knots in $x_{1:n}$, the fact that the
representation \eqref{eq:ffb_tv_mat} holds for {\it some} matrix
\smash{$\W^{k+1}_n$} is essentially definitional. To see this, we can expand  
$f$ in a basis for this linear space, \smash{$f=\sum_{j=1}^n \alpha_j
  g_j$}, then observe that, for some matrix $\Q \in \R^{n \times n}$ (that
depends on this basis, but not on $f$),
\begin{align*}
\TV(D^k f) &= \TV\bigg( \sum_{j=1}^n \alpha_j D^k g_j \bigg) \\
&= \|\Q \alpha\|_1 \\
&= \big\|\Q \, \G^{-1} f(x_{1:n})\big\|_1.
\end{align*}
In the second line we used the fact that each $D^k g_j$ is a piecewise constant
function (with knots in $x_{1:n}$), and in the third line we simply multiplied
by $\G \in \R^{n \times n}$ and its inverse, which has entries $\G_{ij} =
g_j(x_i)$. Now in the last line above, if we multiplied by \smash{$\D^m_n$} and
its ``inverse'' (in quotes, since this matrix is not square, thus strictly
speaking, not invertible), then this would yield a result as in
\eqref{eq:ffb_tv_mat} for a particular matrix \smash{$\W^{k+1}_n$} (defined in
terms of \smash{$\Q,\G^{-1}$}, and the ``inverse'' of $\D^m_n$). But to be
clear, the fact that \eqref{eq:ffb_tv_mat} holds for a {\it diagonal} matrix
\smash{$\W^{k+1}_n$} is what makes the result special, and is tied to the
matching derivatives property that is uniquely satisfed $k$th degree discrete
splines. For example, the corresponding matrix \smash{$\W^{k+1}_n$} would not 
be diagonal for $k$th degree splines.
\end{remark}

\subsection{$L_2$-Sobolev functionals}

Now we show that for a $k$th degree discrete spline, where $k=2m-1$, the
integral of the square of its $m$th derivative can be written in terms of a
certain quadratic form of its $m$th discrete derivatives at the design
points. This integral is (the square of) the seminorm naturally associated with
the $L_2$-Sobolev space $\cW^{m,2}([a,b])$.

\begin{theorem}
\label{thm:ffb_sobolev}
For any odd $k = 2m-1 \geq 1$, and any $k$th degree discrete spline $f \in
\cH^k_n$ (with knots in \smash{$x_{(k+1):(n-1)}$}), as defined in
\eqref{eq:ffb_span}, it holds that 
\begin{equation}
\label{eq:ffb_sobolev}
\int_a^b (D^m f)(x)^2 \, dx =
\big\|(\V^m_n)^{\hspace{-1pt}\frac{1}{2}} \D^m_n f(x_{1:n}) \big\|_2^2,     
\end{equation}
where $f(x_{1:n})=(f(x_1),\ldots,f(x_n)) \in \R^n$ is the vector of evaluations
of $f$ at the design points, and \smash{$\D^m_n \in \R^{(n-m) \times n}$} is
the $m$th order discrete derivative matrix, as in \eqref{eq:discrete_deriv_mat}. 
Moreover, \smash{$\V^m_n \in \R^{(n-m) \times (n-m)}$} is a symmetric banded 
matrix (that depends only on $x_{1:n}$) of bandwidth $2m-1$. 
\end{theorem}

The proof of Theorem \ref{thm:ffb_sobolev} is somewhat intricate and is
deferred to Appendix \ref{app:ffb_sobolev}. It relies on several key properties 
underlying discrete splines, specifically, the recursive property of the falling
factorial basis in Lemma \ref{lem:ffb_lateral_rec}, and the dual relationship in
Lemma \ref{lem:dual_basis}. 

\begin{remark}
As before (similar to Remark \ref{rem:ffb_tv_def} on the total variation
representation result), once we assume that $f$ lies in an $n$-dimensional
linear space, the fact the representation \eqref{eq:ffb_sobolev} holds for {\it
  some} matrix \smash{$\V^m_n$} is essentially definitional. We can see this by
expanding $f$ in terms of a basis for this linear space, \smash{$f=\sum_{j=1}^n
\alpha_j g_j$}, then observing that
\begin{align*}
\int_a^b (D^m f)(x)^2 \, dx 
&= \int_a^b \sum_{i,j=1}^n \alpha_i \alpha_j (D^m g_i)(x) (D^m g_j)(x) \, dx \\ 
&= \alpha^\T \Q \alpha \\
&= f(x_{1:n})^\T \G^{-\T} \Q \, \G^{-1} f(x_{1:n}),
\end{align*}
where $\Q,\G \in \R^{n\times n}$ have entries \smash{$\Q_{ij} = \int_a^b (D^m
  g_i)(x) (D^m g_j)(x) \, dx$} and \smash{$\G_{ij}=g_j(x_i)$}. In the last line 
above, if we multiplied by \smash{$\D^m_n$} and its ``inverse'' (in quotes,
because this matrix is not square, hence not invertible), then this would yield a 
result as in \eqref{eq:ffb_sobolev} for a particular matrix \smash{$\V^m_n$}
(defined in terms of \smash{$\Q,\G^{-1}$}, and the ``inverse'' of $\D^m_n$). To
be clear, the fact that \eqref{eq:ffb_sobolev} holds for a {\it banded} matrix 
\smash{$\V^m_n$} is highly nontrivial, and this appears to be special to the
space of $k$th degree discrete splines. For example, the corresponding matrix 
\smash{$\V^m_n$} would not be banded for $k$th degree splines. On the other
hand, for splines, the {\it inverse} of this matrix turns out to be banded;
recall Theorem \ref{thm:nsp_sobolev}.  
\end{remark}

\begin{remark}
It is worth noting that the nature of the result in Theorem
\ref{thm:ffb_sobolev} is, at a high level, quite different from previous
results in this paper. Thus far, the core underlying property enjoyed by $k$th 
degree discrete splines has been the fact that their $k$th derivatives and $k$th
discrete derivatives match everywhere, as stated in Corollary
\ref{cor:deriv_match}. This led to the dual basis result in Lemma
\ref{lem:dual_basis}, the implicit form interpolation result in Corollary
\ref{cor:ffb_interp_implicit}, and the total variation representation result in
Theorem \ref{thm:ffb_tv}. Meanwhile, the $L_2$-Sobolev representation result in
Theorem \ref{thm:ffb_sobolev} is a statement about connecting a functional of
$m$th derivatives of $k$th degree discrete splines, where $k=2m-1$, to their
$m$th discrete derivatives. In other words, this connects derivatives and
discrete derivatives whose order does not match the degree of the piecewise
polynomial. That this is still possible (and yields a relatively simple and
computationally efficient form) reveals another new feature of discrete splines,
and brings hope that discrete splines may harbor even more results of this
type (discrete-continuous connections) that are yet to be discovered.
\end{remark}

The form of the matrix \smash{$\V^m_n$} in \eqref{eq:ffb_sobolev} can be made
explicit. This is a consequence of the proof of Theorem \ref{thm:ffb_sobolev}.

\begin{lemma}
\label{lem:ffb_sobolev_vmat}
The matrix \smash{$\V^m_n \in \R^{(n-m) \times (n-m)}$} from Theorem
\ref{thm:ffb_sobolev} can be defined via recursion, in the following manner. 
First define a matrix $\M \in \R^{(n-m) \times (n-m)}$ to have entries        
\begin{equation}
\label{eq:ffb_sobolev_mmat}
\M_{ij} = \int_a^b (D^m h^k_{i+m})(x) (D^m h^k_{j+m})(x) \, dx,
\end{equation}
where recall \smash{$h^k_j$}, $j=1,\ldots,n$ are the falling factorial basis 
functions in \eqref{eq:ffb}. For a matrix $\A$ and positive integers $i,j$,
introduce the notation   
$$
\A(i,j) = 
\begin{cases}
\A_{ij} & \text{if $\A$ has at least $i$ rows and $j$ columns} \\
0 & \text{otherwise},
\end{cases}
$$
as well as \smash{$\delta^r_{ij}(\A) = \A(i,j)-\A(i+1,j)$} and  
\smash{$\delta^c_{ij}(\A) = \A(i,j)-\A(i,j+1)$}. Then \smash{$\V^m_n=\V^{m,m}$}
is the termination point of a $2m$-step recursion, initialized at
\smash{$\V^{0,0}=\M$}, and defined as follows:  
\begin{alignat}{2}
\label{eq:ffb_sobolev_vmat1}
\V^{\ell,0}_{ij} &= 
\begin{cases}
\V^{\ell-1,0}_{ij} & \text{if $i \leq m-\ell$} \\
\displaystyle
\delta^r_{ij}(\V^{\ell-1,0}) \cdot \frac{2m-\ell}{x_{i+m} - x_{i-(m-\ell)}} &
\text{if $i > m-\ell$},
\end{cases} & \qquad\qquad \text{for $\ell=1,\ldots,m-1$}, \\
\label{eq:ffb_sobolev_vmat2}
\V^{m,0}_{ij} &= \delta^r_{ij}(\V^{m-1,0}) \\
\label{eq:ffb_sobolev_vmat3}
\V^{m,\ell}_{ij} &= 
\begin{cases}
\V^{m,\ell-1}_{ij} & \text{if $j \leq m-\ell$} \vspace{3pt} \\
\displaystyle
\delta^c_{ij}(\V^{m,\ell-1}) \cdot \frac{2m-\ell}{x_{j+m} - x_{j-(m-\ell)}} &
\text{if $j > m-\ell$},
\end{cases} & \qquad\qquad \text{for $\ell=1,\ldots,m-1$}, \\  
\label{eq:ffb_sobolev_vmat4}
\V^{m,m}_{ij} &= \delta^c_{ij}(\V^{m,m-1}).
\end{alignat}
\end{lemma}

Furthermore, as we show next, the matrix $\M$ in \eqref{eq:ffb_sobolev_mmat} can
be expressed in an explicit form (circumventing the need for numerical 
integration). The proof is an application of integration by parts and is given
in Appendix \ref{app:ffb_sobolev_mmat}. 

\begin{lemma}
\label{lem:ffb_sobolev_mmat}
The entries of the matrix $\M \in \R^{(n-m) \times (n-m)}$ from Lemma  
\ref{lem:ffb_sobolev_vmat} can be written explicitly, for $i \geq j$, as
\begin{equation}
\label{eq:ffb_sobolev_mmat2}
\M_{ij} = 
\begin{cases}
\displaystyle
\Bigg(\sum_{\ell=1}^{i-1} (-1)^{\ell-1}
(D^{m+\ell-1} h^k_{i+m}) (x) (D^{m-\ell} h^k_{j+m})(x)
+ (-1)^{i-1} (D^{m-i} h^k_{j+m})(x) \Bigg) \Bigg|_a^b  
& \text{if $i \leq m$} \\
\displaystyle
\Bigg(\sum_{\ell=1}^{m-1} (-1)^{\ell-1}
(D^{m+\ell-1} h^k_{i+m}) (x) (D^{m-\ell} h^k_{j+m})(x)
+ (-1)^{m-1} h^k_{j+m}(x) \Bigg) \Bigg|_{x_{i+m-1}}^b 
& \text{if $i > m$},
\end{cases}
\end{equation}
where recall the derivatives of the falling factorial basis functions are given
explicitly in \eqref{eq:ffb_deriv}, and we use the notation
$$
f(x) \Big|_s^t = \big(f^-(t)-f^+(s)\big). 
$$
as well as \smash{$f^-(x) = \lim_{t \to x^-} f(t)$} and  
\smash{$f^+(x) = \lim_{t \to x^+} f(t)$}.
\end{lemma}

We conclude this subsection by generalizing Theorem \ref{thm:ffb_sobolev}.
Inspection of its proof shows that the only property of the integration
operator (defining the Sobolev functional) that is actually used in Theorem
\ref{thm:ffb_sobolev} (and Lemma \ref{lem:ffb_sobolev_vmat}) is linearity; we
can therefore substantially generalize this representational result as follows.  

\begin{theorem}
\label{thm:ffb_sobolev_gen}
Let $L$ be a linear functional (acting on functions over $[a,b]$). For any odd
$k = 2m-1 \geq 1$, and any $k$th degree discrete spline $f \in \cH^k_n$, as
defined in \eqref{eq:ffb_span}, it holds that     
\begin{equation}
\label{eq:ffb_sobolev_gen}
L (D^m f)^2 = \big\|(\V^m_{n,L})^{\hspace{-1pt}\frac{1}{2}} \D^m_n f(x_{1:n})
\big\|_2^2,       
\end{equation}
where $f(x_{1:n})=(f(x_1),\ldots,f(x_n)) \in \R^n$ is the vector of evaluations
of $f$ at the design points, and \smash{$\D^m_n \in \R^{(n-m) \times n}$} is
the $m$th order discrete derivative matrix, as in \eqref{eq:discrete_deriv_mat}. 
Further, \smash{$\V^m_{n,L} \in \R^{(n-m) \times (n-m)}$} is a symmetric banded  
matrix (depending only on $x_{1:n}$ and $L$) of bandwidth $2m-1$. As before,
it can be defined recursively: \smash{$\V^m_{n,L}$} is the termination point of
the recursion in \eqref{eq:ffb_sobolev_vmat1}--\eqref{eq:ffb_sobolev_vmat4}, but
now initialized at the matrix $\M$ with entries
\begin{equation}
\label{eq:ffb_sobolev_gen_mmat}
\M_{ij} = L (D^m h^k_{i+m}) (D^m h^k_{j+m}),
\end{equation}
where \smash{$h^k_j$}, $j=1,\ldots,n$ are the falling factorial basis functions
in \eqref{eq:ffb}.
\end{theorem}

\begin{remark}
Theorem \ref{thm:ffb_sobolev_gen} allows for a generic linear operator $L$, and
hence covers, for example, a weighted $L_2$-Sobolev functional of the form
\smash{$\int_a^b (D^m f)(x)^2 w(x) \, dx$} for a weight function $w$. We could 
further generalize this to a functional defined by integration with respect to
an arbitrary measure $\mu$ on $[a,b]$ (Lebesgue-Stieltjes integration).
For such a class of functionals, some version of integration by parts, and thus 
an explicit result for the entries of $\M$ in \eqref{eq:ffb_sobolev_gen_mmat},
analogous to Lemma \ref{lem:ffb_sobolev_vmat}, would still be possible.  

We emphasize once more that the proof of Theorem \ref{thm:ffb_sobolev_gen}
follows immediately from that of Theorem \ref{thm:ffb_sobolev}. It is not clear
to us that the spline result in \eqref{eq:nsp_sobolev} from Theorem
\ref{thm:nsp_sobolev}, due to \citet{schoenberg1964spline}, would extend as
seamlessly to an arbitrary linear functional $L$. The proof is closely tied to
the Peano representation of the B-spline, and therefore for an arbitrary linear
functional $L$, the B-spline itself would need to be replaced by an appropriate
kernel.
\end{remark}


\section{Approximation}
\label{sec:approximation}

Approximation theory is a vast subject, and is particularly well-developed for
splines; see, for example,  Chapters 6 and 7 of \citet{schumaker2007spline};
or Chapters 5, 12, and 13 of \citet{devore1993constructive}. Assuming an   
evenly-spaced design, Chapter 8.5 of \citet{schumaker2007spline} develops 
approximation results for discrete splines that are completely analogous to  
standard spline approximation theory. Roughly speaking, Schumaker shows that 
discrete splines obtain the same order of approximation as splines, once we
measure approximation error and smoothness in suitable discrete-time notions.

Extending these results to arbitrary design points seems nontrivial, although it  
is reasonable to expect that similar approximation results should hold in this
case. Instead of pursuing this line of argument, in this section, we give some
very simple (crude) approximation results for discrete splines, by bounding
their distance to splines and then invoking standard spline approximation
results. The intent is not to give approximation results that are of the optimal 
order---in fact, the  approximation rates obtained will be grossly
suboptimal---but  ``good enough'' for typical use in nonparametric statistical
theory (for example, for bounding the approximation error in trend filtering, as  
discussed in the next section). A finer analysis of discrete spline
approximation may be the topic of future work. 

\subsection{Proximity of truncated power and falling factorial bases} 

We can easily bound the $L_\infty$ distance between certain truncated power and
falling factorial basis functions, as we show next. Denote by \smash{$\cG^k_n 
  = \S^k(x_{(k+1):(n-1)}, [a,b])$}, the space of $k$th degree splines on $[a,b]$ 
with knots in $x_{(k+1):(n-1)}$. As a basis for \smash{$\cG^k_n$}, recall that
we have the truncated power basis \smash{$g^k_j$}, $j=1,\ldots,n$, as in 
\eqref{eq:tpb}, but with $t_{1:r}=x_{(k+1):(n-1)}$ (to be explicit, 
\smash{$g^k_j(x)=(x-x_{j-1})_+^k/k!$}, for each $j=k+2,\ldots,n$).  
The first part \eqref{eq:ffb_tpb_approx1} of the result below is a trivial
strengthening of Lemma 4 in \citet{wang2014falling}, and the second part
\eqref{eq:ffb_tpb_approx2} can be found in the proof of Lemma 13 in 
\citet{sadhanala2019additive}. 

\begin{lemma}
\label{lem:ffb_tpb_approx}
For design points $a \leq x_1 < \cdots < x_n \leq b$, let  \smash{$\delta_n =
  \max_{i=1,\ldots,n-1} \;  (x_{i+1}-x_i)$} denote the maximum gap between
adjacent points. For $k \geq 0$, let \smash{$g^k_j$, $j=1,\ldots,n$} denote
the truncated power basis for \smash{$\cG^k_n$}, as in \eqref{eq:tpb} (but 
with $t_{1:r}=x_{(k+1):(n-1)}$), and \smash{$h^k_j$, $j=1,\ldots,n$} denote 
the falling factorial basis for \smash{$\cH^k_n$}, as in \eqref{eq:ffb}. For
$k=0$ or $k=1$, and each $j=k+2,\ldots,n$, recall that \smash{$g^k_j=h^k_j$},
and hence \smash{$\cG^k_n=\cH^k_n$}. Meanwhile, for $k \geq 2$, and each
$j=k+2,\ldots,n$,    
\begin{equation}
\label{eq:ffb_tpb_approx1}
\|g^k_j - h^k_j \|_{L_\infty} \leq \frac{k (b-a)^{k-1}}{(k-1)!} \delta_n,  
\end{equation}
where \smash{$\|f\|_{L_\infty}=\sup_{x \in [a,b]} \; |f(x)|$} denotes the $L_\infty$ 
norm of a function $f$ on $[a,b]$. Hence for each spline \smash{$g \in \cG^k_n$},
there exists a discrete spline \smash{$h \in \cH^k_n$} such that 
\begin{equation}
\label{eq:ffb_tpb_approx2}
\TV(D^k h) = \TV(D^k g), \quad \text{and} \quad 
\|g-h\|_{L_\infty} \leq \frac{k (b-a)^{k-1}}{(k-1)!} \delta_n \cdot \TV(D^k g).
\end{equation}
\end{lemma}

\begin{proof}
The proof is simple. For each $j=k+2,\ldots,n$, consider for $x > x_{j-1}$,  
\begin{align}
\nonumber
k! \cdot |g^k_j(x) - h^k_j(x)| 
&= \prod_{\ell=j-k}^{j-1} (x-x_\ell) - (x-x_{j-1})^k \\
\nonumber
&\leq (x-x_{j-k})^k - (x-x_{j-1})^k \\
\nonumber
&= (x_{j-1}-x_{j-k}) \sum_{\ell=1}^k (x-x_{j-k})^{\ell-1}
  (x-x_{j-1})^{k-\ell} \\
\nonumber
&\leq k \delta_n (x-x_{j-k})^{k-1} \\
\label{eq:ffb_tpb_approx3}
&\leq k^2 \delta_n (b-a)^{k-1}.
\end{align}
This proves the first part \eqref{eq:ffb_tpb_approx1}. As for the second part
\eqref{eq:ffb_tpb_approx2}, write \smash{$g=\sum_{j=1}^n \alpha_j g^k_j$}, and
then define 
$$
h=\sum_{j=1}^{k+1} \frac{\alpha_j}{(j-1)!} x^{j-1} + 
\sum_{j=k+2}^n \alpha_j h^k_j,
$$ 
Note that \smash{$h \in \cH^k_n$}, and we have specified its polynomial part 
to match that of $g$. We have   
$$
\TV(D^k h) = \TV(D^k g) = \|\alpha_{(k+2):n}\|_1.
$$
Furthermore, using \eqref{eq:ffb_tpb_approx3}, for any $x \in [a,b]$,
$$
|g(x) - h(x)| \leq \sum_{j=k+2}^n |\alpha_j| |g^k_j(x) - h^k_j(x)| 
\leq \frac{k}{(k-1)!} \delta_n \cdot \TV(D^k f),
$$
which completes the proof.
\end{proof}

\subsection{Approximation of bounded variation functions}

Next we show how to couple Lemma \ref{lem:ffb_tpb_approx} with standard spline
approximation theory to derive discrete spline approximation results for
functions whose derivatives are of bounded variation. First we state the spline 
approximation result; for completeness we give its proof in Appendix
\ref{app:tpb_bv_approx} (similar arguments were used in the proof of
Proposition 7 of \citet{mammen1997locally}). 

\begin{lemma}
\label{lem:tpb_bv_approx}
Let $f$ be a function that is $k$ times weakly differentiable on $[0,1]$, such
that $D^k f$ is of bounded variation. Also let $0 \leq x_1 < \cdots < x_n \leq
1$ be arbitrary design points. Then there exists a $k$th degree spline
\smash{$g \in \cG^k_n$}, with knots in \smash{$x_{(k+1):(n-1)}$}, such that for
$k=0$ or $k=1$, 
\begin{equation}
\label{eq:tpb_bv_approx1}
\TV(D^k g) \leq \TV(D^k f), \quad \text{and} \quad 
g(x_i) = f(x_i), \; i=1,\ldots,n,
\end{equation}
and for $k \geq 2$, 
\begin{equation}
\label{eq:tpb_bv_approx2}
\TV(D^k g) \leq a_k \TV(D^k f), \quad \text{and} \quad 
\|f-g\|_{L_\infty} \leq b_k \delta_n^k \cdot \TV(D^k f), 
\end{equation}
where \smash{$\delta_n = \max_{i=1,\ldots,n-1} \; (x_{i+1}-x_i)$} denotes the
maximum gap between adjacent design points, and $a_k,b_k>0$ are constants that 
depend only on $k$.
\end{lemma}

Combining Lemmas \ref{lem:ffb_tpb_approx} and \ref{lem:tpb_bv_approx} and using
the triangle inequality leads immediately to the following result.

\begin{lemma}
\label{lem:ffb_bv_approx}
Let $f$ be a function that is $k$ times weakly differentiable on $[0,1]$, such
that $D^k f$ is of bounded variation. Also let $0 \leq x_1 < \cdots < x_n \leq
1$ be arbitrary design points. Then there exists a $k$th degree discrete spline 
\smash{$h \in \cH^k_n$}, with knots in \smash{$x_{(k+1):(n-1)}$}, such that for
$k=0$ or $k=1$, 
\begin{equation}
\label{eq:ffb_bv_approx1}
\TV(D^k h) \leq \TV(D^k f), \quad \text{and} \quad 
h(x_i) = f(x_i), \; i=1,\ldots,n,
\end{equation}
and for $k \geq 2$, 
\begin{equation}
\label{eq:ffb_bv_approx2}
\TV(D^k h) \leq a_k \TV(D^k f), \quad \text{and} \quad 
\|f-h\|_{L_\infty} \leq c_k \delta_n \cdot \TV(D^k f), 
\end{equation}
where \smash{$\delta_n = \max_{i=1,\ldots,n-1} \; (x_{i+1}-x_i)$} denotes the
maximum gap between adjacent design points, and $a_k,c_k>0$ are constants that 
depend only on $k$ (note $a_k$ is the same constant as in Lemma
\ref{lem:tpb_bv_approx}). 
\end{lemma}

\begin{remark}
The approximation bound for discrete splines in \eqref{eq:ffb_bv_approx2} scales
with $\delta_n$, which is weaker than the order $\delta_n^k$ approximation we
can obtain with splines, in \eqref{eq:tpb_bv_approx2}. It is reasonable to
believe that discrete splines can also obtain an order $\delta_n^k$
approximation, with a finer analysis. Before we discuss this further, we
emphasize once more that an order $\delta_n$ approximation is ``good enough''
for our eventual statistical purposes, as discussed in the next section, because
it will be on the order of $\log{n}/n$ with high probability when the design
points are sorted i.i.d.\ draws from a continous distribution on $[0,1]$ (for
example, Lemma 5 in \citet{wang2014falling}), and this is of (much) smaller
order than the sought estimation error rates, which (on the $L_2$ scale, not
squared $L_2$ scale) will always be of the form $n^{-r}$ for $r<1/2$. 

Now, the culprit---the reason that \eqref{eq:ffb_bv_approx2} ``suffers'' a rate
of $\delta_n$ and not $\delta_n^k$---is the use of truncated power and falling
factorial bases in Lemma \ref{lem:ffb_tpb_approx}. Fixing any $j \geq k+2$, the
fact \smash{$g^k_j,h^k_j$} do not have local support means that the factor of 
\smash{$(x-x_{j-k})^{k-1}$} in the line preceding \eqref{eq:ffb_tpb_approx3} can 
grow to a large (constant) order, as $x$ moves away from the shared knot point
$x_{j-k}$, and thus in a uniform sense over all $x > x_{j-k}$ (and all $j \geq 
k+2$), we can only bound it by $(b-a)^{k-1}$, as done in
\eqref{eq:ffb_tpb_approx3}. A way to fix this issue would be to instead
consider locally-supported bases, that is, to switch over to comparing B-splines
and discrete B-splines: with the appropriate pairing, each basis function
(B-spline and DB-spline) would be supported on the same interval containing
$k+2$ design points, which would have width at most $(k+2)\delta_n$. This
should bring the $L_\infty$ distance between pairs of basis functions down to
the desired order of $\delta_n^k$. 

However, a better way forward, to refining approximation results, seems to be to
analyze discrete splines directly (not just analyze their approximation
capacity via their proximity to splines). For this, we imagine DB-splines
should also play a prominent role: for example, it is not hard to see that the
map $P$ defined by \smash{$Pf = \sum_{i=1}^n f(x_i) N^k_i$}, where
\smash{$N^k_i$}, $i=1,\ldots,n$ is the DB-spline basis in
\eqref{eq:discrete_nbs} (written explicitly in \eqref{eq:discrete_nbs_expl}), is
a bounded linear projector onto the space \smash{$\cH^k_n$}.  (We
mean bounded with respect to the $L_\infty$ norm, that is, \smash{$\|P\| =   
  \sup_{\|g\|_{L_\infty}  \leq 1} \; \|Pg\|_{L_\infty} < \infty$}.)  Thus it
achieves within a global constant factor of the optimal approximation error 
(pointwise for each function $f$): for any $h \in \cH^k_n$, we have 
\smash{$\|f-Pf\|_{L_\infty} \leq \|f-h\|_{L_\infty} +   \|Pf-Ph\|_{L_\infty}   
  \leq (1+\|P\|) \|f-h\|_{L_\infty}$}, which implies  
$$
\|f-Pf\|_{L_\infty} \leq (1+\|P\|) \cdot \inf_{h \in \cH^k_n} \;
\|f-h\|_{L_\infty}. 
$$
\end{remark}

\section{Trend filtering}
\label{sec:trend_filter}

In this section, we revisit trend filtering, in light of our developments on
discrete splines in the previous sections. The following subsections outline
some computational improvements, and then introduce a variant of trend filtering
based on discrete natural splines (which often shows better boundary behavior).
Before this, we briefly revisit some aspects of its interpretation and
estimation theory, to highlight the application of the matching derivatives
result (from Corollary \ref{cor:deriv_match}) and approximation guarantees
(from Lemma \ref{lem:ffb_bv_approx}). 

\paragraph{Penalizing differences of $k$th discrete derivatives.}  

In the trend filtering problem \eqref{eq:trend_filter}, where \smash{$\D^{k+1}_n
  \in \R^{(n-k-1) \times n}$} is the $(k+1)$st order discrete derivative matrix,
as in \eqref{eq:discrete_deriv_mat}, and \smash{$\W^{k+1}_n \in \R^{(n-k-1)
    \times (n-k-1)}$} the $(k+1)$st order diagonal weight matrix, as in 
\eqref{eq:weight_mat}, note that its penalty can be written as
\begin{align}
\label{eq:trend_filter_wpen2a}
\big\|\W^{k+1}_n \D^{k+1}_n \theta \big\|_1
&= \sum_{i=1}^{n-k-1} \big| (\D^{k+1}_n \theta)_i \big| \cdot
\frac{x_{i+k+1} - x_i}{k+1} \\
\label{eq:trend_filter_wpen2b}
&= \sum_{i=1}^{n-k-1} \big| (\D^k_n \theta)_{i+1} - (\D^k_n\theta)_i \big|. 
\end{align}
The first line was given previously in \eqref{eq:trend_filter_wpen}, and we copy
it here for convenience; the second line is due to the recursive definition
\eqref{eq:discrete_deriv_mat} of the discrete derivative matrices. In other
words, we can precisely interpret the trend filtering penalty as an absolute sum
of {\it differences} of $k$th discrete derivatives of $\theta$ at adjacent
design points. This provides the most direct path to the continuous-time 
formulation of trend filtering: for the unique $k$th degree discrete spline $f$
with $f(x_{1:n})=\theta$, it is immedate that \eqref{eq:trend_filter_wpen2b} is
an absolute sum of its $k$th derivatives at adjacent design points, once we
recall the matching derivatives property from Corollary \ref{cor:deriv_match}; 
and as $D^k f$ is piecewise constant with knots at the design points, it is easy
to see that this equals $\TV(D^k f)$. That is, it is easy to work backwards
from \eqref{eq:trend_filter_wpen2b} through the steps \eqref{eq:ffb_tv_sum2},
\eqref{eq:ffb_tv_sum1}, and \eqref{eq:ffb_tv}. The conclusion is, of course, as 
before: the trend filtering problem \eqref{eq:trend_filter} is equivalent the
variational problem \eqref{eq:trend_filter_cont}, where we restrict the
optimization domain in the locally adaptive regression spline problem to the 
space \smash{$\cH^k_n$} of $k$th degree splines with knots in $x_{(k+1):(n-1)}$.  

\paragraph{Estimation theory via oracle inequalities.}

\citet{tibshirani2014adaptive} established estimation error bounds for trend
filtering by first proving that the trend filtering and (restricted) locally
adaptive regression spline estimators, in \eqref{eq:trend_filter_cont} and
\eqref{eq:local_spline_rest}, are ``close'' (in the $\ell_2$ distance defined
with respect to the design points $x_{1:n}$), and then invoking existing
estimation results for the (restricted) locally adaptive regression spline
estimator from \citet{mammen1997locally}. These bounds were refined for
arbitrary design points in \citet{wang2014falling}. The conclusion is that the
trend filtering estimator \smash{$\hf$} in \eqref{eq:trend_filter_cont} achieves
(under mild conditions on the design points) the minimax error rate in
\eqref{eq:minimax_rate}, over the class of functions $\cV^k$ whose $k$th weak
derivative has total variation bounded by a constant $C<\infty$.

It was later shown in \citet{sadhanala2019additive} that the same result could
be proved more directly, without a need to bound the distance between the trend    
filtering and (restricted) locally adaptive spline estimators. The setting in
\citet{sadhanala2019additive} is more general (additive models, where the 
dimension of the design points is allowed to grow with $n$); here we relay the 
implication of their results, namely, Theorem 1 and Corollary 1, for (univariate)
trend filtering, and explain where the approximation result from Lemma
\ref{lem:ffb_bv_approx} enters the picture. If $x_i$, $i=1,\ldots,n$ are sorted
i.i.d.\ draws from a continuous distribution $[0,1]$, and $y_i =
f_0(x_i)+\epsilon_i$, $i=1,\ldots,n$ for uniformly sub-Gaussian errors
$\epsilon_i$, $i=1,\ldots,n$ with mean zero and variance-proxy $\sigma^2>0$,
then there are constants $c_1,c_2,c_3,n_0>0$ depending only on $k,\sigma$ such
that for all $c \geq c_1$, $n \geq n_0$, and \smash{$\lambda \geq c
n^{\frac{1}{2k+3}}$}, the solution \smash{$\hf$} in the trend filtering problem 
\eqref{eq:trend_filter_cont} satisfies\footnote{To be clear, the result in
  \eqref{eq:trend_filter_err_bound} is of a somewhat classical 
  oracle-inequality-type flavor, and similar results can be found in many other  
  papers; the theoretical novelty in \citet{sadhanala2019additive} lies in the
  analysis of additive models with growing dimension, which is given in their
  Theorem 2 and Corollary 2.}
\begin{equation}
\label{eq:trend_filter_err_bound}
\frac{1}{n} \big\| \hf(x_{1:n}) - f_0(x_{1:n}) \big\|_2^2 \leq 
\frac{1}{n} \big\| h(x_{1:n}) - f_0(x_{1:n}) \big\|_2^2 + 
\frac{6\lambda}{n} \max\{1, \TV(D^k h)\}.
\end{equation}
with probability at least \smash{$1-\exp(-c_2c)-\exp(-c_3\sqrt{n})$},
simultaneously over all \smash{$h \in \cH^k_n$} such that 
$(1/n) \| h(x_{1:n}) - f_0(x_{1:n}) \|_2^2 \leq C$. The first term on the
right-hand side in \eqref{eq:trend_filter_err_bound} is the approximation error,
and can be controlled using Lemma \ref{lem:ffb_bv_approx}. When $k=0$ or $k=1$,
we can see from \eqref{eq:ffb_bv_approx1} that we can set it exactly to zero.
When $k \geq 2$, assuming the underlying regression function $f_0$ satisfies
\smash{$\TV(D^k    f_0) \leq 1$}, we can see from \eqref{eq:ffb_bv_approx2} that
we can choose $h$ so that  
$$
\frac{1}{n} \big\| h(x_{1:n}) - f_0(x_{1:n}) \big\|_2^2 \leq 
\big\| h(x_{1:n}) - f_0(x_{1:n}) \big\|_{L_\infty}^2 \leq c_k^2 \delta_n^2.
$$
When the density of the design points is bounded below by a positive constant,
it can be shown (see Lemma 5 of \citet{wang2014falling}) that $\delta_n$ is on
the order of $\log{n}/n$ with high probability. The right-hand side in the
display above is thus on the order of $(\log{n}/n)^2$ with high probability, and
so the first term in \eqref{eq:trend_filter_err_bound} is negligible compared to 
the second. All in all, for any $k \geq 0$, we get that for \smash{$\TV(D^k
  f_0) \leq 1$} and \smash{$\lambda=cn^{\frac{1}{2k+3}}$}, we can choose $h$ 
so that the first term in \eqref{eq:trend_filter_err_bound} is negligible and
the second term is on the order of \smash{$n^{-\frac{2k+2}{2k+3}}$} (where we
have used the bound on $\TV(D^k h)$ from \eqref{eq:ffb_bv_approx1} or 
\eqref{eq:ffb_bv_approx2}). This establishes that trend filtering achieves the
desired minimax estimation error rate.

\subsection{Computational improvements}
\label{sec:trend_filter_comp} 

We discuss computational implications of our developments on discrete splines
for trend filtering.

\paragraph{Efficient interpolation.} 

To state the obvious, both the explicit and implicit interpolation formulae,
from Theorem \ref{thm:ffb_interp} and Corollary \ref{cor:ffb_interp_implicit},
respectively, can be applied directly to trend filtering. Starting with the
discrete-time solution \smash{$\htheta$} from \eqref{eq:trend_filter}, we can
efficiently compute the unique $k$th degree discrete spline interpolant
\smash{$\hf$} to these values, that is, efficiently evaluate \smash{$\hf(x)$} at
any point $x$. The two different perspectives each have their strengths,
explained below.

\begin{itemize}
\item  To use the explicit formula \eqref{eq:ffb_interp}, note that we only need
  to store the $k+1$ polynomial coefficients, \smash{$(\Delta^{k+1}_n
    \hf)(x_i)$}, $i=1,\ldots,k+1$, and the coefficients corresponding to the
  active knots, \smash{$(\Delta^{k+1}_n \hf)(x_i)$}, $i \in I$, where 
  $$
  I = \Big\{i \geq k+2 : (\Delta^{k+1}_n \hf)(x_i) \not= 0\Big\}.
  $$
  As for the design points, in order to use \eqref{eq:ffb_interp}, we
  similarly only need to store $x_{1:(k+1)}$ as well as $x_{(i-k-1):i}$, $i
  \in I$. Thus for $r=|I|$ active knots, we need $O(r+k)$ memory 
  and $O((r+k)k)$ operations to compute \smash{$\hf(x)$} via
  \eqref{eq:ffb_interp}.  

\item To use the implicit formulae \eqref{eq:ffb_interp_implicit1},
  \eqref{eq:ffb_interp_implicit2}, we need to store all evaluations
  \smash{$\htheta=\hf(x_{1:n})$}, and all design points $x_{1:n}$, that is, we
  require $O(n)$ memory. Given this, to compute \smash{$\hf(x)$} we then need 
  to locate $x$ among the design points, which is at most $O(\log{n})$
  operations (via binary search), and solve a single linear system in one
  unknown, which costs $O(k)$ operations to set up. Hence the total cost of 
  finding \smash{$\hf(x)$} via \eqref{eq:ffb_interp_implicit1},
  \eqref{eq:ffb_interp_implicit2} is $O(\log{n}+k)$ operations (or even smaller,
  down to $O(k)$ operations if the design points are evenly-spaced, because
  then locating $x$ among the design points could be done with integer
  division). The implicit interpolation strategy is therefore more efficient
  when memory is not a concern and the number of active knots $r$ is large (at
  least $r=\Omega(\log{n})$). 
\end{itemize}

\paragraph{DB-spline polishing.}   Given the trend filtering solution
\smash{$\htheta$} in \eqref{eq:trend_filter}, let \smash{$\C^{k+1}_n=\W^{k+1}_n
  \, \D^{k+1}_n$}, and define the set of active coordinates \smash{$I=\{i : 
  (\C^{k+1}_n \htheta)_i \not= 0\}$} and vector of active signs
\smash{$s=\sign((\C^{k+1}_n \htheta)_I)$}. Based on the Karush-Kuhn-Tucker (KKT)
conditions for \eqref{eq:trend_filter} (see \citet{tibshirani2011solution} or
\citet{tibshirani2012degrees}), it can be shown that     
\begin{equation}
\label{eq:trend_filter_sol1}
\htheta = \big(\I_n - (\C^{k+1}_n)_{I^c}^\dagger
(\C^{k+1}_n)_{I^c}\big) \big(y - (\C^{k+1}_n)_I^\T s\big),
\end{equation}
where \smash{$(\C^{k+1}_n)_S$} denotes the submatrix formed by retaining the
rows of \smash{$\C^{k+1}_n$} in a set $S$. Recall the extended version of
\smash{$\C^{k+1}_n$}, namely, \smash{$\A^{k+1}_n=\Z^{k+1}_n \, \B^{k+1}_n$} from 
Lemma \ref{lem:ffb_discrete_deriv_pinv}, and define a set $J = \{1,\ldots,k+1\}
\cup \{i+k+1 : i \in I\}$. Then \smash{$(\C^{k+1}_n)_{I^c} =
  (\A^{k+1}_n)_{J^c}$}, and \eqref{eq:trend_filter_sol1} is the projection of
\smash{$y - (\C^{k+1}_n)_I^\T s$} onto \smash{$\nul((\A^{k+1}_n)_{J^c})$}. 
Thus, by the same logic as that in Section \ref{sec:least_squares} (recall the 
equivalence of \eqref{eq:discrete_deriv_system} and
\eqref{eq:discrete_nbs_system}), we can rewrite \eqref{eq:trend_filter_sol1} as      
\begin{equation}
\label{eq:trend_filter_sol2}
\htheta = \N^k_T (\N^k_T)^\dagger \big(y - (\C^{k+1}_n)_I^\T s\big), 
\end{equation}
where $T=\{t_j: j=1,\ldots,r\}$ is the active knot set, with $r=|I|$ and 
$t_j=x_{j+k}$, $j \in I$, and where \smash{$\N^k_T \in \R^{n \times (r+k+1)}$} 
is the DB-spline basis matrix with entries \smash{$(\N^k_T)_{ij} =
  N^k_j(x_i)$}, for \smash{$N^k_j$}, $j=1,\ldots,n$ as defined in
\eqref{eq:discrete_nbs_sk1}, \eqref{eq:discrete_nbs_sk2}. We argued in
Section \ref{sec:least_squares} that linear systems in DB-splines (like
\eqref{eq:trend_filter_sol2}) have the same computational cost yet a
significantly better degree of stability than linear systems in discrete 
derivatives (like \eqref{eq:trend_filter_sol1}). Hence, a very simple idea
for improving the numerical accuracy in trend filtering solutions is as follows:   
form a candidate solution \smash{$\htheta$}, keep only the active set $I$ and 
active signs $s$, and then {\it polish} the solution using DB-splines
\eqref{eq:trend_filter_sol2} (note that this requires $O(nk^2)$ operations, due 
to the bandedness of \smash{$\N^k_T$}). 

\paragraph{DB-spline ADMM.}

Instead of just using DB-splines post-optimization (to polish an
already-computed trend filtering solution), a more advanced idea would be to use
DB-splines to improve stability over the course of optimization directly. As an
example, we consider a specialized augmented Lagrangian method of multipliers
(ADMM) for trend filtering due to \citet{ramdas2016fast}. To derive this
algorithm, we first rewrite \eqref{eq:trend_filter}, using the recursion 
\eqref{eq:discrete_deriv_mat}, as
\begin{equation}
\label{eq:trend_filter_admm_prob}
\minimize_{\theta,z} \; \frac{1}{2} \|y - \theta\|_2^2 + 
\lambda \big\|\widebar\D_{n-k} z \big\|_1 
\quad \st \quad z = \D^k_n \theta,
\end{equation}
and define the augmented Lagrangian, for a parameter $\rho>0$,
\begin{equation}
\label{eq:aug_lagrangian}
L(\theta,z,u) = \frac{1}{2} \|y - \theta\|_2^2 + 
\lambda \big\|\widebar\D_{n-k} z \big\|_1 +
\frac{\rho}{2} \big\| z - \D^k_n \theta + u \big\|_2^2
-\frac{\rho}{2} \|u\|_2^2.
\end{equation}
Minimizing over $\theta$, then $z$, then taking a gradient ascent step with 
respect to the dual variable $u$, gives the updates
\begin{align}
\label{eq:trend_filter_admm_theta}
\theta^+ &= \big(\I_n + \rho (\D^k_n)^\T \D^k_n \big)^{-1} 
\big(y + (\D^k_n)^\T (z + u)\big), \\
\label{eq:trend_filter_admm_z}
z^+ &= \argmin_z \bigg\{ \frac{1}{2} \big\| z - \D^k_n \theta + u \big\|_2^2 + 
\frac{\lambda}{\rho} \big\|\widebar\D_{n-k} \, z \big\|_1 \bigg\}, \\
\label{eq:trend_filter_admm_u}
u^+ &= u + z - \D^k_n \theta.
\end{align}
The $z$-update in \eqref{eq:trend_filter_admm_z} may look at first like the most
expensive step, but it can be done with super-efficient, linear-time algorithms
for total variation denoising (such algorithms take advantage of the simple
pairwise difference structure in the $\ell_1$ penalty), for example, based on
dynamic programming \citep{johnson2013dynamic}. The $\theta$-update in
\eqref{eq:trend_filter_admm_theta} is just a banded linear system solve, which
is again linear-time, but it is (perhaps surprisingly) the more problematic
update in practice due to poor conditioning of the discrete derivative matrices. 

Recalling the notable empirical benefits in using DB-splines for similar systems
(see Figure \ref{fig:sim_dbs} in Section \ref{sec:least_squares}), it is
reasonable to believe that DB-splines could provide a big improvement in
stability if used within this ADMM algorithm as well. The trick is to first
define a {\it working active set} based on an intermediate value of $z$, namely,
$$
I = \{ i : (\widebar{D}_{n-k} \, z)_i \not= 0 \}.
$$  
This could, for example, be computed after running a handful of the ADMM
iterations in
\eqref{eq:trend_filter_admm_theta}--\eqref{eq:trend_filter_admm_u}. We then
restrict our attention to optimization over 
\smash{$z \in  \nul((\widebar{D}_{n-k})_{I^c}$} and 
\smash{$\theta \in \nul((\D^{k+1}_n)_{I^c})$} 
(and upon convergence, we check the KKT conditions for the full problem
\eqref{eq:trend_filter_admm_prob}; if not satisfied then we increase the
working active set appropriately and repeat). With this restriction, the
$z$-update \eqref{eq:trend_filter_admm_z} just becomes a lower-dimensional total
variation denoising problem that can still be solved by dynamic
programming. More importantly, the $\theta$-update
\eqref{eq:trend_filter_admm_theta} can be now rewritten as
\begin{equation}
\label{eq:trend_filter_admm_theta_new}
\theta^+ = \N^k_T \Big((\N^k_T)^\T\big(\I_n + \rho (\D^k_n)^\T
\D^k_n \big) \N^k_T \Big)^{-1} (\N^k_T)^\T \big(y + \rho (\D^k_n)^\T 
(z + u)\big), 
\end{equation}
Here \smash{$\N^k_T \in \R^{n \times (r+k)}$} is the DB-spline basis matrix
defined with respect to the active knots $T=\{t_j: j=1,\ldots,r\}$, with
$r=|I|$ and $t_j=x_{j+k}$, $j \in I$. The step
\eqref{eq:trend_filter_admm_theta_new} is still a banded linear system solve, 
and thus still linear-time, but is much better-conditioned (the DB-spline basis 
matrix \smash{$\N^k_T$} acts something like a rectangular preconditioner).
Careful implementation and comparisons are left to future work. 

\subsection{Natural trend filtering}

For odd $k=2m-1 \geq 1$, consider further restricting the domain in the
continuous-time trend filtering problem \eqref{eq:trend_filter_cont} to the 
space \smash{$\cN^k_n$} of $k$th degree discrete natural splines on $[a,b]$ with
knots $x_{(k+1):(n-1)}$ (as defined in Section
\ref{sec:discrete_natural_spline}):  
\begin{equation}
\label{eq:trend_filter_nat_cont}
\minimize_{f \in \cN^k_n} \; \frac{1}{2} \sum_{i=1}^n \big(y_i - f(x_i)\big)^2 +  
\lambda \, \TV(D^k f).
\end{equation}
As motivation for this, recall the smoothing spline problem
\eqref{eq:smooth_spline} inherently gives rise to a natural spline as its
solution, which can have better boundary behavior (than a normal spline without
any boundary constraints). In fact, looking back at Figure \ref{fig:adapt}, we
can see evidence of this: despite deficiencies in coping with heterogeneous 
smoothness, the smoothing spline estimates (bottom row) have better boundary 
behavior than trend filtering (top right)---see, in particular, the very right
side of the domain. 

The estimator defined by \eqref{eq:trend_filter_nat_cont}, which we call {\it
  natural trend filtering}, can be recast in a familiar discrete-time form:  
\begin{equation}
\label{eq:trend_filter_nat}
\begin{alignedat}{2}
&\minimize_\theta &&\frac{1}{2} \|y - \theta\|_2^2 + 
\lambda \big\|\W^{k+1}_n \D^{k+1}_n \theta \big\|_1 \\
&\st \quad && \theta_{1:m} = \P^m_1 \theta_{(m+1):(k+1)} \\
&&& \theta_{(n-m+1):n} = \P^m_2 \theta_{(n-k):(n-m)},
\end{alignedat}
\end{equation}
Here $\P^m_1 \in \R^{m \times m}$ is a matrix that performs polynomial
extrapolation from function values on $x_{(m+1):(k+1)}$ to values on $x_{1:m}$,
that is, for any polynomial $p$ of degree $m-1$,   
$$
p(x_{1:m}) = \P^m_1 p(x_{(m+1):(k+1)}),
$$
and similarly, $\P^m_2 \in \R^{m \times m}$ performs polynomial extrapolation 
from $x_{(n-k):(n-m)}$ to $x_{(n-m+1):n}$. Observe that
\eqref{eq:trend_filter_nat} is just a standard trend filtering problem where the
first $m$ and last $m$ coordinates of $\theta$ are just linear combinations of
the second $m$ and second-to-last $m$, respectively. Computationally, this is
only a small variant on trend filtering (that is, it would require only a small
tweak on existing optimization approaches for trend filtering). Thanks to
the development of the DB-spline basis for discrete natural splines (see
\eqref{eq:discrete_natural_nbs} in Lemma \ref{lem:discrete_natural_nbs}), the
stability advantages of using DB-splines for trend filtering, as outlined in the
last subsection, should carry over here as well.  Finally, Figure
\ref{fig:nat} displays natural trend filtering fitted to the same data as in
Figure \ref{fig:adapt}, where we can indeed see that the boundary behavior
improves on the right side of the domain. 

\begin{figure}[tb]
\centering
\includegraphics[width=0.495\textwidth]{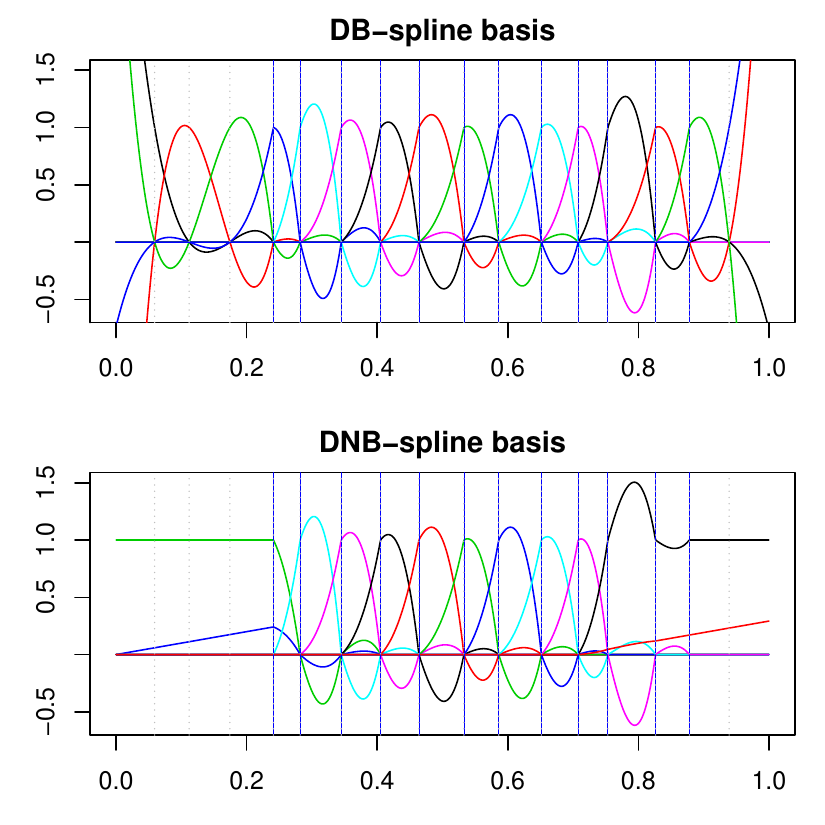}
\includegraphics[width=0.495\textwidth]{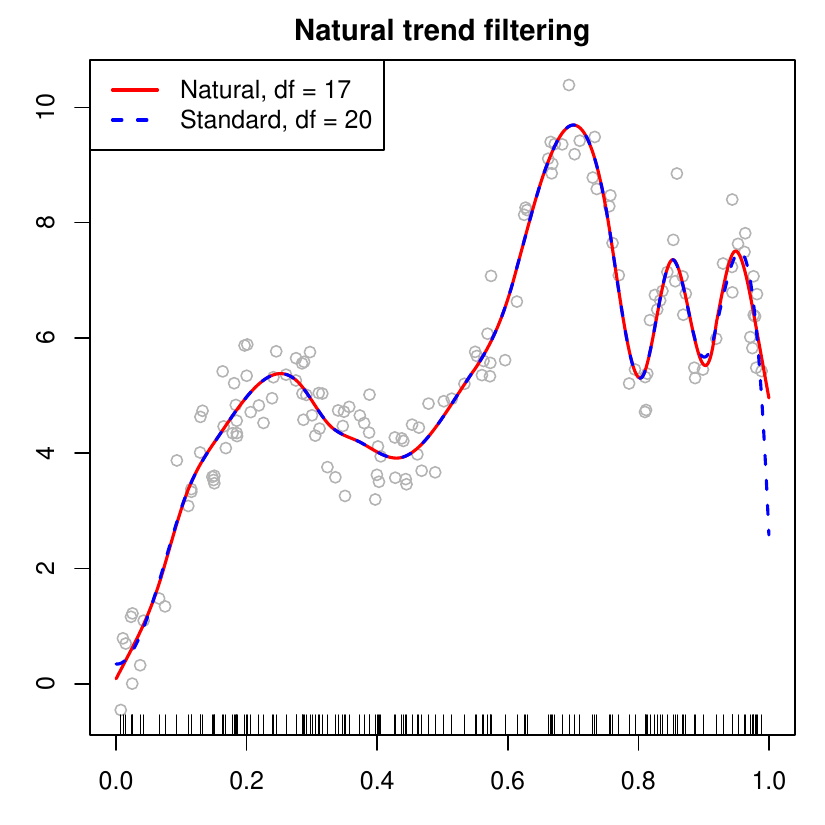}
\caption{\small Left panel: comparison of discrete B-spline (DB-spline) and
  discrete natural B-spline (DNB-spline) bases, for \smash{$\cH^k_n$} and 
  \smash{$\cN^k_n$}, respectively, where $k=3$ and the $n=16$ design points
  marked by dotted gray vertical lines. The knots are marked by blue vertical
  lines. We can see that the middle $n-4$ DB-splines are also DNB-splines; and
  the key difference is that the two leftmost and rightmost DNB-splines are
  linear beyond the boundary knots. Right panel: comparison of natural trend
  filtering and trend filtering on the same data example as in Figure
  \ref{fig:adapt}. They are computed using the same value of $\lambda$; we can 
  see that natural trend filtering, in solid red, essentially matches trend 
  filtering on the entirety of the domain, but it has better boundary behavior
  and uses fewer degrees of freedom (informally, it does not ``waste'' degrees
  of freedom in fitting at the extreme ends of the domain)).}     
\label{fig:nat}
\end{figure}

\section{BW filtering}
\label{sec:bw_filter}

We revisit Bohlmann-Whittaker (BW) filtering, focusing on the case of arbitrary 
design points. We first define a (slight) variant of the classical BW filter
with a weighted penalty, then develop connections to the smoothing spline.

\subsection{Weighted BW filtering}

Recall that for a unit-spaced design, the BW filter is defined in terms of an
quadratic program with a squared $\ell_2$ penalty on forward differences, as
given in \eqref{eq:bw_filter_unit}. For an arbitrary design $x_{1:n}$,
Whittaker proposed to replace forward differences by divided differences, in a
footnote of his famous 1923 paper \citep{whittaker1923new}, resulting in
\eqref{eq:bw_filter_old}. As we alluded to in Section
\ref{sec:bw_filter_review}, we will argue in what follows that it is in several
ways more natural to replace the penalty in \eqref{eq:bw_filter_old} with the 
weighted version \eqref{eq:bw_filter_wpen}, so that the problem becomes   
\begin{equation}
\label{eq:bw_filter}
\minimize_\theta \; \|y - \theta\|_2^2 + \lambda 
\big\| (\W^m_n)^{\hspace{-1pt}\frac{1}{2}} \D^m_n \theta\big\|_2^2.
\end{equation}
Here, \smash{$\D^m_n \in \R^{(n-m) \times n}$} is the $m$th order 
discrete derivative matrix, as in \eqref{eq:discrete_deriv_mat}, and 
\smash{$\W^m_n \in \R^{(n-m) \times (n-m)}$} is the $m$th order diagonal weight
matrix, as in \eqref{eq:weight_mat}. For convenience, we copy over
\eqref{eq:bw_filter_wpen}, to emphasize once again that the form of the 
penalty in \eqref{eq:bw_filter} is    
\begin{equation}
\label{eq:bw_filter_wpen2} 
\big\| (\W^m_n)^{\hspace{-1pt}\frac{1}{2}} \D^m_n \theta\big\|_2^2 = 
\sum_{i=1}^{n-m} (\D^m_n \theta)_i^2 \cdot \frac{x_{i+m} - x_i}{m}. 
\end{equation}
We note once again the strong similarity between the weighted BW filter in
\eqref{eq:bw_filter} and trend filtering in \eqref{eq:trend_filter}, that is,
the strong similarity between their penalties in \eqref{eq:bw_filter_wpen2}
and \eqref{eq:trend_filter_wpen2a}, respectively---the latter uses a weighted 
squared $\ell_2$ norm of discrete derivatives (divided differences), while the 
former uses a weighted $\ell_1$ norm. 

We now list three reasons why the weighted BW problem \eqref{eq:bw_filter} may
be preferable to the classical unweighted one \eqref{eq:bw_filter_old} for
arbitrary designs (for evenly-spaced design points, the two penalties are equal
up to a global constant, which can be absorbed into the tuning parameter; that
is, problems \eqref{eq:bw_filter} and \eqref{eq:bw_filter_old} are equivalent
modulo a rescaling of $\lambda$).       

\begin{enumerate}
\item For $m=1$ and $k=1$, Theorem \ref{thm:nsp_sobolev} tells us for  
  any natural linear spline $f$ on $[a,b]$ with knots at the design points
  $x_{1:n}$, we have the exact representation  
  \begin{equation}
    \label{eq:ffb_sobolev_m1}
  \int_a^b (Df)(x)^2 \, dx = \sum_{i=1}^{n-1} \big(\D_n f(x_{1:n})\big)_i^2
  \cdot (x_{i+1}-x_i). 
  \end{equation}
  This means that for $m=1$, the smoothing spline problem
  \eqref{eq:smooth_spline} is equivalent to the weighted BW problem
  \eqref{eq:bw_filter}. That is, to be perfectly explicit (and to emphasize the
  appealing simplicity of the conclusion), the following two problems are
  equivalent:  
  \begin{gather*}
    \minimize_f \; \sum_{i=1}^n \big(y_i-f(x_i)\big)^2 + \lambda \int_a^b
  (Df)(x)^2 \, dx,  \\
  \minimize_\theta \; \|y-\theta\|_2^2 + \lambda \sum_{i=1}^{n-1}
  (\theta_i-\theta_{i+1})^2 \cdot (x_{i+1}-x_i),
  \end{gather*}
  in the sense that their solutions satisfy \smash{$\htheta=\hf(x_{1:n})$}. 

\item For $m=2$ and $k=3$, we prove in Theorem \ref{thm:ss_bw_bound}
  in the next subsection that the weighted BW filter and smoothing spline
  are ``close'' in $\ell_2$ distance (for enough large values of their tuning 
  parameters). This enables the weighted BW filter to inherit the favorable
  estimation properties of the smoothing spline (over the appropriate
  $L_2$-Sobolev classes), as we show in Corollary \ref{cor:ss_bw_bound}. 

\item Empirically, the weighted BW filter seems to track the smoothing spline
  more closely than the unweighted BW filter does, for arbitrary design
  points. The differences here are not huge (both versions of the
  discrete-time BW filter are typically quite close to the smoothing spline), but
  still, the differences can be noticeable. Figure \ref{fig:bw} gives an
  example. 
\end{enumerate}

\begin{figure}[tb]
\includegraphics[width=0.99\textwidth]{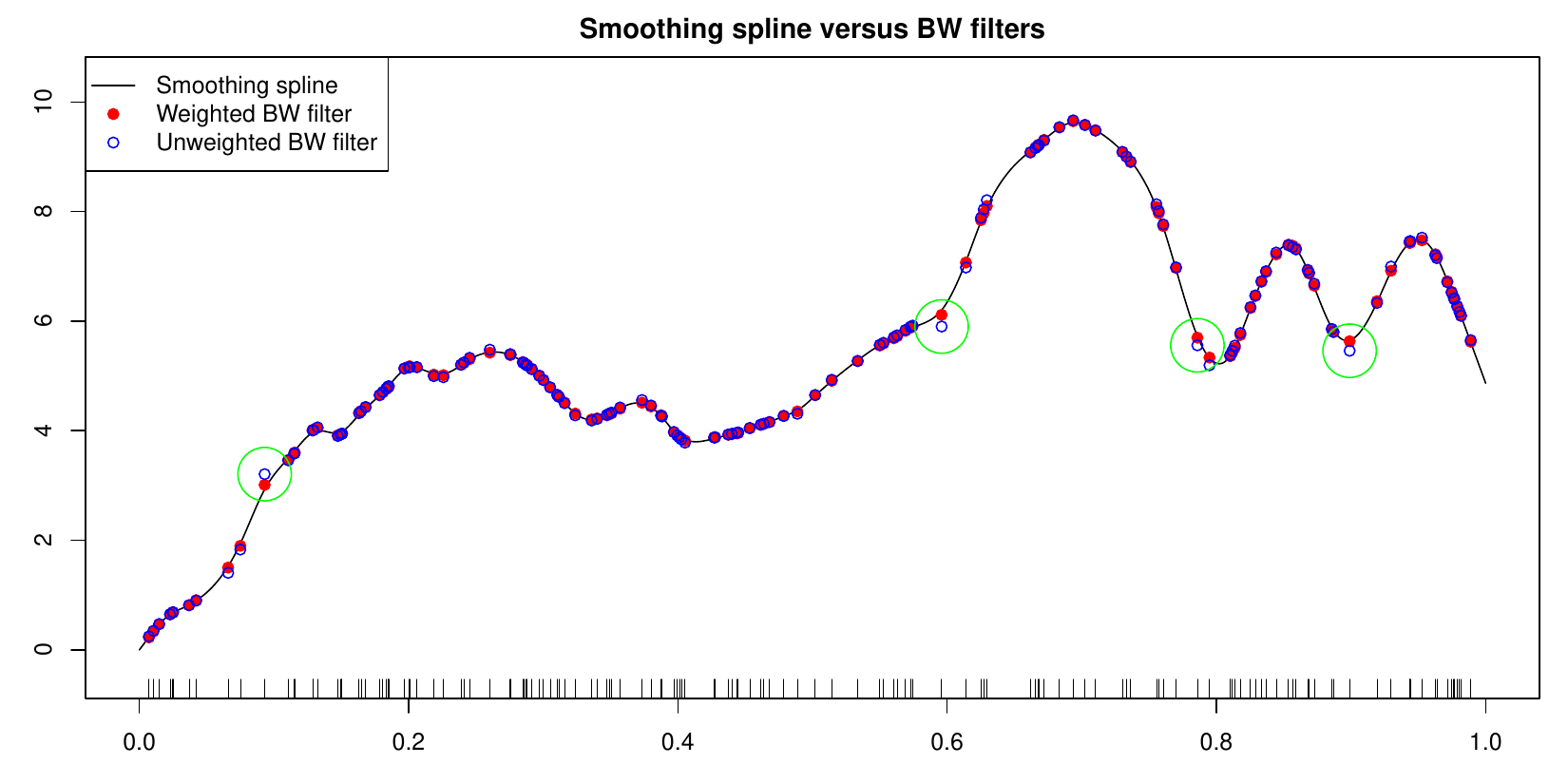}
\caption{\small Comparison of spline smoothing and BW filtering, both weighted
  (see \eqref{eq:bw_filter}) and unweighted (see \eqref{eq:bw_filter_old})
  variants, on the same example data as in Figure \ref{fig:adapt}, with $m=2$ in
  all cases. Note that the $n=150$ design points, marked by ticks on the
  horizontal axis, are not evenly-spaced, and hence there is a meaningful
  difference between the two versions of the BW filter. The smoothing spline
  solution, plotted as a black line, was fit at the same tuning parameter value
  $\lambda$ as in the bottom right plot of Figure \ref{fig:adapt} (where it had
  30 df). Both BW filter solutions are shown as a sequence of discrete points,
  at \smash{$(x_i,\htheta_i)$}, $i=1,\ldots,n$; being truly discrete-time
  estimators, this is the purest representation of their behavior. The weighted
  BW filter, plotted as solid red points, was fit at the same tuning parameter
  value $\lambda$ as the smoothing spline; the unweighted BW filter, plotted as    
  hollow blue points, was fit at the value $\lambda/n$ (which seems to be the
  single best choice of tuning parameter to adjust for the difference in scale
  of its penalty). We can see that, for the most part, the hollow blue points
  are essentially directly on top of the solid red points, which are on top of
  the smoothing spline curve. However, in a few locations (specifically, around
  0.1, 0.6, 0.8, and 0.9 on the x-axis, highlighted by green circles), the
  hollow blue points are noticeably far from the solid red points, and in these
  locations the former is farther than the latter from the smoothing spline
  curve.} 
\label{fig:bw}
\end{figure}

\subsection{Bounds on the distance between solutions} 
\label{sec:bw_err_bound}

To study the distance between smoothing spline and weighted BW filtering
solutions, it helps to first recall a notion of similarity between matrices:
positive semidefinite matrices $\A,\B \in \R^{n \times n}$ are said to be {\it 
  $(\sigma,\tau)$-spectrally-similar}, for $0 < \tau \leq 1 \leq \sigma$,
provided that
\begin{equation}
\label{eq:spec_sim}
\tau u^\T \B u \leq u^\T \A u \leq \sigma u^\T \B u, 
\quad \text{for all $u \in \R^n$}. 
\end{equation}
Spectral similarity is commonly studied in certain areas of theoretical computer  
science, specifically in the literature on graph sparsification (see, for example, 
\citet{batson2013spectral} for a nice review). The next result is both a
simplification and sharpening of Theorem 1 in \citet{sadhanala2016graph}. Its
proof follows from direct examination of the stationarity conditions for
optimality and application of \eqref{eq:spec_sim}, and is given in Appendix
\ref{app:spec_sim_bound}.  

\begin{lemma}
\label{lem:spec_sim_bound}
Let $\A,\B$ be $(\sigma,\tau)$-spectrally-similar, and let $\htheta_a,\htheta_b$ 
denote solutions in the quadratic problems  
\begin{gather}
\label{eq:quad_opt1}
\minimize_\theta \; \|y-\theta\|_2^2 + \lambda_a \theta^\T \A \theta, \\
\label{eq:quad_opt2}
\minimize_\theta \; \|y-\theta\|_2^2 + \lambda_b \theta^\T \B \theta, 
\end{gather}
respectively. Then for any $\lambda_a,\lambda_b \geq 0$, it holds that
\begin{equation}
\label{eq:quad_sol_bound1}
\|\htheta_a - \htheta_b\|_2^2 \leq 
\frac{1}{2} (\lambda_b/\tau - \lambda_a) \htheta_a^\T \A \htheta_a +  
\frac{1}{2} (\sigma \lambda_a - \lambda_b) \htheta_b^\T \B \htheta_b. 
\end{equation}
In particular, for any $\lambda_b \geq \sigma \lambda_a$, it holds that 
\begin{equation}
\label{eq:quad_sol_bound2}
\|\htheta_a - \htheta_b\|_2^2 \leq 
\frac{1}{2} (1/\tau - 1/\sigma) \lambda_b \htheta_a^\T \A \htheta_a. 
\end{equation}
\end{lemma}

We now show that the matrices featured in the quadratic penalties in the
smoothing spline and weighted BW filtering problems,
\eqref{eq:smooth_spline_discrete} and \eqref{eq:bw_filter}, are spectrally
similar for $m=2$, and then apply Lemma  \ref{lem:spec_sim_bound} to bound the
$\ell_2$ distance between the corresponding solutions. The proof is given in
Appendix \ref{app:ss_bw_bound}.\footnote{We thank Yining Wang for his help with
  the spectral similarity result.}   

\begin{theorem}
\label{thm:ss_bw_bound}
For $m=2$, and any (distinct) set of design points $x_{1:n}$, the tridiagonal 
matrix \smash{$\K^2_n$} defined in \eqref{eq:nsp_sobolev_kmat_m2} and the
diagonal matrix \smash{$\W^2_n=\diag((x_3-x_1)/2, \ldots, (x_n-x_{n-2})/2)$} are 
$(3,1)$-spectrally-similar. Thus Lemma \ref{lem:spec_sim_bound} gives the
following conclusion: if \smash{$\hf$} is the solution in the cubic smoothing
spline problem \eqref{eq:smooth_spline} with tuning parameter $\lambda_a$, and
\smash{$\htheta$} is the solution in the weighted cubic BW filtering problem
\eqref{eq:bw_filter} with tuning parameter $\lambda_b \geq 3\lambda_a$, then       
\begin{equation}
\label{eq:ss_bw_bound1}
\big\| \hf(x_{1:n}) - \htheta \big\|_2^2 \leq \frac{\lambda_b}{3}
\int_a^b (D^2 \hf)(x)^2 \, dx.
\end{equation}
\end{theorem}

\begin{remark}
To achieve the bound in \eqref{eq:ss_bw_bound1} in Theorem
\ref{thm:ss_bw_bound}, we take the weighted BW filter tuning parameter
$\lambda_b$ to be at least three times the smoothing spline tuning parameter
$\lambda_a$. This is the result of applying \eqref{eq:quad_sol_bound2} in Lemma 
\ref{lem:spec_sim_bound}. Of course, empirically, and conceptually, we are more
likely to believe that taking $\lambda_a=\lambda_b$ will lead to the most
similar solutions; with this choice, the result in \eqref{eq:quad_sol_bound1}
translates to (in the context of the smoothing spline and weighted BW
filtering): 
\begin{equation}
\label{eq:ss_bw_bound2}
\big\| \hf(x_{1:n}) - \htheta \big\|_2^2 \leq \lambda_a
\big\| (\W^2_n)^{\hspace{-1pt}\frac{1}{2}} \D^2_n \htheta \big\|_2^2, 
\end{equation}
which might also be a useful bound. However, the reason we chose to state
\eqref{eq:ss_bw_bound1} in the theorem, rather than \eqref{eq:ss_bw_bound2}, is 
that the former has the $L_2$-Sobolev penalty of \smash{$\hf$} on the right-hand
side, which can be controlled by leveraging classical nonparametric regression 
theory, as we show next.
\end{remark}

Our next result uses known bounds on the estimation error of the cubic smoothing
spline over $L_2$-Sobolev classes, along with \eqref{eq:ss_bw_bound2} and the
triangle inequality, to establish a similar result for the weighted cubic BW
filter. 

\begin{corollary}
\label{cor:ss_bw_bound}
Assume that the design points $x_i$, $i=1,\ldots,n$ are drawn from a continuous
distribution on $[0,1]$, and that the responses follow the model
$$
y_i = f_0(x_i) + \epsilon_i, \quad i=1,\ldots,n,
$$
for uniformly sub-Gaussian errors $\epsilon_i$, $i=1,\ldots,n$ with mean zero
and unit variance, independent of the design points. Further assume that $f_0$
has two weak derivatives, and that \smash{$\int_0^1 (D^2 f_0)(x)^2 \, dx \leq
  C_n^2$} for $C_n \geq 1$. Recall that there are universal constants
$c_1,c_2,c_3,n_0>0$ such that for all $c \geq c_1$ and $n \geq n_0$, the cubic
smoothing spline solution in \eqref{eq:smooth_spline} (that is, $m=2$) with 
\smash{$\lambda \geq cn^{\frac{1}{5}} C_n^{-\frac{8}{5}}$} satisfies   
\begin{gather}
\label{eq:ss_err_bound}
\frac{1}{n} \big\| \hf(x_{1:n}) - f_0(x_{1:n}) \big\|_2^2 \leq
\frac{8\lambda}{n} C_n^2, \\
\label{eq:ss_pen_bound}
\int_0^1 (D^2 \hf)(x)^2 \, dx \leq 5C_n^2,
\end{gather}
with probability at least \smash{$1-\exp(-c_2 c) - \exp(-c_3 \sqrt{n})$}.
Setting \smash{$\lambda = cn^{\frac{1}{5}} C_n^{-\frac{8}{5}}$}, the right-hand
side in \eqref{eq:ss_err_bound} becomes \smash{$8c n^{-\frac{4}{5}}
  C_n^{\frac{2}{5}}$} which matches the minimax optimal error rate (in squared
$L_2$ norm) for estimation over the space 
$$
\cW^{2,2}(C_n; [0,1]) = \bigg\{ f : [0,1] \to \R : \int_0^1 (D^2 f_0)(x)^2 \, dx
\leq C_n^2 \bigg\}.  
$$

A consequence of the above result and Theorem \ref{thm:ss_bw_bound} is as
follows: for all $c \geq c_1$ and $n \geq n_0$, the weighted cubic BW filtering
solution in \eqref{eq:bw_filter} (that is, $m=2$) with \smash{$\lambda = 3
  cn^{\frac{1}{5}} C_n^{-\frac{8}{5}}$} satisfies     
\begin{equation}
\label{eq:bw_err_bound}
\frac{1}{n} \big\| \theta(x_{1:n}) - f_0(x_{1:n}) \big\|_2^2 \leq
26 c n^{-\frac{4}{5}} C_n^{\frac{2}{5}},
\end{equation}
with probability at least \smash{$1-\exp(-c_2 c) - \exp(-c_3 \sqrt{n})$},
again matching the minimax error rate over \smash{$\cW^{2,2}(C_n; [0,1])$}. 
\end{corollary}

We omit the proof of Corollary \ref{cor:ss_bw_bound}, as 
\eqref{eq:bw_err_bound} follows immediately from \eqref{eq:ss_bw_bound1}, 
\eqref{eq:ss_err_bound}, \eqref{eq:ss_pen_bound}, and the simple inequality
\smash{$\|u+v\|_2^2 \leq 2\|u\|_2^2 + 2\|v\|_2^2$}. To be clear, the smoothing
spline error bound \eqref{eq:ss_err_bound}, penalty bound
\eqref{eq:ss_pen_bound}, and claims of minimax optimality are well-known (and
are not intended to be portrayed as original contributions in the corollary);
for example, see Chapter 10.1 of \citet{vandegeer2000empirical} for a statement
of \eqref{eq:ss_err_bound}, \eqref{eq:ss_pen_bound} in $O_P$ (bounded in 
probability) form; the results in Corollary \ref{cor:ss_bw_bound}, written in 
finite-sample form, are a consequence of Theorem 1 in
\citet{sadhanala2019additive}. 
For the minimax lower bound over the $L_2$-Sobolev class 
\smash{$\cW^{2,2}(C_n; [0,1])$}, see, for example, Chapter 2.6.1 of
\citet{tsybakov2009introduction}. It is not really suprising that the weighted 
BW filter achieves minimax optimal error rates over the appropriate Sobolev
classes, however it is of course reassuring to know that this is the case. As
far as we can tell, this seems to be a new result, despite the fact that the BW
filter has a very long history.

\subsection{Connections to discrete splines}
\label{sec:bw_discrete_spline}

Unlike trend filtering, which bears a very clear connection to discrete
splines, the connections between the (weighted) BW filter and discrete splines
appear to be more subtle. Recall that the $\ell_1$ case, for a $k$th degree
discrete spline $f$, the total variation penalty \smash{$\TV(D^k f)$} is simply
the trend filtering penalty \eqref{eq:trend_filter_wpen2a} acting on
$\theta=f(x_{1:n})$ (Theorem \ref{thm:ffb_tv}). In the $\ell_2$ case, for a
$k$th degree discrete spline $f$, with $k=2m-1$, the $L_2$-Sobolev penalty
\smash{$\int_a^b (D^m f)(x)^2 \, dx$} is a quadratic form of the $m$th discrete
derivatives of $\theta=f(x_{1:n})$ (Theorem \ref{thm:ffb_sobolev}), but this
quadratic form is {\it not} the BW penalty, either unweighted \smash{$\|\D^m_n
  \theta\|_2^2$}, or weighted \eqref{eq:bw_filter_wpen2}. It is instead
\smash{$\|(\V^m_n)^{\hspace{-1pt}\frac{1}{2}} \D^m_n \theta \|_2^2$}, where
\smash{$\V^m_n \in \R^{(n-m) \times (n-m)}$} is a banded matrix (a function of
$x_{1:n}$ only), of bandwidth $2m-1$.  

For completeness, recall that for that a $k$th degree spline, the penalty
\smash{$\int_a^b (D^m f)(x)^2 \, dx$} is also a quadratic form of the 
$m$th discrete derivatives of $\theta=f(x_{1:n})$ (Theorem
\ref{thm:nsp_sobolev}), of the form 
\smash{$\|(\K^m_n)^{\hspace{-1pt}\frac{1}{2}} \D^m_n \theta \|_2^2$}, where
\smash{$\K^m_n \in \R^{(n-m) \times (n-m)}$} is a matrix (a function of
$x_{1:n}$ only) with a banded {\it inverse}, and is therefore itself dense. 

One way to roughly interpret and compare these penalties on discrete derivatives
is as follows. Both can be seen as
\begin{equation}
\label{eq:sobolev_conv}
\big\|\A^\frac{1}{2} \D^m_n \theta \big\|_2^2 = 
\sum_{i,\tau=-\infty}^\infty \A_{i,i-\tau} (\D^m_n \theta)_i (\D^m_n
\theta)_{i-\tau}.  
\end{equation}
for a symmetric matrix $\A \in \R^{(n-m) \times (n-m)}$, where for notational
convenience we simply set the entries of $\A$ or \smash{$\D^m_n \theta$} to zero
when we index beyond their inherent ranges. That is, when the continuous-time
penalty \smash{$\int_a^b (D^m f)(x)^2 \, dx$} gets translated into
discrete-time, we see that the discrete-time equivalent \eqref{eq:sobolev_conv}
``blurs'' the derivatives before it aggregates them; more precisely, the
discrete-time equivalent \eqref{eq:sobolev_conv} measures the weighted $\ell_2$
norm of the product of \smash{$\D^m_n \theta$} and its convolution, weighted
here by a (two-dimensional) kernel $\A$. The weighted BW penalty
\smash{$\|(\W^m_n)^{\hspace{-1pt}\frac{1}{2}} \D^m_n \theta \|_2^2$} performs
no such ``blurring'' (it measures the weighted $\ell_2$ norm of \smash{$\D^m_n  
\theta$} times itself). Therefore we might view the discrete spline
discretization of the Sobolev penalty, 
\smash{$\|(\V^m_n)^{\hspace{-1pt}\frac{1}{2}} \D^m_n \theta \|_2^2$}, as being 
``closer'' to the weighted BW penalty, as its kernel \smash{$\V^m_n$} performs
less ``blurring'' (it has bandwidth $2m-1$), versus the spline discretization,
\smash{$\|(\K^m_n)^{\hspace{-1pt}\frac{1}{2}} \D^m_n \theta \|_2^2$}, whose
kernel \smash{$\K^m_n$} performs more ``blurring'' (it is supported everywhere).

An important exception is the linear case, $m=1$, in which all three
penalties---from the weighted BW filter, spline discretization, and discrete
spline discretization---coincide. The equivalence of the first two was already
noted in \eqref{eq:ffb_sobolev_m1}. The next lemma gives the equivalence of 
the third, by calculating the explicit form of \smash{$\V^m_n$} for
$m=1$.\footnote{This should not be a surprise: for degree $k=1$, discrete 
  splines are splines, and Lemma \ref{lem:ffb_sobolev_vmat_m1} is really just a
  sanity check.}      
Its proof is elementary and is deferred until Appendix
\ref{app:ffb_sobolev_vmat_m1}.      

\begin{lemma}
\label{lem:ffb_sobolev_vmat_m1}
For $m=1$, the matrix $\V_n \in \R^{(n-1) \times (n-1)}$ from Theorem
\ref{thm:ffb_sobolev} has entries
\begin{equation}
\label{eq:ffb_sobolev_vmat_m1}
(\V_n)_{ii} = 
\begin{cases}
x_2-a & \text{if $i=1$} \\
x_{i+1}-x_i & \text{if $i \geq 2$}.
\end{cases}
\end{equation}
Thus when $a=x_1$, we see that the matrix $\V_n$ in
\eqref{eq:ffb_sobolev_vmat_m1} is the same as the matrix $\K_n$ in
\eqref{eq:nsp_sobolev_kmat_m1}, which is the same as $\W_n$ in
\eqref{eq:weight_mat} with $m=1$.
\end{lemma}

The next case to consider would of course be the cubic case, $m=2$. As it turns
out, deriving the explicit form of \smash{$\V^m_n$} for $m=2$ requires a
formidable calculation. The recursion in Lemma
\ref{lem:ffb_sobolev_vmat}---though conceptually straightforward---is
practically challenging to carry out, since it involves some rather complicated
algebraic calculations. However, it can be done for evenly-spaced design points
$x_{i+1}-x_i=v>0$, $i=1,\ldots,n-1$, with $a=x_1$ and $b=x_n$: \footnote{We 
  thank Pratik Patil for his help in checking the result
  \eqref{eq:ffb_sobolev_vmat_m2_even}.} 
\renewcommand\arraystretch{1.2}
\begin{equation}
\label{eq:ffb_sobolev_vmat_m2_even}
\V^2_n = \left[\begin{array}{rrrrrrrr}
3 & -\nfrac{3}{2} & 0 & 0 & \ldots & 0 & 0 & 0 \\
-\nfrac{3}{2} & \nfrac{10}{3} & -\nfrac{5}{6} & 0 & \ldots & 0 & 0 & 0 \\
0 & -\nfrac{5}{6} & \nfrac{8}{3} & -\nfrac{5}{6} & \ldots & 0 & 0 & 0 \\
\vdots & & & & & & & \\
0 & 0 & 0 & 0 & \ldots & -\nfrac{5}{6} & \nfrac{8}{3} & -\nfrac{5}{6} \\
0 & 0 & 0 & 0 & \ldots & 0 & -\nfrac{5}{6} & \nfrac{7}{3} 
\end{array}\right] \cdot v.
\end{equation}
\renewcommand\arraystretch{1}
For comparison, in this case, we have from \eqref{eq:nsp_sobolev_kmat_m2}:  
\renewcommand\arraystretch{1.2}
\begin{equation}
\label{eq:nsp_sobolev_kmat_m2_even}
\K^2_n = \left[\begin{array}{rrrrrrrr}
\nfrac{2}{3} & \nfrac{1}{6} & 0 & 0 & \ldots & 0 & 0 & 0 \\
\nfrac{1}{6} & \nfrac{2}{3} & \nfrac{1}{6} & 0 & \ldots & 0 & 0 & 0 \\
0 & \nfrac{1}{6} & \nfrac{2}{3} & \nfrac{1}{6} & \ldots & 0 & 0 & 0 \\
\vdots & & & & & & & \\
0 & 0 & 0 & 0 & \ldots & \nfrac{1}{6} & \nfrac{2}{3} & \nfrac{1}{6} \\
0 & 0 & 0 & 0 & \ldots & 0 & \nfrac{1}{6} & \nfrac{2}{3} 
\end{array}\right]^{-1} \cdot v.
\end{equation}
\renewcommand\arraystretch{1}

For the case of arbitrary design points, we can carry out the recursion defining
\smash{$\V^m_n$} in Lemma \ref{lem:ffb_sobolev_vmat} with symbolic computation
software. Our current attempts have resulted in somewhat compact expressions for
the elements of \smash{$\V^m_n$}, but they do not appear simple enough to be
useful (amenable to further interpretation or analysis). We may report on this
in more detail at a future time.


\section{Discussion}

This paper began as an attempt to better understand the connections between 
trend filtering and discrete splines, and it grew into something broader: an
attempt to better understand some fundamental properties of discrete splines,
and offer some new perspectives on them. Though discrete splines were first 
studied 50 years ago, there still seems to be some fruitful directions left to 
explore. For example, the approximation results in Section
\ref{sec:approximation} are weak (though recall, they are sufficient for the 
intended statistical applications) and could most certainly be improved. The
use of discrete B-splines within trend filtering optimization algorithms,
described in Section \ref{sec:trend_filter_comp}, should be investigated
thoroughly, as it should improve their stability. As for more open directions,
it may be possible to use discrete splines to approximately (and efficiently) 
solve certain differential equations. Lastly, the multivariate case is of great
interest and importance. 

\subsection*{Acknowledgements}

We are grateful to Yu-Xiang Wang for his many, many insights and inspiring
conversations over the years. His lasting enthusiasm helped fuel our own
interest in ``getting to the bottom'' of the falling factorial basis, and
putting this paper together. We also thank our other collaborators on trend
filtering papers: Aaditya Ramdas, Veeranjaneyulu Sadhanala, James Sharpnack, and
Alex Smola. Finally, we are grateful to Addison Hu, Alden Green, Pratik Patil,
Veeranjaneyulu Sadhanala, and Yu-Xiang Wang for their helpful comments and
feedback on this paper as a whole; to Yining Wang for his help with Section
\ref{sec:bw_err_bound}; and to Pratik Patil for his help with Section
\ref{sec:bw_discrete_spline}. This paper is based upon work supported by the
National Science Foundation under Grant No.\ DMS-1554123.  

\appendix
\newpage
\section{Notation table}
\label{app:notation}

\renewcommand\arraystretch{1.115}
\begin{table}[H]
\hspace{-45pt}
\begin{tabular}{|p{0.26\textwidth}|
p{0.13\textwidth}|p{0.26\textwidth}|
p{0.13\textwidth}|p{0.28\textwidth}|} 
\hline
\multicolumn{1}{|c|}{Discrete object} & 
\multicolumn{1}{|c|}{Reference} & 
\multicolumn{1}{|c|}{Continuum object} &
\multicolumn{1}{|c|}{Reference} & 
\multicolumn{1}{|c|}{Notes}
\\ \hline \hline
\multicolumn{5}{|c|}{\textbf{Operators} \vphantom{$\Big($}} 
\\ \hline 
$\Delta^k_n = \Delta^k(\cdot; x_{1:n})$, $k$th order discrete differentiation
with respect to design point $x_{1:n}$ & 
\eqref{eq:discrete_deriv} &   
$D^k$, $k$th order differentiation & -- & 
$(\Delta^k_n f)(x) = (D^k f)(x)$, for $f \in \cH^k_n$ and $x > x_k$ (Corollary
\ref{cor:deriv_match})  
\\ \hline
$S^k_n = S^k(\cdot; x_{1:n})$, $k$th order discrete integration &
\eqref{eq:cum_sum}, \eqref{eq:discrete_integ_rec} & 
$I^k$, $k$th order integration & -- & 
$S^k_n = (D^k_n)^{-1}$ (Lemma  \ref{lem:discrete_deriv_integ_inv}) 
\\ \hline \hline
\multicolumn{5}{|c|}{\textbf{Spaces} \vphantom{$\Big($}} 
\\ \hline 
$\DS^k_n(t_{1:r})$, $k$th degree discrete splines with knots $t_{1:r}$ and   
  design points $x_{1:n}$ &    
Definition \ref{def:discrete_spline} & 
$\S^k(t_{1:r})$, $k$th degree splines with knots $t_{1:r}$ & 
Definition \ref{def:spline} &  
These spaces coincide for $k=0$ and $k=1$ 
\\ \hline \hline
$\cH^k_n=\DS^k_n (x_{(k+1):(n-1)})$ & -- & 
$\cG^k_n=\S^k(x_{(k+1):(n-1)})$ & -- & 
Abbreviations for the ``canonical'' spaces, with knots $x_{(k+1):(n-1)}$   
\\ \hline \hline
\multicolumn{5}{|c|}{\textbf{Bases} \vphantom{$\Big($}} 
\\ \hline 
$h^k_j$, $j=1,\ldots,n$, $k$th degree falling factorial basis for $\cH^k_n$ &
\eqref{eq:ffb} &   
$g^k_j$, $j=1,\ldots,n$, $k$th degree truncated power basis for $\cG^k_n$ & 
\eqref{eq:tpb} &   
Falling factorials can be seen as truncated Newton polynomials, and have dual 
relationship to discrete differentiation (Lemma \ref{lem:dual_basis}) 
\\ \hline
$Q^k_j$ and $N^k_j$, $j=1,\ldots,n$, unnormalized and normalized $k$th degree
DB-spline basis for $\cH^k_n$ &
\eqref{eq:discrete_bs_evals}, \eqref{eq:discrete_nbs_evals}, 
\eqref{eq:discrete_nbs} &   
$P^k_j$ and $M^k_j$, $j=1,\ldots,n$, unnormalized and normalized $k$th degree
B-spline basis for $\cG^k_n$ &   
\eqref{eq:bs}, \eqref{eq:nbs}, \newline \eqref{eq:nbsb} &
The basis in \eqref{eq:nbsb} is actually defined for an arbitrary knot set
$t_{1:r}$; for arbitrary knots in the DB-spline setting, see
\eqref{eq:discrete_nbs_sk1}, \eqref{eq:discrete_nbs_sk2}  
\\ \hline \hline
\multicolumn{5}{|c|}{\textbf{Matrices} \vphantom{$\Big($}} 
\\ \hline
$\D^k_n$, $k$th order discrete derivative matrix with respect to design points
$x_{1:n}$ & \eqref{eq:diff_mat}, \eqref{eq:weight_mat},
\eqref{eq:discrete_deriv_mat} & 
-- & -- &
Multiplying by a vector of evaluations gives discrete derivatives at design 
points $x_{(k+1):n}$, as in \eqref{eq:discrete_deriv_conn1}
\\ \hline
$\B^k_n$, $k$th order extended discrete derivative matrix with respect to design
points $x_{1:n}$ & 
\eqref{eq:diff_mat_ext}, \eqref{eq:weight_mat_ext}, 
\eqref{eq:discrete_deriv_mat_ext} &
-- & -- & 
Multiplying by a vector of evaluations gives discrete derivatives at all
design points $x_{1:n}$, as in \eqref{eq:discrete_deriv_conn2} 
\\ \hline
$\H^k_n$, $k$th degree falling factorial basis matrix with respect to design
points $x_{1:n}$ & 
Basis in $k$th degree trend filter \eqref{eq:trend_filter_basis} &   
$\G^k_n$, $k$th degree truncated pow- er basis matrix with respect to design 
points $x_{1:n}$ & 
Basis in $k$th degree restricted locally adaptive regression spline 
\eqref{eq:local_spline_basis} 
& $\H^k_n = (\Z^{k+1}_n \, \B^{k+1}_n)^{-1}$, see
\eqref{eq:ffb_discrete_deriv_inv}; results in fast algorithms for matrix
computations in $\H^k_n$, see Appendix \ref{app:fast_mult} 
\\ \hline \hline
\multicolumn{5}{|c|}{\textbf{Smoothness functionals} \vphantom{$\Big($}}  
\\ \hline
$\sum_{i=1}^{n-k-1} | (\D^k_n \theta)_{i+1} - (\D^k_n\theta)_i|$ 
$= \|\W^{k+1}_n \D^{k+1}_n \theta\|_1$, $k$th order discrete total variation of
vector $\theta$ &
Penalty in $k$th degree trend filter \eqref{eq:trend_filter} & 
$\TV(D^k f)$, $k$th order total variation of function $f$ &
Penalty in $k$th degree locally adaptive regres- sion spline
\eqref{eq:local_spline} & 
Equal for $\theta = f(x_{1:n})$ and $f \in \cH^k_n$ (Theorem \ref{thm:ffb_tv}) 
\\ \hline  
$\sum_{i=1}^{n-m} (\D^m_n \theta)_i^2 (x_{i+m} - x_i)/m$
$= \|(\W^m_n)^{\hspace{-1pt}\frac{1}{2}} \D^m_n \theta\|_2^2$, $m$th order 
discrete Sobolev seminorm of vector $\theta$ &  
Penalty in $k$th degree BW filter \eqref{eq:bw_filter}, for $k=2m-1$ &
$\int_a^b (D^m f)(x)^2 \, dx$, $m$th order Sobolev seminorm of $f$ & 
Penalty in $k$th degree smoothing spline \eqref{eq:smooth_spline}, for $k=2m-1$
& Equal for $\theta = f(x_{1:n})$ and $m = 1$ (Lemma
\ref{lem:ffb_sobolev_vmat_m1}), but not in general; see also Theorem
\ref{thm:ffb_sobolev} 
\\ \hline 
\end{tabular}
\caption{Main notation, and discrete-continuum analogies/equivalences in this
  paper. We omit notational dependence on the domain $[a,b]$ for simplicity.}     
\label{tab:notation}
\end{table}
\renewcommand\arraystretch{1}

\newpage
\section{Proofs}
\label{app:proofs}

\subsection{Proof of Theorem \ref{thm:nsp_sobolev}}
\label{app:nsp_sobolev}

Since $f$ is a natural spline of degree $2m-1$ with knots in $x_{1:n}$, we know
that $D^m f$ is a spline of degree $m-1$ with knots in $x_{1:n}$, and moreover,
it is supported on $[x_1,x_n]$. Thus we can expand \smash{$D^m f =
\sum_{i=1}^{n-m} \alpha_i P^{m-1}_i$} for coefficients $\alpha_i$,
$i=1,\ldots,n-m$, and
\begin{equation}
\label{eq:nsp_sobolev1}
\int_a^b (D^m f)(x)^2 \, dx = \alpha^\T \Q \alpha,
\end{equation}
where $\Q \in \R^{(n-m) \times (n-m)}$ has entries \smash{$\Q_{ij} = \int_a^b 
  P^{m-1}_i(x) P^{m-1}_j(x) \, dx$}. But we can also write 
\begin{align}
\nonumber
\int_a^b (D^m f)(x)^2 \, dx 
&= \int_a^b (D^m f)(x) \sum_{i=1}^{n-m} \alpha_i P^{m-1}_i(x) \, dx \\
\nonumber
&= \sum_{i=1}^{n-m} \alpha_i \int_a^b (D^m f)(x) P^{m-1}_i(x) \, dx \\ 
\nonumber
&= \frac{1}{m} \sum_{i=1}^{n-m} \alpha_i \big(\D^m_n f(x_{1:n})\big)_i \\
\label{eq:nsp_sobolev2}
&= \frac{1}{m} \alpha^\T \D^m_n f(x_{1:n}),
\end{align}
where in the third line, we used the Peano representation for B-splines, as
described in \eqref{eq:peano} in Appendix \ref{app:bs}, which implies that for 
$i=1,\ldots,n-m$, 
$$
(m-1)! \cdot f[x_i,\ldots,x_{i+m}] = \int_a^b (D^m f)(x) P^{m-1}_i(x) \,
dx. 
$$ 
Comparing \eqref{eq:nsp_sobolev1} and \eqref{eq:nsp_sobolev2}, we learn that
\smash{$\Q \alpha = \D^m_n f(x_{1:n})/m$}, that is, \smash{$\alpha = \Q^{-1}
  \D^m_n f(x_{1:n})/m$}, and therefore    
\begin{align*}
\int_a^b (D^m f)(x)^2 \, dx 
&= \frac{1}{m^2} \big(\Q^{-1} \D^m_n f(x_{1:n})\big)^\T \Q \, \Q^{-1} 
\D^m_n f(x_{1:n}) \\  
&= \frac{1}{m^2} \big(\D^m_n f(x_{1:n})\big)^\T \Q^{-1} \D^m_n f(x_{1:n}), 
\end{align*}
which establishes \eqref{eq:nsp_sobolev}, \eqref{eq:nsp_sobolev_kmat} with
\smash{$\K^m_n = (1/m^2)\Q^{-1}$}, that is, \smash{$(\K^m_n)^{-1} = m^2 \Q$}. 

When $m=1$, for each $i=1,\ldots,n-1$, we have the simple form for the constant
B-spline:  
$$
P^0_i(x) = 
\begin{cases}
\displaystyle
\frac{1}{x_{i+1}-x_i} & \text{if $x \in I_i$} \\
0 & \text{otherwise}.
\end{cases}
$$
where $I_1=[x_1,x_2]$, and $I_i=(x_i,x_{i+1}]$ for $i=2,\ldots,n-1$. The result
\eqref{eq:nsp_sobolev_kmat_m1} comes from straightforward calculation of
\smash{$\int_a^b P^0_i(x)^2 \, dx$}. Lastly, when $m=2$, for each
$i=1,\ldots,n-2$, we have the linear B-spline:   
$$
P^1_i(x) = 
\begin{cases}
\displaystyle
\frac{x-x_i}{(x_{i+2}-x_i)(x_{i+1}-x_i)} & \text{if $x \in I_i^-$} \\ 
\displaystyle
\frac{x_{i+2}-x}{(x_{i+2}-x_i)(x_{i+2}-x_{i+1})} & \text{if $x \in I_i^+$} \\  
0 & \text{otherwise},
\end{cases}
$$
where \smash{$I_1^-=[x_1,x_2]$}, \smash{$I_i^-=(x_i,x_{i+1}]$} for
$i=2,\ldots,n-2$, and \smash{$I_i^+=(x_{i+1},x_{i+2}]$} for $i=1,\ldots,n-2$.
The two cases in  \eqref{eq:nsp_sobolev_kmat_m2} again come from
straightforward calculation of the integrals \smash{$\int_a^b P^1_i(x)^2 \,
  dx$} and \smash{$\int_a^b P^1_i(x) P^1_{i-1}(x) \, dx$}, which completes 
the proof.     


\subsection{Proof of the linear combination formulation
  \eqref{eq:discrete_integ_linear}} 
\label{app:discrete_integ_linear}

Denote by $g(x)$ the right-hand side of \eqref{eq:discrete_integ_linear}. We
will show that \smash{$\Delta^k_n g=f$}. Note by Lemma
\ref{lem:discrete_deriv_integ_inv}, this would imply \smash{$g=S^k_n f$},
proving \eqref{eq:discrete_integ_linear}. An inductive argument similar to that
in the proof of Lemma \ref{lem:ffb_lateral_rec} shows that, for $x \in
(x_i,x_{i+1}]$ and $i \geq k$,  
\begin{multline*}
(\Delta^k_n g)(x) = 
\sum_{j=1}^k (\Delta^k_n h^{k-1}_j)(x) \cdot f(x_j) \;+
\sum_{j=k+1}^i (\Delta^k_n h^{k-1}_j)(x) \cdot \frac{x_j-x_{j-k}}{k} \cdot
f(x_j) \\
+\, (\Delta^k_n h^{k-1}_{i+1})(x) \cdot \frac{x-x_{i-k+1}}{k} \cdot f(x). 
\end{multline*}
By Lemmas \ref{lem:ffb_discrete_deriv_extra_pp} and
\ref{lem:ffb_discrete_deriv_extra_poly}, all discrete derivatives here are 
zero except the last, which is \smash{$(\Delta^k_n h^{k-1}_{i+1})(x) 
  (x-x_{i-k+1})/k = 1$}. Thus we have shown \smash{$(\Delta^k_n g)(x) =
  f(x)$}. Similarly, for $x \in (x_i,x_{i+1}]$ and $i < k$, 
$$
(\Delta^k_n g)(x) = 
\sum_{j=1}^i (\Delta^k_n h^{k-1}_j)(x) \cdot f(x_j) 
\,+\, (\Delta^k_n h^{k-1}_{i+1})(x) \cdot f(x),
$$
and by Lemma \eqref{eq:ffb_discrete_deriv_extra_poly}, all discrete derivatives
here are zero except the last, which is \smash{$(\Delta^k_n h^{k-1}_{i+1})(x) =
  1$}. For $x \leq x_1$, we have $g(x) = f(x)$ by definition. This establishes
the desired claim and completes the proof.  

\subsection{Proof of Lemma \ref{lem:discrete_deriv_integ_inv}}
\label{app:discrete_deriv_integ_inv}

We use induction, beginning with $k=1$. Using \eqref{eq:cum_sum},
\eqref{eq:weight_map}, we can express the first order discrete integral operator 
$S_n$ more explicitly as 
\begin{equation}
\label{eq:discrete_integ_rec1}
(S_n f)(x) =
\begin{cases}
\displaystyle
f(x_1) + \sum_{j=2}^i f(x_j) (x_j-x_{j-1}) + f(x) (x-x_i) 
& \text{if $x \in (x_i,x_{i+1}]$} \\
f(x) & \text{if $x \leq x_1$}.
\end{cases}
\end{equation}
Compare \eqref{eq:discrete_deriv_rec1} and \eqref{eq:discrete_integ_rec1}. For
$x \leq x_1$, clearly $(\Delta_n S_n f)(x) = f(x)$ and $(S_n \Delta_n f)(x) =
f(x)$, and for $x \in (x_i, x_{i+1}]$,   
\begin{align*}
(\Delta_n S_n f)(x) &= \frac{(S_n f)(x) - (S_n f)(x_i)}{x - x_i} \\
&= \frac{f(x_1) + \sum_{j=2}^i f(x_j) (x_j-x_{j-1}) + f(x) (x-x_i) -  
\big(f(x_1) + \sum_{j=2}^i f(x_j) (x_j-x_{j-1})\big)}{x-x_i} \\
&= f(x),
\end{align*}
and also
\begin{align*}
(S_n \Delta_n f)(x) &= f(x_1) + \sum_{j=2}^i (\Delta_n f)(x_j) \cdot
  (x_j-x_{j-1}) + (\Delta_n f)(x) \cdot (x-x_i) \\ 
&= f(x_1) + \sum_{j=2}^i\big(f(x_j)-f(x_{j-1})\big) + f(x)-f(x_i) \\
&= f(x).
\end{align*}
Now assume the result is true for the order $k-1$ operators. Then, we have from 
\eqref{eq:discrete_deriv_rec}, \eqref{eq:discrete_integ_rec},
$$
\Delta^k_n \circ S^k_n = 
(W^k_n)^{-1} \circ \widebar\Delta_{n-k+1} \circ \Delta^{k-1}_n \circ S^{k-1}_n
\circ \widebar{S}_{n-k+1} \circ W^k_n = \Id,
$$
and also 
$$
S^k_n \circ \Delta^k_n =
S^{k-1}_n \circ \widebar{S}_{n-k+1} \circ W^k_n \circ (W^k_n)^{-1} \circ 
\widebar\Delta_{n-k+1} \circ \Delta^{k-1}_n = \Id, 
$$
where $\Id$ denotes the identity operator. This completes the proof.

\subsection{Proof of Lemma \ref{lem:ffb_lateral_rec}}
\label{app:ffb_lateral_rec}

\paragraph{The case $d=0$.}  

Beginning with the case $d=0$, the desired result in \eqref{eq:ffb_lateral_rec}
reads  
$$
\frac{1}{k!} \prod_{m=j-k}^{j-1} (x-x_m) = \sum_{\ell=j}^i
\frac{1}{(k-1)!} \prod_{m=\ell-k+1}^{\ell-1} (x-x_m) 
\frac{x_\ell-x_{\ell-k}}{k} +
\frac{1}{(k-1)!} \prod_{m=i-k+2}^i (x-x_m) \frac{x-x_{i-k+1}}{k},
$$
or more succintly, 
$$
\eta(x; x_{(j-k):(j-1)}) = \sum_{\ell=j}^i \eta(x; x_{(\ell-k+1):(\ell-1)})
(x_\ell-x_{\ell-k}) + \eta(x; x_{(i-k+2):i}), 
$$
The above display is a consequence of an elementary result
\eqref{eq:newton_poly_diff} on Newton polynomials. We state and prove this
result next, which we note completes the proof for the case $d=0$. 

\begin{lemma}
\label{lem:newton_poly_diff}
For any $k \geq 1$, and points $t_{1:r}$ with $r \geq k$, the Newton polynomials
defined in \eqref{eq:newton_poly} satisfy, at any $x$,      
\begin{equation}
\label{eq:newton_poly_diff}
\eta(x; t_{1:k}) - \eta(x; t_{(r-k+1):r}) = 
\sum_{\ell=k+1}^r \eta(x; t_{(\ell-k+1):(\ell-1)}) (t_\ell-t_{\ell-k}).  
\end{equation}
\end{lemma}

\begin{proof}
Observe that
\begin{equation}
\label{eq:newton_poly_diff0}
\eta(x; t_{1:k}) - \eta(x; t_{2:(k+1)}) = \eta(x; t_{2:k}) \big((x-t_1) -
(x-t_{k+1})\big) = \eta(x; t_{2:k}) (t_{k+1}-t_1). 
\end{equation}
Therefore
\begin{multline*}
\eta(x; t_{1:k}) - \eta(x; t_{(r-k+1):r}) = 
\underbrace{\eta(x; t_{1:k}) - \eta(x; t_{2:(k+1)})}_{a_1} +  
\underbrace{\eta(x; t_{2:(k+1)}) - \eta(x; t_{3:(k+2)})}_{a_2} + \cdots \\ 
+ \underbrace{\eta(x; t_{(r-k):r-1}) - \eta(x; t_{(r-k+1):r})}_{a_{r-k}}. 
\end{multline*}
In a similar manner to \eqref{eq:newton_poly_diff0}, for each $i=1,\ldots,k$,
we have $a_i=\eta(x;t_{(i+1):(i+k-1)}) (t_{i+k}-t_i)$, and the result follows,
after making the substitution $\ell=i+k$. 
\end{proof}

\paragraph{The case $d \geq 1$.}  

We now prove the result \eqref{eq:ffb_lateral_rec} for $d \geq 1$ by induction.
The base case was shown above, for $d=0$. Assume the result holds for discrete
derivatives of order $d-1$. If $x \leq x_d$ (or $d>n$), then
\smash{$(\Delta^d_n f)(x)=(\Delta^{d-1}_n f)(x)$} for all functions $f$ and thus
the desired result holds trivially. Hence assume $x>x_d$ (which implies that $i
\geq d$). By the inductive hypothesis,
\begin{align*}
&(\Delta^{d-1}_n h^k_j)(x) - (\Delta^{d-1}_n h^k_j)(x_i) \\
&= \sum_{\ell=j}^i (\Delta^{d-1}_n h^{k-1}_\ell)(x)
  \cdot \frac{x_\ell-x_{\ell-k}}{k} + (\Delta^{d-1}_n h^{k-1}_{i+1})(x) 
  \cdot \frac{x-x_{i-k+1}}{k} - \sum_{\ell=j}^i (\Delta^{d-1}_n
  h^{k-1}_\ell)(x_i) \cdot \frac{x_\ell-x_{\ell-k}}{k} \\ 
&=\sum_{\ell=j}^i \big((\Delta^{d-1}_n h^{k-1}_\ell)(x) -
 (\Delta^{d-1}_n h^{k-1}_\ell)(x_i)\big) \cdot
  \frac{x_\ell-x_{\ell-k}}{k} + \big((\Delta^{d-1}_n h^{k-1}_{i+1})(x) -  
  (\Delta^{d-1}_n h^{k-1}_{i+1})(x_i)\big) \cdot \frac{x-x_{i-k+1}}{k},
\end{align*}
where in the last line we used the fact that \smash{$h^{k-1}_{i+1}=0$} on
$[a,x_i]$, and thus \smash{$(\Delta^{d-1}_n h^{k-1}_{i+1})(x_i)=0$}. This
means, using \eqref{eq:discrete_deriv_rec2},  
\begin{align*}
&(\Delta^d_n h^k_j)(x) = \frac{(\Delta^{d-1}_n h^k_j)(x) - (\Delta^{d-1}_n 
  h^k_j)(x_i)} {(x-x_{i-d+1})/d} \\
&=\sum_{\ell=j}^i \frac{(\Delta^{d-1}_n h^{k-1}_\ell)(x) -
 (\Delta^{d-1}_n h^{k-1}_\ell)(x_i)} {(x-x_{i-d+1})/d} \cdot
  \frac{x_\ell-x_{\ell-k}}{k} + \frac{(\Delta^{d-1}_n h^{k-1}_{i+1})(x) -  
  (\Delta^{d-1}_n h^{k-1}_{i+1})(x_i)} {(x-x_{i-d+1})/d} \cdot
  \frac{x-x_{i-k+1}}{k} \\
&= \sum_{\ell=j}^i (\Delta^d_n h^{k-1}_\ell)(x) \cdot
  \frac{x_\ell-x_{\ell-k}}{k} + (\Delta^d_n h^{k-1}_{i+1})(x) \cdot 
  \frac{x-x_{i-k+1}}{k},
\end{align*}
as desired. This completes the proof.

\subsection{Lemma \ref{lem:poly_implicit} (helper result for the proof of
  Corollary \ref{cor:ffb_interp_implicit_conv})}   
\label{app:poly_implicit}

\begin{lemma}
\label{lem:poly_implicit}
Given distinct points $t_i \in [a,b]$, $i=1,\ldots,r$ and evaluations $f(t_i)$,  
$i=1,\ldots,r$, if $f$ satisfies  
$$
f[t_1,\ldots,t_r,x] = 0, \quad \text{for $x \in [a,b]$},
$$
then $f$ is a polynomial of degree $r$. 
\end{lemma}

\begin{proof}
We will actually prove a more general result, namely, that if $f$ satisfies 
\begin{equation}
\label{eq:poly_implicit_gen}
f[t_1,\ldots,t_r,x] = p_\ell(x), \quad \text{for $x \in [a,b]$},
\end{equation}
where $p_\ell$ is a polynomial of degree $\ell$, then $f$ is a polynomial of
degree $r+\ell$. We use induction on $r$. For $r=0$, the statement  
\eqref{eq:poly_implicit_gen} clearly holds for all $\ell$, because $f[x]=f(x)$
(a zeroth order divided difference is simply evaluation). Now assume
\eqref{eq:poly_implicit_gen} holds for any $r-1$ centers and all degrees $\ell$. 
Then 
$$
p_\ell(x) = f[t_1,\ldots,t_r,x] = \frac{f[t_2,\ldots,t_r,x] - f[t_1,\ldots,t_r]}
{x-t_1},  
$$ 
which means $f[t_2,\ldots,t_r,x] = (x-t_1) p_\ell(x) + f[t_1,\ldots,t_r]$. As
the right-hand side is a polynomial of degree $\ell+1$, the inductive hypothesis
implies that $f$ is a polynomial of degree $r-1+\ell+1 = r+\ell$, completing 
the proof.
\end{proof}

\subsection{Proof of Theorem \ref{thm:ffb_sobolev}}
\label{app:ffb_sobolev}

Let \smash{$h^k_j$}, $j=1,\ldots,n$ denote the falling factorial basis, as in
\eqref{eq:ffb}. Consider expanding $f$ in this basis, \smash{$f=\sum_{j=1}^n 
  \alpha_j  h^k_j$}. Define $\Q \in \R^{n \times n}$ to have entries        
\begin{equation}
\label{eq:ffb_sobolev_qmat}
\Q_{ij} = \int_a^b (D^m h^k_i)(x) (D^m h^k_j)(x) \, dx.
\end{equation}
Observe
\begin{align}
\nonumber
\int_a^b (D^m f)(x)^2 \, dx 
&= \int_a^b \sum_{i,j=1}^n \alpha_i \alpha_j (D^m h^k_i)(x) 
  (D^m h^k_j)(x) \, dx \\  
\nonumber
&= \alpha^\T \Q \alpha \\
\nonumber
&= f(x_{1:n})^\T (\H^k_n)^{-\T} \Q (\H^k_n)^{-1} f(x_{1:n}) \\ 
\label{eq:ffb_sobolev1}
&= f(x_{1:n})^\T (\B^{k+1}_n)^\T \Z^{k+1}_n \, \Q \, \Z^{k+1}_n \, 
 \B^{k+1}_n f(x_{1:n}). 
\end{align}
In the third line above we used the expansion \smash{$f(x_{1:n}) = \H^k_n
  \alpha$}, where \smash{$\H^k_n$} is the $k$th degree falling factorial basis
with entries \smash{$(\H^k_n)_{ij} = h^k_j(x_i)$}, and in the fourth line we
applied the inverse relationship in \eqref{eq:ffb_discrete_deriv_inv}, where
\smash{$\B^{k+1}_n$} is the $(k+1)$st order extended discrete derivative
matrix in \eqref{eq:discrete_deriv_mat_ext} and \smash{$\Z^{k+1}_n$} is the
extended weight matrix in \eqref{eq:weight_mat_ext}. Now note that we can
unravel the recursion in \eqref{eq:discrete_deriv_mat_ext} to yield
\begin{equation}
\label{eq:ffb_sobolev2}
\B^{k+1}_n = (\Z^{k+1}_n)^{-1} \underbrace{\widebar\B_{n,k+1} (\Z^k_n)^{-1} 
  \widebar\B_{n,k} \cdots (\Z^{m+1}_n)^{-1} \widebar\B_{n,m+1}}_{\F} \B^m_n,  
\end{equation}
and returning to \eqref{eq:ffb_sobolev1}, we get
\begin{equation}
\label{eq:ffb_sobolev3}
\int_a^b (D^m f)(x)^2 \, dx =
 f(x_{1:n})^\T (\B^m_n)^\T  \F^\T \Q \, \F \, \B^m_n f(x_{1:n}).  
\end{equation}
We break up the remainder of the proof up into parts for readability.

\paragraph{Reducing \eqref{eq:ffb_sobolev3} to involve only discrete
  derivatives.} 

First we show that the right-hand side in \eqref{eq:ffb_sobolev3} really depends
on the discrete derivatives \smash{$\D^m_n f(x_{1:n})$} only (as opposed to
extended discrete derivatives \smash{$\B^m_n f(x_{1:n})$}). As the first $m$
basis functions \smash{$h^k_1,\ldots,h^k_m$} are polynomials of degree at most
$m-1$, note that their $m$th derivatives are zero, and hence we can write
$$
\Q = \left[\begin{array}{cc}
0 & 0 \\
0 & \M
\end{array}\right],
$$
where $\M \in \R^{(n-m) \times (n-m)}$ has entries as in
\eqref{eq:ffb_sobolev_mmat}. Furthermore, note that $\F$ as defined in
\eqref{eq:ffb_sobolev2} can be written as 
$$
\F = \left[\begin{array}{cc}
\I_m & 0 \\
0 & \G
\end{array}\right],
$$
for a matrix $\G \in \R^{(n-m) \times (n-m)}$. Therefore 
\begin{equation}
\label{eq:ffb_sobolev_fqf1}
\F^\T \Q \, \F = \left[\begin{array}{cc} 
0 & 0 \\
0 & \G^\T \M \, \G 
\end{array}\right],
\end{equation}
and hence \eqref{eq:ffb_sobolev3} reduces to
\begin{equation}
\label{eq:ffb_sobolev4}
\int_a^b (D^m f)(x)^2 \, dx = f(x_{1:n})^\T (\D^m_n)^\T  
 \underbrace{\G^\T \M \, \G}_{\V^m_n} \D^m_n f(x_{1:n}), 
\end{equation}
recalling that \smash{$\D^m_n$} is exactly given by the last $n-m$ rows of
\smash{$\B^m_n$}. 

\paragraph{Casting $\F^\T \Q \, \F$ in terms of scaled differences.} 

Next we prove that \smash{$\V^m_n=\G^\T \M \, \G$}, as defined in
\eqref{eq:ffb_sobolev4}, is a banded matrix. To prevent unnecessary indexing 
difficulties, we will actually just work directly with \smash{$\F^\T \Q \, \F$},
and then in the end, due to \eqref{eq:ffb_sobolev_fqf1}, we will be able to read
off the desired result according to the lower-right submatrix of \smash{$\F^\T
  \Q \, \F$}, of dimension  $(n-m) \times (n-m)$.  Observe that 
\begin{equation}
\label{eq:ffb_sobolev_fqf2}
\F^\T \Q \, \F = (\widebar\B_{n,m+1})^\T (\Z^{m+1}_n)^{-1} \cdots  
(\widebar\B_{n,k})^\T (\Z^k_n)^{-1} (\widebar\B_{n,k+1})^\T \Q \,
\widebar\B_{n,k+1} (\Z^k_n)^{-1} \widebar\B_{n,k} \cdots (\Z^{m+1}_n)^{-1} 
\widebar\B_{n,m+1}. 
\end{equation}
To study this, it helps to recall the notation introduced in Lemma
\ref{lem:ffb_sobolev_vmat}: for a matrix $\A$ and positive integers $i,j$, let  
$$
\A(i,j) = 
\begin{cases}
\A_{ij} & \text{if $\A$ has at least $i$ rows and $j$ columns} \\
0 & \text{otherwise},
\end{cases}
$$
as well as 
\begin{align*}
\delta^r_{ij}(\A) &= \A(i,j)-\A(i+1,j), \\
\delta^c_{ij}(\A) &= \A(i,j)-\A(i,j+1).
\end{align*}
Now to compute \eqref{eq:ffb_sobolev_fqf2}, we first compute the product 
$$
\F^\T \Q  = (\widebar\B_{n,m+1})^\T (\Z^{m+1}_n)^{-1} \cdots  
(\widebar\B_{n,k})^\T (\Z^k_n)^{-1} (\widebar\B_{n,k+1})^\T \Q.
$$
We will work ``from right to left''. From \eqref{eq:diff_mat_ext}, we have
$$
(\widebar\B_{n,k+1})^\T = 
\left[\begin{array}{rrrrrrrrr}
1 & 0 & \ldots & 0 & \multicolumn{5}{c}{\multirow{4}{*}{0}} \\
0 & 1 & \ldots & 0 & \multicolumn{5}{c}{} \\
\vdots & & & & \multicolumn{5}{c}{} \\
0 & 0 & \ldots & 1 & \multicolumn{5}{c}{} \\
\multicolumn{3}{c}{\multirow{5}{*}{0}} & 1 & -1 & 0 & \ldots & 0 & 0 \\ 
\multicolumn{3}{c}{} & 0 & 1 & -1 & \ldots & 0 & 0 \\
\multicolumn{3}{c}{} & \vdots & & & & & \\
\multicolumn{3}{c}{} & 0 & 0 & 0 & \ldots & 1 & -1 \\
\multicolumn{3}{c}{} & 0 & 0 & 0 & \ldots & 0 & 1 
\end{array}\right]
\renewcommand\arraystretch{1.1}
\begin{array}{ll}
\left.\vphantom{\begin{array}{c} 1 \\ 0 \\ \cdots \\ 0 \end{array}}
\right\} & \hspace{-5pt} \text{$k$ rows} \\
\left.\vphantom{\begin{array}{c} 1 \\ 0 \\ \cdots \\ 0 \\ 0 \end{array}}
\right\} & \hspace{-5pt} \text{$n-k$ rows}
\end{array}
\renewcommand\arraystretch{1}
$$
This shows left multiplication by \smash{$(\widebar\B_{n,k+1})^\T$} gives
row-wise differences, \smash{$((\widebar\B_{n,k+1})^\T \A)_{ij} =
\delta^r_{ij}(\A)$}, for $i > k$. Further, from \eqref{eq:weight_mat_ext}, we
can see that left multiplication by \smash{$(\Z^k_n)^{-1}$} applies a row-wise 
scaling, \smash{$(\Z^k_n)^{-1} \A = \A_{ij} \cdot k/(x_i-x_{i-k})$}, for $i >
k$. Thus letting \smash{$\U^{1,0}=(\Z^k_n)^{-1} (\widebar\B_{n,k+1})^\T \Q$},
its entries are:
$$
\U^{1,0}_{ij} = 
\begin{cases}
\Q_{ij} & \text{if $i \leq k$} \vspace{3pt} \\
\displaystyle
\delta^r_{ij}(\Q) \cdot \frac{k}{x_i - x_{i-k}} & \text{if $i > k$}. 
\end{cases}
$$
The next two products to consider are left multiplication by
\smash{$(\widebar\B_{n,k})^\T$} and by \smash{$(\Z^{k-1}_n)^{-1}$}, which act 
similarly (they again produce row-wise differencing and scaling, respectively).
Continuing on in this same manner, we get that $\F^\T \Q = \U^{m,0}$,
where $\U^{\ell,0}$, $\ell=1,\ldots,m-1$ satisfy the recursion relation (setting 
$\U^{0,0}=\Q$ for convenience):  
\begin{equation}
\label{eq:ffb_sobolev_umat1}
\U^{\ell,0}_{ij} = 
\begin{cases}
\U^{\ell-1,0}_{ij} & \text{if $i \leq k+1-\ell$} \vspace{3pt} \\
\displaystyle
\delta^r_{ij}(\U^{\ell-1,0}) \cdot \frac{k+1-\ell}{x_i - x_{i-(k+1-\ell)}} & 
\text{if $i > k+1-\ell$},
\end{cases}
\end{equation}
and where (using $k+1-m=m$):
\begin{equation}
\label{eq:ffb_sobolev_umat2}
\U^{m,0}_{ij} = 
\begin{cases}
\U^{m-1,0}_{ij} & \text{if $i \leq m$} \\
\delta^r_{ij}(\U^{m-1,0}) & \text{if $i > m$}.
\end{cases}
\end{equation}
The expressions \eqref{eq:ffb_sobolev_umat1}, \eqref{eq:ffb_sobolev_umat2} are  
equivalent to \eqref{eq:ffb_sobolev_vmat1}, \eqref{eq:ffb_sobolev_vmat2}, the
row-wise recursion in Lemma \ref{lem:ffb_sobolev_vmat} (the main difference is
that Lemma \ref{lem:ffb_sobolev_vmat} is concerned with the lower-right $(n-m)
\times (n-m)$ submatrices of these matrices, and so these recursive expressions
are written with $i,j$ replaced by $i+m,j+m$, respectively).

The other half of computing \eqref{eq:ffb_sobolev_fqf2} is of course to compute
the product  
$$
\F^\T \Q \, \F = \F^\T \Q \, \widebar\B_{n,k+1} (\Z^k_n)^{-1} \widebar\B_{n,k}  
\cdots (\Z^{m+1}_n)^{-1} \widebar\B_{n,m+1}.
$$
Working now ``from left to right'', this calculation proceeds analogously to the
case just covered, but with column-wise instead of row-wise updates, and we get 
$\F^\T \Q \, \F = \U^{m,m}$, where $\U^{m,\ell}$, $\ell=1,\ldots,m-1$ satisfy
the recursion: 
\begin{equation}
\label{eq:ffb_sobolev_umat3}
\U^{m,\ell}_{ij} = 
\begin{cases}
\U^{m,\ell-1}_{ij} & \text{if $j \leq k+1-\ell$} \vspace{3pt} \\
\displaystyle
\delta^c_{ij}(\U^{m,\ell-1}) \cdot \frac{k+1-\ell}{x_j - x_{j-(k+1-\ell)}} & 
\text{if $j > k+1-\ell$},
\end{cases}
\end{equation}
and where: 
\begin{equation}
\label{eq:ffb_sobolev_umat4}
\U^{m,m}_{ij} = 
\begin{cases}
\U^{m,m-1}_{ij} & \text{if $j \leq m$} \\
\delta^c_{ij}(\U^{m,m-1}) & \text{if $j > m$}.
\end{cases}
\end{equation}
Similarly, \eqref{eq:ffb_sobolev_umat3}, \eqref{eq:ffb_sobolev_umat4} are
equivalent to \eqref{eq:ffb_sobolev_vmat3}, \eqref{eq:ffb_sobolev_vmat4}, 
the column-wise recursion in in Lemma \ref{lem:ffb_sobolev_vmat} (again, the
difference is that Lemma \ref{lem:ffb_sobolev_vmat} is written in terms of the
lower-right $(n-m) \times (n-m)$ submatrices). This establishes the result in
Lemma \ref{lem:ffb_sobolev_vmat}. 
 
\paragraph{Exchanging the order of scaled differencing with integration and
  differentiation.}    

Now that we have shown how to explicitly write the entries of $\F^\T \Q \,
\F$ via recursion, it remains to prove bandedness. To this end, for each $x \in
[a,b]$, define $\Q^x \in \R^{n \times n}$ to have entries \smash{$\Q^x_{ij} =
  (D^m h^k_i)(x) (D^m h^k_j)(x)$}, and note that by linearity of integration,
$$
\F^\T \Q \, \F = \int_a^b \F^\T \Q^x \, \F \, dx,
$$
where the integral on the right-hand side above is meant to be interpreted
elementwise. Furthermore, defining $a^x \in \R^n$ to have entries 
\smash{$a^x_i = (D^m h^k_i)(x)$}, we have \smash{$\Q^x = a^x (a^x)^\T$},
and defining $b^x \in \R^n$ to have entries \smash{$b^x_i = h^k_i(x)$}, note
that by linearity of differentiation, 
$$
\F^\T a^x = D^m \F^\T b^x, 
$$
where again the derivative on the right-hand side is meant to be interpreted
elementwise. This means that
$$
\F^\T \Q^x \, \F = (D^m \F^\T b^x) (D^m \F^\T b^x)^\T.
$$

By the same logic as that given above (see the development of
\eqref{eq:ffb_sobolev_umat1}, \eqref{eq:ffb_sobolev_umat2}), we can view $\F^\T 
b^x$ as the endpoint of an $m$-step recursion. First initialize $u^{x,0}=b^x$,
and define for $\ell=1,\ldots,m-1$,  
\begin{equation}
\label{eq:ffb_sobolev_uvec1}
u^{x,\ell}_i = 
\begin{cases}
u^{x,\ell-1}_i & \text{if $i \leq k+1-\ell$} \vspace{3pt} \\
\displaystyle
(u^{x,\ell-1}_i - u^{x,\ell-1}_{i+1}) \cdot \frac{k+1-\ell}{x_i -
  x_{i-(k+1-\ell)}} & \text{if $i > k+1-\ell$},   
\end{cases}
\end{equation}
as well as
\begin{equation}
\label{eq:ffb_sobolev_uvec2}
u^{x,m}_i = 
\begin{cases}
u^{x,m-1}_i& \text{if $i \leq m$} \\
u^{x,m-1}_i - u^{x,m-1}_{i+1} & \text{if $i > m$}.
\end{cases}
\end{equation}
Here, we set \smash{$u^{x,\ell}_{n+1}=0$}, $\ell=1,\ldots,m$, for convenience.
Then as before, this recursion terminates at $u^{x,m}=\F^\T b^x$. 

In what follows, we will show that 
\begin{equation}
\label{eq:ffb_sobolev_banded}
(D^m u^{x,m}_i) (D^m u^{x,m}_j) = 0, 
\quad \text{for $x\in [a,b]$ and $|i-j|>m$}.  
\end{equation}
Clearly this would imply that $(\F^\T \Q^x \, \F)_{ij} = 0$ for $x \in
[a,b]$ and $|i-j|>m$, and so $(\F^\T \Q \, \F)_{ij} = 0$ for $|i-j|>m$;
focusing on the lower-right submatrix of dimension $(n-m) \times (n-m)$,
this would mean \smash{$(\G^\T \M \, \G)_{ij} = (\V^m_n)_{ij} = 0$} for
$|i-j|>m$, which is the claimed bandedness property of \smash{$\V^m_n$}. 

\paragraph{Proof of the bandedness property \eqref{eq:ffb_sobolev_banded} for
  $i > k+1, j > k+1$.}  

Consider $i > k+1$. At the first iteration of the recursion 
\eqref{eq:ffb_sobolev_uvec1}, \eqref{eq:ffb_sobolev_uvec2}, we get
\begin{equation}
\label{eq:ffb_sobolev_uvec3}
u^{x,1}_i = \big(h^k_i(x) - h^k_{i+1}(x)\big) \cdot \frac{k}{x_i-x_{i-k}},  
\end{equation}
where we set \smash{$h^k_{n+1}=0$} for notational convenience. Next we present     
a helpful lemma, which is an application of the elementary result in
Lemma \ref{lem:newton_poly_diff}, on differences of Newton polynomials (recall
this serves as the main driver behind the proof of Lemma
\ref{lem:ffb_lateral_rec}). Since \eqref{eq:ffb_diff_pp} is a direct consequence
of \eqref{eq:newton_poly_diff} (more specifically, a direct consequence of the
special case highlighted in \eqref{eq:newton_poly_diff0}), we state the lemma 
without proof. 

\begin{lemma}
\label{lem:ffb_diff_pp}
For any $k \geq 1$, the piecewise polynomials in the $k$th degree falling
factorial basis, given in the second line of \eqref{eq:ffb}, satisfy for each
$k+2 \leq i \leq n-1$,  
\begin{equation}
\label{eq:ffb_diff_pp}
h^k_i(x) - h^k_{i+1}(x) = h^{k-1}_i(x) \cdot \frac{x_i - x_{i-k}}{k}, \quad  
\text{for $x \notin (x_{i-1},x_i]$}.
\end{equation}
\end{lemma}


Fix $i \leq n-m$. Applying Lemma \ref{lem:ffb_diff_pp} to
\eqref{eq:ffb_sobolev_uvec3}, we see that for $x \notin (x_{i-1},x_i]$, we
have simply \smash{$u^{x,1}_i = h^{k-1}_i(x)$}. By the same argument, for $x
\notin (x_{i-1},x_{i+1}]$,    
\begin{align*}
u^{x,2}_i &= (u^{x,1}_i - u^{x,1}_{i+1}) 
\cdot \frac{k-1}{x_i - x_{i-(k-1)}} \\
&= \big(h^{k-1}_i(x) - h^{k-1}_{i+1}(x)\big) 
\cdot \frac{k-1}{x_i -  x_{i-(k-1)}} \\
&= h^{k-2}_i(x).
\end{align*}
Iterating this argument over \smash{$u^{x,\ell}_i$}, $\ell=3,\ldots,m$, we get
that for $x \notin (x_{i-1},x_{i+m-1}]$,  
\begin{align*}
u^{x,m}_i &= u^{x,m-1}_i - u^{x,m-1}_{i+1} \\
&= h^m_i(x) - h^m_{i+1}(x) \\
&= h^{m-1}_i(x) \cdot \frac{x_i-x_{i-m}}{m}.
\end{align*}
As \smash{$h^{m-1}_i=0$} on $[a,x_{i-1}]$ and it is a polynomial of degree $m-1$
on $(x_{i-1},b]$, we therefore conclude that \smash{$D^m u^{x,m}_i =0$} for $x
\notin (x_{i-1},x_{i+m-1}]$. 

For $i \geq n-m+1$, note that we can still argue \smash{$u^{x,m}_i=0$} for $x 
\leq x_{i-1}$, as \smash{$u^{x,m}_i$} is just a linear combination of the 
evaluations \smash{$h^k_i(x),h^k_{i+1}(x),\ldots,h^k_n(x)$}, each of which are
zero. Thus, introducing the convenient notation \smash{$\bar{x}_i = x_i$}
for $i \leq n-1$ and \smash{$\bar{x}_i=b$} for $i \geq n$, we can still write 
\smash{$D^m u^{x,m}_i =0$} for \smash{$x \notin (x_{i-1},\bar{x}_{i+m-1}]$}. 

Putting this together, we see that for $i>k+1,j>k+1$, the product 
\smash{$(D^m u^{x,m}_i) (D^m u^{x,m}_j)$} can only be nonzero if 
\smash{$x \notin (x_{i-1},\bar{x}_{i+m-1}] \cap (x_{j-1},\bar{x}_{j+m-1}]$},
which can only happen (this intersection is only nonempty) if $|i-j| \leq  
m$. This proves \eqref{eq:ffb_sobolev_banded} for $i>k+1, j>k+1$.    

\paragraph{Proof of the bandedness property \eqref{eq:ffb_sobolev_banded} for
  $i \leq k+1, j > k+1$.}  

Consider $i=k+1$. At the first iteration of the recursion
\eqref{eq:ffb_sobolev_uvec1}, \eqref{eq:ffb_sobolev_uvec2}, we get 
\begin{equation}
\label{eq:ffb_sobolev_uvec4}
u^{x,1}_{k+1} = \big(h^k_{k+1}(x) - h^k_{k+2}(x)\big) \cdot
\frac{k}{x_{k+1}-x_1}.
\end{equation}
We give another helpful lemma, similar to Lemma \ref{lem:ffb_diff_pp}. As
\eqref{eq:ffb_diff_poly} is again a direct consequence of
\eqref{eq:newton_poly_diff} from Lemma \ref{lem:newton_poly_diff} (indeed a
direct consequence of the special case in \eqref{eq:newton_poly_diff0}), we
state the lemma without proof.

\begin{lemma}
\label{lem:ffb_diff_poly}
For any $k \geq 1$, the last of the pure polynomials and the first of
the piecewise polynomials in the $k$th degree falling factorial basis, given in
\eqref{eq:ffb}, satisfy  
\begin{equation}
\label{eq:ffb_diff_poly}
h^k_{k+1}(x) - h^k_{k+2}(x) = h^{k-1}_{k+1}(x) \cdot \frac{x_{k+1} - x_1}{k},
\quad \text{for $x > x_{k+1}$}.
\end{equation}
\end{lemma}

Applying Lemma \ref{lem:ffb_diff_poly} to \eqref{eq:ffb_sobolev_uvec4}, we see
that for $x > x_{k+1}$, it holds that \smash{$u^{x,2}_{k+1} = h^{k-1}_{k+1}(x)$}.
Combined with our insights from the recursion for the case $i>k+1$ developed
previously, at the next iteration we see that for $x > x_{k+2}$,   
\begin{align*}
u^{x,2}_{k+1} &= (u^{x,1}_{k+1} - u^{x,1}_{k+2}) 
\cdot \frac{k-1}{x_{k+1} - x_2} \\
&= \big(h^{k-1}_i(x) - h^{k-1}_{i+1}(x)\big) 
\cdot \frac{k-1}{x_{k+1} -  x_2} \\
&= h^{k-2}_{k+1}(x).
\end{align*}
Iterating this argument over \smash{$u^{x,\ell}_i$}, $\ell=3,\ldots,m$, we get
that for $x > x_{k+m}$,  
\begin{align*}
u^{x,m}_{k+1} &= u^{x,m-1}_{k+1} - u^{x,m-1}_{k+2} \\
&= h^m_{k+1}(x) - h^m_{k+2}(x) \\
&= h^{m-1}_{k+1}(x) \cdot \frac{x_{k+1}-x_{k+1-m}}{m}.
\end{align*}
and as before, we conclude that \smash{$D^m u^{x,m}_{k+1} =0$} for $x > 
x_{k+m}$. 

For $i<k+1$, the same argument applies, but just lagged by some number of
iterations (for $\ell=1,\ldots,k+1-i$, we stay at \smash{$u^{x,\ell}_i =
  h^k_i(x)$}, then for $\ell=k+2-i$, we get \smash{$u^{x,\ell}_i = (h^k_i(x) -
  h^k_{i+1}(x)) \cdot (i-1)/(x_i-x_1)$}, so Lemma \ref{lem:ffb_diff_poly} can 
be applied, and so forth), which leads us to \smash{$D^m u^{x,m}_i =0$} for $x >
x_{i+m-1}$.    

Finally, for $i \leq k+1$ and $|i-j|>m$, we examine the product \smash{$(D^m
u^{x,m}_i) (D^m u^{x,m}_j)$}. As $|i-j|>m$, we must have either $j<m$ or
$j>k+1$. For $j<m$, we have already shown $(\F^\T \Q\, \F)_{ij}=0$, and so for
our ultimate purpose (of establishing \eqref{eq:ffb_sobolev_banded} to establish
bandedness of $\F^\T \Q\, \F$), we only need to consider the case $j>k+1$. But
then (from our analysis in the last part) we know \smash{$(D^m u^{x,m}_j)=0$}
for $x \leq x_{j-1}$, whereas (from our analysis in the current part)
\smash{$(D^m u^{x,m}_i)=0$} for $x > x_{i+m-1}$, and since $x_{j-1}>x_{i+m-1}$,
we end up with \smash{$(D^m u^{x,m}_i) (D^m u^{x,m}_j) = 0$} for all $x$. This
establishes the desired property \eqref{eq:ffb_sobolev_banded} over all $i,j$,
and completes the proof of the theorem.

\subsection{Proof of Lemma \ref{lem:ffb_sobolev_mmat}}
\label{app:ffb_sobolev_mmat}

To avoid unnecessary indexing difficulties, we will work directly on the
entries of $\Q$, defined in \eqref{eq:ffb_sobolev_qmat}, and then we will be
able to read off the result for the entries of $\M$, defined in
\eqref{eq:ffb_sobolev_mmat}, by inspecting the lower-right submatrix of
dimension $(n-m) \times (n-m)$. Fix $i \geq j$, with $i>2m$. Applying 
integration by parts on each subinterval of $[a,b]$ in which the product
\smash{$(D^m h^k_i) (D^m h^k_j)$} is continuous, we get  
$$
\int_a^b (D^m h^k_i)(x) (D^m h^k_j)(x) \, dx = 
(D^m h^k_i)(x) (D^{m-1} h^k_j)(x) \Big|_{a,x_{j-1},x_{i-1}}^{x_{j-1},x_{i-1},b}  
- \int_a^b (D^{m+1} h^k_i)(x) (D^{m-1} h^k_j)(x) \, dx, 
$$
where we use the notation 
$$
f(x) \Big|_{a_1,\ldots,a_r}^{b_1,\ldots,b_r} = \sum_{i=1}^r
\big(f^-(b_i)-f^+(a_i)\big). 
$$
as well as \smash{$f^-(x) = \lim_{t \to x^-} f(t)$} and  
\smash{$f^+(x) = \lim_{t \to x^+} f(t)$}. As \smash{$h^k_i$} and \smash{$h^k_j$}
are supported on $(x_{i-1},b]$ and $(x_{j-1},b]$, respectively, so are there
derivatives, and as $x_{i-1} \geq x_{j-1}$ (since $i \geq j$) the second to last
display reduces to  
$$
\int_a^b (D^m h^k_i)(x) (D^m h^k_j)(x) \, dx = 
(D^m h^k_i)(x) (D^{m-1} h^k_j)(x) \Big|_{x_{i-1}}^b 
- \int_a^b (D^{m+1} h^k_i)(x) (D^{m-1} h^k_j)(x) \, dx, 
$$
Applying integration by parts $m-2$ more times (and using $k=2m-1$) yields 
\begin{align}
\nonumber
\int_a^b (D^m h^k_i)(x) &(D^m h^k_j)(x) \, dx \\
&= \sum_{\ell=1}^{m-1} (-1)^{\ell-1} (D^{m+\ell-1} h^k_i)(x) (D^{m-\ell}
  h^k_j)(x) \Big|_{x_{i-1}}^b + (-1)^{m-1} \int_a^b (D^k h^k_i)(x) (Dh^k_j)(x)
  \, dx \\    
\label{eq:ffb_sobolev_qmat2a}
&= \sum_{\ell=1}^{m-1} (-1)^{\ell-1} (D^{m+\ell-1} h^k_i)(x) (D^{m-\ell}
  h^k_j)(x) \Big|_{x_{i-1}}^b + (-1)^{m-1} \big(h^k_j(b) - h^k_j(x_{i-1})\big), 
\end{align}
where in the second line we used $(D^k h^k_i)(x) = 1\{x > x_{i-1}\}$ and the
fundamental theorem of calculus. The result for the case $i \leq 2m$ is 
similar, the only difference being that we apply integration by parts a total of
$i-m-1$ (rather than $m-1$ times), giving
\begin{equation}
\label{eq:ffb_sobolev_qmat2b}
\int_a^b (D^m h^k_i)(x) (D^m h^k_j)(x) \, dx 
= \sum_{\ell=1}^{i-m-1} (-1)^{\ell-1} (D^{m+\ell-1} h^k_i)(x) (D^{m-\ell}
  h^k_j)(x) \Big|_a^b + (-1)^{i-m-1} \big(h^k_j(b) - h^k_j(a)\big).
\end{equation}
Putting together \eqref{eq:ffb_sobolev_qmat2a}, \eqref{eq:ffb_sobolev_qmat2b}
establishes the desired result \eqref{eq:ffb_sobolev_mmat2} (recalling that
the latter is cast in terms of the lower-right $(n-m) \times
(n-m)$ submatrix of $\Q$, and is hence given by replacing $i,j$ with $i+m,j+m$,
respectively).   

\subsection{Proof of Lemma \ref{lem:tpb_bv_approx}}
\label{app:tpb_bv_approx}

For $k=0$ or $k=1$, we can use elementary piecewise constant 
or continous piecewise linear interpolation. For $k=0$, we set $g$ to be the
piecewise constant function that has knots in $x_{1:(n-1)}$, and
$g(x_i)=f(x_i)$, $i=1\ldots,n$; note clearly, $\TV(g) \leq \TV(f)$. For $k=1$,
we again set $g$ to be the continous piecewise linear function with knots in   
$x_{2:(n-1)}$, and $g(x_i)=f(x_i)$, $i=1\ldots,n$; still clearly, $\TV(Dg) \leq
\TV(Df)$. This proves \eqref{eq:tpb_bv_approx1}. 

For $k \geq 2$, we can appeal to well-known approximation results for $k$th
degree splines, for example, Theorem 6.20 of \citet{schumaker2007spline}.
First we construct a quasi-uniform partition from $x_{(k+1):(n-1)}$, call it
\smash{$x^*_{1:r} \subseteq x_{(k+1):(n-1)}$}, such that \smash{$\delta_n/2 \leq
  \max_{i=1,\ldots,r-1} \; (y_{i+1}-y_i) \leq 3\delta_n/2$}, and an extended
partition $y_{1:(r+2k+2)}$,   
$$
y_1 = \cdots = y_{k+1} = a, 
\quad y_{k+2} = x^*_1 < \cdots < y_{r+k+1} = x^*_r, \quad   
y_{r+k+2} = \cdots = y_{r+2k+2} = b.
$$
Now for each $\ell=k+1,\ldots,r+k+1$, define $I_\ell=[y_\ell,y_{\ell+1}]$ and 
\smash{$\bar{I}_\ell=[y_{\ell-k},y_{\ell+k+1}]$}. Then there exists a $k$th
degree spline $g$ with knots in \smash{$x^*_{1:r}$}, such that, for any
$d=0,\ldots,k$, and a constant $b_k>0$ that depends only on $k$,
\begin{equation}
\label{eq:local_approx}
\|D^d (f-g)\|_{L_\infty(\bar{I}_\ell)} \leq b_k \delta_n^{k-d} 
\omega(D^k f; \delta_n)_{L_\infty(\bar{I}_\ell)},
\end{equation}
Here \smash{$\|h\|_{L_\infty(I)}=\sup_{x\in I} \; |f(x)|$} denotes the
$L_\infty$ norm of a function $h$ an interval $I$, and 
$$
\omega(h; v) _{L_\infty(I)}=\sup_{x,y \in I, \, |x-y| \leq v} \; |h(x)-h(y)|
$$ 
denotes the modulus of continuity of $h$ on $I$. Note that \smash{$\omega(D^k
  f; \delta_n)_{L_\infty(\bar{I}_\ell)} \leq \TV(D^k f)$}. Thus setting $d=0$
in \eqref{eq:local_approx}, and taking a maximum over $\ell=k+1,\ldots,r+k+1$,
we get \smash{$\|f-g\|_{L_\infty} \leq b_k \delta_n^k \cdot \TV(D^k f)$}.
Further, the importance of the result in \eqref{eq:local_approx} is that it is
{\it local} and hence allows us to make statements about total variation as
well. Observe
\begin{align*}
\TV(D^k g) &= \sum_{i=k+2}^{r+k+2} |D^k g(y_i) - D^k g(y_{i-1})| \\
&\leq \sum_{i=k+2}^{r+k+2} \Big(|D^k f(y_i) - D^k g(y_i)| + 
  |D^k f(y_{i-1}) - D^k g(y_{i-1})| + |D^k f(y_i) - D^k f(y_{i-1})|\Big) \\
&\leq \underbrace{\big(2(k+2)b_k +1 \big)}_{a_k} \cdot \, \TV(D^k f),  
\end{align*}
In the last step above, we applied \eqref{eq:local_approx} with $d=k$, and the
fact that each interval \smash{$\bar{I}_\ell$} can contain at most $k+2$ of the
points $y_i$, $i=k+1,\ldots,r+k+2$. This proves  \eqref{eq:tpb_bv_approx2}. 

\subsection{Proof of Lemma \ref{lem:spec_sim_bound}}
\label{app:spec_sim_bound}

Observe that, by adding and subtracting $y$ and expanding, 
\begin{equation}
\label{eq:quad_sol_expand}
\|\htheta_a - \htheta_b\|_2^2 = 
(y - \htheta_a)^\T (\htheta_b - \htheta_a) + 
(y - \htheta_b)^\T (\htheta_a - \htheta_b).
\end{equation}
By the stationarity condition for problem \eqref{eq:quad_opt1}, we have
\smash{$y - \htheta_a = \lambda_a \A \htheta_a$}, so that 
\begin{align*}
(y - \htheta_a)^\T (\htheta_b - \htheta_a) 
&\leq \lambda_a \htheta_a^\T \A \htheta_b - 
\lambda_a \htheta_a^\T \A \htheta_a \\   
&\leq \frac{1}{2} \lambda_a \htheta_b^\T \A \htheta_b -   
\frac{1}{2} \lambda_a \htheta_a^\T \A \htheta_a,
\end{align*}
where in the second line we used the inequality $u^\T \A v \leq u^\T \A u / 2 + 
v^\T \A v / 2$. By the same logic, 
$$
(y - \htheta_b)^\T (\htheta_a - \htheta_b) \leq 
\frac{1}{2} \lambda_b \htheta_a^\T \B \htheta_a -  
\frac{1}{2} \lambda_b \htheta_b^\T \B \htheta_b.
$$
Applying the conclusion in the last two displays to \eqref{eq:quad_sol_expand}, 
\begin{align*}
\|\htheta_a - \htheta_b\|_2^2 
&\leq \frac{1}{2} \lambda_a \htheta_b^\T \A \htheta_b -  
\frac{1}{2} \lambda_a \htheta_a^\T \A \htheta_a +
\frac{1}{2} \lambda_b \htheta_a^\T \B \htheta_a -  
\frac{1}{2} \lambda_b \htheta_b^\T \B \htheta_b \\
&\leq \frac{1}{2} \sigma \lambda_a \htheta_b^\T \B \htheta_b -   
\frac{1}{2} \lambda_a \htheta_a^\T \A \htheta_a +
\frac{1}{2} (\lambda_b/\tau) \htheta_a^\T \A \htheta_a -  
\frac{1}{2} \lambda_b \htheta_b^\T \B \htheta_b, 
\end{align*}
where in the second line we twice used the spectral similarity property 
\eqref{eq:spec_sim}. The desired result follows by grouping terms. 

\subsection{Proof of Theorem \ref{thm:ss_bw_bound}}
\label{app:ss_bw_bound}

Note that 
\begin{align*}
\text{$\K^2_n, \W^2_n$ are $(\sigma,\tau)$-spectrally-similar} 
&\iff \text{$(\K^2_n)^{-1}, (\W^2_n)^{-1}$ are
 $(1/\sigma,1/\tau)$-spectrally-similar} \\ 
&\iff \text{$\W^2_n (\K^2_n)^{-1} \W^2_n, \W^2_n$ are
 $(1/\sigma,1/\tau)$-spectrally-similar}.
\end{align*}
Set \smash{$\A = \W^2_n (\K^2_n)^{-1} \W^2_n$}. From
\eqref{eq:nsp_sobolev_kmat_m2}, we can see that  
$$
\A_{ij} = 
\begin{cases}
\displaystyle
\frac{x_{i+2}-x_i}{3} & \text{if $i=j$} \\
\displaystyle
\frac{x_{i+1}-x_i}{6} & \text{if $i=j+1$}. 
\end{cases}
$$
Now define $a_i=(x_{i+2}-x_i)/3$ and $b_i=(x_{i+2}-x_{i+1})/6$, for
$i=1,\ldots,n-2$. Also denote $q_i=(x_{i+2}-x_i)/2$, for
$i=1,\ldots,n-2$. Fix $u \in \R^n$. For notational convenience, set
$b_0=u_0=0$ and $u_{n-1}=0$. Then 
\begin{align*}
u^\T \A u 
&= \sum_{i=1}^{n-2} \Big( a_iu_i^2 + b_{i-1}u_{i-1}u_i + b_iu_iu_{i+1} \Big) \\   
&\leq \sum_{i=1}^{n-2} \bigg( a_iu_i^2 + \frac{b_{i-1}}{2}(u_{i-1}^2 + u_i^2) +  
\frac{b_i}{2}(u_i^2 + u_{i+1}^2) \bigg) \\
&= \sum_{i=1}^{n-2} (a_i+b_{i-1}+b_i) u_i^2 \\
&= \sum_{i=1}^{n-2} q_i u_i^2 - \frac{x_2-x_1}{6} u_1^2 -
 \frac{x_{n-1}-x_{n-2}}{6} u_{n-2}^2 \\
&\leq \sum_{i=1}^{n-2} q_i u_i^2.
\end{align*}
In the second line above, we used $2st \leq s^2 + t^2$, and in the fourth we
used $a_i+b_{i-1}+b_i = q_i$, for $i=1,\ldots,n-2$. This shows that we can take 
$1/\tau = 1$, that is, $\tau=1$.  

As for the other direction, using $2st \geq -s^2 - t^2$, we have
\begin{align*}
u^\T Wu 
&\geq \sum_{i=1}^n\bigg( a_iu_i^2 - \frac{b_{i-1}}{2}(u_{i-1}^2+u_i^2) -
  \frac{b_i}{2}(u_i^2+u_{i+1}^2) \bigg) \\ 
&= \sum_{i=1}^{n-2} (a_i-b_{i-1}-b_i) u_i^2 \\
&= \frac{1}{2} \sum_{i=1}^{n-2} q_i u_i^2 + \frac{x_2-x_1}{6} u_1^2 +
 \frac{x_{n-1}-x_{n-2}}{6} u_{n-2}^2 \\
&\geq \frac{1}{3} \sum_{i=1}^{n-2} q_i u_i^2,
\end{align*}
where in the third line we used the fact that $a_i - b_{i-1} - b_i = q_i/3$, for
$i=1,\ldots,n-2$. This shows that we can take $1/\sigma=1/3$, that is,
$\sigma=3$, which completes the proof.

\subsection{Proof of Lemma \ref{lem:ffb_sobolev_vmat_m1}}
\label{app:ffb_sobolev_vmat_m1}

To keep indexing simple in the current case of $m=1$, we will compute the
entries of the matrix $\Q$ in \eqref{eq:ffb_sobolev_qmat}, then carry out
the recursion \eqref{eq:ffb_sobolev_umat1}--\eqref{eq:ffb_sobolev_umat4}, and
the desired matrix $\V_n$ will be given be reading off the lower-right $(n-1)
\times (n-1)$ submatrix of the result. Consider $i \geq j$. For $i \geq 3$, 
observe that      
\begin{align*}
\Q_{ij} &= \int_a^b (Dh^1_i)(x) (Dh^1_j)(x) \, dx \\
&= \int_a^b 1\{x > x_{i-1}\} \, dx \\
&= b-x_{i-1}.
\end{align*}
Meanwhile, for $i = 2$, by a similar calculation, $\Q_{ij}=b-a$. Therefore, 
introducing the convenient notation \smash{$\bar{x}_i = x_i$}
for $i \geq 3$ and \smash{$\bar{x}_i=a$} for $i = 2$, we get 
$$
\Q_{ij} = b - \bar{x}_{i-1},
$$
for all $i \geq 2$. We know that the result of the recursion in
\eqref{eq:ffb_sobolev_umat1}--\eqref{eq:ffb_sobolev_umat4} will be diagonal.
As $m=1$, this recursion reduces to simply \eqref{eq:ffb_sobolev_umat2},
\eqref{eq:ffb_sobolev_umat4}, which together give
\begin{align*}
\U^{1,1}_{ii} &= (\Q_{ii}-\Q_{i+1,i}) - (\Q_{i,i+1}-\Q_{i+1,i+1}) \\
&= \big((b - \bar{x}_{i-1}) - (b - \bar{x}_i)\big) - 
\big((b - \bar{x}_i) - (b - \bar{x}_i)\big) \\
&= \bar{x}_i - \bar{x}_{i-1}.
\end{align*}
This proves \eqref{eq:ffb_sobolev_vmat_m1} (recalling that this is written in
terms of \smash{$\V_n=\V^{1,1}$}, the lower-right $(n-1) \times (n-1)$ submatrix
of  \smash{$\U^{1,1}$}, and so for \eqref{eq:ffb_sobolev_vmat_m1} we simply
replace $i$ with $i+1$).   

\newpage
\section{B-splines and discrete B-splines}
\label{app:bs_dbs}

\subsection{B-splines}
\label{app:bs}

Though the truncated power basis \eqref{eq:tpb} is the simplest basis for
splines, the {\it B-spline basis} is just as fundamental, as it was ``there at
the very beginning'', appearing in Schoenberg's original paper on splines  
\citep{schoenberg1946contributions1}. Here we are quoting
\citet{deboor1976splines}, who gives a masterful survey of the history and
properties of B-splines (and points out that the name ``B-spline'' is derived
from Schoenberg's use of the term ``basic spline'', to further advocate for the
idea that B-splines can be seen as {\it the} basis for splines). A key feature
of B-splines is that they have local support, and are thus extremely useful for
computational purposes.

\paragraph{Peano representation.}

There are different ways to construct B-splines; here we cover a construction
based on what is called the {\it Peano representation} for B-splines (see, for 
example, Theorem 4.23 in \citet{schumaker2007spline}). If $f$ is a $k+1$ times
differentiable function $f$ on an interval $[a,b]$ (and its $(k+1)$st derivative
is integrable), then by Taylor expansion
$$
f(z) = \sum_{i=0}^k \frac{1}{i!} (D^i f)(a) (z-a)^i + 
\int_a^z \frac{1}{k!} (D^{k+1} f)(x) (z-x)^k \, dx.
$$
Note that we can rewrite this as
$$
f(z) = \sum_{i=0}^k \frac{1}{i!} (D^i f)(a) (z-a)^i + 
\int_a^b \frac{1}{k!} (D^{k+1} f)(x) (z-x)^k_+ \, dx. 
$$
Next we take a divided difference with respect to arbitrary centers 
$z_1,\ldots,z_{k+2} \in [a,b]$, where we assume without a loss of generality
that $z_1 < \cdots < z_{k+2}$. Then by linearity we can exchange divided
differentiation with integration, yielding 
\begin{equation}
\label{eq:peano}
k! \cdot f[z_1,\ldots,z_{k+2}] = \int_a^b (D^{k+1} f)(x)
\underbrace{(\cdot-x)^k_+[z_1,\ldots,z_{k+2}]}_{P^k(x; z_{1:(k+2)})} \, dx, 
\end{equation}
where we have also used the fact that a $(k+1)$st order divided difference (with
respect to any $k+2$ centers) of a $k$th degree polynomial is zero (for example,
see \eqref{eq:deriv_match_poly}), and lastly, we multiplied both sides by  
$k!$. To be clear, the notation \smash{$(\cdot - x)^k_+[z_1,\ldots,z_{k+2}]$}
means that we are taking the divided difference of the function \smash{$z
  \mapsto (z - x)^k_+$} with respect to centers $z_1,\ldots,z_{k+2}$.  

\paragraph{B-spline definition.}  

The result in \eqref{eq:peano} shows that the $(k+1)$st divided difference of
any (smooth enough) function $f$ can be written as a weighted average of
its $(k+1)$st derivative, in a local neighborhood around the corresponding
centers, where the weighting is given by a universal kernel \smash{$P^k(\cdot;
  z_{1:(k+2)})$} (that does not depend on $f$), which is called the {\it Peano 
  kernel} formulation for the B-spline; to be explicit, this is
\begin{equation}
\label{eq:bs_orig}
P^k(x; z_{1:(k+2)}) = (\cdot - x)^k_+[z_1,\ldots,z_{k+2}].
\end{equation}
Since 
$$
(z-x)^k_+ - (-1)^{k+1} (x-z)^k_+ =  (z-x)^k,
$$
and any $(k+1)$st order divided difference of the $k$th degree polynomial $z
\mapsto (z-x)^k$ is zero, we can rewrite the above \eqref{eq:bs_orig} as:  
\begin{equation}
\label{eq:bs}
P^k(x; z_{1:(k+2)}) = (-1)^{k+1} (x - \cdot)^k_+[z_1,\ldots,z_{k+2}].
\end{equation}
The function \smash{$P^k(\cdot; z_{1:(k+2)})$} is called a $k$th degree {\it
  B-spline} with knots $z_{1:(k+2)}$. It is a linear combination of $k$th
degree truncated power functions and is hence indeed a $k$th degree spline.  

It is often more convenient to deal with the {\it normalized B-spline}:
\begin{equation}
\label{eq:nbs}
M^k(x; z_{1:(k+2)}) = (-1)^{k+1} (z_{k+2}-z_1) 
(x - \cdot)^k_+[z_1,\ldots,z_{k+2}]. 
\end{equation}
It is easy to show that 
\begin{equation}
\label{eq:nbs_supp}
\text{$M^k(\cdot; z_{1:(k+2)})$ is supported on $[z_1,z_{k+2}]$, and  
$M^k(x; z_{1:(k+2)})>0$ for $x \in (z_1,z_{k+2})$}.  
\end{equation}
To see the support result, note that for $x > z_{k+2}$, we are taking a divided
difference of all zeros, which of course zero, and for $x < z_1$, we are taking 
a $(k+1)$st order divided difference of a polynomial of degree $k$, which is
again zero. To see the positivity result, we can, for example, appeal to 
induction on $k$ and the recursion to come later in \eqref{eq:nbs_rec}.  


\paragraph{B-spline basis.}  

To build a local basis for $\S^k(t_{1:r}, [a,b])$, the space of $k$th degree
splines with knots $t_{1:r}$, where we assume $a < t_1 < \cdots < t_r < b$, 
we first define boundary knots  
$$
t_{-k} < \cdots < t_{-1} < t_0 = a, \quad \text{and} \quad 
b = t_{r+1} < t_{r+2} < \cdots < t_{r+k+1}. 
$$
(Any such values for $t_{-k},\ldots,t_0$ and $t_{r+1},\ldots,t_{r+k+1}$ will
suffice to produce a basis; in fact, setting $t_{-k}=\cdots=t_0$ and
$t_{r+1}=\cdots=t_{r+k+1}$ would suffice, though this would require us to 
understand how to properly interpret divided differences with repeated centers;
as in Definition 2.49 of \citet{schumaker2007spline}.)  We then define the
normalized B-spline basis \smash{$M^k_j$}, $j=1,\ldots,r+k+1$ for  
$\S^k(t_{1:r},[a,b])$ by   
\begin{equation}
\label{eq:nbsb}
M^k_j = M^k(\cdot ; t_{(j-k-1):j}) \Big|_{[a,b]}, 
\quad j=1,\ldots,r+k+1. 
\end{equation}
It is clear that each \smash{$M^k_j$}, $j=1,\ldots,r+k+1$ is a $k$th degree    
spline with knots in $t_{1:r}$; hence to verify that they are a basis for
\smash{$\S^k(t_{1:r},[a,b])$}, we only need to show their linear independence, 
which is straightforward using the structure of their supports (for example, see  
Theorem 4.18 of \citet{schumaker2007spline}).

For concreteness, we note that the 0th degree normalized B-splines basis for
$\S^0(t_{1:r}, [a,b])$ is simply  
\begin{equation}
\label{eq:nbsb_k0}
M^0_j = 1_{I_j}, \quad j=1,\ldots,r+1.
\end{equation}
Here $I_0=[t_0,t_1]$ and $I_i=(t_i,t_{i+1}]$, $i=1,\ldots,r$, and we use
$t_{r+1}=b$ for notational convenience. We note that this particular choice for
the half-open intervals (left- versus right-side open) is arbitrary, but
consistent with our definition of the truncated power basis \eqref{eq:tpb}
when $k=0$. Figure \ref{fig:bs} shows example normalized B-splines of degrees 0 
through 3.   

\paragraph{Recursive formulation.} 

B-splines satisfy a recursion relation that can be seen directly from the
recursive nature of divided differences: for any $k \geq 1$ and centers $z_1 <
\cdots < z_{k+2}$,  
\begin{align*}
(x - \cdot)^k_+ [z_1,\ldots,z_{k+2}]
&= \frac{(x - \cdot)^k_+[z_2,\ldots,z_{k+2}] - 
(x -\cdot)^k_+[z_1,\ldots,z_{k+1}]}{z_{k+2} - z_1} \\
&= \frac{(x-z_{k+2})(x - \cdot)^{k-1}_+[z_2,\ldots,z_{k+2}] 
- (x-z_1) (x - \cdot)^{k-1}_+[z_1,\ldots,z_{k+1}] }{z_{k+2} - z_1},  
\end{align*} 
where in the second line we applied the Leibniz rule for divided differences 
(for example, Theorem 2.52 of \citet{schumaker2007spline}),
\smash{$fg [z_1,\ldots,z_{k+1}] = \sum_{i=1}^{k+1} f[z_1,\ldots,z_i] 
g[z_i,\ldots,z_{k+1}]$}, to conclude that 
\begin{align*}
(x - \cdot)^k_+ [z_1,\ldots,z_{k+1}] &= (x-z_1) \cdot 
(x - \cdot)^{k-1}_+ [z_1,\ldots,z_{k+1}] \\
(x - \cdot)^k_+ [z_2,\ldots,z_{k+2}] &= (x - \cdot)^{k-1}_+ 
  [z_2,\ldots,z_{k+2}] \cdot (x-z_{k+2}). 
\end{align*}
Translating the above recursion over to normalized B-splines, we get 
\begin{equation}
\label{eq:nbs_rec}
M^k(x; z_{1:(k+2)}) = \frac{x-z_1}{z_{k+1}-z_1} \cdot 
M^{k-1}(x; z_{1:(k+1)}) + \frac{z_{k+2}-x}{z_{k+2}-z_2} \cdot 
M^{k-1}(x; z_{2:(k+2)}),  
\end{equation}
which means that for the normalized basis, 
\begin{equation}
\label{eq:nbsb_rec}
M^k_j(x) = \frac{x-t_{j-k-1}}{t_{j-1}-t_{j-k-1}} \cdot
M^{k-1}_{j-1}(x) + \frac{t_j-x}{t_j-t_{j-k}} \cdot M^{k-1}_j(x), 
\quad j=1,\ldots,r+k+1.  
\end{equation}
Above, we naturally interpret \smash{$M^{k-1}_0 = M^{k-1}(\cdot; 
  t_{-k:0})|_{[a,b]}$} and \smash{$M^{k-1}_{r+k+1} = M^{k-1}(\cdot;   
  t_{(r+1):(r+k+1)})|_{[a,b]}$}. 

The above recursions are very important, both for verifying numerous properties
of B-splines and for computational purposes. In fact, many authors prefer to
use recursion to define a B-spline basis in the first place: they start with
\eqref{eq:nbsb_k0} for $k=0$, and then invoke \eqref{eq:nbsb_rec} for all  
$k \geq 1$. 

\subsection{Discrete B-splines}
\label{app:discrete_bs_even}

Here we will assume the design points are evenly-spaced, taking the form
$[a,b]_v=\{a,a+v,\ldots,b\}$ for $v>0$ and $b=a+Nv$. As covered in Chapter 8.5
of \citet{schumaker2007spline}, in this evenly-spaced case, discrete B-splines
can be developed in a similar fashion to B-splines. Below we will jump directly
into defining the discrete B-spline, which is at face value just a small 
variation on the definition of the usual B-spline given above. 
Chapter 8.5 of \citet{schumaker2007spline} develops several properties 
for discrete B-splines (for evenly-spaced design points)---such as a Peano
kernel result for the discrete B-spline, with respect to a discrete
integral---that we do not cover here, for simplicity.

\paragraph{Discrete B-spline definition.}

Let $z_{1:(k+2)} \subseteq [a,b]_v$. Assume without a loss of generality   
that $z_1 < \cdots < z_{k+2}$, and also $z_{k+2} \leq b-kv$. We define 
the $k$th degree {\it discrete B-spline} or DB-spline with knots
$z_1,\ldots,z_{k+2}$ by  
\begin{equation}
\label{eq:discrete_bs_even_orig}
U^k(x; z_{1:(k+2)}) = \Big((\cdot - x)^{k,v} \cdot 1\{\cdot > x\}\Big) 
[z_1,\ldots,z_{k+2}],   
\end{equation}
where now we denote by $(z)^{k,v} = z (z+v) \cdots (z+(k-1)v)$ the rising 
factorial polynomial of degree $k$ with gap $v$, which we take to be equal
to 1 when $k=0$. To be clear, the notation $((\cdot - x)^{k,v} \cdot 
1\{\cdot > x\}) [z_1,\ldots,z_{k+2}]$ means that we are taking the divided
difference of the function $z \mapsto (z - x)^{k,v} \cdot 1\{z>x\}$ with 
respect to the centers $z_1,\ldots,z_{k+2}$.
Since 
$$
(z-x)^{k,v} \cdot 1\{z>x\} - (-1)^{k+1} (x-z)_{k,v} \cdot 1\{x>z\}  =
(z-x)^{k,v}, 
$$
and any $(k+1)$st order divided difference of the $k$th degree polynomial $z
\mapsto (z-x)^{k,v}$ is zero, we can equivalently rewrite
\eqref{eq:discrete_bs_even_orig} as:    
\begin{equation}
\label{eq:discrete_bs_even}
U^k(x; z_{1:(k+2)}) = (-1)^{k+1} \Big((x - \cdot)_{k,v} \cdot 1\{x >
\cdot\}\Big) [z_1,\ldots,z_{k+2}].  
\end{equation}
We see \eqref{eq:discrete_bs_even} is just as in the usual B-spline definition 
\eqref{eq:bs}, but with a truncated falling factorial polynomial instead of a
truncated power function. Also, note \smash{$U^k(\cdot; z_{1:(k+2)})$} is a
linear combination of $k$th degree truncated falling factorial polynomials and
is hence a $k$th degree discrete spline. 

\begin{figure}[p]
\centering
\includegraphics[width=0.495\textwidth]{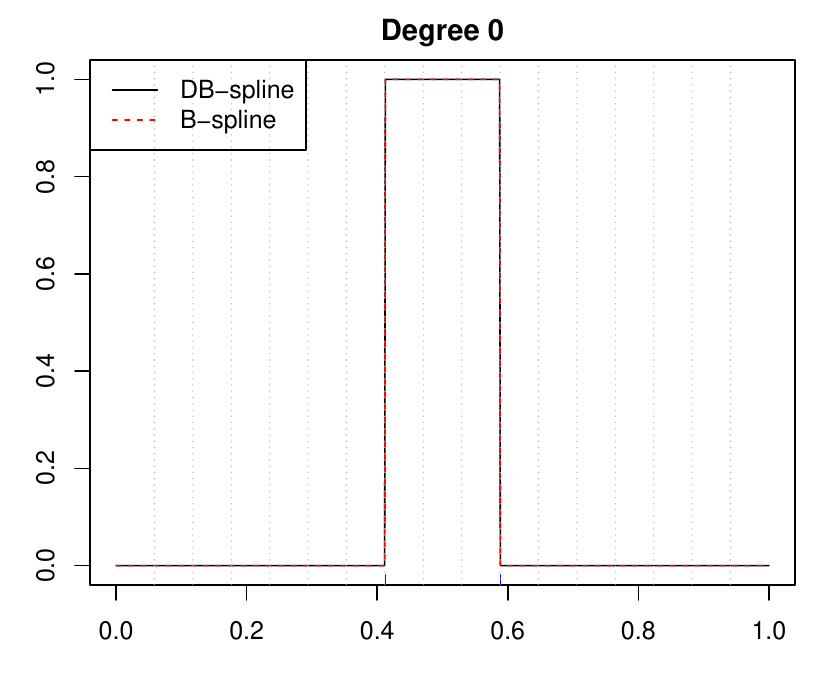} 
\includegraphics[width=0.495\textwidth]{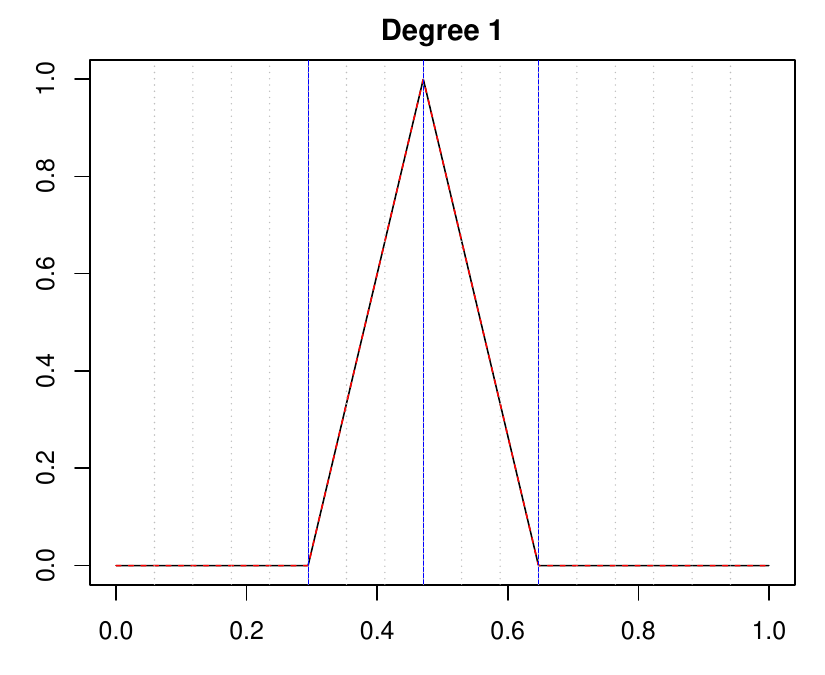} 
\includegraphics[width=0.495\textwidth]{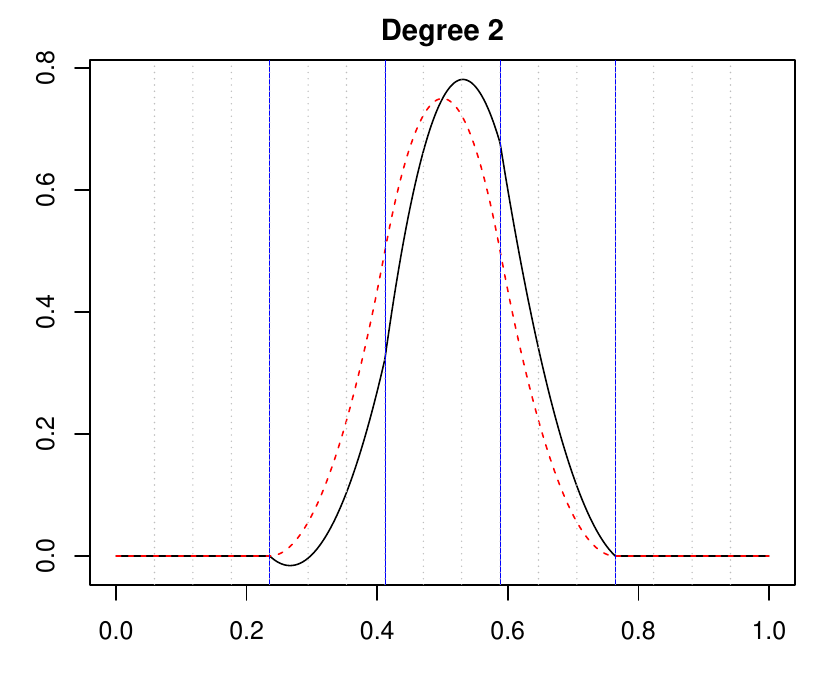} 
\includegraphics[width=0.495\textwidth]{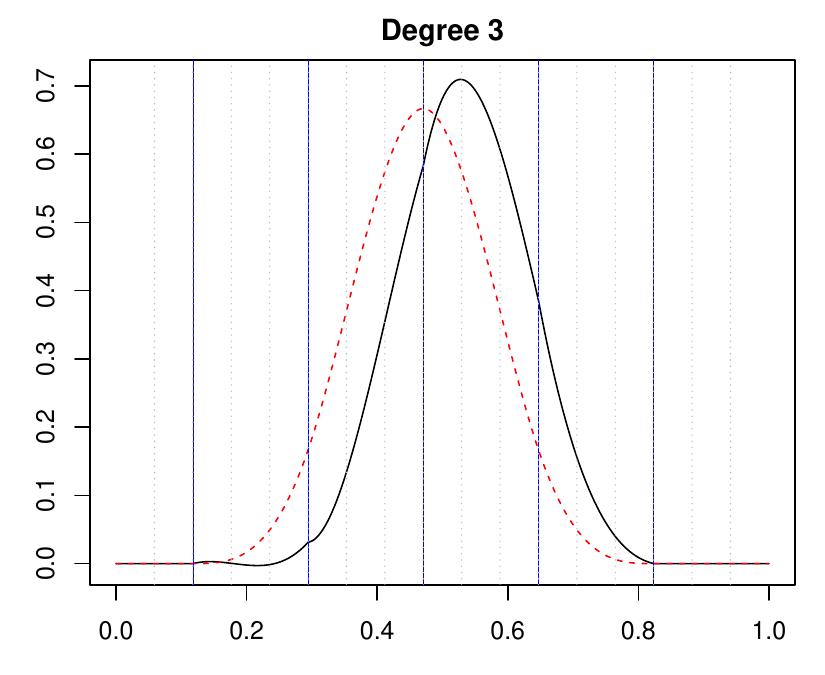} 
\includegraphics[width=0.495\textwidth]{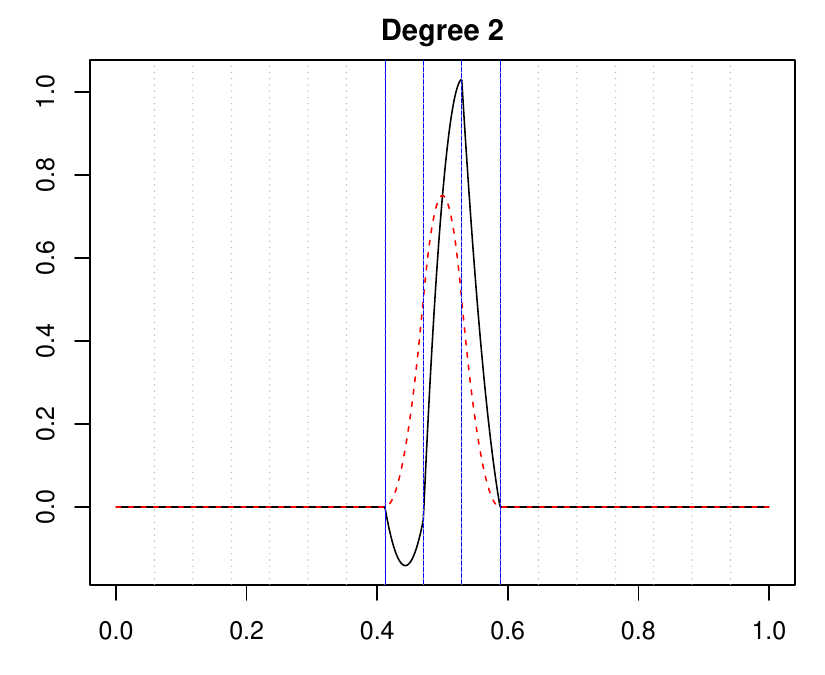} 
\includegraphics[width=0.495\textwidth]{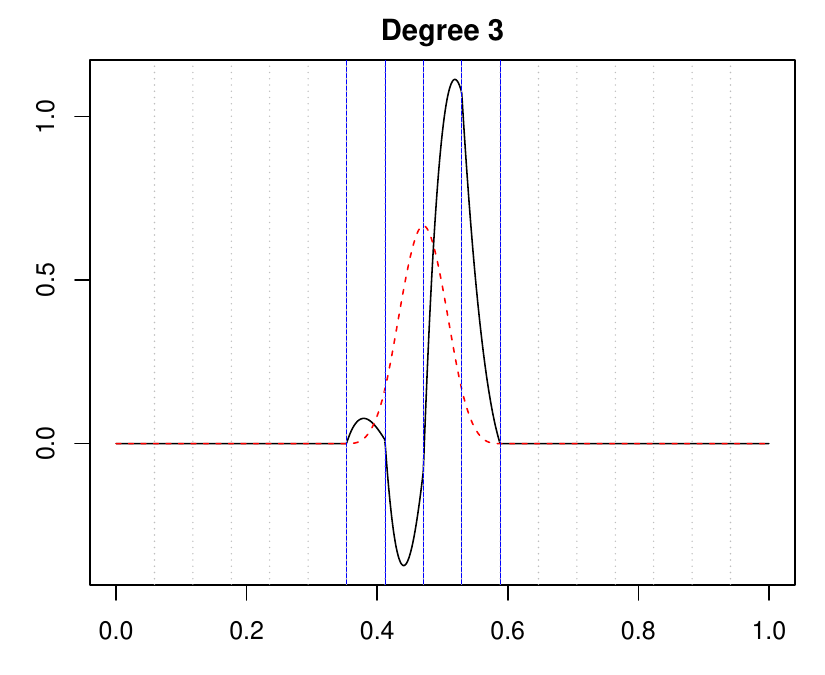} 
\caption{\small Normalized DB-splines in black, and normalized B-splines in
  dashed red, of degrees 0 through 3. In each example, the $n=16$ design points
  are evenly-spaced between 0 and 1, and marked by dotted vertical lines. The
  knot points are marked by blue vertical lines (except for $k=0$, as here these
  would obscure the B-splines, so in this case we use small blue ticks on the
  horizontal axis). In the bottom row, the knots are closer together; we can see
  that the DB-splines of degrees 2 and 3 have negative ``ripples'' near their
  leftmost knots, which is much more noticeable when the knot points are 
  closer together.}    
\label{fig:bs}
\end{figure}

As before, it is convenient to define the {\it normalized discrete B-spline} or
normalized DB-spline:
\begin{equation}
\label{eq:discrete_nbs_even}
V^k(x; z_{1:(k+2)}) = (-1)^{k+1} (z_{k+2}-z_1) \Big((x - \cdot)_{k,v}
\cdot 1\{x > \cdot\}\Big) [z_1,\ldots,z_{k+2}].  
\end{equation}
We must emphasize that
$$
V^k(x; z_{1:(k+2)}) = M^k(x; z_{1:(k+2)}) \quad 
\text{for $k=0$ or $k=1$} 
$$
(and the same for the unnormalized versions). This should not be a surprise, as
discrete splines are themselves exactly splines for degrees $k=0$ and
$k=1$. Back to a general degree $k \geq 0$, it is easy to show that    
\begin{equation}
\label{eq:discrete_nbs_even_supp}
\text{$V^k(\cdot; z_{1:(k+2)})$ is supported on $[z_1,z_{k+2}]$}. 
\end{equation}
Curiously, $V^k(\cdot; z_{1:(k+2)})$ is no longer positive on the whole interval
$(z_1,z_{k+2})$: for $k \geq 2$, it has a negative ``ripple'' close to the 
leftmost knot $z_1$. This is more pronounced when the knots are closer together
(separated by fewer design points), see Figure \ref{fig:bs}.

\paragraph{Discrete B-spline basis.}

To develop a local basis for \smash{$\DS^k_v(t_{1:r}, [a,b]_v)$}, the space of 
$k$th degree discrete splines with knots in $t_{1:r}$, where $a < t_1 < 
\cdots < t_r < b$, and also $t_{1:r} \subseteq [a,b]_v$ and $t_r \leq b-kv$, 
we first define boundary knots    
$$
t_{-k} < \cdots < t_{-1} < t_0 = a, \quad \text{and} \quad 
b = t_{r+1} < t_{r+2} < \cdots < t_{r+k+1},
$$
as before. We then define the normalized discrete B-spline basis
\smash{$V^k_j$}, $j=1,\ldots,r+k+1$ for \smash{$\DS^k_v(t_{1:r}, [a,b]_v)$}
by   
\begin{equation}
\label{discrete_nbsb_even}
V^k_j = V^k(\cdot; t_{(j-k-1):j}) \Big|_{[a,b]}, 
\quad j=1,\ldots,r+k+1. 
\end{equation}
It is clear that each \smash{$V^k_j$}, $j=1,\ldots,r+k+1$ is a $k$th
degree discrete spline with knots in $t_{1:r}$; hence to verify that they form a
basis for \smash{$\DS^k_n(t_{1:r}, [a,b])$}, we only need to show their linear 
independence, which follows from similar arguments to the result for the usual
B-splines (see also Theorem 8.55 of \citet{schumaker2007spline}). 

\paragraph{Recursive formulation.} 

To derive a recursion for discrete B-splines, we proceed as in the usual
B-spline case, using the recursion that underlies divided differences: for any
$k \geq 1$ and centers $z_1 < \cdots < z_{k+2}$ (such that $z_{1:(k+2)}
\subseteq [a,b]_v$ and $z_{k+2} \leq b-kv$), 
\begin{align*}
\Big((x - \cdot)_{k,v} \cdot 1\{x > \cdot\}\Big) [z_1,&\ldots,z_{k+2}] \\ 
&= \frac{((x - \cdot)_{k,v} \cdot 1\{x > \cdot\}) [z_2,\ldots,z_{k+2}] - 
((x - \cdot)_{k,v} \cdot 1\{x > \cdot\}) [z_1,\ldots,z_{k+1}]}{z_{k+2}-z_1} \\ 
&= \bigg((x-z_{k+2}-(k-1)v) \cdot ((x - \cdot)_{k-1,v} \cdot 1\{x > \cdot\})
  [z_2,\ldots,z_{k+2}] \\
&\quad -  (x-z_1-(k-1)v) \cdot ((x - \cdot)_{k-1,v} \cdot 1\{x > 
  \cdot\}) [z_1,\ldots,z_{k+1}]\bigg) / (z_{k+2}-z_1),
\end{align*}
where as before, in the second line, we applied the Leibniz rule for divided 
differences to conclude 
\begin{align*}
\Big((x - \cdot)_{k,v} \cdot 1\{x > \cdot\}\Big)[z_1,\ldots,z_{k+1}]
&= \big(x-z_1-(k-1)v\big) \cdot \Big((x - \cdot)_{k-1,v} \cdot 1\{x >
  \cdot\}\Big) [z_1,\ldots,z_{k+1}] \\
\Big((x - \cdot)_{k,v} \cdot 1\{x > \cdot\}\Big)[z_2,\ldots,z_{k+2}]
&= \Big((x - \cdot)_{k-1,v} \cdot 1\{x > \cdot\}\Big) [z_1,\ldots,z_{k+1}] 
  \cdot \big(x-z_{k+2}-(k-1)v\big).
\end{align*}
Translating the above recursion over normalized DB-splines, we get 
\begin{equation}
\label{eq:discrete_nbs_even_rec}
V^k(x; z_{1:(k+2)}) = \frac{x-z_1-(k-1)v}{z_{k+1}-z_1} \cdot
V^{k-1}(x; z_{1:(k+1)}) + \frac{z_{k+2}+(k-1)v - x}{z_{k+2}-z_2} \cdot
V^{k-1}(x; z_{2:(k+2)}), 
\end{equation}
which means that for the normalized basis, 
\begin{equation}
\label{eq:discrete_nbsb_even_rec}
V^k_j(x) = \frac{x-t_{j-k-1}-(k-1)v}{t_{j-1}-t_{j-k-1}} \cdot
V^{k-1}_{j-1}(x) + \frac{t_j+(k-1)v-x}{t_j-t_{j-k}} \cdot 
V^{k-1}_j(x), \quad j=1,\ldots,r+k+1.    
\end{equation}
Above, we naturally interpret \smash{$V^{k-1}_0 = V^{k-1}(\cdot;  
  t_{-k:0})|_{[a,b]}$} and \smash{$V^{k-1}_{r+k+1} = V^{k-1}(\cdot;   
  t_{(r+1):(r+k+1)})|_{[a,b]}$}. 

\newpage 
\section{Fast matrix multiplication}
\label{app:fast_mult}

We recall the details of the algorithms from \citet{wang2014falling} for fast 
multiplication by 
\smash{$\H^k_n, (\H^k_n)^{-1},(\H^k_n)^\T,(\H^k_n)^{-\T}$}, in Algorithms
\ref{alg:hk}--\ref{alg:hkt_inv}. In each case, multiplication takes $O(nk)$
operations (at most $4nk$ operations), and is done in-place (no new memory
required). We use $\cumsum$ to denote the cumulative sum operator,  
$\cumsum(v) = (v_1, v_1+v_2, \ldots,  v_1+\cdots+v_n)$, for $v \in \R^n$, and 
$\diff$ for the pairwise difference operator, $\diff(v) = (v_2-v_1, v_3-v_2,
\ldots, v_n-v_{n-1})$. We also use $\rev$ for the reverse operator, $\rev(v) =
(v_n,\ldots,v_1)$, and $\odot$ for elementwise multiplication between vectors.     

\begin{algorithm}[h]
\caption{Multiplication by $\H^k_n$}
\label{alg:hk}
\begin{algorithmic}
\STATE {\bf Input:} Integer degree $k \geq 0$, design points $x_{1:n}$ 
(assumed in sorted order), vector to be multiplied $v \in \R^n$. 
\STATE {\bf Output:} $v$ is overwritten by \smash{$\H^k_n v$}. 
\FOR {$i=k$ to $0$}
\STATE $v_{(i+1):n} = \cumsum(v_{(i+1):n})$
\IF {$i \neq 0$}
\STATE $v_{(i+1):n} = v_{(i+1):n} \odot \frac{x_{(i+1):n}-x_{1:(n-i)}}{i}$ 
\ENDIF
\ENDFOR
\STATE Return $v$.
\end{algorithmic}
\end{algorithm}

\vspace{-12pt}
\begin{algorithm}[h]
\caption{Multiplication by $(\H^k_n)^{-1}$}
\label{alg:hk_inv}
\begin{algorithmic}
\STATE {\bf Input:} Integer degree $k \geq 0$, design points $x_{1:n}$  
(assumed in sorted order), vector to be multiplied $v \in \R^n$. 
\STATE {\bf Output:} $v$ is overwritten by \smash{$(\H^k_n)^{-1} v$}. 
\FOR {$i=0$ to $k$}
\IF {$i \neq 0$}
\STATE $v_{(i+1):n} = v_{(i+1):n} \odot \frac{i}{x_{(i+1):n}-x_{1:(n-i)}}$
\ENDIF
\STATE $v_{(i+2):n} = \diff(v_{(i+1):n})$
\ENDFOR
\STATE Return $v$.
\end{algorithmic}
\end{algorithm}

\vspace{-12pt}
\begin{algorithm}[h]
\caption{Multiplication by $(\H^k_n)^\T$}
\label{alg:hkt}
\begin{algorithmic}
\STATE {\bf Input:} Integer degree $k \geq 0$, design points $x_{1:n}$  
(assumed in sorted order), vector to be multiplied $v \in \R^n$. 
\STATE {\bf Output:} $v$ is overwritten by \smash{$(\H^k_n)^{-1} v$}. 
\FOR {$i=0$ to $k$}
\IF {$i \neq 0$}
\STATE $v_{(i+1):n} = v_{(i+1):n} \odot \frac{x_{(i+1):n}-x_{1:(n-i)}}{i}$ 
\ENDIF
\STATE $v_{(i+1):n} x= \rev(\cumsum(\rev(v_{(i+1):n})))$
\ENDFOR
\STATE Return $v$.
\end{algorithmic}
\end{algorithm}

\vspace{-12pt}
\begin{algorithm}[h!]
\caption{Multiplication by $(\H^k_n)^{-\T}$}
\label{alg:hkt_inv}
\begin{algorithmic}
\STATE {\bf Input:} Integer degree $k \geq 0$, design points $x_{1:n}$  
(assumed in sorted order), vector to be multiplied $v \in \R^n$. 
\STATE {\bf Output:} $v$ is overwritten by \smash{$(\H^k_n)^{-\T} v$}. 
\FOR {$i=k$ to $0$}
\STATE $v_{(i+1):n-1}= \rev(\diff(\rev(v_{(i+1):n})))$
\IF {$i \neq 0$}
\STATE $v_{(i+1):n} \odot \frac{i}{x_{(i+1):n}-x_{1:(n-i)}}$
\ENDIF
\ENDFOR
\STATE Return $v$.
\end{algorithmic}
\end{algorithm}

\newpage
\RaggedRight 
\bibliographystyle{plainnat}
\bibliography{ryantibs}

\end{document}

%% file: ryantibs.tex
\usepackage[utf8]{inputenc} 
\usepackage[T1]{fontenc}    
\usepackage{booktabs}       
\usepackage{nicefrac}       
\usepackage{microtype}      
\usepackage{times}             

\usepackage[round]{natbib}
\usepackage{amssymb,amsmath,amsthm,bbm}
\usepackage[margin=1in]{geometry}
\usepackage{verbatim,float,url,dsfont}
\usepackage{graphicx,subfigure,psfrag}
\usepackage{algorithm,algorithmic}
\usepackage{mathtools,enumitem}
\usepackage[colorlinks=true,citecolor=blue,urlcolor=blue,linkcolor=blue]{hyperref}
\usepackage{multirow}

\newtheorem{theorem}{Theorem}
\newtheorem{lemma}{Lemma}
\newtheorem{corollary}{Corollary}

\theoremstyle{definition}
\newtheorem{remark}{Remark}
\newtheorem{definition}{Definition}

\newtheorem*{assumption*}{\assumptionnumber}
\providecommand{\assumptionnumber}{}


\newcommand{\argmin}{\mathop{\mathrm{argmin}}}

\newcommand{\minimize}{\mathop{\mathrm{minimize}}}
\newcommand{\st}{\mathop{\mathrm{subject\,\,to}}}

\def\R{\mathbb{R}}
\def\C{\mathbb{C}}
\def\E{\mathbb{E}}
\def\P{\mathbb{P}}

\def\col{\mathrm{col}}

\def\nul{\mathrm{null}}

\def\spa{\mathrm{span}}
\def\sign{\mathrm{sign}}

\def\diag{\mathrm{diag}}

\def\hy{\hat{y}}
\def\hf{\hat{f}}

\def\halpha{\hat{\alpha}}

\def\htheta{\hat{\theta}}

\def\cG{\mathcal{G}}

\def\cH{\mathcal{H}}

\def\cN{\mathcal{N}}

\def\cV{\mathcal{V}}
\def\cW{\mathcal{W}}
\def\cX{\mathcal{X}}

\makeatletter
\newcommand*\rel@kern[1]{\kern#1\dimexpr\macc@kerna}
\newcommand*\widebar[1]{%
  \begingroup
  \def\mathaccent##1##2{%
    \rel@kern{0.8}%
    \overline{\rel@kern{-0.8}\macc@nucleus\rel@kern{0.2}}%
    \rel@kern{-0.2}%
  }%
  \macc@depth\@ne
  \let\math@bgroup\@empty \let\math@egroup\macc@set@skewchar
  \mathsurround\z@ \frozen@everymath{\mathgroup\macc@group\relax}%
  \macc@set@skewchar\relax
  \let\mathaccentV\macc@nested@a
  \macc@nested@a\relax111{#1}%
  \endgroup
}
\makeatother